\documentclass[a4paper,twoside,11pt]{amsbook}

\usepackage{amsmath}
\usepackage{geometry}
\usepackage{amsthm}
\usepackage{amsfonts}
\usepackage{amssymb}
\usepackage{amsbsy}
\usepackage{graphics}
\usepackage{stmaryrd}
\usepackage{amsfonts}
\usepackage{setspace}
\usepackage{enumerate}
\usepackage{bm}
\usepackage[perpage]{footmisc}

\usepackage[normalem]{ulem} 

\newtheorem{theorem}{Theorem}[chapter]
\newtheorem{lemma}[theorem]{Lemma}

\theoremstyle{definition}
\newtheorem{definition}[theorem]{Definition}
\newtheorem{example}[theorem]{Example}

\theoremstyle{remark}
\newtheorem{remark}[theorem]{Remark}

\numberwithin{section}{chapter}
\numberwithin{equation}{chapter}

\newtheorem{proposition}[theorem]{Proposition}
\newtheorem{corollary}[theorem]{Corollary}
\newtheorem{conjecture}[theorem]{Conjecture}
\newtheorem{problem}[theorem]{Problem}

\newcommand{\R}{\mathbb{R}}
\newcommand{\C}{\mathbb{C}}
\newcommand{\Z}{\mathbb{Z}}
\newcommand{\N}{\mathbb{N}}
\newcommand{\tor}{\mathbb{T}} 
\newcommand\T{\mathbb T}  
\newcommand{\disk}{\mathbb{D}} 
\newcommand\D{\mathbb D}  
\newcommand{\bH}{\mathbb{H}} 
\newcommand\Ha{{\mathbb C}^+}
\newcommand{\Dom}{D} 

\newcommand{\cP}{{\rm\bf P}}  
\newcommand{\ID}{{\rm \bf ID}} 
\newcommand{\IA}{{\rm\bf IA}} 
\newcommand{\SD}{{\rm\bf SD}} 
\newcommand{\Univ}{{\rm\bf Univ}} 
\newcommand\UM{{\rm\bf UM}} 
\newcommand{\Star}{{\rm\bf Star}} 
\newcommand{\cH}{{\rm\bf H}} 
\newcommand{\LL}{{\rm\bf L}}     
\newcommand{\cL}{\mathcal{L}} 
\newcommand{\GGC}{{\rm\bf GGC}} 
\newcommand{\EGGC}{{\rm\bf EGGC}} 
\newcommand{\cR}{{\rm\bf Ray}} 

\renewcommand{\Re}{\text{\normalfont Re}} 
\renewcommand{\Im}{\text{\normalfont Im}} 
\renewcommand{\epsilon}{\varepsilon}

\newcommand\DF{\mathcal{D}} 
\newcommand{\DD}{D}  
\newcommand{\wto}{\overset{\rm w}{\rightarrow}} 
\newcommand{\KM}{M} 
\newcommand{\gammabm}{{\bm\gamma}} 
\newcommand\haar{\mathbf{h}} 
\newcommand\pk{\mathbf{pk}} 
\newcommand{\bmu}{\overline{\mu}} 

\newcommand{\bfb}{\mathbf{b}}     
\newcommand{\bfm}{\mathbf{m}}     
\newcommand{\bff}{\mathbf{f}}     
\newcommand\MP{\mathbf{mp}} 
\newcommand{\bfu}{\mathbf{u}} 
\newcommand{\AS}{\mathbf{as}} 

\def\submm{\circlearrowright}
\def\mmm{\circlearrowright}
\def\mm{\circlearrowright}

\makeatletter
  \def\mathcomposite{%
     \@ifstar
        {\def\@mathcomposite@option{%
            \baselineskip\z@skip\lineskiplimit-\maxdimen}%
         \@mathcomposite}%
        {\let\@mathcomposite@option\offinterlineskip
         \@mathcomposite}}
  \def\@mathcomposite{%
     \@ifnextchar[\@@mathcomposite{\@@mathcomposite[0]}}
  \def\@@mathcomposite[#1]#2#3#4{%
     #2{\mathchoice
        {\@mathcomposite@{#1}{#3}{#4}\displaystyle{1}}%
        {\@mathcomposite@{#1}{#3}{#4}\textstyle{1}}%
        {\@mathcomposite@{#1}{#3}{#4}%
         \scriptstyle\defaultscriptratio}%
        {\@mathcomposite@{#1}{#3}{#4}%
         \scriptscriptstyle\defaultscriptscriptratio}}}
  \def\@mathcomposite@#1#2#3#4#5{%
     \vcenter{\m@th\@mathcomposite@option
        \dimen@\f@size\p@\dimen@#1\dimen@\dimen@#5\dimen@
        \divide\dimen@ 18
        \edef\@mathcomposite@skipamount{\the\dimen@}%
        \ialign{\hfil$#4##$\hfil\cr
           #2\crcr
           \noalign{\vskip\@mathcomposite@skipamount}%
           #3\crcr}}}
  \makeatother

\newcommand{\utimes}{\kern-0.1ex\mathcomposite[-12]{\mathrel}{\cup}{\times}}                           
\newcommand{\hutimes}{\kern-0.1ex\mathcomposite[-13.2]{\mathrel}{\cup}{\times}\kern-0.1ex}    

\usepackage{imakeidx}
\makeindex

\begin{document}

\title{Monotone Increment Processes, Classical Markov Processes, and Loewner Chains}


\author{Uwe Franz}
\address{D\'epartement de math\'ematiques de Besan\c{c}on,
Universit\'e de Bourgogne Franche-Comt\'e, 16, route de Gray, F-25 030
Besan\c{c}on cedex, France.\newline
Email: uwe.franz@univ-fcomte.fr.\newline
 Website: http://lmb.univ-fcomte.fr/uwe-franz.}

\author{Takahiro Hasebe}
\address{Department of Mathematics, Hokkaido University, North 10, West 8, Kita-ku, Sapporo 060-0810, Japan.\newline
Email: thasebe@math.sci.hokudai.ac.jp.}

\author{Sebastian Schlei{\ss}inger}
\address{University of W\"urzburg, Emil-Fischer-Stra{\ss}e 40, 97074 W\"urzburg, Germany.\newline
 Email: sebastian.schleissinger@mathematik.uni-wuerzburg.de.}

\thanks{U.F.\ was supported by the French "Investissements d'Avenir" program, project ISITE-BFC (contract ANR-15-IDEX-03), by MAEDI/MENESR and JSPS through the SAKURA program, and by the ANR Project No. ANR-19-CE40-0002.\newline \newline
 T.H.\ was supported by JSPS Grant-in-Aid for Young Scientists (B) 15K17549, (A) 17H04823 and by JSPS and MAEDI Japan--France Integrated Action Program (SAKURA). \newline \newline
 S.S.\ was supported by the German Research Foundation (DFG), project no. 401281084.}



\date{\today}

\begin{abstract}
We prove one-to-one correspondences between certain decreasing Loewner chains in the upper half-plane, a special class of real-valued Markov processes, and quantum stochastic processes with
monotonically independent additive increments. This leads us to a detailed investigation of probability measures on $\R$ with univalent Cauchy transform. We discuss several subclasses of such measures and
obtain characterizations in terms of analytic and geometric properties of the corresponding Cauchy transforms.

Furthermore, we obtain analogous results for the setting of decreasing Loewner chains in the unit disk, which correspond to quantum stochastic processes of unitary operators with
monotonically independent multiplicative increments. 
\end{abstract}

\maketitle

\tableofcontents


%

\chapter*{Preface}

In non-commutative probability there exist several notions of independence of (non-commutative) random variables. For each notion of independence we also have a notion of (non-commutative) stochastic processes with independent increments. Here we focus on monotone independence, introduced by Muraki, and therefore on processes with monotonically independent increments, simply called monotone increment processes.

The marginal distributions of a monotone increment process, via the reciprocal\\ Cauchy transform, give rise to a decreasing Loewner chain, which is a family of univalent self-mappings of the upper half-plane with decreasing range and with some normalization at infinity.

Extending work of Biane, Letac-Malouche, and Franz-Muraki, we show how we can associate to each such Loewner chain a classical Markov process. The Markov processes that arise in this way are characterized by a special form of space-homogeneity.

Finally, given such a Markov process we can reconstruct the monotone increment process from it, as a non-commutative stochastic process consisting of (possibly unbounded) self-adjoint operators.

The above constructions give one-to-one correspondences between the three classes of objects: monotone increment processes, certain decreasing Loewner chains, and certain Markov processes. The Loewner chains we encounter have Denjoy-Wolff fixed points at infinity, and the classical theory of Loewner chains is not sufficient to treat them. Fortunately, recent work of Bracci, Contreras, D{\'{\i}}az-Madrigal, and Gumenyuk extended the theory of Loewner chains to the case where the Denjoy-Wolff fixed points of the mappings lie on the boundary, and this theory is quite well suited to our purpose.

We prove that the set of marginal distributions of monotone increment processes is exactly the set of probability measures with univalent Cauchy transform. Thus our bijection leads to a probabilistic interpretation of geometric function theory.

In geometric function theory, one important class of holomorphic functions is the set of starlike functions, i.e., univalent functions whose ranges are starlike domains with respect to some point. We identify the set of starlike Cauchy transforms (with respect to the origin) as the set of monotonically selfdecomposable distributions.  Furthermore, these distributions can be also characterized as the limit distributions in a limit theorem for the monotone convolution, in analogy to a classical limit theorem due to Paul L\'evy.

Similar results hold for unitary processes with monotonically independent multiplicative increments. The corresponding Loewner chains are radial and technically easier to treat. We can associate to unitary monotone increment processes classical Markov processes taking values in the unit circle, and obtain again one-to-one correspondences, this time between unitary monotone increment processes, a class of radial Loewner chains, and a class of Markov processes with values in the unit circle. We also discuss geometric properties and limit theorems for distributions on the unit circle and the multiplicative monotone convolution.\\

Uwe Franz, Takahiro Hasebe, Sebastian Schlei{\ss}inger



\chapter{Introduction}

\section{Quantum probability}

In quantum probability or non-commutative probability theory, random variables are regarded abstractly as elements of a unital *-algebra $\mathcal{A}$ over $\C$ together with a \emph{state} $\Phi$,
i.e. a linear functional $\Phi: \mathcal{A}\to \C$ with $\Phi(X^*X)\geq 0$ and $\Phi(1)=1,$ which corresponds to the classical expectation.
The pair $(\mathcal{A}, \Phi)$ is called an \index{quantum probability space!abstract}\emph{abstract quantum probability space}. An element $X \in \mathcal{A}$ is called a \index{random variable!abstract}\emph{random variable}.

\begin{example}\label{prel_ex_classical}Classical probability spaces fit into this setting as follows. Let $(\Omega, \mathcal{F}, \mathbb P)$ be a classical probability space. Then $\mathcal{A}=\bigcap_{1\leq p<\infty}L^p(\Omega, \C)$ is a unital *-algebra with * defined by $X^*(\omega):= \overline{X(\omega)}$, and
 $\Phi(X)=\int_\Omega X(\omega) {\rm d} \mathbb P(\omega)$ defines a state on $\mathcal{A}.$
\end{example}

\begin{example}\label{prel_ex_matrix} For $n\in\N$, let $\mathcal{A}=M_n(\C)$ be the *-algebra of complex $n\times n$ matrices, with * being the conjugate transpose, and let $\Phi(X)=\frac1{n}\text{Tr}(X)$. Then $(\mathcal{A}, \Phi)$ is an abstract quantum probability space.
\end{example}

It is now of interest to translate common notions from classical probability theory into this non-commutative setting. For instance, the distribution of  a random variable $X\in \mathcal{A}$ can be defined abstractly as the set $\{\Phi((X^{\epsilon_1})^{k_1}  \cdots (X^{\epsilon_n})^{k_n}) \mid \epsilon_i \in \{1,*\}, k_i \in \N, n\in\N\}$, called the *-moments of $X$. In particular, $\Phi(X^n)$ is called the \emph{$n$th moment} of $X\in \mathcal{A}$. If $\mathcal{A}$ is a $C^*$-algebra and $X$ is self-adjoint (namely $X=X^*$), then its moments $(\Phi(X^n))_{n\in\N}$ define a unique probability measure $\mu$ on $\R$ via
\begin{equation}\label{eq_moments}
\Phi(X^n)= \int_\R x^n \,\mu({\rm d}x), \qquad n \in \N.
\end{equation}
The measure $\mu$ is called the \emph{distribution} of $X$. As $X$ is an element of a $C^*$-algebra, the distribution $\mu$ of $X$ has compact support.

For our purposes, it will be mostly sufficient (the exception is Section \ref{alternative_constructions}) to work with the following concrete and most prominent example of a quantum probability space.
\begin{definition}\label{def_qu_space}
 A  \index{quantum probability space!concrete}\emph{(concrete) quantum probability space} $(H,\xi)$ consists of a Hilbert space $H$ and a unit vector $\xi\in H,$ which defines the vector state
 $\Phi_\xi:B(H)\to\mathbb{C}$ via \[\Phi_\xi(X)=\langle\xi, X\xi\rangle.\]
 Here, $B(H)$ denotes the space of all bounded linear operators on $H$ and we use inner products which are linear in the  second argument.
\end{definition}

Under some mild condition, e.g., if for each $X\in\mathcal{A}$ there exists a constant $C_X>0$ such that
\[
\forall Y\in \mathcal{A},\, \Phi(Y^*X^*XY) \le C_X \Phi(Y^*Y),
\]
we can realize an abstract quantum probability space $(\mathcal{A},\Phi)$ as a subalgebra of $B(H)$, acting on a concrete quantum probability space $(H,\xi)$. This is an immediate consequence of the GNS representation theorem. Even without any such a condition we can realize any abstract quantum probability space as an algebra of possibly unbounded linear operators acting on a pre-Hilbert space.

One advantage of using concrete quantum probability spaces is that we can include some unbounded operators as random variables.

\newpage

\begin{definition}\label{def_random_var_extended}${}$ Let $(H,\xi)$ be a concrete quantum probability space.
\begin{enumerate}[\rm(1)]
\item A \index{random variable!normal}\emph{normal random variable} is a densely defined closed operator $X$ such that $X X^* = X^* X$,  i.e.,
\begin{eqnarray*}
\Dom(XX^*)&:=&\{v\in \Dom(X^*): X^*v\in \Dom(X)\} \\
&=& \{v\in \Dom(X): Xv\in \Dom(X^*)\}=:\Dom(X^*X),
\end{eqnarray*}
and $XX^*$ and $X^*X$ agree on their domain.

\item In particular, if $X$ is self-adjoint/unitary, then we call it a \index{random variable!self-adjoint}\index{random variable!unitary}\emph{self-adjoint/unitary random variable}.

\item If $X$ is an essentially self-adjoint operator, then we call it an \index{random variable!self-adjoint}\index{random variable!essentially self-adjoint}\emph{essentially self-adjoint random variable}.
\end{enumerate}
\end{definition}
\begin{example}\label{ex:multiplication_op} Let $(\Omega, \mathcal{F}, \mathbb P)$ be a classical probability space. Then $(L^2(\Omega, \mathcal{F}, \mathbb P), \mathbf1_\Omega)$ is a concrete quantum probability space, where $\mathbf1_\Omega$ is the constant function on $\Omega$ taking the value $1$. 
If $f:\Omega\to\C$ is an $\mathcal{F}$-measurable function, then the multiplication operator $X\colon 
h\mapsto fh$ defined for $h$ in the dense domain
$$
\{h \in L^2(\Omega, \mathcal{F}, \mathbb P): fh \in L^2(\Omega, \mathcal{F}, \mathbb P)\}
$$
is a normal random variable. If $f$ takes only real values, then the random variable $X$ is self-adjoint.
\end{example}

Our random variables $X$ will be possibly unbounded operators, and so the domain of $X^n$ may not contain $\xi$ for some $n$. Thus we cannot define the distribution of $X$ by \eqref{eq_moments} in this case. We can generalize the definition by using resolvents and the Cauchy transform.

\begin{definition}\label{def_random_variable}
Let $X$ be an essentially self-adjoint random variable on a concrete quantum probability space $(H,\xi)$ and consider its closure $\overline{X}$.
The \index{distribution!of a self-adjoint random variable}distribution of $X$ is the unique probability measure $\mu$ on $\R$ such that
\begin{equation}\label{eq:Cauchy}
\Phi_\xi((z-\overline{X})^{-1}) = \int_{\R}\frac{1}{z-x}  \,\mu({\rm d}x) =: G_\mu(z), \qquad z \in \C^+:=\{w\in\C: \Im(w)>0\}.
\end{equation}
The function $G_\mu(z)$, which will also be denoted by $G_X$, is called the \index{Cauchy transform}\emph{Cauchy transform} of $\mu$ or of $X$. The \index{F-transform@$F$-transform}$F$-transform (or reciprocal Cauchy transform) of $\mu$ or of $X$ (denoted by $F_\mu$ or $F_X$) is defined to be the inverse of the Cauchy transform, i.e. as the mapping
\[
F_\mu\colon\Ha\to \Ha, \quad F_\mu(z) = \frac{1}{G_\mu(z)}.
\]
\end{definition}

\begin{remark} If $X$ is self-adjoint and bounded, then the distribution $\mu$ of $X$, as defined above, is indeed
the unique probability measure on $\R$ with moments $\Phi_\xi(X^n)$ as previously defined in \eqref{eq_moments}, because
\[
\Phi_\xi((z-X)^{-1}) = \frac1{z}\Phi_\xi((I-X/z)^{-1}) = \sum_{k=0}^\infty \frac{\Phi_\xi(X^k)}{z^{k+1}}
\]
and
\[
G_\mu(z) = \sum_{k=0}^\infty \frac{\int_\R x^k \,\mu({\rm d}x)}{z^{k+1}}
\]
for all $z\in\C$ with $|z|$ large enough (in fact $|z|> \|X\|$).
\end{remark}

\begin{example} In the setting of Example \ref{ex:multiplication_op}, if $f$ is real-valued, then the distribution of the multiplication operator $X$ is exactly the distribution of $f$ in the usual sense of probability theory.
\end{example}

For the basics of quantum probability we refer the reader to introductions such as \cite{Att,  meyer}.

\section{Monotone independence}\label{sec:monotone_independence}

Muraki has shown in \cite{MR2016316} that
the tensor, Boolean, free, monotone, and anti-monotone independences are the only five possible universal notions of independence in non-commutative probability theory.
We study monotone independence in this paper. This independence was introduced by Muraki \cite{M00,M01,Mur01b} based on earlier work on  monotone Fock spaces \cite{MR1467953, MR1462227,MR1483010, MR1455615}.

In what follows we denote by $C_b(S)$ the set of all continuous and bounded functions $f\colon S\to\C$, where $S$ is a topological space. For a normal random variable $X$ and $f\in C_b(\C)$, $f(X)$ is defined via functional calculus. If $X$ is self-adjoint, we can define $f(X)$ in the same way for $f\in C_b(\R)$.

\begin{definition}${}$\label{def-mon} Let $(H, \xi)$ be a concrete quantum probability space.
\begin{enumerate}[\rm(1)]
\item A family of *-subalgebras $(\mathcal{A}_\iota)_{\iota\in I}$ of $B(H)$  indexed by a linearly ordered set $I$
is called \emph{monotonically independent} if the following two conditions are satisfied.
\begin{description}
\item[(i)] \label{monotone_indep1}
For any $r,s\in\mathbb{N}\cup\{0\}$, $i_1,\ldots,i_r,j,k_1\ldots,k_s\in I$ with
\[
i_1>\cdots>i_r>j < k_s<\cdots<k_1\footnote{If $r=0$, then we just assume $j<k_s < \cdots <k_1$, and similarly for the case $s=0$.}
\]
and for any $X_1\in\mathcal{A}_{i_1},\ldots,X_r\in\mathcal{A}_{i_r}$, $Y\in\mathcal{A}_j$, $Z_1\in\mathcal{A}_{k_1},\ldots,Z_s\in\mathcal{A}_{k_s}$, we have
\[
\Phi_\xi(X_1\cdots X_r Y Z_s\cdots Z_1)=\Phi_\xi(X_1)\cdots \Phi_\xi(X_r)\Phi_\xi(Y) \Phi_\xi(Z_s)\cdots \Phi_\xi(Z_1).
\]
\item[(ii)]
For any $i,j,k\in I$ with $i<j>k$ and any $X\in\mathcal{A}_i$, $Y\in\mathcal{A}_j$, $Z\in\mathcal{A}_k$ we have
\[
XYZ = \Phi_\xi(Y) XZ.
\]
\end{description}
\item\label{enu:monotone_independence2} A family $(X_\iota)_{\iota\in I}$ of normal random variables indexed by a linearly ordered set $I$ is called \emph{monotonically independent} if the family
$(\mathcal A_\iota)_{\iota \in I}$  of *-algebras is monotonically independent, where
\[
\mathcal A_\iota =\{ f(X_{\iota})\mid f\in C_b(\C), f(0)=0\}.
\]
\end{enumerate}
\end{definition}

\begin{remark}
The following definition of monotone independence is also commonly used in the literature:
\begin{description}
\item[(iii)] For any $n\in\N$, $i_1, \dots, i_n \in I$ and any $X_1\in\mathcal{A}_{i_1},\ldots,X_n\in\mathcal{A}_{i_n}$, we have
\[
\Phi_\xi(X_1\cdots X_n) = \Phi_\xi(X_p)\Phi_\xi(X_1\cdots X_{p - 1}X_{p + 1}\cdots X_n)
\]
whenever $p$ is such that $i_{p - 1} < i_p > i_{p + 1}$, where the first or the last inequality is eliminated if $p = 1$ or $p = n$ respectively.
\end{description}

It can be checked that (i) and (ii) imply (iii). We prefer (i) and (ii) since our operator model satisfies these stronger conditions (see Theorem \ref{thm:additive_monotone_construction}). As noted in \cite[Remark 3.2 (c)]{franz07b}, the condition (iii) is equivalent to (i) and (ii) if the vacuum vector $\xi$ is cyclic regarding the algebra generated by $\mathcal{A}_i, i\in I$.
\end{remark}

\begin{remark}\label{rem:unit} Monotone (and anti-monotone) independence of two random variables is defined for ordered pairs $(X,Y)$, while tensor, free and Boolean independences do not need an order. Indeed, it is easy to see that $(X,I)$ is monotonically independent for all
random variables $X$, where $I\in B(H)$ denotes the identity. However, if $(I,X)$ is monotonically independent for $X\in B(H)$, then we have $X=IXI = \Phi_\xi(X)I,$ i.e. $X$ is a multiple of the identity. This also explains why we take functions $f \in C_b(\C)$ such that $f(0)=0$ in \eqref{enu:monotone_independence2}. If we remove the condition $f(0)=0$, then we can take $f \equiv 1$ and so $X_\iota$ must be multiples of the identity for all but the maximal index.
\end{remark}

Once a notion of independence of random variables is defined, one can introduce many concepts similar to those in probability theory: convolution of probability measures, central limit theorems,  quantum stochastic processes with independent increments, and quantum stochastic differential equations. For quantum independent increment processes, see the two books \cite{one, barndorff-nielsen+al}. We also refer to \cite{Oba17}, where the author shows how independences in quantum probability theory can be applied to the analysis of graphs. The different notions of independence appear in connection with certain products for graphs.

Assume that $(X,Y)$ is a pair of monotonically independent self-adjoint random variables on a concrete quantum probability space $(H,\xi)$ such that $X+Y$ is essentially self-adjoint. If $\mu$ and $\nu$ denote the distributions of $X$ and $Y$ respectively, then it can be shown that the distribution $\lambda$ of $X+Y$ can be computed
by
\begin{equation}\label{eq:monotone_convolution}
F_\lambda = F_\mu \circ F_\nu;
\end{equation}
see Lemma \ref{lem:mon-conv} below. 
Conversely, given two probability measures $\mu$ and $\nu$ on $\R$, one can always find monotonically independent self-adjoint operators $X$ and $Y$ with the distributions $\mu$  and $\mu$, respectively (e.g.\ use the operators in \cite[Proposition 3.9]{franz07b}).

Thus the formula \eqref{eq:monotone_convolution} defines the binary operation $\mu \rhd \nu := \lambda$, called the \index{monotone convolution!additive $\rhd$}\emph{(additive) monotone convolution} of probability measures $\mu$ and $\nu$ on $\R$.

\begin{remark} Monotone convolution was originally defined by Muraki in \cite{M00}. He first derived  formula \eqref{eq:monotone_convolution} for compactly supported probability measures by computing the moments of $(X+Y)^n$ when $X$ and $Y$ are monotonically independent bounded self-adjoint random variables \cite[Theorem 3.1]{M00}.  Then he extended the definition of monotone convolution to arbitrary probability measures via complex analysis \cite[Theorem 3.5]{M00}. Franz \cite{franz07b} constructed an unbounded self-adjoint operator model for monotone convolution of arbitrary probability measures as mentioned above.
\end{remark}

A \index{quantum stochastic process}non-commutative stochastic process or a quantum process is simply a family $(X_t)_{t\geq 0}$ of
random variables. In this work we study the following
monotone increment processes.

\begin{definition}\label{def_saip0}
Let $(H,\xi)$ be a concrete quantum probability space and $(X_t)_{t\ge 0}$ a family of essentially self-adjoint operators on $H$ with $X_0=0$.  We
call $(X_t)$ a \index{monotone increment process!additive (SAIP)}\emph{self-adjoint additive
monotone increment process (SAIP)} if the following conditions are satisfied.
\begin{itemize}
\item[(a)] The increment $X_t-X_s$ with domain $\Dom(X_t)\cap \Dom(X_s)$ is essentially self-adjoint for every $0\le s\le t$. 
	\item[(b)] $\Dom(X_s)\cap \Dom(X_t) \cap \Dom(X_u)$ is dense in $H$ and is a core for the increment $X_u-X_s$ for every $0\le s\le t \le u$. 
	\item[(c)] The mapping $(s,t)\mapsto \mu_{st}$ is continuous w.r.t.\ weak convergence, where $\mu_{st}$ denotes the distribution of the increment $X_t-X_s$.
	\item[(d)]The tuple
    \[
(X_{t_1},X_{t_2}-X_{t_1},\ldots,X_{t_n}-X_{t_{n-1}})
\]
is monotonically independent for all $n\in\mathbb{N}$ and all $t_1,\ldots,t_n\in\mathbb{R}$ s.t.\ $0\le t_1\le t_2\le\cdots\le t_n$.
\end{itemize}
Furthermore if $X_t-X_s$ has the same distribution as $X_{t-s}$ for all $0\leq s \leq t$ (the condition of \emph{stationary increments}), then $(X_t)_{t\geq0}$ is called a \index{monotone L\'evy process!additive}\emph{monotone L\'evy process.}
\end{definition}

\section{Summary of results}

The first goal is to establish one-to-one correspondences between SAIPs, some class of classical (in general time-inhomogeneous) Markov processes, and Loewner chains, motivated by or extending the past works  \cite{Bia98,franz+muraki04,franz07b,letac+malouche00,monotone}. We say that a probability kernel $k(x, \cdot),x \in \R$, is \emph{monotonically homogeneous} ($\rhd$-homogeneous, for short) if $\delta_x\rhd k(y, \cdot) = k(x+y,\cdot)$ for all $x,y \in \R$, and that a Markov process $(M_t)_{t\geq0}$ on $\R$ with transition kernels $\{k_{st}\}_{0\leq s\leq t}$ is \index{Markov process!$\rhd$-homogeneous}\emph{$\rhd$-homogeneous} if each $k_{st}$ is $\rhd$-homogeneous and the mapping $(s,t)\mapsto k_{st}(x,\cdot)$ is continuous w.r.t. weak convergence for every $x\in\R$.
{}From the complex analysis side, we call a decreasing Loewner chain $(F_t:\Ha\to\Ha)_{t\geq 0}$ in the upper half-plane $\Ha$ an \index{Loewner chain!additive}additive Loewner chain if $F_t'(\infty) =1$ in the sense of a non-tangential limit, see Definition \ref{EV_def:evolution_family}.

\newpage

The first main result of this paper can be summarized as follows.

\begin{theorem}\label{intro_thm0}
 We establish one-to-one correspondences between the following objects:

\begin{enumerate}[\rm(1)]
\item SAIPs $(X_t)_{t\ge 0}$ up to equivalence,
\item
additive Loewner chains $(F_{t})_{t\geq 0}$ in $\mathbb{C}^+$,
\item
real-valued $\rhd$-homogeneous Markov processes $(M_t)_{t\ge 0}$ with $M_0=0$ up to equivalence.
\end{enumerate}
\end{theorem}

The details of the correspondences in Theorem \ref{intro_thm0} are as follows.  If $(X_t)_{t\geq 0}$ is a SAIP, then the reciprocal Cauchy transforms $(F_{X_t})_{t\geq0}$ form an additive Loewner chain in $\Ha$. Given an additive Loewner chain $(F_t)_{t\geq0}$ in $\Ha$, the Markov transition kernels $(k_{st})_{0\leq s\leq t}$ defined by the identity 
\[ \int_{y\in\R} \frac{1}{z-y} k_{st}(x,{\rm d}y) =\frac{1}{F_s^{-1}\circ F_t(z)-x}, \quad z\in\Ha, x\in\R, \] determine a unique real-valued $\rhd$-homogeneous Markov process $(M_t)_{t\geq0}$.\\  Finally, if $(M_t)_{t\geq0}$ is a real-valued $\rhd$-homogeneous Markov process on $(\Omega, \mathcal{F}, \mathbb P)$ with filtration $(\mathcal{F}_t)_{t\geq0}$, then the non-commutative stochastic process $(X_t)_{t\geq0}$ defined by \[X_th=\mathbb E [M_th |\mathcal{F}_t]\] for $h\in L^2(\Omega, \mathcal{F}, \mathbb P)$ satisfying the condition $M_th\in L^2(\Omega, \mathcal{F}, \mathbb P)$ is a SAIP. Note that the conditional expectation $P_t=\mathbb E [\,\cdot\, |\mathcal{F}_t]$ is viewed as the orthogonal projection from $L^2(\Omega, \mathcal{F}, \mathbb P)$ onto the subspace $L^2(\Omega, \mathcal{F}_t,\mathbb P)$. Thus, with a slight abuse of notation, we may write $X_t=P_tM_t$ by viewing the function $M_t$ as a multiplication operator on
$L^2(\Omega, \mathcal{F}, \mathbb P)$ in the sense of Example \ref{ex:multiplication_op}.

\begin{remark}
In the literature constructions of SAIPs have been limited to the case of bounded operators. In \cite{MR1462227}, Muraki constructed a \index{Brownian motion!monotone}monotone Brownian motion, i.e. a SAIP $(X_t)_{t\geq0}$ where the distribution of $X_t -X_s$ is the \index{arcsine distribution}arcsine distribution with mean 0 and variance $t-s$.  More generally, monotone L\'evy processes consisting of bounded self-adjoint operators have been constructed in \cite[Theorem 4.1]{franz+muraki04}. Jekel \cite[Theorem 6.25]{jekel} constructed (operator-valued) bounded monotone increment processes on a monotone Fock space.  Our construction based on classical Markov processes is different from all of them and has the advantage that we can include any unbounded processes. Yet other constructions of monotone L\'evy processes with finite moments are discussed in Section  \ref{alternative_constructions}.
\end{remark}

The class of (in particular, stationary) $\rhd$-homogeneous Markov processes may be of independent interest in terms of probability theory, so we study their  further properties. We will prove that they have
\begin{itemize}
\item the Feller property,
\item an explicit formula for the infinitesimal generator,
\item a martingale property.
\end{itemize}

It is another remarkable fact that a probability measure $\mu$ can occur as marginal distribution of an SAIP iff its Cauchy transform $G_\mu=1/F_\mu$ is univalent.

\begin{theorem}\label{intro_thm2}Let $\mu$ be a probability measure on $\R$. The following statements are equivalent.
\begin{enumerate}[\rm(1)]
\item\label{intro_thm2-1} \index{F-transform@$F$-transform}\index{Cauchy transform!univalent}$F_\mu$ is univalent.

\item \label{intro_thm2-2}There exists a SAIP $(X_t)_{t\geq0}$ such that the distribution of $X_1$ is $\mu$.

\item\label{intro_thm2-3} There exists an additive Loewner chain $(F_t)_{t\geq0}$ in $\Ha$ such that $F_1 = F_\mu$.
\end{enumerate}
\end{theorem}

The idea of the proof of Theorem \ref{intro_thm2} is as follows. The equivalence between  \eqref{intro_thm2-2} and \eqref{intro_thm2-3} is a part of Theorem \ref{intro_thm0}. Under suitable Cayley transforms and a suitable time change, the Loewner chain $(F_t)$ can be  transformed into a Loewner chain on the unit disk that is differentiable regarding $t$ almost everywhere and satisfies Loewner's partial differential equation (Section \ref{subsec_LPDE}). Then we can use recent work on Loewner chains \cite{contreres+al2014} to prove the equivalence between \eqref{intro_thm2-1} and \eqref{intro_thm2-3}.

\begin{remark}
The equivalence between \eqref{intro_thm2-1} and \eqref{intro_thm2-2} of Theorem \ref{intro_thm2} is to be compared with classical probability (see \cite[Theorems 7.10 and 9.1]{Sat13}): given a stochastic process $(Y_t)_{t\geq0}$ with independent (not necessarily stationary) increments, $Y_0=0$ and suitable continuity properties, the distribution of $Y_1$ is infinitely divisible; conversely any infinitely divisible distribution can be realized as such. The same statement is true if we consider L\'evy processes (namely, if we assume stationary increments). However, as we will see in Example \ref{semicircle not ID}, there exists a probability measure $\mu$ which is not monotonically infinitely divisible but $F_\mu$ is univalent. Therefore there exists a gap between the laws of SAIPs and those of monotone L\'evy processes.
\end{remark}

\begin{remark}R. Bauer has studied univalent Cauchy transforms in \cite{Bau05} and he has also regarded Loewner's differential equation from a quantum probabilistic point of view, see \cite{bauer03} and \cite{bauer04}. The relation to monotone independence is also discussed in \cite{monotone}.
\end{remark}

The study of univalent functions is a classical subject of \emph{geometric function theory}, which investigates
analytic functions in terms of their geometric properties. From this point of view, it is interesting to classify
univalent functions via the geometry of their image domains.
The fact that every probability measure $\mu$ on $\R$ is uniquely determined by its $F$-transform (due to the Stieltjes-Perron inversion formula)  motivates the investigation of relations between properties of the measure $\mu$ and analytic/geometric properties of  $F_\mu$ (or of other transforms such as the Voiculescu transform). 

Theorem \ref{intro_thm2} shows that univalence of $F$-transforms can indeed be interpreted in a probabilistic way,
which leads us to the question whether the many well-known subclasses of univalent functions\footnote{Such as convex, starlike, and spirallike functions, slit mappings, mappings with quasiconformal extensions, etc.} have reasonable interpretations in the context of monotone independence. In Sections \ref{section_limits_additive} and \ref{sectionunitcircle}, we investigate some subclasses from this point of view: limits of infinitesimal arrays,
infinitely divisible distributions, unimodal distributions, and selfdecomposable distributions. We compare such classes to the analogues obtained
by switching to classical and free independence, and we give various illustrating examples.

\begin{definition}
A family $\{\mu_{n,j}\}_{1 \leq j \leq k_n, 1\leq n}$ of probability measures on $\R$ is called an \index{infinitesimal array!of measures on $\R$}infinitesimal array if $k_n \to \infty$ as $n\to \infty$ and for any $\delta>0$
\begin{equation*}
\lim_{n\to \infty}\sup_{1 \leq j \leq k_n} \mu_{n,j}([-\delta,\delta]^c)=0.
\end{equation*}
\end{definition}
The reason for requiring $k_n\to \infty$ is that if $\{k_n\}_{n\ge1}$ were bounded, then the limit distributions in Theorem  \ref{intro_thm3}\eqref{intro_thm3-1} below would be trivially $\delta_0$.

We first establish the following limit theorem with respect to monotone convolution.

\begin{theorem}[Theorem \ref{thminfinitesimalarray}]\label{intro_thm3} For a probability measure $\nu$ on $\R$, denote by $\sigma^2(\nu)$ its variance. 
\begin{enumerate}[\rm(1)]
\item\label{intro_thm3-1} If $\mu$ is a probability measure such that \index{F-transform@$F$-transform}\index{Cauchy transform!univalent}$F_\mu$ is univalent, then there exists an infinitesimal array $\{\mu_{n,j}\}_{1 \leq j \leq k_n, 1\leq n}$ such that
\begin{equation*}
\mu_{n,1} \rhd \mu_{n,2} \rhd \cdots \rhd \mu_{n,k_n}
\end{equation*}
 converges weakly to $\mu$ as $n\to\infty$.

\item\label{intro_thm3-2} If an infinitesimal array $\{\mu_{n,j}\}_{1 \leq j \leq k_n, 1\leq n}$ satisfies the variance condition
$$
\sup_{1\leq j\leq k_n}\sigma^2(\mu_{n,j}) \to 0 \quad \text{as~} n\to\infty
$$
and if $\mu_{n,1} \rhd \mu_{n,2} \rhd \cdots \rhd \mu_{n,k_n} $ converges weakly to a probability measure $\mu$, then $F_\mu$ is univalent.
\end{enumerate}
\end{theorem}
The proof of \eqref{intro_thm3-1} is just an application of Theorem \ref{intro_thm2}.  Furthermore, an analogous and more complete limit theorem for multiplicative monotone convolution on the unit circle is proved in Theorem \ref{Khintchineunit}.

The most interesting and complete result is the following analytic-geometric characterization of selfdecomposable distributions.
A probability measure $\mu$ on $\R$ is called \index{selfdecomposable distribution!additive monotone convolution}\emph{monotonically selfdecomposable} if for every $c\in (0,1)$ there exists a probability measure $\bmu^c$ on $\R$ such that
\begin{equation*}
\mu = (\DD_c\mu) \rhd \bmu^c,
\end{equation*}
where  $(\DD_c\mu)(A)=\mu(A/c)$ for Borel sets $A \subset \R$. We obtain the following characterization.

\begin{theorem}[Theorems \ref{msd2} and \ref{starlike1}] Let $\mu$ be a probability measure on $\R$. The following statements are equivalent.
\begin{enumerate}[\rm (a)]
 \item $\mu$ is monotonically selfdecomposable.
 \item $F_\mu$ is univalent and \index{F-transform@$F$-transform}\index{Cauchy transform!starlike}starlike w.r.t. $\infty$ in the sense that $c\cdot F_\mu(\Ha)\subset F_\mu(\Ha)$ for all $c\in (1,\infty).$
 \item
 $\Im\left(\frac{F'_\mu(z)}{F_\mu(z)}\right)\leq 0$ for all $z\in \Ha.$
 \item
 There exists a probability measure $\nu$ on $\R$ satisfying the integrability condition\\
$
\int_{\R} \log(1+|x|)\,\nu({\rm d}x)<\infty
$
such that
$$
F_\mu(z) = \exp\left(\int_\R \log(z-x) \, \nu({\rm d}x)\right), \qquad z \in \C^+.
$$
\end{enumerate}
\end{theorem}
The mapping $\nu \mapsto \mu$ defined by (d) is called the Markov transform, see Section \ref{label_markov_transform} for further details. The above result says that the set of Markov transforms of probability measures (with the above integrability condition) is exactly the set of monotonically selfdecomposable distributions.

Furthermore, the analogue of L\'evy's limit theorem (see \cite[Theorem 56]{Lev54} or \cite[\S29, Theorem 1]{GK54}) holds for monotone convolution.

\begin{theorem}[Theorem \ref{SL}] If $\mu$ is a weak limit of probability measures
\begin{equation*}
\DD_{b_n}(\mu_1\rhd \cdots  \rhd \mu_n),  \qquad n\to \infty,
\end{equation*}
where $b_n$ are positive real numbers and $\mu_{n}$ are probability measures on $\R$ such that\\
 $\{\DD_{b_n}(\mu_k)\}_{1 \leq k \leq n, 1\leq n}$ forms an infinitesimal array, then $\mu$ is monotonically selfdecomposable. Conversely, any monotonically selfdecomposable distribution can be obtained as such a limit.
\end{theorem}

Another subclass of probability measures having univalent $F$-transforms is given by unimodal distributions. A (Borel) measure $\mu$ on $\R$ is said to be \index{unimodal distribution!on $\R$}\emph{unimodal} with mode
$c \in \R$ if there exist a non-decreasing function $f\colon(-\infty, c)\mapsto [0,\infty)$ and a non-increasing function $g\colon(c,\infty)\mapsto [0,\infty)$ and $\lambda\in[0,\infty]$ such that
\begin{equation*}
\mu({\rm d}x) = f(x)\mathbf{1}_{(-\infty,c)}(x)\,{\rm d}x + \lambda\delta_c +  g(x)\mathbf{1}_{(c,\infty)}(x)\,{\rm d}x.
\end{equation*}

We give a characterization of unimodal distributions as follows. The result is in fact a combination of the past works of Khintchine \cite{K38}, Kaplan \cite{K52}, and Hengartner and Schober \cite{HS70}.

\begin{theorem}[Theorem \ref{anotherkhintchine}] Let $\mu$ be a probability measure on $\R$.  The following are equivalent.
\begin{enumerate}[\rm(1)]
\item $\mu$ is unimodal with mode $c$.
\item $\Im((z-c)G_\mu'(z)) \geq0$ for all $z\in\C^+.$
\item There exists an $\R$-valued random variable $X$, independent of a uniform random variable $U$ on $(0,1)$, such that $\mu$ is the law of $U X +c$.  

\item The following three assertions hold:
\begin{itemize}
\item $G_\mu$ is univalent in $\C^+$.
\item $G_\mu(\C^+)$ is horizontally convex, namely if $z_1,z_2 \in G_\mu(\C^+)$ with the same imaginary part, then $(1-t)z_1+t z_2 \in G_\mu(\C^+)$ for any $t\in(0,1)$.
\item There exist points $z_n \in \C^+$ such that $z_n \to c$ and
$$
\lim_{n\to\infty} \Im (G_\mu(z_n)) = \inf_{z\in \C^+} \Im (G_\mu(z)).
$$
\end{itemize}
\end{enumerate}
\end{theorem}

 By considering unitary operators, one can translate most of the notions we discussed into a multiplicative setting,
 which leads to a multiplicative convolution for probability measures on the unit
 circle, Loewner chains in the unit disk, and unitary multiplicative monotone increment processes. We obtain analogous
results for this setting and thus both cases will be treated simultaneously throughout this work.

\section{Organization of the paper} The remaining sections are structured in the following way. \\

\begin{itemize}
\item Section \ref{section_preliminaries}: We give the necessary definitions and notations, and we recall several properties of $F$-transforms and $\eta$-transforms.\\

\item Section \ref{EV_sec_Loewner}: We review the theory of Loewner chains and their relation to the Loewner differential equation (Section \ref{subsec_LPDE}) and we prove the equivalence between \eqref{intro_thm2-1} and \eqref{intro_thm2-3} in Theorem \ref{intro_thm2} (Sections \ref{sec_EV_univalence}, \ref{EV_section_embeddings}).\\

\item Section \ref{section_processes}: We prove Theorem \ref{intro_thm0}. The construction is based on a one-to-one correspondence between
additive Loewner chains and $\rhd$-homogeneous Markov processes (Sections \ref{section_additive_processes} - \ref{summary_additive}). Furthermore, we explain how these Markov processes
arise from free increment processes (Section \ref{sec:Free-Markov_additive}) and we clarify the Feller property, give explicit formula of the generator, and construct martingales that are naturally associated to $\rhd$-homogeneous Markov processes (Section \ref{analysis_Markov_additive}). In Section \ref{alternative_constructions}, we look at alternative constructions of SAIPs via quantum stochastic differential equations.\\

\item Section \ref{section_processes_mult}: This section is the multiplicative analogue of Section \ref{section_processes}. We establish one-to-one correspondences between the unitary version of SAIPs, the Loewner chains on the unit disk fixing the origin, and certain Markov processes on the unit circle (Sections \ref{from_mult_processes_to_Loewner} - \ref{summary_multiplicative}). We also obtain such Markov processes from multiplicative free increment processes (Section \ref{mult_con_from_free}) and we analyze their properties (Section \ref{analysis_mult_markov}).\\

\item Section \ref{section_limits_additive}: We study several subclasses of probability measures on $\R$ in the context of monotone, free, Boolean, and the classical convolution. These classes are: limits of infinitesimal arrays (Section \ref{inf_arrays_additive}),  infinitely divisible distributions (Sections \ref{subsec_mon_inf_div_add}, \ref{subsec_free_inf_div_add}), unimodal distributions (Section \ref{subsec_unimodal_add}), and selfdecomposable distributions (Sections \ref{sec SD}, \ref{sec CSD}, \ref{sec FSD}, \ref{additive_selfdecom}).\\

\item Section \ref{sectionunitcircle}: This section is the multiplicative analogue of Section \ref{section_limits_additive} and deals with probability measures on the unit circle. In Section \ref{subsec_khin_mult}, we characterize the class of limits of monotone infinitesimal arrays. We furthermore investigate infinitely divisible distributions (Sections \ref{mon_inf_mult}, \ref{FIDunit}) and unimodal distributions (Section \ref{subsec_unimodal_mult}).
\end{itemize}



\chapter{Preliminaries}\label{section_preliminaries}


\section{The four other additive convolutions}\label{intro_convolutions}
We denote by $\cP(\R)$ the family of all Borel probability
measures on $\mathbb{R}.$

As mentioned in Section \ref{sec:monotone_independence}, monotone convolution is a binary operation on $\cP(\R)$ defined by
\begin{equation}
F_{\mu\rhd \nu}(z)=(F_\mu \circ F_\nu)(z), \quad z\in \Ha.
\end{equation}
As already mentioned, there are only five notions of independence in a certain sense:
monotone, anti-monotone, Boolean, free, and tensor independence; see \cite{MR2016316}.
All of these notions lead to additive convolutions of probability measures on $\R$ by regarding the sum of independent self-adjoint random variables. In particular, anti-monotone independence is simply defined by reversing the order in Definition \ref{def-mon} and leads to the \index{anti-monotone convolution!additive $\lhd$}anti-monotone convolution $\mu\lhd \nu$ given by
\[ F_{\mu\lhd \nu} = F_\nu \circ F_\mu. \]
We will also encounter the \index{Boolean convolution!additive $\uplus$}additive Boolean convolution $\mu\uplus \nu$, which was introduced in \cite{SW97}.
 Let $B_\mu=z-F_\mu(z)$.
Then $\mu\uplus \nu$ is characterized by
\begin{equation}\label{bool_convolutions_addi}
B_{\mu\uplus \nu}(z) = B_\mu(z) + B_\nu(z), \quad z\in\Ha.
 \end{equation}

Next we look at \index{free convolution!additive $\boxplus$}free convolution. For $\lambda,M>0$, the truncated cone is the domain
\begin{equation*}
\Gamma_{\lambda,M} = \{z \in \C^+: \Im(z)>M, |\Re(z)| <\lambda \Im(z)\}.
\end{equation*}
Bercovici and Voiculescu \cite{BV93} showed that for any $\lambda>0$, there exist $\lambda',M',M>0$ such that $F_\mu$ is univalent in the truncated cone $\Gamma_{\lambda',M'}$
and $F_\mu(\Gamma_{\lambda',M'}) \supset \Gamma_{\lambda,M}$. Hence the right compositional inverse map $F_\mu^{-1}$ may be defined in $\Gamma_{\lambda,M}$.
The \index{Voiculescu transform}\emph{Voiculescu transform} of $\mu$ is  then defined by
\begin{equation}\label{eq:Voiculescu_transform}
\varphi_{\mu }(z) =F_{\mu }^{-1}(z)-z,\qquad z \in \Gamma_{\lambda,M}.
\end{equation}
For $\mu,\nu\in\cP(\R)$, the additive free convolution $\boxplus$ is characterized by
\begin{equation*}
\varphi_{\mu\boxplus\nu}(z) = \varphi_{\mu}(z) + \varphi_{\nu}(z)
\end{equation*}
on the intersection of the domains.

Finally, tensor convolution is nothing but the \index{classical convolution!additive $\ast$}classical convolution $\mu \ast \nu$ characterized by
\[ \int_\R e^{ixz}\, (\mu\ast \nu)({\rm d}x) = \int_\R e^{ixz} \,\mu({\rm d}x) \cdot  \int_\R e^{ixz} \,\nu({\rm d}x), \qquad z \in \R.
\]

\section{The five multiplicative convolutions}\label{mono_convolutions_mult}
 Denote by $\cP(\T)$ the family of all Borel probability measures on $\T:=\partial \D,$ where $\D=\{z\in\C\,|\, |z|<1\}$ is the unit disk. We denote the moments of probability measures on $\T$ by
$$
m_n(\mu)=\int_\T x^n \,\mu({\rm d}x).
$$
Let $\cP_{\times}(\tor)$ be the set of probability measures on $\tor$ with non-zero mean.

If $U$ is a unitary random variable on a concrete quantum probability space $(H,\xi)$, then there exists a unique probability measure $\mu\in \cP(\T)$ such that
\begin{equation}
\Phi_\xi(U^n) = \int_{\T} x^n \,\mu({\rm d}x), \qquad n \in \N,
\end{equation}
which is referred to as the \index{distribution!of a unitary random variable}\emph{distribution} of $U$. In this case rather than the Cauchy transform, the \index{moment generating function}moment generating function $\psi_\mu$ and the \index{eta-transform@$\eta$-transform}$\eta$-transform
\begin{equation}\label{def:psi-eta}
\psi_\mu(z)=\int_{\T}\frac{xz}{1-xz}\, \mu({\rm d}x) \qquad \text{and} \qquad \eta_\mu(z)=\frac{\psi_\mu(z)}{1+\psi_\mu(z)}, \qquad z \in \D,
\end{equation}
are more useful. It can be seen that $\eta_\mu$ maps $\D$ into $\D$ with
$$
\eta_\mu(0)=0\qquad  \text{and} \qquad \eta_{\mu}'(0)=\int_{\T}x\, \mu({\rm d}x).
$$
Then the distribution $\mu$ of a unitary random variable $U$ is equivalently characterized by
\begin{equation}
\Phi_\xi\left(\frac{z U}{1-zU}\right) = \psi_\mu(z), \qquad z \in \D.
\end{equation}

Now let $U,V$ be two unitary random variables, independent in some sense. A natural convolution arises as the distribution of $UV$, called a multiplicative convolution. However, for the monotone, anti-monotone, and Boolean cases, we do not simply assume $(U,V)$ is independent, since if we do so, then the distribution of
$UV$ is given by a rather trivial expression, see \cite[p. 930]{B05} for the monotone case. In \cite{B05}, Bercovici considers another situation, which leads to a more interesting convolution of probability measures on the unit circle: the \index{monotone convolution!multiplicative $\submm$}\emph{multiplicative monotone convolution} $\mu \submm \nu$ of probability measures $\mu,\nu\in\cP(\T)$ is defined by
\begin{equation}\label{eq:multiplicative_monotone_convolution}
\eta_{\mu \mmm \nu}:=\eta_\mu \circ \eta_\nu.
\end{equation}
Now assume that $(U-I, V)$ is monotonically independent (this is not equivalent to assuming $(U,V)$ is monotonically independent since in Definition \ref{def-mon} \eqref{enu:monotone_independence2} we are not allowed to take $f(x)=x\pm1$), and let $\mu$ and $\nu$ be the distributions of $U$ and $V$ respectively. Then the distribution of $UV$ is given by $\mu \submm \nu$, see \cite[Corollary 2.3]{B05}, \cite[Corollary 4.2]{franz-mult}, and the distribution of $VU$ is also equal to $\mu \submm \nu$, see \cite[Corollary 4.2]{franz-mult}.

\index{anti-monotone convolution!multiplicative}Anti-monotone convolution on $\T$ is simply defined by reversing the order in \eqref{eq:multiplicative_monotone_convolution}.

For Boolean convolution, let $h_\mu(z)=\eta_\mu(z)/z$ which can be analytically defined in $\D$ since $\eta_\mu(0)=0$. Then the \index{Boolean convolution!multiplicative $\hutimes$}multiplicative Boolean convolution $\mu\hutimes\nu$ is defined via
\begin{equation}\label{bool_convolutions_mult}
h_{\mu\hutimes\nu}(z)= h_\mu(z)\cdot h_\nu(z), \qquad z \in \D.
\end{equation}
The Boolean convolution $\mu \hutimes \nu$ is the distribution of $UV$ where $U$ and $V$ are unitary random variables with distributions $\mu$ and $\nu$ respectively, such that $U-I$ and $V-I$ are Boolean independent, see \cite{Ber06,franz_boolean}. Similarly to the monotone case, if we simply assume $(U,V)$ is Boolean independent, then the distribution of $UV$ is rather trivial, see \cite[p.\ 6]{franz_boolean}.

\begin{remark}
Why we assume the independence of $(U-I,V)$ or $(U-I,V-I)$ can be somehow explained from the viewpoint of so-called \emph{conditionally monotone independence} or \emph{conditionally free independence}, see \cite[Proposition 2.1]{Has13}.
\end{remark}
\begin{remark}\label{u-I}
The monotone independence of $(U-I,V)$ is equivalent to the monotone independence of $(U-I,V-I)$.
\end{remark}

The remaining cases are free and tensor independences. In these cases we define multiplicative convolutions by simply assuming that $U$ and $V$ are tensor/freely independent. Note that for any $a, b \in \C$, tensor or free independence of $U$ and $V$ is equivalent to tensor, resp.\ free  independence of $U +a I$ and $V+b I$.

For $\mu\in\cP_{\times}(\tor)$ the series expansion of $\eta_\mu(z)$ is $m_1(\mu)z + O(z^2)$, and hence the inverse series $\eta_\mu^{-1}$ exists and is convergent in an open neighborhood of $0$. Then we define $\Sigma_\mu$ by
\begin{equation*}
\Sigma_\mu(z)=\frac{1}{z}\eta_\mu^{-1}(z)
\end{equation*}
in the neighborhood of $0$ where $\eta_\mu^{-1}$ is defined.
For $\mu,\nu\in\cP_{\times}(\tor)$ Voiculescu \cite{Voi87} characterized \index{free convolution!multiplicative $\boxtimes$}multiplicative free convolution $\boxtimes$ by
\begin{equation}\label{FM}
\Sigma_{\mu \boxtimes \nu}(z) = \Sigma_{\mu}(z) \Sigma_{\nu}(z)
\end{equation}
in a neighborhood of $0$.

Finally, \index{classical convolution!multiplicative $\circledast$}classical multiplicative convolution $\circledast$ on $\tor$ is defined by
\begin{equation*}
(\mu\circledast\nu)(A) = \int_{\tor^2} \mathbf{1}_A(\xi \zeta) \mu({\rm d}\xi) \nu({\rm d}\zeta)
\end{equation*}
for Borel sets $A \subset \tor$. This convolution can be characterized in a simple way using the moments:
\begin{equation*}
m_n(\mu\circledast\nu) =m_n(\mu) m_n(\nu), \qquad n\in\N.
\end{equation*}

The free convolution $\mu\boxtimes \nu$ and the classical convolution $\mu\circledast\nu$ are the distributions of the product $UV$ of two unitary random variables $U,V$ distributed according to $\mu$ and $\nu$, when these random variables are free, respectively, tensor independent.

The arc-length measure $\haar=\frac{{\rm d}\theta}{2\pi}$, which is the normalized Haar measure on $\tor$, is a singular object in the context of convolutions on $\tor$. It has the property that
\begin{equation}\label{Haar}
\haar \submm \mu = \mu \submm \haar = \haar \hutimes \mu =\haar \boxtimes \mu= \haar \circledast \mu=\haar
\end{equation}
for any $\mu \in \cP(\tor)$. All moments of $\haar$ are zero, and hence $\psi_\haar= \eta_\haar=0$.
Actually the $\eta$-transform $\eta_\mu$ is constant if and only if $\mu=\haar$.


\section{Properties of $F$-transforms}\index{F-transform@$F$-transform|(} As we saw in Sections \ref{sec:monotone_independence} and \ref{intro_convolutions}, the $F$-transform characterizes the convolutions of probability measures except the classical case. This section summarizes useful facts about the $F$-transform.

A holomorphic function $F:\Ha\to \Ha \cup \R$ is called a Pick function. Note that either $F(\Ha)\subseteq \Ha$ or $F$ is constant. Any Pick function can be written as
\begin{equation}\label{EV_eq:1}
F(z)=az+b+\int_\mathbb{R}\frac{1+xz}{x-z} \rho({\rm d}x) \quad \index{Pick-Nevanlinna representation}\text{(Pick-Nevanlinna representation)},
\end{equation}
where $a\geq0, b\in \mathbb{R}$ and $\rho$ is a finite non-negative (Borel) measure on $\mathbb{R}$; see \cite[Thm. F.1]{MR2953553}.
With the one-point compactification $\widehat\R:=\R\cup\{\infty\}$, the above formula may be written in the form
\begin{equation}\label{eq:compactification}
F(z)=b + \int_{\widehat\R}\frac{1+xz}{x-z}\widehat\rho({\rm d}x),
\end{equation}
where $\widehat\rho|_{\R} = \rho$ and $\widehat\rho(\{\infty\})=a$.
The triplet $(a,b,\rho)$ is uniquely determined by the formulas
\begin{align*}
a &= \lim_{y\to\infty}\frac{F(i y)}{i y}, \\
F(i) &= b + i (a+\rho(\R)),
\end{align*}
and the \index{Stieltjes inversion formula}Stieltjes inversion formula
\begin{align}
\frac{1}{2}\tau(\{\alpha\})+\frac{1}{2}\tau(\{\beta\})+\tau((\alpha,\beta)) &= \frac{1}{\pi} \lim_{\epsilon\downarrow0}\int_\alpha^\beta \Im[F(x+i \epsilon)]\, {\rm d}x, \label{Stieltjes} \\
\tau(\{\alpha\})&=-\lim_{\epsilon\downarrow0} i\epsilon F(\alpha+i\epsilon), \label{eq:atom}
\end{align}
where $-\infty< \alpha<\beta<\infty$ and $\tau({\rm d}x) = (1+x^2) \rho ({\rm d}x)$. The number $a$ is also called the \emph{angular derivative} of $F$ at $\infty$ and it is also denoted by $F'(\infty)$.
If $F$ is not an automorphism of $\Ha,$ then the iterates $F^{\circ n}$ converge locally uniformly in $\Ha$ to a point $\tau \in \Ha \cup \widehat{\R}$, the \emph{Denjoy-Wolff point} of $F$; see \cite[The Grand Iteration Theorem]{shapiro}.

$F$-transforms of probability measures can be characterized as follows.
\begin{lemma}[Prop. 2.1 and 2.2 in \cite{M92}]\label{Julia} Let $F\colon \C^+ \to \C^+\cup\R$ be holomorphic. Then the following are equivalent.
\begin{enumerate}[\rm(1)]
\item There exists a probability measure $\mu$ on $\R$ such that $F=F_\mu$.

\item\label{Julia2} $\lim_{y\to \infty}\frac{F(i y)}{i y}=1$. 

\item $F$ has the Pick-Nevanlinna representation
\begin{equation}\label{EV_eq:2}
F(z) =  z + b + \int_{\R}\frac{1+xz}{x-z}\rho({\rm d}x),
\end{equation}
where $b \in \R$ and $\rho$ is a finite, non-negative measure on $\R$.
\end{enumerate}
Note that if these equivalent conditions hold, then $\Im(F(z)) \geq \Im(z)$ for all $z\in\C^+$. Furthermore, $\mu$ has mean zero and finite variance $\sigma^2$ if and only if there exists a
finite non-negative measure $\tau$ on $\mathbb{R}$ with $\tau(\R)=\sigma^2$
such that
\begin{equation}\label{EV_finite_var}
F_\mu(z) = z + \int_\mathbb{R} \frac1{x-z}\tau({\rm d}x).
\end{equation}
\end{lemma}
\begin{remark}\label{rm_finite_var}
 Condition \eqref{EV_finite_var} is furthermore equivalent to the normalization
 \[F_\mu(z) = z - \frac{\sigma^2}{z} + {\scriptstyle\mathcal{O}}(1/|z|)\]
 as $z\to\infty$ non-tangentially in $\Ha,$ see \cite[Lemma 1]{MR1201130}.
\end{remark}

\begin{lemma}\label{first_moment}
If the first moment of $\mu$ exists, then
 \[ \int_\R x\, \mu({\rm d}x) =  \lim_{y\to\infty}(iy-F_\mu(iy)).\]
  \end{lemma}
  \begin{proof}
   The proof is essentially given in  the proof of \cite[Theorem 2.2]{M92}:
 \[
 \begin{split}
 \frac{iy}{F_\mu(iy)}(iy-F_\mu(iy))
 &=-\frac{y^2}{F_\mu(iy)}-iy =  y^2 \int_\R\left( \frac1{x-iy}+\frac1{iy}\right)\mu({\rm d}x) \\
 &= y^2 \int_\R \frac{y^2(x+iy)-iy(x^2+y^2)}{(x^2+y^2)y^2}\,\mu({\rm d}x) =
 \int_\R \frac{xy^2-iyx^2}{x^2+y^2}\,\mu({\rm d}x).
 \end{split}
 \]
 The dominated convergence theorem, with the simple inequality $|xy| \leq (x^2+y^2)/2$,  implies that
 \[
\lim_{y\to\infty}\int_\R \frac{xy^2 - i y x^2}{x^2+y^2}\,\mu({\rm d}x) =  \int_\R x \,\mu({\rm d}x).
 \]
Then Lemma \ref{Julia} \eqref{Julia2} finishes the proof.
  \end{proof}

The next lemma deals with convergence of a sequence of $F$-transforms.

\begin{lemma}\label{lemmaconvergence}
Let $F, F_n$ be analytic mappings from $\C^+$ to $\C^+ \cup \R$. They have the Pick-Nevanlinna representations
\[
\begin{split}
&F(z) =  b + \int_{\widehat\R}\frac{1+xz}{x-z}\widehat\rho({\rm d}x) = a z + b + \int_{\R}\frac{1+xz}{x-z}\rho({\rm d}x), \\
&F_n(z) =  b_n + \int_{\widehat\R}\frac{1+xz}{x-z}\widehat\rho_n({\rm d}x) = a_n z + b_n + \int_{\R}\frac{1+xz}{x-z}\rho_n({\rm d}x),
\end{split}
\]
where $b \in \R$, $\widehat\rho$ is a non-negative finite, non-negative measure on $\widehat\R$, $a =\widehat\rho(\{\infty\})\geq 0$ and similarly for $b_n, \widehat\rho_n, a_n$.
Then the following are equivalent:
\begin{enumerate}[\rm(1)]
\item $F_n$ converges to $F$ locally uniformly on $\C^+$.
\item There is a sequence of distinct points $\{z_n \}_{n\geq 1}$ converging to a point of $\C^+$ such that
$F_n(z_k)$ converges to $F(z_k)$ as $n \to \infty$ for any $k \geq 1$.
\item $b_n$ converges to $b$ and $\widehat\rho_n$ converges weakly to $\widehat\rho$ on $\widehat\R$.
\end{enumerate}
In particular, if $F=F_\mu$ and $F_n = F_{\mu_n}$ for probability measures $\mu,\mu_n$ on $\R$, or equivalently if $a=a_n = 1$, the above conditions are also equivalent to the weak convergence $\mu_n \wto\mu$.
\end{lemma}
\begin{proof}
$(2) \Rightarrow (1)$ and $(3)$: by using a dilation $z \mapsto \lambda z$ and a translation $z \mapsto z + a$, we can always assume that $z_1 = i$. Since $F_n(i) = b_n + i\widehat\rho_n(\widehat\R)$, the sequence $b_n$ converges to $b $, and $\widehat\rho_n(\widehat\R)$ converges to $\widehat\rho(\widehat\R)$. In particular, they are bounded.
We can then prove that $F_n$ is uniformly bounded on any compact subset of $\C^+$.
Given a subsequence of $F_n$, Montel's theorem and Helly's theorem show the existence of a further subsequence, denoted as $F_{m(k)}$, such that $F_{m(k)}$ and $\widehat\rho_{m(k)}$ converge to limits $\widetilde{F}$ and $\widetilde{\widehat\rho}$, respectively. Then $\widetilde{F}$ coincides with $F$ on $\{z_p \}_{p\geq 1}$, so that $F = \widetilde{F}$ by the identity theorem and $\widehat\rho = \widetilde{\widehat\rho}$ by the uniqueness of the Pick-Nevanlinna representation. Therefore, (1) and (3) follow.

The implication $(1) \Rightarrow (2)$ is immediate and $(3) \Rightarrow (2)$ follows by the definition of weak convergence.

By \cite[Theorem 2.5]{M92}, locally uniform convergence of $F_{\mu_n}$ is equivalent to the weak convergence of $\mu_n$.
\end{proof}
\index{F-transform@$F$-transform|)} 

\section{Cauchy transform}

\index{Cauchy transform}

While the $F$-transform is central in our considerations, the Cauchy transform \eqref{eq:Cauchy} is also useful in some cases. 
Let $\mu$ be a probability measure on $\R$. The Cauchy transform $G_\mu$ can be expressed as $-F(z)$, where $F$ is a Pick function of the form \eqref{EV_eq:1} with $a=0, (1+x^2)\rho({\rm d}x)=\mu({\rm d}x)$ and a suitable $b$. Then the Stieltjes inversion formula \eqref{Stieltjes}--\eqref{eq:atom} implies
\begin{align}
\frac{1}{2}\mu(\{\alpha\})+\frac{1}{2}\mu(\{\beta\})+\mu((\alpha,\beta)) 
&= -\frac{1}{\pi} \lim_{\epsilon\downarrow0}\int_\alpha^\beta \Im[G_\mu(x+i \epsilon)]\, {\rm d}x, \label{Stieltjes2} \\
\mu(\{\alpha\}) &= \lim_{\epsilon\downarrow0} i\epsilon G_\mu(\alpha+i\epsilon), \qquad \alpha \in \R, \label{eq:atom2}
\end{align}
which are useful for computing $\mu$. This in particular implies that for two probability measures $\mu$ and $\nu$, $\mu = \nu$ if and only if $G_\mu=G_\nu$. 

A characterization of the Cauchy transform can be derived from that of the $F$-transform in Lemma \ref{Julia}.  
\begin{proposition} Let $G\colon \C^+ \to (-\C^+)\cup\R$ be holomorphic. Then the following are equivalent.
\begin{enumerate}[\rm(1)]
\item There exists a probability measure $\mu$ on $\R$ such that $G=G_\mu$.
\item $\lim_{y\to \infty}i yG(i y)=1$. 
\end{enumerate}
\end{proposition}

\section[$\eta$-transforms and moment generating functions]{Properties of $\eta$-transforms and moment generating functions}

\index{eta-transform@$\eta$-transform|(}
\index{moment generating function|(}${}$

The $\eta$-transform \eqref{def:psi-eta} is used to characterize multiplicative convolutions (see Section \ref{mono_convolutions_mult}). We collect several facts about it.
Hereafter $\bH$ denotes the right half-plane of $\C$.
We quote the following result from \cite[p.91]{A65}, which can be obtained from \eqref{eq:compactification} and \eqref{Stieltjes} with a suitable Moebius transformation.

\begin{lemma}\label{inversionunit}
Let $f\colon\disk \to \bH\cup i\R$ be an analytic function (called a Herglotz function). Then $f$ can be represented as
\[
f(z) = ib + \int_\tor \frac{\xi+z}{\xi-z}\rho({\rm d}\xi),
\]
where $b \in \R$ and $\rho$ is a finite (Borel) non-negative measure. Then $b = \Im (f(0))$, $\rho(\tor)=\Re (f(0))$ and
\begin{align*}
\frac{1}{2}\rho(\partial A) + \rho(A) &= \lim_{r \uparrow 1}\int_A \Re(f(r \xi))\,\haar({\rm d}\xi), \\
\rho(\{\alpha\}) &= \frac{1}{2}\lim_{r\uparrow1}(1-r)f(r\alpha), \qquad \alpha \in \tor, 
\end{align*}
for every open arc $A \subset \tor$, where $\partial A$ consists of the two endpoints of $A$. Here, $\haar$ denotes the normalized
Haar measure on $\T$.
\end{lemma}

Since the moment generating function \eqref{def:psi-eta} of a probability measure $\mu$ on $\tor$ can be expressed as
\begin{equation}\label{eq:psi2}
\psi_\mu(z)=-\frac{1}{2}+\frac{1}{2}\int_\tor \frac{\xi^{-1}+z}{\xi^{-1}-z}\,\mu({\rm d}\xi),
\end{equation}
the inversion formula of Lemma \ref{inversionunit} gives 
\begin{align}
\frac{1}{2}\mu(\partial A)+ \mu(A)&= \lim_{r \uparrow 1}\int_A \Re (2\psi_\mu(r \xi^{-1})+1)\,\haar({\rm d}\xi), \label{eq:Herglotz-inv1}\\
\mu(\{\alpha\}) &= \lim_{r\uparrow1}(1-r)\psi_\mu(r\alpha), \qquad \alpha \in \tor, \label{eq:atom3}
\end{align}
where $A$ is an open arc of $\tor$ and $\partial A$ consists of two endpoints of $A$. 

Lemma \ref{inversionunit} yields the following analytic characterization of $\psi_\mu$.
\begin{lemma}\label{characterizationpsi}
Let $\psi\colon\disk \to \C$ be an analytic function. The following conditions are
equivalent.
\begin{enumerate}[\rm(1)]
\item There exists a probability measure $\mu$ on $\tor$ such that $\psi=\psi_\mu$.
\item $\psi(0)=0$ and $\Re[\psi(z)] \geq -\frac{1}{2}$ for all $z\in \D$.
\end{enumerate}
\end{lemma}

A more useful characterization in terms of $\eta_\mu$ follows from Lemma \ref{characterizationpsi} (see \cite[Proposition 3.2]{BB05}). 

\begin{lemma}\label{characterizationeta}
Let $\eta\colon\disk \to \C$ be an analytic function. The following conditions are
equivalent.
\begin{enumerate}[\rm(1)]
\item There exists a probability measure $\mu$ on $\tor$ such that $\eta=\eta_\mu$.
\item $\eta(0)=0$ and $\eta$ maps $\disk$ into $\disk$.
\item $|\eta(z)|\leq |z|$ for all $z\in\disk$.
\end{enumerate}
\end{lemma}

The following result can be obtained from Lemma \ref{lemmaconvergence} and a Moebius transformation sending $\C^+$ onto $\disk$. Notice that the equivalence of the last condition \eqref{moment_converge} has no analogue in Lemma \ref{lemmaconvergence}, but the proof simply follows by dominated convergence theorem.
\begin{lemma}\label{lemmaconvergenceunit}
Let $f,f_n$ be analytic maps from $\disk$ to $\bH \cup i\R$ with the Herglotz representations
\[
\begin{split}
&f(z) =   i b + \int_{-\pi}^\pi\frac{\xi+z}{\xi-z}\rho({\rm d}\xi),  \\
&f_n(z) =   i b_n + \int_{\tor}\frac{\xi+z}{\xi-z}\rho_n({\rm d}\xi),
\end{split}
\]
where $b,b_n \in \R$ and $\rho,\rho_n$ are finite, non-negative measures on $\tor$.
Then the following are equivalent:
\begin{enumerate}[\rm(1)]
\item $f_n$ converges to $f$ locally uniformly on $\disk$.
\item\label{unit2} There is a sequence of distinct points $\{z_n \}_{n\geq 1}\subset \tor$ converging to a point of $\disk$ such that
$$
\lim_{n\to\infty}f_n(z_k)=f(z_k) \quad \text{for any~} k \in \N.
$$
\item $\lim_{n\to\infty} b_n =b$ and $\rho_n \wto \rho$.
\end{enumerate}
In particular, if $f=\psi_\mu+1/2$ and $f_n = \psi_{\mu_n}+1/2$ for probability measures $\mu,\mu_n$ on $\tor$, then
 the above conditions are also equivalent to any one of the following conditions:
 \begin{enumerate}[\rm(1)]\setcounter{enumi}{3}
 \item  $\eta_{\mu_n}$ converges to $\eta_\mu$ locally uniformly on $\disk$.
\item The weak convergence $\mu_n\wto\mu$ holds.
\item\label{moment_converge} Moments of any degree converge, i.e.
$
\lim_{n\to\infty}m_k(\mu_n) = m_k(\mu)$ for any $k\in\N.$
\end{enumerate}
\end{lemma}

\begin{corollary}\label{weakconv}
The convolutions $\circledast, \boxtimes,\utimes,\submm \colon \cP(\tor) \times \cP(\tor) \to \cP(\tor)$ are weakly continuous.
\end{corollary}
\begin{proof}
For $\star$ being any one of the three convolutions, the moments $m_n(\mu \star \nu)$ can be represented by polynomials on $m_k(\mu), m_k(\nu), k=1,2,\dots,n$. Then Lemma \ref{lemmaconvergenceunit} \eqref{moment_converge} implies the conclusion.
\end{proof}

\index{eta-transform@$\eta$-transform|)}
\index{moment generating function|)}



\chapter{Loewner Chains} \label{EV_sec_Loewner}

 In what follows, $D$ stands either for the upper half-plane
$\Ha$ or the unit disk $\D.$\\

The distributions of processes with monotonically independent increments will lead us to certain families of
holomorphic mappings. These families turn out to be decreasing \emph{Loewner chains}.

\begin{definition}\label{EV_def:evolution_family}${}$
\begin{enumerate}[\rm (1)]
\item  Let $(f_{st})_{0\leq s \leq t}$ be a family of holomorphic self-mappings $f_{st}:D\to D$ satisfying
 \begin{enumerate}[\quad(a)]
\item $f_{ss}(z)=z$ for all $z\in D$ and $s\geq 0$,
\item $f_{su} = f_{st} \circ f_{tu}$ for all $0\leq s \leq t \leq u$,
\item  $(s,t)\mapsto f_{st}$ is continuous with respect to locally uniform convergence.
\end{enumerate}
The family $(f_t)_{t\geq 0}:=(f_{0t})_{t\geq 0}$ is called a \index{Loewner chain}\emph{(decreasing) Loewner chain on $D$}. We will call the mappings $f_{st}$ the \index{transition mappings}\emph{transition mappings} of the Loewner chain.

\item We call a Loewner chain $(f_{t})_{t\geq 0}$ an \index{Loewner chain!additive}\index{F-transform@$F$-transform}\emph{additive Loewner chain} if $D=\Ha$ and \\
$\lim_{y \to \infty} f_{st}(iy)/ (iy)=1$, or equivalently
$$f_{st} = F_{\mu_{st}}$$
 for all $0\leq s\leq t$, where each $\mu_{st}$ is a probability measure on $\R$. 
	
\item We call a Loewner chain $(f_{t})_{t\geq 0}$ a \index{Loewner chain!multiplicative}\emph{multiplicative Loewner chain}\index{eta-transform@$\eta$-transform} if $D=\D$ and $f_{st}(0)=0$, or equivalently,
$$f_{st} = \eta_{\mu_{st}}$$
for all $0\leq s\leq t$, where each $\mu_{st}$ is a probability measure on $\T$.

\end{enumerate}

\end{definition}

Some remarks concerning this definition are in order.

\begin{remark}
Due to property (b), the domains $f_t(D)$ are decreasing, i.e. $f_t(D)\subseteq f_s(D)$ for all $s\leq t.$
In Loewner theory, the term \emph{Loewner chain} usually refers to \emph{increasing} domains described by a family $(f_t)$ of
\emph{univalent} functions.
In Section \ref{sec_EV_univalence} we will see that Loewner chains always consist of univalent functions.
\end{remark}

\begin{remark}\label{EV_remark_ev}
In Loewner theory, a family $(\phi_{st})_{0\leq s\leq t}$ of holomorphic mappings $\phi_{st}:D\to D$ is called an \emph{evolution family on $D$}
if
 \begin{itemize}
\item[(a)] $\phi_{ss}(z)=z$ for all $z\in D$ and all $s \geq 0$,
\item[(b)] $\phi_{su} = \phi_{tu} \circ \phi_{st}$ whenever $0\leq s\leq t\leq u$,
\item[(c)] $(s,t)\mapsto \phi_{st}$ is continuous with respect to locally uniform convergence.
\end{itemize}
If (b) is replaced by $\phi_{su} = \phi_{st} \circ \phi_{tu}$, then the family is usually called a
\emph{reverse evolution family}. Thus the transition mappings of a decreasing Loewner chain form a reverse evolution family. 
\end{remark}

\begin{remark}
In case of an additive or a multiplicative Loewner chain, condition (c) is equivalent to 
$$\text{$(s,t)\mapsto\mu_{st}$ is continuous with respect to weak convergence}$$
due to Lemma \ref{lemmaconvergence} and Lemma \ref{lemmaconvergenceunit}.
\end{remark}

\newpage

\begin{lemma}\label{non_constant}All transition mappings $f_{st}$ of a Loewner chain are non-constant.
\end{lemma}
\begin{proof}
Assume that some $f_{st}$ is constant. Let $T>s$ be the smallest time such that $f_{sT}$ is constant. Then
\[f_{sT} = f_{s,T-\varepsilon} \circ f_{T-\varepsilon, T}\]
for all $\varepsilon \in (0,T-s]$ and we see that $f_{T-\varepsilon,T}$ must be constant for every $\varepsilon \in (0,T-s]$, because $f_{s, T-\varepsilon}$ is non-constant.  But $f_{T-\varepsilon,T}(z) \to z$ locally uniformly as $\varepsilon\to 0$, a contradiction.
\end{proof}


\begin{example}
 Assume that the Loewner chain $(f_{t})$ is a semigroup, i.e. it satisfies
 $$f_{t+s} = f_{t}\circ f_{s}$$ for all $s,t\geq 0.$ In this case, the following limit exists for every $z\in D$:
 $$G(z):=\lim_{t\searrow 0} \frac{f_t(z)-z}{t}.$$
 The function $G$ is also holomorphic and it is called the \emph{infinitesimal generator of the semigroup
 $(f_t)_{t\geq0}$}.\footnote{Note that sometimes
$-G(z)$ is called the infinitesimal generator of the semigroup.} The function $f_t$ can be recovered from $G$ by solving the initial
value problem \begin{equation}\label{EV_semigroup} \frac{{\rm d}}{{\rm d}t}f_t = G(f_t), \quad f_0(z)=z \in D.\end{equation}
The family of all infinitesimal generators on $\D$ can be represented quite explicitly by the
 \index{Berkson-Porta formula}Berkson-Porta formula (\cite{BP78}):
 \begin{equation}\label{EV_berkson_porta}G(z) = (\tau-z)(1-\overline{\tau}z)p(z)
 \end{equation}
 with $\tau\in \overline{\D}$ and $p:\D\to \C$ is holomorphic with $\Re (p(z)) \geq 0$  for all $z\in \D$.
 
 By using the Cayley mapping $C:\Ha\to\D, C(z)=\frac{z-i}{z+i},$ we obtain semigroups and infinitesimal generators $H$
 on $\Ha$ with the general form
 \[ H(z) = -\frac{i}{2}(i + z)^2 (\tau-C(z))(1-\overline{\tau}C(z))p(C(z)).\]
 In particular, for $\tau=1$ we obtain $H(z)=2 i p(C(z))$, which shows that every holomorphic $H:\Ha \to \Ha \cup \R$ is an infinitesimal generator on $\Ha$.
\end{example}

In this section, we explain the relation of Loewner chains to the Loewner differential equation,
and we prove two important facts: Each element $f_{t}$ of an additive or a multiplicative Loewner chain is a univalent function,
and conversely, every univalent $F$-transform/$\eta$-transform can be embedded into an
additive/a multiplicative Loewner chain.

\section{The Loewner differential equation}\label{subsec_LPDE}

In 1923, C. Loewner introduced a differential equation for univalent functions
 to attack the so called Bieberbach conjecture. Afterwards, his ideas have been extended to more general
 settings, with recent applications in stochastic geometry (Schramm-Loewner evolution). We refer to \cite{AbateBracci:2010}
for an historical overview of Loewner theory.\\

 First we look at Loewner chains with a certain regularity property, which lead to a one-to-one correspondence
 with so called Herglotz vector fields via Loewner's differential equation.
 Most of the following definitions and statements can be found in \cite{MR2995431}.

 \begin{definition}
Let $d \in [1,\infty]$. A family $(f_{t})_{t \geq 0}$ of non-constant holomorphic
mappings $f_{t}:D\to D$ is called a
\index{Loewner chain!of order $d$}\emph{Loewner chain of order $d$} if it satisfies the conditions of Definition \ref{EV_def:evolution_family} with
$(c)$ replaced by the condition
\begin{itemize}
\item[(c')] for any $z \in D$ and any $S>0$ there exists a non-negative function
$k_{z,S} \in L^d([0,S],\mathbb{R})$ such that
$$|f_{st}(z) - f_{su}(z)| \leq \int_t^u k_{z,S}(\xi)
\,{\rm d}\xi$$ for all
$0\leq s\leq t \leq u \leq S$.
\end{itemize}
\end{definition}

\begin{example}\label{EV_not_abs_continuous}
 Let $D=\Ha$ and $f_{st}(z) = z + C(t)-C(s),$ where $C:[0,\infty)\to \mathbb{R}$ is
 continuous but not absolutely continuous. Then $(f_{0t})$ is an additive Loewner chain
 with $\mu_{t}=\delta_{C(0)-C(t)}$, and we have
$$|f_{st}(z)-f_{su}(z)|= |C(t)-C(u)|.$$ Hence $(f_{t})$ is not a Loewner chain of any order $d.$
\end{example}

Property (c') ensures that $t\mapsto f_{st}$ is differentiable almost everywhere. For a precise statement,
 we also need the following notion.

\begin{definition} A \index{Herglotz vector field}\emph{Herglotz vector field of order $d \in [1,\infty]$ on $D$} is a function
$M:D\times [0,\infty) \to \mathbb{C}$ with the following properties:
\begin{itemize}
	\item[(i)] The function $t\mapsto M(z,t)$ is measurable for every $z\in D.$
	\item[(ii)] The function $z\mapsto M(z,t)$ is holomorphic for every $t\in[0,\infty).$
	\item[(iii)] For any compact set $K \subset D$ and for all $S>0$ there exists a non-negative function
$k_{K,S} \in L^d([0,S],\mathbb{R})$ such that $|M(z,t)| \leq k_{K,S}(t)$ for all $z \in K$ and for almost
every $t \in [0,S]$.
	\item[(iv)] $M(\cdot, t)$ is an infinitesimal generator on $D$ for a.e. $t\geq0.$
\end{itemize}
We call $M$ an \index{Herglotz vector field!additive}\emph{additive Herglotz vector field} if $D=\Ha$ and, for a.e. $t\geq 0,$ $M(\cdot, t)$ has the form
\begin{equation}\label{EV_eq:4}
M(z,t)=\gamma_t + \int_\mathbb{R}\frac{1+xz}{x-z} \rho_t({\rm d}x),
\end{equation}
 where $\gamma_t\in\mathbb{R}$ and $\rho_t$ is a finite non-negative Borel measure on $\mathbb{R}$.

 We call $M$ a \index{Herglotz vector field!multiplicative}\emph{multiplicative Herglotz vector field} if $D=\D$ and, for a.e. $t\geq 0,$ $M(\cdot, t)$ has the form
\begin{equation*}
M(z,t)= -zp_t(z),
\end{equation*}
where $p_t:\D\to \C$ is holomorphic with $\Re (p_t(z)) \geq 0$  for all $z\in \D$.
\end{definition}

Now we have the following one-to-one correspondence.

\begin{theorem}
A Loewner chain $(f_{t})_{t\geq0}$ of order $d$ satisfies the Loewner partial differential equation
\begin{equation}\label{EV_Loewner}
\frac{\partial}{\partial t} f_{t}(z) = \frac{\partial}{\partial z}f_{t}(z)\cdot M(z,t) \quad \text{for a.e. $t\geq 0$, $f_{0}(z)=z\in D,$}
\end{equation}
for a Herglotz vector field $M$ of order $d$. Conversely, the unique solution to \eqref{EV_Loewner}
for a given Herglotz vector field of order $d$ is always a Loewner chain of order $d$.

Moreover, each element $f_t:D\to D$ of a Loewner chain of order $d$ is a \index{univalent function}univalent function.
\end{theorem}
\begin{proof}
In \cite{contreres+al2014}, Loewner chains consist of univalent functions by definition. However, the proof
of \cite[Theorem 3.2]{contreres+al2014} does not use this property and proves equation \eqref{EV_Loewner}. This can also be seen by looking at the family
$(f_{T-t,T-s})_{0\leq s\leq t\leq T}$ for some fixed $T>0$. It can be verified that it forms an evolution family (see Remark \ref{EV_remark_ev}) and we obtain \eqref{EV_Loewner} from \cite[Theorem 1.1]{MR2995431}. 

Conversely, by \cite[Theorems 1.11]{contreres+al2014}, the unique solution to \eqref{EV_Loewner} yields a Loew\-ner chain of order $d$ consisting of univalent functions.
\end{proof}

\begin{remark}From the relation $f_t = f_s \circ f_{st}$ we obtain
\begin{equation*}
\frac{\partial}{\partial t} f_{st}(z) = \frac{\partial}{\partial z}f_{st}(z)\cdot M(z,t) \quad \text{for a.e. $t\geq s$, $f_{ss}(z)=z\in D.$}
\end{equation*}
Furthermore, we can also differentiate $f_{st}$ w.r.t. $s$ and obtain
 \begin{equation}\label{EV_ord} \frac{\partial}{\partial s} f_{st}(z) = -M(f_{st}(z), s) \quad \text{for a.e. $s\leq t$, $f_{tt}(z)=z\in D$}.
\end{equation}
Conversely, this equation has a unique solution, which gives the transition mappings of a decreasing Loewner chain of order $d$,
see again \cite[Theorems 1.11 and 3.2]{contreres+al2014}.
\end{remark}

Our special Loewner chains now satisfy the following relationship.

\begin{proposition}
 Let $(f_{t})$ be an additive Loewner chain of order $d$.
 Then $(f_t)$ satisfies \eqref{EV_Loewner} for an additive Herglotz vector field $M$ of order $d$. 
 
Conversely, let $M$ be an additive Herglotz vector field of order $d$. Then the solution $f_{t}$ to \eqref{EV_Loewner}
is an additive Loewner chain of order $d$.
\end{proposition}
\begin{proof} ``$\Longrightarrow$'':
{}From  representation \eqref{EV_eq:2} we obtain $$\Im(f_{su}(z))=
\Im(f_{st}(f_{tu}(z))) \geq \Im(f_{tu}(z))$$ for all $0\leq s\leq t\leq u.$ So
$s\mapsto \Im(f_{st}(z))$ is non-increasing for every $z\in\Ha$. From \eqref{EV_ord} we see that
$\Im(M(z,t))\geq 0$ for almost every $t\geq 0$ and every $z\in\Ha$. Hence, $M(\cdot, t)$ has the form \eqref{EV_eq:1} for a.e. $t\geq 0.$ (See also \cite[Thm. 8.1]{MR2995431}.) 

Assume that $M'(\infty,t)>0$ for a set $I\subset [0,T]$ of positive Lebesgue measure. Then, by \cite[Thm. 1.1]{bracci+al2015}, we obtain that
$ f_{T}'(\infty) >1$, a contradiction. This proves that $M'(\infty,t)=0$ for a.e. $t\geq 0$, i.e. $M$ is an additive Herglotz vector field.

\vspace{3mm}

``$\Longleftarrow$'': We have to show that every
$f_{t}$ can be written as $f_{t}=F_{\mu_{t}}$ for a probability measure $\mu_{t}$.
 By \eqref{EV_eq:1}, we can write $f_{t}$ as
$$f_{t}(z) = A_{t} + B_{t} z + \int_\mathbb{R} \frac{1+xz}{x-z} \sigma_{t}({\rm d}x).$$
 Furthermore, as $M(\cdot, t)$ has the form \eqref{EV_eq:4} for a.e. $t\geq0,$ $M(\cdot, t)$ has a ``boundary regular null point'' at $\infty$ with dilation $0$ for a.e.\ $t\geq 0$, see \cite[Def. 2.6]{bracci+al2015} which handles the unit disk case.
 
By \cite[Thm. 1.1]{bracci+al2015}, the ``spectral function'' of $f_{t}$ at $\infty$ is equal to 0, which translates in our setting to
$$B_{t} = 1$$ for every $t\geq 0$ (note that $B_{t}=f_{t}'(\infty)$, which corresponds to $f'_{t}(\sigma)$ in
 \cite{bracci+al2015}, must be a non-negative real number by \cite[Thm. 2.2 (vi)]{bracci+al2015}). 
Hence, \eqref{EV_eq:2} implies $f_{t}=F_{\mu_{t}}$ for a probability measure $\mu_{t}$ for every $t\geq 0.$
\end{proof}

\begin{proposition}\label{EV_multiplicative_Loewner}
 Let $(f_{t})$ be a multiplicative Loewner chain of order $d$.
 Then $(f_{t})$ satisfies \eqref{EV_Loewner} for a multiplicative Herglotz vector field $M$ of order $d$.
 
Conversely, let $M$ be a multiplicative Herglotz vector field of order $d$. Then the solution $(f_{t})$ to \eqref{EV_Loewner}
is a multiplicative Loewner chain of order $d$.
\end{proposition}
\begin{proof}
By using the \index{Berkson-Porta formula}Berkson-Porta formula \eqref{EV_berkson_porta}, it is easy to see that
the property $f_{t}(0)=0$ for all $t \geq 0$ is equivalent to $\tau=0$ for a.e. generator $M(\cdot,t)$.
\end{proof}

\section{Normalized Loewner chains}

As we saw in Example \ref{EV_not_abs_continuous}, not every Loewner chain satisfies Loewner's differential equation. However,
certain normalizations guarantee the differentiability of additive and multiplicative Loewner chains.

\begin{proposition}\label{EV_normal_add}
 Let $(f_{t})$ be an \index{Loewner chain!additive}additive Loewner chain such that the first and second moments
 of all $\mu_{t}$ exist with
 \begin{equation*}
\text{$\int_\mathbb{R} x \mu_{t}({\rm d}x)=0$ and $\int_\mathbb{R} x^2 \mu_{t}({\rm d}x)=t$ for all $t\geq0.$}
 \end{equation*}
 Then $(f_{t})$ satisfies
 \eqref{EV_Loewner} for an \index{Herglotz vector field!additive}additive Herglotz vector field $M$ of the form
$$
M(z,t)= \int_\mathbb{R}\frac1{u-z}\,\tau_t({\rm d}u),
$$
where $\tau_t$ is a probability measure for a.e. $t\geq0.$

Conversely, let $M$ be a  Herglotz vector field of the above form. Then the solution $(f_{t})$ to \eqref{EV_Loewner}
is an additive Loewner chain having the above normalization.
\end{proposition}
\begin{proof}
See \cite[Prop. 3.6]{monotone}. We note that the normalization implies that
\begin{equation*}\label{EV_regularity2}|f_{s}(z) - f_{t}(z)| \leq \frac{t-s}{\Im(z)},\end{equation*}
for all $0\leq s \leq t$ and $z\in \Ha,$ see \cite[p. 1214]{MR1201130}.
Hence, $(f_{t})$ is an additive Loewner chain of order $\infty$.
\end{proof}

\begin{proposition}\label{EV_normal_mult}
 Let $(f_t)$ be a \index{Loewner chain!multiplicative}multiplicative Loewner chain such that
 \begin{equation*}
\text{$f'_{t}(0)=\int_{\T} x \mu_{t}({\rm d}x)=e^{-t}$ for all $t\geq 0.$}
\end{equation*}
Then $(f_{t})$ satisfies
 \eqref{EV_Loewner} for a \index{Herglotz vector field!multiplicative}multiplicative Herglotz vector field $M(z,t) = -zp_t(z)$ with $p_t(0)=1$ for a.e. $t\geq0.$ 
 
 Conversely, let $M$ be a Herglotz vector field of the above form. Then the solution $(f_t)$ to \eqref{EV_Loewner}
is a multiplicative Loewner chain having the above normalization.
\end{proposition}
\begin{proof}
Let $0\leq s\leq t\leq u$. The normalization implies that
$$|f_{tu}(z)-z| \leq
2|z| \frac{1+|z|}{1-|z|}(1-e^{u-t}),$$
see the proof of Lemma 6.1 in \cite{P75}.
Let $K=r\overline{\D}\subset \D$ for some $r\in[0,1).$ As $(f_{st})_{0\leq s\leq t}$ is a normal family, we find
a Lipschitz constant $L$ for this family w.r.t. the set $K$. Let $z\in K$. Then
also $f_{tu}(z)\in K$ because of the Schwarz lemma and we obtain
$$ |f_{su}(z)-f_{st}(z)|= |f_{st}(f_{tu})-f_{st}(z)|\leq L|f_{tu}(z)-z| \leq
2L|z| \frac{1+|z|}{1-|z|}(1-e^{u-t}).$$

So $(f_{t})$ is a Loewner chain of any order $d.$
The statement is now an easy consequence of Proposition \ref{EV_multiplicative_Loewner}.
By looking at the first coefficient of the power series expansion of $f_t(z)$ in \eqref{EV_Loewner}, we find that
\begin{equation}\label{eq:diff}
\text{$\frac{\partial}{\partial t}f_{t}'(0) = -f_{t}'(0)\cdot p_t(0)$\; for a.e. $t\geq 0,$\; i.e.
$f_{t}'(0)=e^{-\int_0^t p_\tau(0)\, {\rm d}\tau}.$} 
\end{equation}
The converse statement follows from \eqref{eq:diff} and Proposition \ref{EV_multiplicative_Loewner}.
\end{proof}

\section{Univalent functions}\label{sec_EV_univalence}\index{univalent function|(} 

Each element of a Loewner chain of order $d$ is a univalent function. A general Loewner chain does not satisfy \eqref{EV_Loewner}, but a suitable reparametrization yields univalence also for this case.

\begin{theorem}\label{EV_univalence}${}$
 Let $(f_{t})$ be a \index{Loewner chain}Loewner chain.
        Then every transition mapping $f_{st}$, in particular every $f_t=f_{0t}$, is a univalent function.
\end{theorem}
\begin{proof} 
Because of $f_t = f_s \circ f_{st}$, it is sufficient to prove that $f_t$ is univalent for all $t\geq0.$ 

\vspace{3mm}

a) Assume that $(f_t)$ is a multiplicative Loewner chain. Let $a_{t}:=f'_{t}(0).$ Due to the Schwarz lemma, we have $|a_{t}|\leq 1$ and, as $f_t = f_s \circ f_{st}$,
$t\mapsto |a_t|$ is non-increasing. Furthermore, as $t\mapsto f_t$ is continuous, also $t\mapsto a_t$ is continuous and we conclude that  $a_{t}\not=0$ for all $t\in[0,\varepsilon]$ and some $\varepsilon>0$. 
 
First, assume that $a_{t}\not=0$ for all $t\geq 0.$ Then we have $0<|a_t|\leq 1$ for all $t\geq 0$ and there exists a uniquely determined continuous function $C:[0,\infty)\to \{z\in\C \,|\, \Re (z)\leq 0\}$ with $C(0)=0$ such that $a_{t}=e^{C(t)}$.
It is easy to see that $$g_{t}(z):= f_{t}(e^{-i\Im(C(t))}z)$$ is also a multiplicative Loewner chain with
$$g'_{t}(0)=e^{\Re(C(t))}.$$ The function $t\mapsto \Re(C(t))$ is
non-increasing and continuous. Note that $\Re(C(t))=\Re(C(s))$, $s\leq t$, implies that $g_{t}=g_{s}$, for  $g_t = g_s \circ g_{st}$ with $g_{st}:\D\to\D$, $g_{st}(0)=0, g'_{st}(0)=1$, i.e. $g_{st}$ is the identity by the Schwarz lemma.
        
We can reparametrize $g_{t}$ to $h_{s}:=g_{\tau(s)}$ such that $h'_{s}(0)=e^{-s}$ for all $s \in [0,S)$ for some $0<S \leq \infty$, where $\tau(s)$ is defined by 
$$
\tau(s) = \inf \{ t \geq0: \Re[C(t)]=-s\}, 
$$  
which is a strictly increasing, possibly discontinuous function. The reparametrization\\ $(h_{s})_{s\in[0,S)}$ is (part of) a multiplicative Loewner chain with the normalization from
Proposition \ref{EV_normal_mult}. Hence, each $h_{s}$ is univalent, which implies that each $f_{t}$ is univalent.    

No assume that $a_{\tau}=0$ for some $\tau > 0$ and
$a_{t}\not=0$ for $t< \tau.$  The previous case implies that $f_{t}$ is univalent for all $t<\tau.$
Hence, $f_{\tau}=\lim_{t\uparrow \tau}f_{t}$ is the limit of univalent
self-mappings of $\D$ fixing $0$. It follows that $f_{\tau}(z)\equiv 0$. This is a contradiction as all elements of a Loewner chain are non-constant by Lemma \ref{non_constant}. 
				
\vspace{3mm}
				
b) Now we consider the general case. If $D=\Ha$, then we can use the Cayley transform $I:\Ha\to \D,$ $I(z)=\frac{z-i}{z+i},$ to transfer the problem to the unit disk, i.e. we define
$F_{st}:= I \circ f_{st} \circ I^{-1},$ which gives transition mappings of a Loewner chain on $\D$.
Hence, we may assume that $(F_t)$ is a Loewner chain on $\D$ with transition mappings $F_{st}$. 
Next we use an idea from Proposition 2.9 in \cite{contreres+al2010}. Fix some $T>0.$ We define
$$ a(t) := F_{tT}(0),
\quad h_t(z):=\frac{z+a(t)}{1+\overline{a(t)}z},$$
for $t\geq 0,$ $z\in \D.$ Note that $h_t$ is an automorphism of $\D$ mapping
$0$ onto $a(t)$. 

Define $(G_{st})_{0\leq s\leq t\leq T}:= (h^{-1}_{s}\circ F_{st} \circ h_t)_{0\leq s\leq t\leq T}$ and $G_t=G_{0t}.$
Then $(G_t)$ is (a part of) a Loewner chain on $\D$ with
\[ G_t(0)=  (h^{-1}_{0}\circ F_t \circ h_t)(0)= (h^{-1}_{0}\circ F_t)(F_{tT}(0))=h^{-1}_{0}(F_T(0))=0.\]
Hence, $(G_t)$ is a multiplicative Loewner chain and a) implies that every $G_t$, $t\in[0,T]$, is univalent.
As $T>0$ can be chosen arbitrarily large, we conclude that every $F_t$ is univalent.
\end{proof}

%

Every univalent $f:\D\to\D$ maps $\D$ conformally onto a simply connected subdomain
of $\D$. Conversely,
if $D\subseteq \D$ is a simply connected subdomain with $0\in D,$ we
find a unique conformal mapping $f:\D\to D$
with $f(0)=0$ and $f'(0)>0$ due to the Riemann mapping theorem.
This mapping is always a univalent $\eta$-transform, see Lemma \ref{characterizationeta}.

The situation for $F$-transforms is more complicated since we need a normalization at infinity. In order that a simply connected domain $\Omega \subset \C$ is the range of some univalent $F$-transform, at least $\Omega$ must contain truncated cones $\Gamma_{\lambda, M_\lambda}$ for all $\lambda>0$, where $M_\lambda>0$ is a function of $\lambda$ (see Section \ref{intro_convolutions}). We do not know whether these conditions are necessary and sufficient.   

We provide some sufficient conditions. We call a closed Jordan curve in the Riemann sphere \emph{Dini-smooth} if it has a parametrization with non-zero and Dini-continuous derivative; see \cite[p.46, 48]{MR1217706}.

\begin{theorem}\label{ftransform_images}${}$
\begin{enumerate}[\rm (a)]
 \item Let $\Omega \subseteq \Ha$ be a simply connected domain and assume that there is a Dini-smooth closed Jordan curve $C$ in the Riemann sphere such that $\infty\in C$
 and one component $\Omega_0$ of $\hat{\C}\setminus C$ is contained in $\Omega.$ Then there exists a probability measure
 $\mu$ on $\R$ with \index{F-transform@$F$-transform}$F_\mu(\Ha)=\Omega.$
 \item Let $\Omega \subseteq\Ha$ be a simply connected domain such that $\Ha\setminus \Omega$ is
a bounded set. Then there exists a unique probability measure $\mu$ on $\R$ with mean $0$
and compact support
such that $F_\mu(\Ha)=\Omega.$
\end{enumerate}
\end{theorem}
\begin{proof}
(a) This statement is basically shown in the proof of \cite[Theorem 6.1]{MR2719792}.
For the sake of completeness, we extract the necessary steps.

 We transfer the problem to the unit disk, i.e. we set $D=I(\Omega)$ and $D_0=I(\Omega_0)$,
 where $I(z)=\frac{z-i}{z+i}$ is the Cayley map. The boundary point $\infty$ now corresponds to the point $1$.
 
 First, we find a conformal mapping $f_0:\D\to D.$
 Next, let $\gamma:[0,1)\to D$ be a continuous curve with $\gamma(t)\to 1$
 as $t\to 1.$ By \cite[Prop. 2.14]{MR1217706}, the curve $f_0^{-1}\circ \gamma$
 is a curve in $\D$ with $f_0^{-1}(\gamma(t))\to p\in\partial\D$ as $t\to 1.$ By precomposing
 $f_0$ with a rotation, we can assume that $p=1.$
 
 The Schwarz-Pick lemma implies $(1-|z|^2)|f_0'(z)|\leq 1-|f_0(z)|^2$ for all $z\in\D.$
 Hence, $f_0$ is a normal function
 and the Lehto-Virtanen theorem  (\cite[Section 4.1, p.71]{MR1217706}) implies that
 $f_0(z)\to 1$ as $z\to 1$ non-tangentially. The Julia-Wolff lemma (\cite[Prop. 4.13]{MR1217706})
 implies that $f_0$ has an angular derivative $f_0'(1)=c\in (0,\infty],$ i.e.
 $(f_0(z)-1)/(z-1)\to c$ as $z\to 1$ non-tangentially.
 
 Actually, $c$ is finite:
 Let $f_1:\D\to D_0.$ By the same arguments, we can assume that $f_1(z)\to 1$
 as $z\to 1$ non-tangentially and the angular derivative$f_1'(1)\in (0,\infty]$
 exists.
Now we use the Dini-smoothness of the curve $C$. By  \cite[Thm. 3.5]{MR1217706}, $f_1'(1)<\infty$ and 
 \cite[Thm. 4.14]{MR1217706} implies that $f_0'(1)<\infty.$
 By going back to the upper half-plane, we find that $f_2:=I^{-1}f\circ I$
 maps $\Ha$ conformally onto $\Omega$ and has the form $\eqref{EV_eq:1}$ with
 $a=\frac1{f_0'(1)}\not=0$ and $b\in\R$. Hence, the map $f(z):=f_2(z/a)$
 maps $\Ha$ conformally onto $\Omega$ and has the form \eqref{EV_eq:2}. We conclude that
 $f=F_\mu$ for some $\mu\in \cP(\R).$ 
 
 \vspace{3mm}

(b) By \cite[Prop. 3.36]{lawler05}, there exists a unique conformal mapping
  $f:\Ha\to \Omega$ with $$f(z)=z - \frac{c}{z} + \mathcal{O}(z^{-2})$$ as
  $z\to\infty,$ $c\geq0.$ By \eqref{EV_eq:2}, we have
  $f=F_\mu$ for some $\mu\in \cP(\R)$. 
  
  Consider the complement $B=\Ha\setminus\Omega$. As $B\cup\{0\}$ is bounded, we find a disc $R\cdot\D$ 
  such that $B\cup\{0\}\subset R\cdot\D$. 
  Now consider the domain $H=f^{-1}(\Ha\setminus \overline{R\D})$. The set $f^{-1}(\partial (R\D) \cap \Ha)$ 
  is clearly a simple curve in $\Ha$. Due to \cite[Prop. 2.14]{MR1217706}, its closure intersects $\R$ at exactly two points 
  $a,b$, $a<b$. Hence, $\partial H = f^{-1}(\partial (R\D))\cup(-\infty,a]\cup[b,\infty)$.
 
   We can now apply  \cite[Theorem 2.6]{MR1217706} to see that $f$ extends continuously to $\overline{H}$ with 
  $f((-\infty,a]\cup[b,\infty))\subset \R\setminus\{0\}$. 
  (In \cite[Theorem 2.6]{MR1217706}, the simply connected domains are assumed 
  to be bounded, but we can simply regard $H\cap R_2\cdot \D$ for any $R_2$ large enough.)
  Hence, $1/f$, which is the Cauchy transform of $\mu$, maps $(-\infty,a]\cup[b,\infty)$ 
  into $\R$ and the Stieltjes-Perron inversion formula implies that $\mu$ has compact support. Due to Lemma \ref{first_moment},
  the first moment of $\mu$ is $0$.
\end{proof}

\begin{remark}\label{rm_capacity}
In (b), $\mu$ has finite variance $\sigma^2=c$.
This value is also called the \index{half-plane capacity}\emph{half-plane capacity} of the ``hull''
$\Ha\setminus \Omega,$ see \cite[Section 3.4]{lawler05}. It has a more or less geometric interpretation, see
\cite{LLN09}. 

More generally, a set $A\subset \Ha$, such that $\Ha\setminus A$ is a simply connected domain, is said to have \emph{finite half-plane capacity $c\geq0$} if there exists a conformal mapping $G:\Ha\to \Ha\setminus A$
with
\[ G(z) = z - \frac{c}{z} + {\scriptstyle\mathcal{O}}(1/|z|) \]
as $z\to\infty$ non-tangentially in $\Ha$. Let $A=\Ha \setminus F_\mu(\Ha)$ for a probability measure $\mu$ on $\R.$ Then
$\mu$ has finite variance $\sigma^2$ if and only if $A$ has finite half-plane capacity $\sigma^2.$
This can be shown by regarding the mapping $G(z)=F_\mu(z-m)$, where $m$ is the first moment of
$\mu$, together with Lemma \ref{Julia}.
\end{remark}

\begin{problem} Characterize the domains $\Omega\subseteq \Ha$
appearing as image domains of univalent $F$-transforms.
\end{problem}

\section{Embeddings}\label{EV_section_embeddings}

Finally we prove the following statements, which are converse to Theorem \ref{EV_univalence}.

\begin{theorem}\label{EV_embed_F}${}$
\begin{enumerate}[\rm(a)]
 \item Let $\mu$ be a probability measure on $\R$ such that \index{F-transform@$F$-transform}$F_\mu$ is univalent.
Then there exists an \index{Loewner chain!additive}additive Loewner chain $f_{t}$ such that $f_{1} = F_\mu.$
\item Let $\mu$ be a probability measure on $\R$  such that $F_\mu$ is univalent and
$$\int_\mathbb{R} x\, \mu({\rm d}x)=0, \int_\mathbb{R} x^2\, \mu({\rm d}x)=:T<\infty.$$
 Then there exists an additive Loewner chain $f_{t}$ satisfying the normalization from Proposition
 \ref{EV_normal_add} such that  $f_{T} = F_\mu.$
\end{enumerate}
\end{theorem}
\begin{proof}
(a) Theorem 1.2 in \cite{bracci+al2015}, with $\Lambda(t)\equiv 0,$ implies that
we can write $F_\mu = f_{0,1}$ where $\{f_{st}\}_{0\leq s\leq t\leq 1}$ is an evolution family in the sense of Remark \ref{EV_remark_ev} and
\begin{itemize}
	\item[(i)] $f_{st}$ has a boundary regular fixed point at $\infty$ for all $0\leq s \leq t.$
	\item[(ii)] $f_{st}'(\infty) = 1$ for all $0\leq s \leq t.$
\end{itemize}
(Note that \cite[Thm. 1.2]{bracci+al2015} only gives $|f_{st}'(\infty)| = 1$. 
However, $f_{st}'(\infty)$ must be non-negative as $\infty$ is a fixed point of $f_{st},$
see again \cite[Thm. 2.2 (vi)]{bracci+al2015}.) 
We conclude that every $f_{st}$ has the form \eqref{EV_eq:2}. 
Finally, the family $(f_t)_{t\geq 0}$ with $f_t = f_{1-t,1}$ for $t\in[0,1]$, $f_t = f_{0,1}$ for $t>1$, is an additive Loewner chain with $f_1 = f_{0,1} = F_\mu.$ 

\vspace{3mm}

(b) This statement follows in a similar way by using \cite[Theorem 5]{MR1201130}.
\end{proof}

\begin{theorem}\label{embed_multiplicative}${}$
\begin{enumerate}[\rm(a)]
\item Let $\mu$ be a probability measure on $\T$ such that \index{eta-transform@$\eta$-transform}$\eta_\mu$ is univalent. Then there exists a \index{Loewner chain!multiplicative}multiplicative Loewner chain $f_{t}$
 such that $f_{1} = \eta_\mu.$
\item Let $\mu$ be a probability measure on $\T$ such that $\eta_\mu$ is univalent and
$$\eta_\mu'(0)=\int_\mathbb{\T} x\, \mu({\rm d}x)=e^{-T}, T>0.$$
  Then there exists a multiplicative Loewner chain $f_{t}$ satisfying the normalization from Proposition
 \ref{EV_normal_mult} and $f_T = \eta_\mu.$
\end{enumerate}
\end{theorem}
\begin{proof} We start with (b). 

\vspace{3mm}

(b) The function $G:=e^T\eta_\mu$ is univalent and satisfies $G(0)=0,$ $G'(0)=1$, i.e. it belongs to the class $S$,
see \cite[p. 11]{P75}. By Problem 3 in \cite[Section 6.1]{P75}, there exists a a family $(g_t)_{t\geq 0}$ of
univalent functions $g_t:\D\to\C$
with $g_t(0)=0, g_t'(0)=e^t$ and $g_s(\D)\subseteq g_t(\D)$ whenever $s\leq t$ (an increasing
Loewner chain) such that $g_0=G$ and $g_t(z)=e^tz$ for all $t\geq T.$
We define $f_t = e^{-T}\cdot g_{T-t}$ for $t\in[0,T]$ and $f_t=f_T(e^{T-t}z)$ for all $t > T$. Then $(f_t)$ is a normalized multiplicative
Loewner chain with $f_T=e^{-T} g_{0}=e^{-T}G=e^{-T} e^T\eta_\mu=\eta_\mu.$ 

For the sake of completeness, we include the solution of \cite[Section 6.1, Problem 3]{P75}, which
is similar to the proof of Theorem 6.1 in \cite{P75}. 
First, define $\eta_n:=e^{T+1/n}\eta_\mu(e^{-1/n}z)$. Then $\eta_n$ belongs to the class $S$ and
$\eta_n(\D)$ is bounded by an analytic Jordan curve contained in the disk $e^{T+1/n}\D$. 
Denote by $R_n$ the ring domain whose boundary is given by
$\partial \eta_n(\D)$ and $\partial (e^{T+1/n}\D).$ Then we find a conformal mapping
$g_n:R_n\to \{z\in\C\,|\, 1<|z|<r_n\}$ for some $r_n>1.$ For $0\leq t< \log r_n,$ denote by
$D_{t,n}$ the simply connected domain
bounded by the analytic curve $g_n^{-1}(e^t\partial \D).$ Denote by $f_{t,n}:\D\to D_{t,n}$
the conformal mapping
with $f_{t,n}(0)=0, f_{t,n}'(0)>0.$ After a reparametrization, we have
that $f_{t,n}'(0)=e^t$ for all $0\leq t < T+1/n$. For $t\geq T+1/n$, we define $f_{t,n}(z)=e^tz.$

Now it is easy to see that $(f_{t,n})_{t\geq 0}$ is an increasing Loewner chain with
$f_{0,n}=\eta_n.$ By \cite[Lemma 6.2]{P75}, the sequence of Loewner chains contains a subsequence that
converges locally uniformly in $\D$, pointwise w.r.t. $t,$ to a Loewner chain $(f_t)_{t\geq 0}.$
We have $f_0(z)=\lim_{n\to\infty} f_{0,n}(z)=\lim_{n\to\infty} e^{T+1/n}\eta_\mu(e^{-1/n}z)= e^T\eta_\mu(z)$ and
$f_{t}(z)=e^tz$ for all $t\geq T.$ 

\vspace{3mm}

(a) If $\eta_\mu$ is the identity, we can simply take the constant Loewner chain $f_t(z)=z$.\\
 If $\eta_\mu$ is not the identity, we find $\alpha\in \R$, $T> 0,$ such that $e^{i\alpha }\eta_\mu'(0)=e^{-T}.$ By (b),
$\eta_\mu(e^{i\alpha} z)=g_{T}(z)$
for a multiplicative Loewner chain $g_{t}$ with $g_{t}'(0)=e^{-t}.$
Let $$f_{t}(z)= g_{t\cdot T}(e^{-i\alpha t}z).$$
It is easy to see that $f_{t}$ is a multiplicative Loewner chain and
$f_{1}(z)=g_T(e^{-i\alpha}z)=\eta_\mu(z).$
\end{proof}
\index{univalent function|)}



\chapter[SAIPs, Markov Processes and Loewner Chains]{Additive Monotone Increment Processes, Classical Markov Processes, and Loewner Chains}\label{section_processes}

In this section we establish a bijection between SAIPs, some class of classical Markov processes and additive Loewner chains. We will encounter many unbounded operators. For two linear operators $X, Y$ on a Hilbert space with domains $\Dom(X)$ and $\Dom(Y)$ respectively, we always assume that the domain of $X+Y$ is the intersection $\Dom(X)\cap \Dom(Y)$ and the domain of $XY$ consists of all vectors $u \in \Dom(Y)$ such that $Y u \in \Dom(X)$ unless specified otherwise. By $X \subset Y$ we mean that $\Dom(X)\subset \Dom(Y)$ and $X=Y$ on $\Dom(X)$.

\section{Quantum stochastic processes}

A non-commutative or quantum stochastic process is a mapping 
$$
[0,\infty) \ni t \mapsto X_t, 
$$
where the values are linear operators on a Hilbert space with some domains. In this paper we focus on  densely defined normal operators or essentially self-adjoint operators and call them normal processes and essentially self-adjoint processes, respectively. For a classical $\C$-valued stochastic process $(Y_t)$ on a underlying probability space $(\Omega, \mathcal F, \mathbb P)$ we can associate for each $t\geq0$ the multiplication operator $X_t\colon f \mapsto Y_t f$ with $\Dom(X_t) = \{f \in L^2(\Omega,\mathcal F, \mathbb P):Y_t f \in L^2(\Omega,\mathcal F, \mathbb P)\}$, which is a densely defined normal operator. The family $(X_t)$ is a standard construction of a non-commutative stochastic process from a classical one.

We introduce an equivalence relation for normal processes. 

\begin{definition}\label{def:equiv_saip}\index{equivalence!of quantum stochastic processes}
Let $(H, \xi)$ and $(H', \xi')$ be concrete quantum probability spaces and let $(X_t)$ and $(X_t')$ be two normal processes on $(H, \xi)$ and $(H', \xi')$ respectively. Then $(X_t)$ and $(X_t')$ are equivalent if the finite dimensional distributions are equal, namely 
\begin{equation}\label{eq:equiv_saip}
\langle \xi, f_1(X_{t_1}) \cdots f_n(X_{t_n})\xi\rangle_H =\langle \xi', f_1(X_{t_1}') \cdots f_n(X_{t_n}')\xi'\rangle_{H'} 
\end{equation}
for all $n\in \N, t_1,\dots, t_n \geq0, f_1,\dots, f_n \in C_b(\C)$. For essentially self-adjoint processes the same definition is adopted by taking the closure.  
\end{definition}

\begin{remark}\label{rem:equiv_saip} For self-adjoint processes, it is enough to take $f_i$ from $C_b(\R)$ to define the finite dimensional distributions. 
Moreover, it is enough to take $f_i$ to be continuous functions on $\R$ with compact support, or even smooth functions with compact support. To see this, for a self-adjoint operator $X$ a general function $f \in C_b(\R)$ can be approximated by continuous functions $g_n$ with compact support such that $g_n(X) \to f(X)$ strongly as $n\to\infty$. Indeed, we take $h_n$ to be a continuous function such that $0 \leq h_n \leq 1$, $h_n=1$ on $[-n,n]$ and $h_n=0$ on $\R \setminus(-2n,2n)$. The functions $g_n=f h_n$ then converge to $f$ pointwise. Denoting by $E(x)$ the spectral projection of $X$ over $(-\infty, x]$, by the dominated convergence theorem one has
\begin{equation*}
\|g_n(X) u -f(X) u\|^2 = \int_\R |g_n(x)-f(x)|^2 \,{\rm d}\langle u,E(x)u\rangle \to 0
\end{equation*}
for all $u \in H$. Finally, since continuous functions with compact support are bounded and uniformly continuous, we can approximate them uniformly by smooth functions with compact support using mollifiers. 
\end{remark}


\section[From SAIPs to Loewner chains]{From additive monotone increment processes to additive Loewner chains}\label{section_additive_processes}\index{monotone increment process!additive (SAIP)|(}

We recall and extend Definition \ref{def_saip0}. 

\begin{definition}\label{def_saip}
Let $(H,\xi)$ be a concrete quantum probability space and $(X_t)_{t\ge 0}$ an essentially self-adjoint process on $H$ with $X_0=0$.  We
call $(X_t)$ an \emph{(essentially) self-adjoint additive
monotone increment process (SAIP)} if the following conditions are satisfied.
\begin{itemize}
\item[(a)] The increment $X_t-X_s$ with domain $\Dom(X_t)\cap \Dom(X_s)$ is essentially self-adjoint for every $0\le s\le t$. 
	\item[(b)] $\Dom(X_s)\cap \Dom(X_t) \cap \Dom(X_u)$ is dense in $H$ and is a core for the increment $X_u-X_s$ for every $0\le s\le t \le u$. 
	\item[(c)] The mapping $(s,t)\mapsto \mu_{st}$ is continuous w.r.t.\ weak convergence, where $\mu_{st}$ denotes the distribution of the increment $X_t-X_s$.
	\item[(d)]The tuple
    \[
(X_{t_1},X_{t_2}-X_{t_1},\ldots,X_{t_n}-X_{t_{n-1}})
\]
is \index{monotone independence}monotonically independent for all $n\in\mathbb{N}$ and all $t_1,\ldots,t_n\in\mathbb{R}$ s.t.\ $0\le t_1\le t_2\le\cdots\le t_n$.
\end{itemize}
Furthermore if $X_t-X_s$ has the same distribution as $X_{t-s}$ for all $0\leq s \leq t$ (the condition of \emph{stationary increments}), then $(X_t)_{t\geq0}$ is called a \index{monotone L\'evy process!additive}\emph{monotone L\'evy process.}
We call $(X_t)$ \emph{normalized} if $\xi\in\Dom(\overline{X_t})$,  $\langle\xi,\overline{X_t}\xi\rangle=0$ and  $\|\overline{X_t}\xi\|^2=t$ for all $t\ge 0$. 
\end{definition}

\begin{remark}
The assumption (b) is needed in the proofs of Theorem \ref{from_add_process_to_Loewner_chains} and Theorem \ref{thm:equiv_saip}. The reason for this requirement is that we want to use certain self-adjointness regarding the decomposition $X_u -X_s \supset (X_u-X_t)+(X_t-X_s)$, i.e., we have equality on the common domain $\Dom(X_s)\cap \Dom(X_t)\cap \Dom(X_u)$.
\end{remark}

We establish the following result. Note that the essential self-adjointness of $X+Y$ is a consequence if $\xi$ is assumed to be cyclic regarding $X$ and $Y$ \cite{franz07b}. 

\begin{lemma}\label{lem:mon-conv}\index{F-transform@$F$-transform}\index{monotone convolution!additive $\rhd$}\index{monotone independence}
Suppose that $(X,Y)$ is a pair of monotonically independent essentially self-adjoint operators on a concrete quantum probability space $(H,\xi)$ such that $X+Y$ is essentially self-adjoint. Then 
$F_{X+Y} = F_X \circ F_Y$. 
\end{lemma}
\begin{proof} 
Let $\{f_n\}$ be a sequence of bounded continuous $\R$-valued functions such that $f_n(0)=0$, $|f_n(x)| \leq |x|$ and $f_n(x) \to x$ for all $x\in\R$. The pair of bounded self-adjoint operators $(f_n(\overline{X}),f_n(\overline{Y}))$ is monotonically independent, and hence $F_{f_n(\overline{X})+f_n(\overline{Y})} =F_{f_n(\overline{X})} \circ F_{f_n(\overline{Y})}$ by Muraki's formula \cite[Theorem 3.1]{M00}. The corresponding arguments in Remark \ref{rem:equiv_saip} show that $\lim_{n\to \infty} f_n(\overline{X}) u = \overline{X} u$ for all $u \in \Dom(\overline{X})$, and hence by \cite[Theorem 1.41]{Ara18} the resolvents of   $f_n(\overline{X})$ converge strongly to that of $\overline{X}$. This implies that $F_{f_n(\overline{X})}(z) \to F_X(z)$ as $n\to\infty$ for every $z\in \C\setminus\R$.  
Similarly, $F_{f_n(\overline{Y})}(z) \to F_Y(z)$. 

A similar argument works for $X+Y$. Namely, for $u \in \Dom(X)\cap \Dom(Y)$ it follows that
$$
\|(f_n(\overline{X}) +f_n(\overline{Y}) - \overline{(X+Y)} )u\| \leq \|(f_n(\overline{X}) - \overline{X})u\| + \|(f_n(\overline{Y}) - \overline{Y})u\| \to 0
$$ 
as $n\to\infty$. Therefore, the resolvents of  $f_n(\overline{X})+f_n(\overline{Y})$ converge strongly to that of $\overline{X+Y}$, and hence 
\[F_{f_n(\overline{X})+f_n(\overline{Y})} (z) \to F_{X+Y}(z).\] 
\end{proof}

\begin{theorem}\label{from_add_process_to_Loewner_chains}
Let $(X_t)_{t\ge 0}$ be a SAIP in a concrete quantum probability space $(H,\xi)$. For $0\le s\le t$, let $F_{st}$ be the $F$-transform of the increment $X_t-X_s$, i.e.
\[
\frac{1}{F_{st}(z)} = \left\langle \xi, \left\{z-\overline{(X_t-X_s)}\right\}^{-1}\xi \right\rangle, 
z\in \Ha,
\]
and let $F_t=F_{0t}$. Then $(F_{t})_{t\geq 0}$ is an \index{Loewner chain!additive}additive Loewner chain with transition mappings
$(F_{st})_{0\leq s\leq t}.$
\end{theorem}
\begin{proof} Let $0 \leq s \leq t \leq u$ and let $X_{st} = X_t-X_s$. Let $D=\Dom(X_s)\cap \Dom(X_t)\cap \Dom(X_u)$. The increments satisfy $X_{su}|_D= X_{st}+X_{tu}$ and, by definition, $\overline{X_{su}|_D}=\overline{X_{su}}$ is self-adjoint. Lemma \ref{lem:mon-conv} shows that $F_{su}(z) = F_{st}(F_{tu}(z))$. 
As $X_{ss}=0$, we have $F_{ss}(z)\equiv z$ and the weak continuity of
 $(s,t)\mapsto \mu_{st}$ implies continuity of $(s,t)\mapsto F_{st}$
by Lemma \ref{lemmaconvergence} (3).
\end{proof}

For two SAIPs $(X_t)$ and $(X_t')$ consisting of bounded operators, it is easy to show that they are equivalent if and only if their distributions of increments coincide. This is because we can take polynomials $f_i$ in \eqref{eq:equiv_saip}. For example, for $t \geq s \geq0$ we can proceed as 
$$
\langle\xi, X_s X_t X_s \xi \rangle = \langle\xi, X_s (X_t-X_s + X_s) X_s \xi \rangle= \langle\xi, (X_t-X_s) \xi \rangle\langle\xi, X_s^2 \xi \rangle +\langle\xi, X_s^3 \xi \rangle, 
$$
and the RHS can be computed via the distributions of increments only. For general unbounded operators, the above arguments are not valid and so we need a trick.  

\begin{theorem}\label{thm:equiv_saip}\index{equivalence!of quantum stochastic processes}  Two SAIPs are equivalent if and only if the distributions of increments coincide. 
\end{theorem}
\begin{proof} We take SAIPs $(X_t)$ on $(H,\xi)$ and $(X_t')$ on $(H', \xi')$. 

The only if part: let $F_{st}$ and $H_{st}$ be the reciprocal Cauchy transforms of $X_t-X_s$ and $X_t'-X_s'$ respectively. We take $n=1$ and $f_1(x) = 1/(z-x)$ in \eqref{eq:equiv_saip} to get $F_{0t}=H_{0t}$. By Theorem \ref{from_add_process_to_Loewner_chains} we obtain $F_{st}=F_{0s}^{-1}\circ F_{0t} = H_{0s}^{-1}\circ H_{0t}=H_{st}$, and hence the distributions of $X_t-X_s$ and $X_t'-X_s'$ are equal. 

For the if part, 
we use Trotter's product formula (see \cite{Tro59} or \cite{Che68} for proofs, and \cite[Theorem VIII.31]{RS80} for further information)
$$
e^{i z \overline{X_t}} = e^{iz \overline{\overline{X_t-X_s} +  \overline{X_s}}} = \text{s-}\lim_{N\to\infty} (e^{iz (\overline{X_t-X_s})/N}e^{iz \overline{X_s}/N})^N, \qquad z \in \R, 0 \leq s \leq t. 
$$
The identity $\overline{X_t} =  \overline{\overline{X_t-X_s} +  \overline{X_s}}$ is used above, and it can be proved as follows.  
Since $X_t|_{\Dom(X_t)\cap \Dom(X_s)}$ is essentially self-adjoint by assumption, $\overline{X_t-X_s}+\overline{X_s}$ is also essentially self-adjoint because the range of $\overline{X_t-X_s}+\overline{X_s} \pm i$ contains that of $X_t|_{\Dom(X_t)\cap \Dom(X_s)} \pm i$ and the latter is dense in $H$; see Corollary to \cite[Theorem VIII.3]{RS80}. By the uniqueness of the self-adjoint extension of $X_t|_{\Dom(X_t)\cap \Dom(X_s)}$, we conclude that $\overline{X_t} =\overline{\overline{X_t-X_s} +  \overline{X_s}}=  \overline{X_t|_{\Dom(X_t)\cap \Dom(X_s)}}$. 

 For example, if $t_1 \leq  t_3 \leq t_2$, then we have 
\begin{align*}
&\langle \xi, e^{i z_1 \overline{X_{t_1}}}  e^{i z_2 \overline{X_{t_2}}} e^{iz_3 \overline{X_{t_3}}}\xi\rangle \\
&\qquad= \lim_{N\to\infty} \langle \xi, e^{i z_1 \overline{X_{t_1} }}  (e^{i z_2 (\overline{X_{t_2} -X_{t_1}})/N} e^{iz_2 \overline{X_{t_1}}/N})^N (e^{iz_3 (\overline{X_{t_3}-X_{t_2}})/N}e^{iz_3 \overline{X_{t_2}}/N})^N\xi\rangle. 
\end{align*}
The product of operators on the RHS can be written as a polynomial of the three bounded operators $e^{iz_3 \overline{(X_{t_3}-X_{t_2})}/N} -I, e^{iz_2 \overline{(X_{t_2}-X_{t_1})}/N} -I, e^{iz_1 \overline{X_{t_1}}} -I$. Then we can use the monotone independence (since the function $e^{izx}-1$ vanishes at $0$) to factorize the inner product, and then compute the value in terms of the distributions of the increments $X_{t_i}-X_{t_{i-1}},i=2,3$ and of $X_{t_1}$. A similar idea can be used for the general case and thus for two SAIPs $(X_t)$ and $(X_t')$ with the same distributions of increments the equation \eqref{eq:equiv_saip} holds for all functions $f_i$ of the form $f_i(x)=e^{i z_i x}, z_i \in \R$. By Remark \ref{rem:equiv_saip} it suffices to show \eqref{eq:equiv_saip} for rapidly decreasing functions, namely $C^\infty$ functions which and whose derivatives decay at infinity faster than any polynomial.  A rapidly decreasing function $f$ can be written as the inverse Fourier transform of a rapidly decreasing function, and so by Fubini's theorem
\begin{align*}
\langle u, f(\overline{X_{t}})v\rangle 
&= \int_\R f(x) \,{\rm d}\langle u,E_t(x)v\rangle 
= \int_\R \left(\int_\R e^{i z x} \hat f(z) \,{\rm d}z\right) {\rm d}\langle u,E_t(x)v\rangle \\
&= \int_\R \langle u,e^{i z \overline{X_{t}}}v\rangle \hat f(z) \,{\rm d }z,  
\end{align*}
for all $t\geq0$ and $u,v \in H$, where $\hat f$ is the Fourier transform of $f$ and $E_t(x)$ is the spectral projection of $\overline{X_t}$ over $(-\infty,x]$. Iterating this we obtain 
$$
\langle \xi, f_1(\overline{X_{t_1}}) \cdots f_n(\overline{X_{t_n}})\xi\rangle 
= \int_{\R^n} \langle \xi,e^{i z_1 \overline{X_{t_1}}}\cdots e^{i z_n \overline{X_{t_n}}}\xi\rangle \hat f(z_1) \cdots \hat f_n(z_n) \,{\rm d}z_1 \cdots {\rm d}z_n 
$$
for all rapidly decreasing functions $f_i$ on $\R$. Thus \eqref{eq:equiv_saip} holds for all functions $f_i \in C_b(\R)$. 
\end{proof}

\index{monotone increment process!additive (SAIP)|)}

\section{Markov processes}\index{Markov process|(}
First we give some basic concepts on Markov processes. This section is based on \cite{RY99} and \cite{kallenberg02}. 
In this section $S$ denotes a locally compact space with countable basis, and $\mathcal S$ denotes the Borel $\sigma$-field.

A \emph{probability kernel} $k$ on $(S,\mathcal{S})$ is a map $k:S\times\mathcal{S}\to[0,1]$ such that
\begin{itemize}
\item[(i)]
$k(x,\,\cdot\,):\mathcal{S}\ni B \mapsto k(x,B)$ is a probability measure for each $x\in S$;
\item[(ii)]
$k(\,\cdot\,,B):S\ni x\mapsto k(x,B)$ is a measurable function for each $B\in\mathcal{S}$.
\end{itemize}
For two probability kernels $k$ and $l$ we can define its composition
$$
(k\star l)(x, B) = \int_S k(x,{\rm d}y) l(y,B)\qquad \mbox{ for }x\in S, B\in\mathcal{S}.
$$

A family $(k_{st})_{0\le s\le t}$ of probability kernels is called \emph{transition kernels} if it satisfies 
\begin{equation}\label{eq:CK}
k_{su} = k_{st}\star k_{tu} ~~\text{and}~~k_{ss}(x,\,\cdot\,)=\delta_x(\cdot)
\end{equation}
for all $0\le s \le t\le u$ and $x\in S$. The former relation is called the \index{Chapman-Kolmogorov relation}Chapman-Kolmogorov relation. It is similar to the compositional relation satisfied by the \index{transition mappings}transition mappings of a Loewner chain (see Definition \ref{EV_def:evolution_family}). Indeed, the Chapman-Kolmogorov relation is a crucial ingredient for establishing a connection to Loewner chains (see Theorem \ref{from_add_process_to_Loewner_chains}). 

\begin{definition}
Let $(k_{st})_{0\le s\le t}$ be a family of transition kernels. A stochastic process $(M_t)_{t\geq0}$  on $(S,\mathcal{S})$ adapted to a filtration $(\mathcal F_t)$ is called a \emph{Markov process with transition kernels $(k_{st})_{0\leq s \leq t}$} if for each $0\leq s\leq t$ and $B \in \mathcal S$ we have 
\begin{equation}\label{eq:transition}
\mathbb P[M_t\in B|\mathcal F_s] = k_{st}(M_s,B)~~a.s. 
\end{equation} 
The distribution $\mathbb P\circ M_0^{-1}$ on $(S,\mathcal{S})$ is called the \emph{initial distribution}. When we simply say a \emph{Markov process}, it is a Markov process with some transition kernels and some filtration. A Markov process is said to be \emph{stationary} if its transition kernels satisfy $k_{st}=k_{0,t-s}$. Then we simply denote $k_{0t}$ by $k_t$ and call $(k_t)_{t\geq0}$ the transition kernels as well. In this case the Chapman-Kolmogorov relation reads $k_s \star k_t = k_{s+t}$ for $s,t\geq0$.
\end{definition}

The equation \eqref{eq:transition} is called the \emph{Markov property}. It is equivalent to 
\begin{equation}\label{eq:Markov2}
\mathbb E[f(M_t)|\mathcal F_s] = \int_{S} f(x)k_{st}(M_s,{\rm d}x)~~a.s.
\end{equation}
for all bounded measurable functions $f\colon S \to \C$. 

It is known that for a distribution $\mu$ on $(S,\mathcal{S})$ and a family of transition kernels $(k_{st})_{0\le s\le t}$ on $(S,\mathcal{S})$ satisfying \eqref{eq:CK}, there exists a Markov process $(M_t)_{t\ge 0}$ with initial distribution $\mu$ and $(k_{st})_{0\le s\le t}$ as transition kernels. Moreover, the Markov process is unique up to finite dimensional distributions, namely, with respect to the following equivalence. 
\begin{definition}\label{def:Markov_equiv}\index{equivalence!of classical stochastic processes}
Two stochastic processes $(M_t)_{t\geq0}$ and $(N_t)_{t\geq0}$ are equivalent if 
\begin{equation}\label{eq:finite-dim}
\mathbb P[(M_{t_1}, \dots,M_{t_n})\in B]=\mathbb P[(N_{t_1}, \dots, N_{t_n})\in B] 
\end{equation}
for all times $t_1, \dots, t_n \geq0$, all $n\in \N$ and all $B \in \mathcal S^n$. 
\end{definition}
Suppose that two Markov processes $(M_t)_{t\geq0}$ and $(N_t)_{t\geq0}$ have the same transition kernels $(k_{st})_{0\le s\le t}$ and initial distribution $\mu$. Then they are equivalent, and actually the above common value \eqref{eq:finite-dim} is given by 
$$
\int_{S^{n+1}} 1_{B}(x_1,\dots, x_n) \mu({\rm d}x_0)k_{0t_1}(x_0,{\rm d}x_1) \cdots k_{t_{n-1}t_n}(x_{n-1},{\rm d}x_n).  
$$
In this paper, we fix the initial distribution to be a delta measure. Then an equivalence class of Markov processes is determined by $(k_{st})$, on which we mainly focus.  

\index{Markov process|)}


\section[From Loewner chains to Markov processes]{From additive Loewner chains to $\rhd$-homogeneous Markov processes}\index{Markov process!$\rhd$-homogeneous|(}

In this subsection we will construct a classical real-valued Markov process from a given additive Loewner chain. The Markov processes obtained in this way have a special space-homogeneity property.  Our construction is the continuous-time version of the Markov chains constructed in \cite{letac+malouche00}.

\begin{definition}\label{def_equi_Markov}
A probability kernel $k$ on $\R$ is called \emph{$\rhd$-homogeneous} if it satisfies
\[
\delta_x\rhd k(y,\,\cdot\,) = k(x+y,\,\cdot\,)
\]
for all $x,y\in\mathbb{R}$. A real-valued Markov process $(M_t)_{t\geq0}$ is called a \emph{$\rhd$-homogeneous Markov process} if its transition kernels $(k_{st})_{0\le s\le t}$
 satisfy the following two conditions:
\begin{itemize}
\item[(a)] The mapping $(s,t)\mapsto k_{st}(x,\cdot)$ is continuous w.r.t.\ weak convergence for all $x\in\R$.
\item[(b)] The kernel $k_{st}$ is \emph{$\rhd$-homogeneous} for all $0\le s\le t$.
\end{itemize}
Later in Section \ref{sec:Feller_additive} we prove that any stationary $\rhd$-homogeneous Markov process has the Feller property, so that it has a cadlag (right continuous with left limits) modification. 
\end{definition}
\begin{remark}
The standard notion of homogeneity is 
$$
\delta_x \ast k(y,\,\cdot\,) = k(x+y,\,\cdot\,), \qquad x,y \in \R,  
$$
where $\ast$ is the classical convolution of probability measures. 
If the transition kernels of a Markov process satisfy the homogeneity, then the process has independent increments (see \cite[Proposition 8.5]{kallenberg02}). Our $\rhd$-homogeneity is a variant of homogeneity, which corresponds to monotonically independent increments of the corresponding non-commutative stochastic process (see Section \ref{subsec:Markov-SAIP}). 
\end{remark}

\begin{theorem}\label{theo-Markov-proc}
Let $(F_t)_{t\geq0}$ be an \index{Loewner chain!additive}additive Loewner chain with \index{transition mappings}transition mappings $(F_{st})_{0\le s\le t}$. Then there exists a $\rhd$-homogeneous Markov process $(M_t)_{t\ge 0}$ whose transition kernels $(k_{st})_{0\le s\le t}$ are given by
\begin{equation}\label{eq:additive_equivariance}
\int_\mathbb{R} \frac{1}{z-y}\,k_{st}(x,{\rm d}y) = \frac{1}{F_{st}(z)-x}
\end{equation}
for $0\le s\le t$, $x\in\mathbb{R}$ and $z\in\Ha$.
\end{theorem}
\begin{proof} Firstly, since $z\mapsto 1/(F_{st}(z)-x)$ maps $\C^+$ into $-\C^+$ and $iy/(F_{st}(iy)-x) \to 1$ as $y\to\infty$, by Lemma \ref{Julia} there exists a probability measure $k_{st}(x,\,\cdot\,)$ such that \eqref{eq:additive_equivariance} holds for each $(s,t,x)$.  Since $F_{ss}(z)=z$, we have $k_{ss}(x,\cdot)=\delta_x(\cdot)$. 

Concerning the measurability of $x\mapsto k_{st}(x,B)$, by the inversion formulas \eqref{Stieltjes2} and \eqref{eq:atom2}, we have
\begin{align*}
k_{st}(x,\{\alpha\}) &= \lim_{\epsilon\downarrow0} \frac{i\epsilon}{F_{st}(\alpha+i\epsilon)-x}, \\
\frac{1}{2}k_{st}(x,\{\alpha\})+ \frac{1}{2}k_{st}(x,\{\beta\})+k_{st}(x,(\alpha,\beta)) &= -\frac{1}{\pi}\lim_{\epsilon\downarrow0}\int_\alpha^\beta \frac{1}{F_{st}(y+i\epsilon)-x}\,{\rm d}y, 
\end{align*}
which imply the measurability of $x\mapsto k_{st}(x,B)$ for open intervals $B$. The remaining arguments are standard in probability theory: the set $\mathcal B$ of all Borel subsets $B \subset \R$ such that $x\mapsto k_{st}(x,B)$ is measurable forms a Dynkin system by monotone convergence, and the set of open intervals is closed under the intersection and is contained in $\mathcal B$, and hence by Dynkin's $\pi-\lambda$ theorem $\mathcal B$ coincides with the set of all Borel subsets of $\R$.

For the existence of a Markov process it suffices to prove that $(k_{st})$ satisfies the Chapman-Kolmogorov relation. For $0 \leq s \leq t \leq u$, by using \eqref{eq:additive_equivariance} we get
\begin{align*}
&\int_{\R_w} \frac{1}{z-w} \int_{\R_y} k_{st}(x,{\rm d}y) k_{tu}(y,{\rm d}w)
= \int_{\R} \frac{1}{F_{tu}(z)-y}\, k_{st}(x,{\rm d}y)\\ &= \frac{1}{F_{st}(F_{tu}(z))-x}
=  \frac{1}{F_{su}(z)-x}
=  \int_{\R} \frac{1}{z-w}\, k_{su}(x,{\rm d}w),
\end{align*}
where the subscripts in $\R_y$ and $\R_w$ indicate the variable to be integrated. By the Stieltjes-Perron inversion \eqref{Stieltjes}--\eqref{eq:atom} we obtain the Chapman-Kolmogorov relation. 

The $\rhd$-homogeneity of the transition kernels follows from the straightforward calculation
$$
F_{\delta_x \rhd k_{st}(y,\cdot)}(z) = F_{k_{st}(y,\cdot)}(z)-x =  F_{st}(z)-(x+y) = F_{k_{st}(x+y,\cdot)}(z).
$$
Finally, Lemma \ref{lemmaconvergence} implies the weak continuity of $(s,t)\mapsto k_{st}(x,\cdot)$ for all $x\in\R$.
\end{proof}

\begin{remark}
In the stationary case ($F_s \circ F_t = F_{s+t}$) the defining formula \eqref{eq:additive_equivariance} becomes
\begin{equation*}
\int_\mathbb{R} \frac{1}{z-y}\,k_{t}(x,{\rm d}y) = \frac{1}{F_{t}(z)-x}, \qquad z \in \C^+, x \in \R.
\end{equation*}
This formula has similarity with the one for continuous-time branching processes in one dimension. In the latter processes, the transition kernels $(l_{t})_{t\geq0}$ are characterized by
\begin{equation*}
\int_0^\infty e^{-\lambda y}\,l_{t}(x,{\rm d}y) = e^{-x v_t(\lambda)}, \qquad \lambda ,x \geq0,
\end{equation*}
where $(v_t)_{t\geq0}$ is a compositional semigroup of mappings on $[0,\infty)$ generated by a certain vector field (see \cite[Section 3]{Li11}).
\end{remark}

Let $(M_t)$ be the $\rhd$-homogeneous Markov process constructed in
Theorem \ref{theo-Markov-proc}. Then \eqref{eq:additive_equivariance} and the Markov property \eqref{eq:Markov2} imply that for each $0 \leq s \leq t$ and $z \in \Ha$ we have 
\begin{equation}\label{eq:Markovianity}
\mathbb E\left[\frac{1}{z-M_t} \Bigg| \mathcal F_s\right] =\frac{1}{F_{st}(z)-M_s}~~a.s.,
\end{equation}
which will be intensively used in the next section. 
\index{Markov process!$\rhd$-homogeneous|)}
\section[From Markov processes to SAIPs]{From $\rhd$-homogeneous Markov processes to additive monotone increment processes}\label{subsec:Markov-SAIP}\index{monotone increment process!additive (SAIP)|(}

We will now come to the main result of this section, namely the construction of a SAIP from a $\rhd$-homogeneous Markov process.
Let $(M_t)_{t\ge 0}$ be a \index{Markov process!$\rhd$-homogeneous}$\rhd$-homogeneous Markov process such that $M_0=0$ and adapted to a filtration $(\mathcal{F}_t)_{t\ge 0}$. Denote by $(\Omega,\mathcal F, \mathbb P)$ the underlying probability space and by $P_t$ the conditional expectation
\[
P_t = \mathbb{E}[\,\cdot\,|\mathcal{F}_t], \qquad t\ge 0.
\]
 We will now show that the family of symmetric operators $(X_t)_{t\ge 0}$ defined by
\begin{align}
X_t&=P_t M_t, \qquad t\ge 0, \label{def_SAIP}\\
\Dom(X_t) &= \{\psi \in L^2: M_t \psi \in L^2\} = \Dom(M_t), 
\end{align}
is a SAIP on $(L^2(\Omega,\mathcal{F},\mathbb P),\mathbf{1}_\Omega)$, where $M_t$ acts by multiplication on $L^2(\Omega,\mathcal{F},\mathbb P)$ and $\mathbf{1}_\Omega$ is the constant function with value $1$ on $\Omega$.
In this section we always regard $M_t$ as the multiplication operator. Notice that $P_t M_t  \subset M_t P_t$ as unbounded operators. 

We denote by $G_{st}$ the Cauchy transform of $k_{st}(0,\,\cdot\,)$ and by $F_{st}$ its reciprocal; namely
\begin{equation*}
G_{st}(z) = \int_\R \frac{1}{z-y}\,k_{st}(0,{\rm d}y), \qquad F_{st}(z)=\frac{1}{G_{st}(z)}.
\end{equation*}
Then the $\rhd$-homogeneity implies that $F_{k_{st}(x,\cdot)}(z) = F_{\delta_x} (F_{st}(z)) =F_{st}(z)-x$ which is exactly \eqref{eq:additive_equivariance}. Therefore \eqref{eq:Markovianity} holds true, which can also be stated as \begin{equation}\label{eq:additive_markovianity}
P_s \frac{1}{z-M_t}P_s = \frac{1}{F_{st}(z)-M_s}P_s. 
\end{equation}

Firstly we check that the domains appearing in this section are dense in $L^2$. 

\begin{lemma}
For every $n\in \N$ and reals $0\leq t_1 <t_2 \cdots < t_n$, the subspace
$$
\bigcap_{i=1}^n \Dom(X_{t_i})
$$
is dense in $L^2$. 
\end{lemma}
\begin{proof} 
For any $\psi \in L^2$ we can take functions $f_k=\mathbf1_{[-k,k]^n}(M_{t_1}, \dots, M_{t_n}) \psi, k\in \N,$ which are contained in $\cap_{i=1}^n \Dom(X_{t_i})
$ and converge to $\psi$ in $L^2$. 
\end{proof}

It is not obvious that the increment $X_t-X_s$ with the dense domain $\Dom(X_t) \cap \Dom(X_s)$ is essentially self-adjoint. We prove this by finding and using an explicit form of the resolvent operator of increments.

\begin{proposition}\label{prop:resolvent}
For $t\ge c\ge s\ge 0$ and $z\in \mathbb{C}\setminus\mathbb{R}$ the operator
\begin{gather*}
R_{st}(z)\psi := \frac{\psi}{z} + \frac{M_t}{z(z-M_t)}P_t\psi - \frac{ M_s\big(F_{st}(z)-M_s\big)}{F_{st}(z)(z-M_t)}P_s\frac{1}{z-M_t}\psi
\end{gather*}
is a bounded linear operator on $L^2$, preserves $D_{sct}=\Dom(X_s)\cap \Dom(X_c)\cap\Dom(X_t)$, and satisfies
$$
\{z-(X_t-X_s)\}R_{st}(z)|_{D_{sct}} = R_{st}(z)\{z-(X_t-X_s)\}|_{D_{sct}}=I_{D_{sct}}.
$$
\end{proposition}
\begin{remark}
The above formula for resolvents looks miraculous, but for bounded Markov processes it can be naturally derived from a series expansion method which is similar to the proof of Proposition \ref{prop:unitary_key} for the unitary case.
\end{remark}
\begin{proof}
First we prove that $R_{st}(z)$ is a bounded linear operator. It suffices to prove that the operator
$$
T\psi :=\frac{ M_s\big(F_{st}(z)-M_s\big)}{z-M_t}P_s\frac{1}{z-M_t}\psi
$$
is bounded. By the conditional Schwarz inequality, we obtain
$$
\left| P_s\frac{1}{z-M_t}\psi \right|^2 \leq \left(P_s \frac{1}{|z-M_t|^2}\mathbf{1}_\Omega\right) P_s |\psi|^2.
$$
We can further compute the first factor as
\begin{align}\label{eq:conditional_square}
P_s \frac{1}{|z-M_t|^2}\mathbf{1}_\Omega
&= \frac{1}{\overline{z} - z} P_s \left(\frac{1}{z-M_t} - \frac{1}{\overline{z}-M_t}\right)\mathbf{1}_\Omega \\
&= \frac{1}{\overline{z} - z}\left(\frac{1}{F_{st}(z)-M_s} - \frac{1}{F_{st}(\overline{z})-M_s}\right)\mathbf{1}_\Omega  \notag\\
&=  \frac{\Im(F_{st}(z))}{\Im(z) |F_{st}(z)-M_s|^2} \mathbf{1}_\Omega,  \notag
\end{align}
and hence
\begin{equation}\label{eq:estimate_f}
|T\psi|^2 \leq \frac{\Im(F_{st}(z))}{\Im(z)}\frac{|M_s|^2}{|z-M_t|^2}P_s |\psi|^2.
\end{equation}
Taking the expectation shows that
\begin{align}\label{eq:bounded} 
\mathbb E|T\psi|^2 = \mathbb E [P_s|T\psi|^2]  
&\leq   \frac{\Im(F_{st}(z))}{\Im(z)} \mathbb E  \left[ \left(M_s^2 P_s |\psi|^2\right) \cdot\left(P_s\frac{1}{|z-M_t|^2}\mathbf{1}_\Omega\right) \right] \notag\\
&=  \left|\frac{\Im(F_{st}(z))}{\Im(z)}\right|^2 \mathbb E  \left[ \left(P_s |\psi|^2\right) \cdot \left(\frac{M_s^2}{|F_{st}(z)-M_s|^2}\mathbf{1}_\Omega\right)\right].  
\end{align}
Since $M_s/(F_{st}(z)-M_s)$ is a bounded random variable and $\mathbb E [P_s |\psi|^2] = \|\psi\|_{L^2}$, we conclude that $T$ is a bounded operator.

We show that $R_{st}(z)$ preserves the space $\Dom(X_s)\cap\Dom(X_c)$ for every $c \in[s,t]$. For this it suffices to show that $T$ does. For $\psi \in \Dom(X_s)\cap \Dom(X_c)$, we imitate the estimate \eqref{eq:bounded} and insert the conditional expectation $P_c$:   
\begin{align*}
\mathbb E|M_cT\psi|^2 &= \mathbb E [P_c|M_cT\psi|^2]  
\leq   \frac{\Im(F_{st}(z))}{\Im(z)} \mathbb E  \left[ \left(M_s^2 P_s |\psi|^2\right) \cdot\left(|M_c|^2P_c\frac{1}{|z-M_t|^2}\mathbf{1}_\Omega\right) \right] \\
&=  \frac{\Im(F_{st}(z))\Im(F_{ct}(z))}{\Im(z)^2} \mathbb E  \left[ \left(P_s |M_s\psi|^2\right) \cdot \left(\frac{|M_c|^2}{|F_{ct}(z)-M_c|^2}\mathbf{1}_\Omega\right)\right].
\end{align*}
Again by the boundedness of $M_c/(F_{ct}(z)-M_c)$ this shows that $M_c T\psi \in L^2$, and hence we conclude that $T\psi \in \Dom(X_c)$. Since we can also take $c=s$ we conclude that $T\psi \in \Dom(X_s)\cap\Dom(X_c)$. This also shows that $T$ and hence $R_{st}$ preserves $D_{sct}$. 
 
For $\psi \in \Dom(X_t) \cap \Dom(X_s)$ expand the quantity
\begin{align*}
&\left\{z-(P_t M_t-P_s M_s)\right\}\left( \frac{\psi}{z} + \frac{M_t}{z(z-M_t)}P_t\psi - \frac{ M_s\big(F_{st}(z)-M_s\big)}{F_{st}(z)(z-M_t)}P_s\frac{1}{z-M_t}\psi\right) \\
&= \psi + \underbrace{\frac{M_t}{z-M_tP_t \vphantom{(F_{st})}} \psi}_{I_1} \underbrace{- \frac{z M_s (F_{st}(z)-M_s)}{F_{st}(z)(z-M_t)}P_s \frac{1}{z-M_t}\psi}_{I_2} \underbrace{+ \frac{M_s}{z\vphantom{(F_{st})}}P_s \psi}_{I_3} \underbrace{+ \frac{M_s}{z\vphantom{(F_{st})}}P_s\frac{M_t}{z-M_t}\psi}_{I_4}\\
&\qquad  \underbrace{- \frac{M_s^2(F_{st}(z)-M_s)}{F_{st}(z)}P_s \frac{1}{z-M_t} P_s \frac{1}{z-M_t}\psi}_{I_5}  \underbrace{- \frac{M_t}{z\vphantom{(F_{st})}}P_t \psi}_{I_6}  \underbrace{- \frac{M_t^2}{z(z-M_t)\vphantom{(F_{st})}}P_t\psi}_{I_7} \\
&\qquad \underbrace{+ \frac{M_sM_t(F_{st}(z)-M_s)}{F_{st}(z)(z-M_t)}P_s \frac{1}{z-M_t}\psi}_{I_8}.
\end{align*}
We can easily show that $I_1 + I_6+I_7=0$.  For $I_5$, the Markov property \eqref{eq:additive_markovianity} shows that
\begin{align*}
P_s \frac{1}{z-M_t} P_s \frac{1}{z-M_t}\psi
&= \frac{1}{F_{st}(z)-M_s}P_s\frac{1}{z-M_t}\psi,
\end{align*}
and then we can prove that
\begin{equation*}
I_2 + I_5 + I_8 = -M_s P_s \frac{1}{z-M_t}\psi.
\end{equation*}
Then it is straightforward to show that $I_2 + I_3+ I_4+ I_5 + I_8=0$. Thus we conclude that $\{z-(X_t-X_s)\} R_{st}(z)=I_{\Dom(X_s)\cap \Dom(X_t)}. $

A similar computation shows that $R_{st}(z)\{z-(X_t-X_s)\}=I_{\Dom(X_s)\cap \Dom(X_t)}$. Since we know that $R_{st}(z)$ preserves $D_{sct}$, the desired formula holds on $D_{sct}$ as well. 
\end{proof}

\begin{proposition}\label{prop:SAIP1} The increment $X_t-X_s$ is essentially self-adjoint for every $0 \leq s \leq t < \infty$, and the resolvent of $\overline{X_t-X_s}$ is given by the bounded operator $R_{st}(z)$. Furthermore, $(X_t-X_s)|_{\Dom(X_s) \cap \Dom(X_c) \cap \Dom(X_t)}$ is essentially self-adjoint for every $0 \leq s \leq c \leq t$. 
\end{proposition}
\begin{proof}
Proposition \ref{prop:resolvent} implies that the range of $z-(X_t-X_s)$ contains $\Dom(X_t) \cap \Dom(X_s)$, which is dense in $L^2$. By Corollary to \cite[Theorem VIII.3]{RS80}, $X_t - X_s$ is essentially self-adjoint. The second statement readily follows from Proposition \ref{prop:resolvent}. 
The last statement follows from the same arguments as for $\Dom(X_t)\cap \Dom(X_s)$. 
\end{proof}

Next we compute the distributions of the increments and show that the increments are monotonically independent.

\begin{proposition}\label{prop:monotone-indep}
We have
\begin{equation}\label{eq:additive_sandwich}
P_s \left\{z-\overline{(X_t-X_s)}\right\}^{-1} P_s = G_{st}(z) P_s
\end{equation}
for $t\ge s\ge 0$ and $z\in \mathbb{C}\setminus\mathbb{R}$. In particular, the distribution of $X_t-X_s$ with respect to the state $\langle \mathbf{1}_\Omega, \cdot \mathbf{1}_\Omega \rangle_{L^2(\Omega)}$ is equal to $k_{st}(0,\cdot)$.
\end{proposition}
\begin{proof} Using properties of the conditional expectation and the Markov property \eqref{eq:additive_markovianity}, for $\psi \in L^2$ we obtain
\begin{align*}
&P_s \big(z-(P_tM_t-P_sM_s)\big)^{-1}P_s\psi \\
&= \frac{1}{z}P_s\psi + P_s\frac{M_t}{z(z-M_t)}P_s\psi - P_s\frac{ M_s\big(F_{st}(z)-M_s\big)}{F_{st}(z)(z-M_t)}P_s\frac{1}{z-M_t}P_s\psi\\
&=  \frac{1}{z}P_s\psi + \frac{1}{z}\left( P_s\psi \right) \cdot \left( P_s\frac{M_t}{z-M_t} \mathbf{1}_\Omega\right) - \left(\frac{M_s(F_{st}(z)-M_s)}{F_{st}(z)} P_s\psi \right) \cdot \left(P_s\frac{1}{z-M_t} \mathbf{1}_\Omega \right)^2 \\
&=  \frac{1}{z}P_s\psi
+\frac{1}{z}\left( P_s\psi \right) \cdot \left( - \mathbf{1}_\Omega+ \frac{z}{F_{st}(z)-M_s} \mathbf{1}_\Omega\right) \\&\qquad- \left(\frac{M_s(F_{st}(z)-M_s)}{F_{st}(z)}P_s\psi \right) \cdot \left(\frac{1}{F_{st}(z)-M_s} \mathbf{1}_\Omega \right)^2\\
&= G_{st}(z)P_s \psi,
\end{align*}
the conclusion. The last statement follows by applying \eqref{eq:additive_sandwich} to the constant function $\mathbf1_\Omega$ and taking the expectation.
\end{proof}

We want to replace the function $1/(z-x)$ in Proposition \ref{prop:monotone-indep} with a more general bounded continuous function $f$ (to obtain Lemma \ref{lem-mon-relations} \eqref{lem-mon-relations2}). For this the following lemma suffices.

\begin{lemma}\label{lem:approx}
The set of functions
$$
D=\text{\rm span}_\C\left\{(z- \cdot)^{-1}: z\in \C \setminus \R\right\}
$$
is dense in 
$$
C_0(\R)=\left\{f \in C(\R): \lim_{x \to\infty} f(x) =\lim_{x \to -\infty} f(x)  =0\right\}
$$
with respect to uniform convergence.
\end{lemma}
\begin{proof} A proof is given in \cite[Lemma 6]{Ans13}. We give another proof based on the Stone-Weierstrass theorem. Since $D$ itself is not a *-algebra, we need a trick. 

Step 1. Let $g_z(x)=(z-x)^{-1}$. For $z\neq w$ we have 
$$
g_z g_w = \frac{1}{w-z} g_z - \frac{1}{w-z}g_w,  
$$
and hence $g_z g_w \in D$. 

Step 2. For $z\in \C \setminus \R$ we can approximate $g_z^2$ uniformly by $g_z g_w$ as $w\to z$ with $w\neq z$, because   
$$
\|g_z^2 - g_z g_w\|_{C_0(\R)}  \leq \frac{|w-z|}{\Im(z)^2 |\Im(w)|}. 
$$
By Step 1, we conclude that $g_z^2 \in \overline{D}$. 

Step 3. The above steps show that $\overline{D}$ is a *-algebra. It also separates points and vanishes nowhere, and hence equals $C_0(\R)$ by Stone-Weierstrass' theorem. 
\end{proof}

Let us collect key relations between increments and conditional expectations to prove $(X_t)$ is a SAIP. 
\begin{lemma}\label{lem-mon-relations}
\begin{enumerate}[\rm(1)]
\item \label{lem-mon-relations1}
For $0\le s\le t\le u$ and $f\in \mathcal C_b(\mathbb{R})$ with  $f(0)=0$, we have
\[
f(\overline{X_t-X_s})P_u = P_u f(\overline{X_t-X_s}) = f(\overline{X_t-X_s}).
\]
\item \label{lem-mon-relations2}
For $0\le s\le t\le u$ and $f\in C_b(\mathbb{R})$, we have
\[
P_s f(\overline{X_u-X_t}) P_s = \left\langle \mathbf{1}_\Omega, f(\overline{X_u-X_t})\mathbf{1}_\Omega \right\rangle_{L^2(\Omega)}P_s.
\]
\end{enumerate}
\end{lemma}
\begin{proof}
(1) In this proof for a closed subspace $K$ of $L^2$ and a linear operator $X$ with domain $\Dom(X)$ we denote by $X|_K$ the operator with domain $\Dom(X)\cap K$. Let $X= X_t-X_s$, $P=P_u$ and $K$ be the range of the orthogonal projection $P$. We can check that, by a property of conditional expectations, $P$ preserves $\Dom(X)$ and we have $X=P X \subset XP$. We can also check, by the definition of closure, that $P$ preserves $\Dom(\overline{X})$ and $\overline{X}=P \overline{X} \subset \overline{X}P$, and thus $\overline{X}$ is reduced by $K$ and by $K^\perp$. By \cite[Theorems 1.38 and 1.40]{Ara18}, we have $\overline{X}|_{K}=\overline{X|_{K}}$ and it is self-adjoint in $K$, and for any $f \in C_b(\R)$ we have $f(\overline{X|_{K}})=f(\overline{X})|_{K}$. 
This implies $f(\overline{X})$ preserves $K$, and similar arguments show that $f(\overline{X})$ preserves $K^{\perp}$. Therefore, $Pf(\overline{X}) = f(\overline{X})P$. 

Similarly, we have $0|_{\Dom(\overline{X})}=(1-P) \overline{X} \subset \overline{X}(1-P)$, and hence $\overline{X}|_{K^\perp}=\overline{X|_{K^\perp}}=0$ and $f(\overline{X})|_{K^\perp}=f(\overline{X|_{K^\perp}})=0$ if $f(0)=0$. This implies that $f(\overline{X}) = f(\overline{X})P$. 

(2) 
For $f \in C_0(\R)$ the statement follows from Proposition \ref{prop:monotone-indep} and Lemma \ref{lem:approx}. For general $f \in C_b(\R)$, we can adapt the corresponding arguments in Remark \ref{rem:equiv_saip} to find a sequence $f_n \in C_0(\R)$ such that $f_n(\overline{X_u-X_t}) \to f(\overline{X_u-X_t})$ strongly. 
\end{proof}

\begin{theorem}\label{thm:additive_monotone_construction}
The family $(X_t)_{t\geq0}$ of essentially self-adjoint operators is a SAIP on the quantum
probability space $(L^2(\Omega,\mathcal{F},\mathbb P),\mathbf{1}_\Omega)$.
\end{theorem}
\begin{proof}
As $M_0=0$, we have $X_0=0$. By Proposition \ref{prop:SAIP1} $\Dom(X_s)\cap\Dom(X_c)\cap \Dom(X_t)$ is a core for $X_t-X_s$ for every $0\leq s\leq c\leq t$. By Proposition \ref{prop:monotone-indep}, the distribution of $X_t-X_s$ is equal to $k_{st}(0,\cdot)$.
The mapping $t\mapsto k_{st}(0,\cdot)$ is continuous by assumption. It remains to show that $(X_t)$ has monotonically independent increments.

Step 1.
Let $t_1,\ldots,t_p, s_1,\ldots, s_p, t'_1,\ldots,t'_q, s'_1,\ldots, s'_q,t,s\geq0$ be such that
\[
t_1\geq s_1\geq t_2\geq  \cdots \geq t_p\geq s_p\geq t\geq s \leq t \leq t_q' \leq \cdots \leq t_2' \leq s_1' \leq t_1',
\]
$f_1,\ldots, f_p,g, h_1,\ldots,h_q\in C_b(\mathbb{R})$ vanishing at $0$, and set
\begin{gather*}
W_1=f_1(\overline{X_{t_1}-X_{s_1}}), \ldots,\, W_p=f_p(\overline{X_{t_p}-X_{s_p}}),\, Y=g(\overline{X_t-X_s}), \\
Z_1=h_1(\overline{X_{t'_1}-X_{s'_1}}),\ldots,\, Z_q=h_q(\overline{X_{t'_q}-X_{s'_q}}).
\end{gather*}
Then we have from Lemma \ref{lem-mon-relations}   \eqref{lem-mon-relations1} and \eqref{lem-mon-relations2}
\begin{align*}
&\langle \mathbf{1}_\Omega, W_1\cdots W_pYZ_q\cdots Z_1 \mathbf{1}_\Omega \rangle\\
&= \langle \mathbf{1}_\Omega, P_{t_2}W_1 P_{t_2} W_2 W_3 \cdots W_p Y Z_q  \cdots Z_2 P_{t'_2} Z_1 P_{t'_2} \mathbf{1}_\Omega \rangle \\
&= \langle \mathbf{1}_\Omega, W_1\mathbf{1}_\Omega\rangle
\langle \mathbf{1}_\Omega, W_2 W_3 \cdots W_p Y Z_q  \cdots Z_2 \mathbf{1}_\Omega \rangle \langle \mathbf{1}_\Omega, Z_1 \mathbf{1}_\Omega \rangle \\
&= \cdots \\
&=  \langle \mathbf{1}_\Omega, W_1\mathbf{1}_\Omega\rangle  \cdots \langle\mathbf{1}_\Omega , W_p\mathbf{1}_\Omega\rangle \langle\mathbf{1}_\Omega ,Y\mathbf{1}_\Omega\rangle \langle\mathbf{1}_\Omega Z_q\mathbf{1}_\Omega\rangle \cdots \langle \mathbf{1}_\Omega, Z_1 \mathbf{1}_\Omega \rangle.
\end{align*}
It is clear that this implies Condition (i) of Definition \ref{def-mon}.

Step 2.
Let $t,s,t',s',t'',s''\in\mathbb{R}$ such that $0\le s'\leq t'\leq s\leq t$ and $0\le s''\leq t''\le s\le t$, $f,g,h\in C_b(\mathbb{R})$ vanishing at $0$ and set $X=f(\overline{X_{t'}-X_{s'}})$, $Y=g(\overline{X_t-X_s})$, $Z=h(\overline{X_{t''}-X_{s''}})$. From Lemma \ref{lem-mon-relations}  \eqref{lem-mon-relations1} and  \eqref{lem-mon-relations2} we get
\[
XYZ = XP_{t'}YP_{t''}Z = \left\langle \mathbf{1}_\Omega, Y\mathbf{1}_\Omega \right\rangle_{L^2(\Omega)}XZ.
\]
This shows that Condition (ii) of Definition \ref{def-mon} is also satisfied and concludes the proof.
\end{proof}

\index{monotone increment process!additive (SAIP)|)}

\section{Summary of the one-to-one correspondences}${}$\label{summary_additive}

All in all, Theorems \ref{from_add_process_to_Loewner_chains}, \ref{theo-Markov-proc}, and
\ref{thm:additive_monotone_construction} yield one-to-one correspondences between 

\begin{enumerate}[\quad(A)]
\item\label{item:A}
\index{Loewner chain!additive}additive Loewner chains $(F_{t})_{t\geq 0}$ in $\mathbb{C}^+$ (Def.\ \ref{EV_def:evolution_family}),

\item \label{item:B}\index{Markov process!$\rhd$-homogeneous}
real-valued $\rhd$-homogeneous Markov processes $(M_t)_{t\ge 0}$ such that $M_0=0$ up to equivalence (Def. \ref{def_equi_Markov} and Def.\ \ref{def:Markov_equiv}), 

\item\label{item:C} \index{monotone increment process!additive (SAIP)}SAIPs: self-adjoint additive monotone increment processes $(X_t)_{t\ge 0}$ up to equivalence (Def.\ \ref{def_saip} and Def.\ \ref{def:equiv_saip}). 
\end{enumerate}
Moreover, the above objects also correspond to:
\begin{enumerate}[\quad(A)]
\setcounter{enumi}{3}
\item\label{item:D} families $(\mu_{st})_{0\leq s \leq t}$ of probability measures on $\R$ such that 
\begin{enumerate}[(i)]
\item $\mu_{tt} = \delta_0$ for all $t\geq0,$ 

\item\label{hemigroup} $\mu_{su} = \mu_{st} \rhd \mu_{tu}$ for all $0\leq s\leq t\leq u,$ 

\item $(s,t)\mapsto \mu_{st}$ is weakly continuous.
\end{enumerate}
\end{enumerate}
\eqref{item:C} $\Rightarrow$ \eqref{item:D}: Given a SAIP $(X_t)$ we define $\mu_{st}$ to be the law of $X_t-X_s$. This is independent of  a choice of a SAIP in the same equivalence class by Theorem \ref{thm:equiv_saip}. \eqref{item:D} $\Rightarrow$ \eqref{item:A}: Given $(\mu_{st})$ we define the \index{transition mappings}transition mappings $F_{st}=F_{\mu_{st}}$. Then $(F_{0t})$ forms an additive Loewner chain. Thus our constructions yield bijections between the objects \eqref{item:A}--\eqref{item:D}. Note that property (ii) in \eqref{item:D} corresponds to the Chapman-Kolmogorov relation for transition kernels of Markov processes.

We call such a family of probability measures that satisfies the three conditions in \eqref{item:D} a \index{hemigroup}\emph{weakly continuous $\rhd$-convolution hemigroup}.

\begin{remark}
For the object \eqref{item:B} we identify the equivalent Markov processes in the sense of Definition \ref{def:Markov_equiv}, and so we are actually looking at the $\rhd$-homogeneous transition kernels $(k_{st})$ with weak continuity on $t\geq s$. 
\end{remark}

\begin{remark}
A SAIP associates an additive Loewner chain, which is a part of a reverse evolution family (a family of transition mappings). If we adopt \emph{anti-monotone independence} instead of monotone independence, we can obtain an evolution family in the sense of Remark \ref{EV_remark_ev}, but then the corresponding probability kernels do not satisfy the standard Chapman-Kolmogorov relation. What we obtain is something that might be called a ``backward Chapman-Kolmogorov relation'', but the authors do not know any theory on such a relation and the corresponding stochastic processes. 
\end{remark}

Furthermore, we can relate properties of the three notions as follows: 

\begin{itemize}
 \item[(1)] The Markov process $(M_t)$ has finite second moment for all $t\ge 0$ if and only if the complement $\mathbb{C}^+\setminus F_t(\mathbb{C}^+)$ has
finite \index{half-plane capacity}half-plane capacity (see Remark \ref{rm_capacity}) for all $t\ge 0$ if and only if $\xi \in \Dom (\overline{X_t})$ for all $t\geq0$. The last equivalence follows from the spectral theory, see \cite[Section 5.3]{MR2953553}. 

 \item[(2)] The Markov process $(M_t)$ has mean $0$ and variance $t$ for all $t\geq0$ if and only if
\[ F_t(z)=z-\frac{t}{z}+{\scriptstyle\mathcal{O}}(1/|z|)
\]
 as $z\to\infty$ non-tangentially in $\Ha$ for all $t\ge 0$ (Remark \ref{rm_finite_var}) if and only if the SAIP $(X_t)$ is normalized (again by the spectral theory). 

\item[(3)] The Markov process $(M_t)$ is stationary if and only if $F_{s} \circ F_t = F_{s+t}$ for all $s,t\geq0$ if and only if the SAIP $(X_t)$ has stationary increments, i.e.\ $X_t - X_s$ has the same distribution as $X_{t-s}$ for all $0 \leq s \leq t$.
\end{itemize}

\section[Additive free increment processes]{Construction of $\rhd$-homogeneous Markov processes from additive free increment processes} \label{sec:Free-Markov_additive}
\index{Markov process!$\rhd$-homogeneous|(}
\index{free increment process!additive|(}
Following the idea of \cite{franz07b} we can construct $\rhd$-homogeneous transition kernels from free additive increment processes.
Assume that $(\mu_{st})_{0\leq s \leq t}$ are the laws of the increments of such a process, i.e.\ a weakly continuous $\boxplus$-convolution hemigroup (defined as in Section \ref{summary_additive} for the monotone convolution).
By the subordination property \cite{Bia98}, there exists a probability measure $\nu_{st}$ such that $\mu_{0t}= \mu_{0s} \rhd \nu_{st}$ for $0\leq s \leq t$. They satisfy $\nu_{tt}=\delta_0$ and the hemigroup property
$$
\nu_{st} \rhd \nu_{tu}=\nu_{su}, \qquad 0 \leq s \leq t \leq u.
$$
The reciprocal Cauchy transform $F_{\nu_{st}}$ can be expressed by use of the Voiculescu transform $\varphi_{\mu_{st}}$: 
\[
\begin{split}
F_{\nu_{st}}(z)
&= F_{\mu_{0s}}^{-1} \circ F_{\mu_{0t}}(z) \\
&=\varphi_{\mu_{0s}}(F_{\mu_{0t}}(z)) + F_{\mu_{0t}}(z) \\
&= \varphi_{\mu_{0t}}(F_{\mu_{0t}}(z)) + F_{\mu_{0t}}(z) - \varphi_{\mu_{st}}(F_{\mu_{0t}}(z)) \\
&= z - \varphi_{\mu_{st}}(F_{\mu_{0t}}(z)), 
\end{split}
\]
on a truncated cone where all the functions and compositions make sense. Eventually the identity  $F_{\nu_{st}}(z)=z - \varphi_{\mu_{st}}(F_{\mu_{0t}}(z))$ holds on $\C^+$ by analytic continuation, because now $\varphi_{\mu_{st}}$ is defined on $\C^+$ by $\boxplus$-infinite divisibility of $\mu_{st}$.   
Using \cite[Proposition 5.7]{BV93}, $\nu_{st}$ is weakly continuous with respect to $(s,t)$. The probability measures $(\nu_{st})_{0 \leq s \leq t}$ therefore form a weakly continuous $\rhd$-convolution hemigroup. We can thus construct a $\rhd$-homogeneous Markov process through the correspondence in Section \ref{summary_additive}.

Note that each $\nu_{st}$ is $\boxplus$-infinitely divisible (see Section \ref{inf_arrays_additive}), since $\varphi_{\nu_{st}}$ can be calculated as
\begin{equation*}
\varphi_{\nu_{st}}(z) = \varphi_{\mu_{st}}(F_{\mu_{0s}}(z)),
\end{equation*}
and so Theorem \ref{thmBV93} is available.

In the particular case of free L\'evy processes, namely stationary additive free increment processes, the Markov transition kernels are related to additive Boolean convolution $\uplus$ \eqref{bool_convolutions_addi}. Let $(\mu_t)_{t\geq 0}$ be a weakly continuous free convolution semigroup with $\mu_0 = \delta_0$.
Then the $F$-transform of $\nu_{st}$ has the explicit formula
\begin{equation}\label{eq:subordination_stationary_additive}
F_{\nu_{st}}(z)= \frac{s}{t}z + \left(1-\frac{s}{t}\right)F_{\mu_t}(z),
\end{equation}
which was essentially calculated in \cite{BB04}. Therefore, we have the formula
\begin{equation*}
\nu_{st}=\mu_t^{\uplus (1-s/t)}.
\end{equation*}

\begin{example} (See \cite[Section 5.3]{Bia98}.)\\
If $\mu_t$ is the centered \index{semicircle distribution}semicircle law with variance $t$, then
\[
F_{\nu_{st}}(z) = \frac{1}{2}\Big( 1 + \frac{s}{t}\Big) z + \frac{1}{2}\Big( 1-\frac{s}{t}\Big) \sqrt{z^2-4t}.
\]
Using \eqref{eq:subordination_stationary_additive} we can compute the transition kernels for the associated Markov process, to get
\begin{equation*}
k_{st}(x,{\rm d}y) = \frac{t-s}{2\pi}\frac{\sqrt{4t-y^2}}{(sy-tx)(y-x) + (t-s)^2} \,\mathbf1_{[-2\sqrt{t},2\sqrt{t}]}(y)\,{\rm d}y + \lambda \delta_a,
\end{equation*}
where
\begin{align*}
a &=
\frac{(t+s)x -\text{sign}(x)(t-s)\sqrt{x^2-4s}}{2s},\\
 \lambda &= \begin{cases} 0, &|x| \leq \frac{t+s}{\sqrt{t}}, \\
 \frac1{2s}\left(t+s-\frac{|x|(t-s)}{\sqrt{x^2-4s}}\right),& |x| > \frac{t+s}{\sqrt{t}}.
 \end{cases}
 \end{align*}
We extend the above expressions of $a$ and $\lambda$ to the case $s=0$ by continuity.
\end{example}

\begin{remark}
D.\ Jekel has shown that not all Loewner chains arise from free increment processes, see \cite[Section 4.4]{jekel}; his counter-example is an additive Loewner chain with $F_1=F_\sigma$ and  $F_2=F_{\sigma\rhd\sigma}$, where $\sigma=\frac{1}{2\pi} \sqrt{4 -x^2} \mathbf{1}_{[-2,2]}(x)\,{\rm d}x$ denotes the centered semicircle distribution with variance $1$. Combined with the one-to-one correspondences from the previous section, this shows that not all real-valued $\rhd$-homogeneous Markov chains can be obtained from additive free increment processes in the way described in this section.
\end{remark}

\index{free increment process!additive|)}

\section[Generators, Feller property, and martingale property]{Generators, Feller property, and martingale property of $\rhd$-homogeneous Markov processes}\label{analysis_Markov_additive}

\subsection{Generators in the stationary case} \label{sec:Hunt}
We now compute \index{Hunt's formula}Hunt's formula for the generator of stationary $\rhd$-homogeneous Markov processes. The computation was given in \cite{franz+muraki04} for compactly supported transition kernels. This section treats a general case.

Let $(k_t)$ be transition kernels for a stationary $\rhd$-homogeneous Markov process. Then the probability measures $\mu_t := k_t(0,\cdot)$ are weakly continuous regarding $t$, and form a $\rhd$-convolution semigroup, namely the relation
$$
\mu_s \rhd \mu_t = \mu_{s+t}
$$
holds for all $s,t\geq0$. This family admits a generator which has an integral representation called the \index{L\'evy-Khintchine representation!monotone}\emph{monotone L\'evy-Khintchine representation} proved by Muraki \cite{M00} in the finite variance case. The general case follows from \index{Berkson-Porta formula}Berkson and Porta's result \cite{BP78}.
\begin{theorem} \label{MLK}
\begin{enumerate}[\rm(1)]
\item\label{MLKC} Let $\{\mu_t\}_{t\geq0}\subset\cP(\R)$ be a weakly continuous $\rhd$-convolution semigroup such that $\mu_0=\delta_0$ and let $F_t=F_{\mu_t}$. Then the right derivative $A(z) = \frac{\rm d}{{\rm d}t}\big|_{t=0} F_{t}(z)$ exists, and $F_t$ satisfies the differential equation
\begin{align}
\frac{{\rm d}}{{\rm d}t} F_t(z) &= A(F_t(z)), \quad F_0(z) = z, \qquad z \in\C^+, t\geq0. \label{eq:diff1}
\end{align}
Moreover, the analytic function $A$ is of the form
\begin{equation}\label{MLK1}
- A(z)=\gamma +\int_{\mathbb{R}}\frac{1+z x}{z-x} \rho({\rm d}x) ,\quad z\in \C^+,
\end{equation}
where $\gamma \in \R$ and $\rho$ is a finite, non-negative measure on $\R$.
The pair $(\gamma,\rho)$ is unique and is called the generating pair.

\item Conversely, given a pair $(\gamma,\rho)$ of a real number and a finite, non-negative measure, define a function $A$ by \eqref{MLK1}. Then the solution to the differential equation \eqref{eq:diff1} defines a flow $\{F_t\}_{t\geq0}$ on $\C^+$, and then there exists a weakly continuous $\rhd$-convolution semigroup $\{\mu_t\}_{t\geq0}$ such that $F_t=F_{\mu_t}$ for all $t\geq0$.
\end{enumerate}
\end{theorem}

Now we relate the monotone L\'evy-Khintchine representation to our Markov processes.
\begin{lemma}\label{lem:generator} Let $(M_t)_{t\geq0}$ be a stationary $\rhd$-homogeneous Markov process with transition kernels $(k_t)_{t\geq0}$. Let $(\gamma,\rho)$ be the generating pair in \eqref{MLK1} associated to the monotone convolution semigroup $\{k_t(0,\cdot)\}_{t\geq0}$.   Then for all $x\in \R$ we have, as $t\downarrow0$,
$$
\frac{1}{t}\frac{(y-x)^2}{1+y^2}\,k_t(x,{\rm d}y) \to \rho({\rm d}y)~~{\it(weakly)},
$$
and
$$
\frac{1}{t}\int_\R \frac{(y-x)(1+x y)}{1+y^2}\, k_t(x,{\rm d}y) \to \gamma.
$$
\end{lemma}
\begin{proof} Let $G_{x,t}$ and $F_{x,t}$ be the Cauchy transform and its reciprocal of the distribution $k_t(x,\cdot)$.
Since the infinitesimal generator $A(z)$ in Equation \eqref{MLK1} is defined by $A(z) = \left.\frac{{\rm d}}{{\rm d}t}\right|_{t=0} F_{0,t}(z)$, we have
$$
\left.\frac{{\rm d}}{{\rm d}t}\right|_{t=0} G_{x,t}(z)  =\left.\frac{{\rm d}}{{\rm d}t}\right|_{t=0} \frac{1}{F_{0,t}(z)-x}= -(z-x)^{-2} A(z)
$$
and then
\begin{align*}
A(z) &= -\lim_{t\to 0} (z-x)^2 \frac{1}{t}\left(\int_\R \frac{1}{z-y}\,k_{t}(x,{\rm d}y) -\frac{1}{z-x} \right) \\
&= \lim_{t\to 0}\frac{1}{t} \int_\R \frac{(z-x)(x-y)}{z-y}\,k_{t}(x,{\rm d}y).
\end{align*}
In particular for $z= x + i v$ we get
$$
\Im(A(x+iv)) = \lim_{t\to 0}\frac{1}{t} \int_\R \frac{(x-y)^2 v}{v^2+(x-y)^2}\,k_{t}(x,{\rm d}y).
$$
Fix $x \in \R$. Since $A(x+ iv) = o(v)$ as $v\to\infty$, for every $\epsilon>0$ there exists $v_0>0$ such that $|A(x+iv_0)| < \epsilon v_0$. This shows that there exists $t_0>0$ such that for all $0<t<t_0$
$$
\frac{1}{t} \int_\R \frac{(x-y)^2}{v_0^2+(x-y)^2}\,k_{t}(x,{\rm d}y) < 2\epsilon.
$$
For some $y_0>0$ we have $1+y^2 \geq \frac{1}{2}(v_0^2+(x-y)^2)$ for all $y \geq y_0$, and hence
$$
\frac{1}{t}\int_{|y|> y_0} \frac{(x-y)^2}{1+y^2}\,k_{t}(x,{\rm d}y) < 4\epsilon, \qquad 0 < t< t_0.
$$
This implies that the family $\{\frac{(x-y)^2}{t(1+y^2)}\,k_{t}(x,{\rm d}y):  0<t< t_0\}$ is tight.  Also this family is uniformly bounded, because there exists a constant $C>0$ not depending on $y$ such that  we have $1 \leq C \frac{1+y^2}{1+(x-y)^2}$ for all $y \in \R$, and hence
$$
\frac{1}{t}\int_{\R} \frac{(x-y)^2}{1+y^2}\,k_{t}(x,{\rm d}y) \leq \frac{C}{t}\int_{\R} \frac{(x-y)^2}{1+(x-y)^2}\,k_{t}(x,{\rm d}y) \to C\,\Im(A(x+i)), \qquad t\downarrow0.
$$
Take a weak limit $\rho'$ of this family. Then
\begin{align*}
&A(z)
= \lim_{t\to 0}\frac{1}{t} \int_\R \frac{(z-x)(x-y)}{z-y}\,k_{t}(x,{\rm d}y) \\
&=  \lim_{t\to 0}\frac{1}{t} \int_\R \left( x-y - \frac{(x-y)^2}{z-y}\right)k_{t}(x,{\rm d}y) \\
&=  \lim_{t\to 0}\frac{1}{t} \int_\R \left(\frac{1+y^2}{y-z} -y\right) \frac{(x-y)^2}{1+y^2} k_{t}(x,{\rm d}y) - \lim_{t\to 0}\frac{1}{t} \int_\R \left( y-x - \frac{(x-y)^2 y}{1+y^2}\right)k_{t}(x,{\rm d}y) \\
&=  \int_\R \frac{1+y z}{y-z}\rho'({\rm d}y) - \lim_{t\to 0}\frac{1}{t} \int_\R \left( y-x - \frac{(x-y)^2 y}{1+y^2}\right)k_{t}(x,{\rm d}y).
\end{align*}
By the uniqueness of the \index{Pick-Nevanlinna representation}Pick-Nevanlinna representation, we conclude that
$$
\rho'=\rho \text{~~and~~} \gamma = \lim_{t\to 0}\frac{1}{t} \int_\R \left( y-x - \frac{(x-y)^2 y}{1+y^2}\right)k_{t}(x,{\rm d}y),
$$
which shows the desired claim.
\end{proof}
In order to state \index{Hunt's formula}Hunt's formula we introduce the free difference quotient $\partial\colon C^1(\R)\to C(\R^2)$ by
\begin{equation}\label{eq:FDQ}
(\partial f) (x,y) = \begin{cases} \frac{f(x)-f(y)}{x-y}, & x \neq y, \\ f'(x), & x=y. \end{cases}
\end{equation}
Then for $f \in C^2(\R)$ we have
\begin{equation}\label{eq:Ans}
(\partial_x\partial f)(x,y) = \begin{cases}\frac{f(y)-f(x)-(y-x)f'(x)}{(y-x)^2} ,&x\neq y, \\ \frac{1}{2}f''(x), & x=y.
\end{cases}
\end{equation}
Notice that the operator \eqref{eq:Ans} was previously used by Anshelevich in a similar context \cite{Ans13}.

Now we can compute the generators of the Markov processes.  
Let $\mathcal B_b(\R)$ be the set of all bounded Borel measurable functions $f \colon \R \to \C$. 
\begin{theorem}\label{thm:Hunt_additive}
Let $(M_t)_{t\geq0}$ be a stationary $\rhd$-homogeneous Markov process with transition kernels $(k_t)_{t\geq0}$. Let $T_t\colon \mathcal B_b(\R)\to \mathcal B_b(\R)$ be its transition semigroup
\begin{equation*}
(T_t f)(x) = \int_\R f(y) \,k_t(x,{\rm d}y), \qquad f \in \mathcal B_b(\R),
\end{equation*}
which satisfies $T_s T_t = T_{s+t}$ for $s,t\geq0$.
The generator of the transition semigroup is then given by
\begin{align*}
(G f)(x) &:= \left.\frac{{\rm d}}{{\rm d} t}\right|_{t=0} (T_t f)(x) \\
&=  \gamma f'(x) + \int_\R\left\{(1+y^2)(\partial_x\partial f)(x,y) +y f'(x)\right\} \rho({\rm d}y)
\end{align*}
for $f \in C_b(\R)\cap C^2(\R)$ and $x\in\R$, where $(\gamma,\rho)$ is the pair in \eqref{MLK1} associated to the monotone convolution semigroup $\{k_t(0,\cdot)\}_{t\geq0}$.
\end{theorem}
\begin{proof} For $f \in C_b(\R)\cap C^2(\R)$ the identity
\begin{align*}
&\int_\R f(y) \,k_t(x,{\rm d}y) -f(x) = \int_{\R\setminus\{x\}} \{ f(y)-f(x)\} \,k_t(x,{\rm d}y) \\
&=\int_{\R\setminus\{x\}}\left(\frac{(1+y^2)\{f(y)-f(x)-(y-x)f'(x)\}}{(y-x)^2} +y f'(x)\right) \frac{(y-x)^2}{1+y^2} k_t(x,{\rm d}y)  \\
& \qquad + f'(x)\int_{\R\setminus\{x\}}\left(y-x -  \frac{(y-x)^2 y}{1+y^2}\right) k_t(x,{\rm d}y) \\
&= \int_\R \left\{(1+y^2)(\partial_x\partial f)(x,y) +y f'(x)\right\} \frac{(y-x)^2}{1+y^2} k_t(x,{\rm d}y) \\
&\qquad + f'(x)\int_\R\left(y-x -  \frac{(y-x)^2 y}{1+y^2}\right) k_t(x,{\rm d}y)
\end{align*}
holds. Lemma \ref{lem:generator} implies the desired formula.
\end{proof}

\begin{remark} The term
$y f'(x)$ in the generator is a compensator: it is placed so that the integral converges. If the measure $\rho$ has a finite first moment, then we can integrate out the compensator to get the reduced form
\begin{align}\label{eq:reduced}
(G f)(x) =  \tilde\gamma f'(x) + \int_\R (1+y^2)(\partial_x\partial f)(x,y)\, \rho({\rm d}y),
\end{align}
where
$$
\tilde\gamma = \gamma + \int_{\R} y \,\rho({\rm d}y).
$$
\end{remark}

\begin{example} \begin{enumerate}
\item If $\gamma=0$ and $\rho = \delta_0$, then the generator is
$$
(G f)(x)
= \begin{cases}
\frac{f(0)-f(x) +x f'(x)}{x^2} ,& x \neq0, \\
\frac{1}{2} f''(0), & x=0,
\end{cases}
$$
and $k_t(0,\cdot)$ is the \index{arcsine distribution}arcsine law with mean $0$ and variance $t$. The infinitesimal generator (see Theorem \ref{MLK}) for the $\rhd$-convolution semigroup $\{k_t(0,\cdot)\}_{t\geq0}$ is given by $A(z) = -\frac{1}{z}$, and hence $F_{k_t(0,\cdot)}(z)=\sqrt{z^2 -2t}$.
The transition probability of the associated stationary Markov process is
\begin{equation*}
k_t(x,{\rm d}y) = \frac{\sqrt{2t-y^2}}{\pi (x^2-y^2 +2t)} \,\mathbf1_{(-\sqrt{2t},\sqrt{2t})}(y)\,{\rm d}y + \frac{|x|}{\sqrt{x^2+2t}}\delta_{\text{sign}(x)\sqrt{x^2+2t}}({\rm d}y).
\end{equation*}

\item If $\gamma=-\lambda/2$ and $\rho= (\lambda /2)\delta_1$ for $\lambda>0$, then
$$
(G f)(x)
= \begin{cases}
\lambda\frac{f(1)-f(x) -(1-x) f'(x)}{(1-x)^2} ,& x \neq1, \\
\frac{\lambda}{2} f''(1), & x=1,
\end{cases}
$$
and $k_t(0,\cdot)$ is the monotone Poisson law with mean $\lambda t$ and variance $\lambda t$ \cite{Mur01b}.

\item 
The monotone stable process with index $\alpha \in(0,2)$ and asymmetry parameter $\beta \in [1-1/\alpha,1/\alpha]\cap[0,1]$ has the distribution characterized by $F_{\mu_t}(z) = (z^\alpha+ t e^{i \alpha \beta \pi})^{1/\alpha}$ (see also Example \ref{mst}). The function $A(z)$ in Theorem \ref{MLK} is given by $\frac{1}{\alpha}e^{i\alpha\beta\pi}z^{1-\alpha}$. The pair $(\gamma, \rho)$ is computed by the Stieltjes inversion formula to be $\gamma=-\alpha^{-1} \sin\alpha(\beta-1/2)\pi$ and
\begin{equation*}
(1+y^2)\rho({\rm d}y) = \frac{\sin(\alpha\beta\pi)}{\alpha\pi} \mathbf1_{(0,\infty)}(y) y^{1-\alpha}\, 
{\rm d}y + \frac{\sin(\alpha(1-\beta)\pi)}{\alpha\pi} \mathbf1_{(-\infty,0)}(y) |y|^{1-\alpha}\,  {\rm d}y.
\end{equation*}
In particular, if $1<\alpha \leq 2$, then $\rho$ has a finite first moment, and one can show that $\tilde \gamma =0$ in the reduced version \eqref{eq:reduced}.

\item If $\alpha=1/2$ and $\beta=1$ in the above example, then the transition kernel $k_t(x,\cdot) = \delta_x \rhd \mu_t$ is explicitly given by
\begin{equation*}
k_t(x, {\rm d}y) = \frac{2t\sqrt{y}}{\pi[y^2+2(t^2-x)y+(t^2+x)^2]} \mathbf1_{(0,\infty)}(y)\,{\rm d}y +\lambda \delta_{-(\sqrt{|x|}-t)^2},
\end{equation*}
where
\begin{align*}
\lambda &= \begin{cases}  1-\frac{t}{\sqrt{|x|}}, & x< -t^2, \\ 0, & x \geq -t^2. \end{cases}
\end{align*}
The generator is given by
\begin{equation*}
(G f)(x) =  -\sqrt{2} f'(x) + \int_{0}^\infty \left( \frac{f(y)-f(x)-(y-x)f'(x)}{(y-x)^2} + \frac{y}{1+y^2}f'(x)\right) \frac{2\sqrt{y}}{\pi}\,{\rm d}y
\end{equation*}
for $f \in C_b(\R)\cap C^2(\R)$ and $x\in\R$.

\end{enumerate}
\end{example}

\subsection{Feller property in the stationary case}\label{sec:Feller_additive}
\index{Feller property|(}
We continue the study of the stationary case. We show below that a stationary $\rhd$-homogeneous Markov process actually is a Feller process. The advantage of Feller processes includes the fact that the process has a cadlag version and  has the strong Markov property. We refer the reader to \cite[Chap.\ 19]{kallenberg02} and \cite[Chapter III]{RY99} for Feller processes and semigroups.

\begin{definition}\label{def:Feller}
Let $A$ be a locally compact Hausdorff space. Let $C_0(A)$ be the Banach space of continuous functions $f\colon A \to \C$ vanishing at infinity, equipped with the uniform norm. A one-parameter semigroup $(Q_t)_{t\geq0}$ of linear operators on $C_0(A)$ is called a \emph{Feller semigroup} if
\begin{enumerate}[\quad(F1)]
\item\label{F1} $0\leq Q_t f \leq 1$ whenever $0 \leq f \leq 1$ and $t\geq0$,
\item\label{F2} $(Q_tf)(x)\to f(x)$ as $t\downarrow0$ for all $f \in C_0(A)$ and $x\in A$.
\end{enumerate}
\end{definition}

Let $S_t := T_t|_{C_0(\R)}$, where $(T_t)_{t\geq0}$ is the transition semigroup on $\mathcal B_b(\R)$ defined in Theorem \ref{thm:Hunt_additive}.
\begin{theorem}\label{thm:Feller}
The family $(S_t)_{t\geq0}$ is a Feller semigroup on $C_0(\R)$.
\end{theorem}
\begin{proof}
We need to prove that $S_t$ defines a linear operator on $C_0(\R)$. The set 
\begin{equation*}
D=\text{\rm span}_\C\left\{\frac{1}{z-x}: z \in \C \setminus\R\right\}
\end{equation*}
is dense in $C_0(\R)$ by Lemma \ref{lem:approx}. By the Markov property it is easy to show that $S_t D \subset D$. Indeed,
\begin{equation}\label{eq:D}
S_t \sum_{i=1}^n \frac{\lambda_i}{z_i - x} = \sum_{i=1}^n \frac{\lambda_i}{F_t(z_i) - x} \in D,
\end{equation}
where $F_t$ is the reciprocal Cauchy transform of $k_t(0,\cdot)$ as before. By approximating $C_0(\R)$ by $D$ and using $\sup_{x \in \R}|S_tf(x)| \leq \|f\|_{C_0(\R)}$ we conclude that $S_t C_0(\R) \subset C_0(\R)$.

The property (F\ref{F1}) is obvious, since each $k_t(x,\cdot)$ is a probability measure on $\R$.
For continuity at $t=0$, first observe that for $g \in D$ and $x\in\R$ we have $(S_t g)(x) \to g(x)$ by using \eqref{eq:D} and the fact that $F_t(z)$ is right-continuous (even right-differentiable) at $t=0$. For a general $f \in C_0(\R)$ we can take $g \in D$ close in uniform norm to $f$, so that
\begin{align*}
|(S_t f)(x) - f(x)|
&\leq |(S_t (f-g))(x)| +|(S_t g)(x)-g(x)| + |g(x)-f(x)|\\
&\leq 2\|f-g\|_{C_0(\R)} +|(S_t g)(x)-g(x)|,
\end{align*}
and hence $(S_t f)(x) \to f(x)$ for all $x\in\R$ as $t\downarrow0$.
\end{proof}
{}From the theory of Feller processes, the family $(S_t)_{t\geq0}$ has the infinitesimal generator $L$ with the domain
\begin{equation*}
\Dom(L) = \left\{f \in C_0(\R): \lim_{t\downarrow0} \frac{S_t f -f}{t} \text{~exists in the uniform norm}\right\}.
\end{equation*}
Then $Lf$ is defined to be the limit above for all $f \in \Dom(L)$. {}From Theorem \ref{thm:Hunt_additive}, for $f \in \Dom(L) \cap C^2(\R)$ the generator $L$ coincides with $G$. We show that $D$ is contained in $\Dom(L)$ (it is clear that $D \subset C^2(\R))$.

\begin{proposition} We have that $D \subset \Dom(L)$, and hence for functions in $D$ the convergence of the Hunt formula in Theorem \ref{thm:Hunt_additive} holds with respect to the uniform norm.
\end{proposition}
\begin{proof} By linearity it suffices to take $g(x) =\frac{1}{z - x} $. Using \eqref{eq:D} we have
\begin{equation*}
(S_t g) (x) - g(x) = - \frac{F_t(z)-z}{(F_t(z)-x)(z-x)}.
\end{equation*}
Let $A(z):= \lim_{t\downarrow0}(F_t(z)-z)/t$ which exists locally uniformly on $\C^+$ from \cite{BP78}. Then
\begin{align*}
&\frac{(S_t g) (x) - g(x)}{t} + \frac{A(z)}{(z-x)^2} \\
&= \frac{1}{z-x} \left[ \left( A(z) -\frac{F_t(z)-z}{t}\right) \frac{1}{F_t(z)-x} + A(z)\left( \frac{1}{z-x} -\frac{1}{F_t(z)-x}\right) \right],
\end{align*}
which converges to zero uniformly in $x \in \R$.
\end{proof}
It is known from \cite[Prop.\ 19.9]{kallenberg02} that any dense subspace of $\Dom(L)$ which is invariant by operators $\{S_t:t\geq0\}$ is a core for $L$. Thus we conclude that $D$ is a core for $L$.

\begin{problem}
Is it true that $C_c^\infty(\R) \subset \Dom(L)$, or more strongly $C_c^2(\R) \subset \Dom(L)$?
\end{problem}

\index{Feller property|)}

\subsection{Martingale property}
\index{Martingale property|(}

A martingale property for (not necessarily stationary) $\rhd$-homogeneous Markov processes follows from \eqref{eq:Markovianity}. 

\begin{proposition}\label{prop:martingale1} Let $(M_t)_{t\ge 0}$ be a $\rhd$-homogeneous Markov process with a filtration $(\mathcal F_t)$ such that $M_0=0$, and let $(F_t)_{t\geq0}$ be the associated additive Loewner chain. Then, for each fixed $u\geq0$ and $z \in F_u(\C^+)$ the process $(N_t^{z})_{0\leq t \leq u}$ defined by 
$$
N_t^{z}=\frac{1}{F_t^{-1}(z)-M_t} 
$$ 
is an $(\mathcal{F}_t)$-martingale.
\end{proposition}
\begin{proof}
Note that each map $F_t$ is univalent by Theorem \ref{EV_univalence}. 
\end{proof}

We show that $\rhd$-homogeneous Markov processes with suitable shifts are martingales. This fact was observed by Wang and Wendler \cite{WW13} in the discrete time setting. We do not assume the finite variance of the process, which was assumed in \cite{WW13}.

\begin{proposition} Let $(M_t)_{t\ge 0}$ be a $\rhd$-homogeneous Markov process with a filtration $(\mathcal F_t)$ such that $M_0=0$ and $\mathbb E |M_t| <\infty$ for all $t>0$. Then the process $(M_t - \mathbb E[M_t])_{t\geq0}$ is an $(\mathcal{F}_t)$-martingale.
\end{proposition}
\begin{proof} There is a proof based on Proposition \ref{prop:martingale1}, but we give another proof. Let $(k_{st})_{0\leq s \leq t }$ be the transition kernels.
The Chapman-Kolmogorov relation implies
$$
\infty > \int_\R |y| \, k_{0t}(0, {\rm d}y) = \int_\R \left( \int_{\R} |y| \,  k_{st}(x, {\rm d}y) \right) k_{0s}(0, {\rm d}x),
$$
which implies that
$$
\int_{\R} |y| \,  k_{st}(x, {\rm d}y) <\infty
$$
for almost all $x$ with respect to $k_{0s}(0, {\rm d}x)=\mathbb P (M_s \in {\rm d}x)$. For such $x$ or $x=0$, the equation $F_{k_{st}(x, \cdot)}(z)=F_{k_{st}(0,\cdot)}(z)-x$ and Lemma \ref{first_moment} show that
\begin{eqnarray*}
&&\int_{\R} y \,  k_{st}(x, {\rm d}y)
=  \lim_{y\to\infty} (iy -F_{k_{st}(0,\cdot)}(iy) +x)= \int_{\R} y \,  k_{st}(0, {\rm d}y) +  x.
\end{eqnarray*}
This shows that, by the Markov property \eqref{eq:Markov2},  
$$
\mathbb E [M_t | \mathcal F_s] = M_s + \int_{\R} y \,  k_{st}(0, {\rm d}y)~~a.s.
$$
Taking the expectation shows that the last integral is equal to $\mathbb E[M_t-M_s]$. 
\end{proof}

\index{Martingale property|)}
\index{Markov process!$\rhd$-homogeneous|)}


\section[Alternative constructions of SAIPs]{Alternative constructions of additive monotone increment processes}\label{alternative_constructions}

One can also use quantum stochastic calculus to construct additive monotone increment processes as operators on the symmetric Fock space. We will present two such constructions here. It should be remarked that both these constructions require stronger conditions than the construction via classical Markov processes that we presented above.

The first construction uses results of \cite{franz-unif} to reduce the problem of constructing monotone increment processes to Sch\"urmann's theory of L\'evy processes on involutive bialgebras \cite{schurmann}. This requires in particular stationarity, i.e., that the distributions of increments depend only on the difference $t-s$. We will give only a very rough sketch of this construction, more details about the tools that have to be used here can be found in Franz' lecture in \cite{franz+skalski}.

The second construction was carried out by Belton \cite{belton-mon}, who proved that a class of vacuum-adapted stochastic integrals produces stochastic processes with monotonically independent increments. We will show how the coefficients can be chosen to insure that the resulting process is associated to a given additive Loewner chain. This construction requires all distributions to be compactly supported. It produces bounded operators on the symmetric Fock space, which can easily be seen to be self-adjoint.

For the prerequisites on quantum stochastic calculus we refer to \cite{partha-book,meyer,lindsay}.


\subsection{Generating functionals and Sch\"urmann triples}
\index{generating functional|(}
\index{Sch\"urmann triple|(}

Let $(\mu_t)_{t\ge 0}$ be a weakly continuous monotone convolution semigroup of probability measures on $\mathbb{R}$ which admit moments of all orders and $\mu_0=\delta_0$. Then we can associate to it a linear functional $\psi$ on the algebra of polynomials $\mathbb{C}[x]$ in one self-adjoint variable by setting
\[
\psi(p) = \lim_{t\searrow0}\frac{1}{t} \left(\int_\mathbb{R} p(x)\mu_t ({\rm d}x) - p(0)\right), \qquad p\in\mathbb{C}[x].
\]
This functional is hermitian, positive on positive polynomials vanishing at $0$, and maps constant functions to zero. We call a functional on $\mathbb{C}[x]$ with these properties a \emph{generating functional}. By a GNS-type construction one can extend any generating functional $\psi:\mathbb{C}[x]\to\mathbb{C}$ to a \emph{Sch\"urmann triple} $(\rho,\eta,\psi)$ of linear maps $\rho:\mathbb{C}[x]\to \mathcal{L}(D)$, $\eta:\mathbb{C}[x]\to D$, $\psi:\mathbb{C}[x]\to\mathbb{C}$. Here $D$ is a pre-Hilbert space, $\mathcal{L}(D)$ denotes the *-algebra of adjointable operators on $D$, i.e.,
\[
\mathcal{L}(D)=\{f:D\to D \mbox{ linear s.t.\  }\exists f^*:D\to D\mbox{ satisfying }\forall u,v\in D: \langle u,fv\rangle = \langle f^*u,v\rangle\},
\]
and the three maps $(\rho,\eta,\psi)$ verify the relations
\begin{align*}
&\rho(pq) = \rho(p)\rho(q), \quad \rho(p^*)=\rho(p)^*, \quad \rho(1)={\rm id}_D, \\
&\eta(pq) = \rho(p)\eta(q) + \eta(p)q(0), \\
&\langle \eta(p^*),\eta(q)\rangle = \psi(pq)-p(0)\psi(q) - \psi(p)q(0)
\end{align*}
for $p,q\in\mathbb{C}[x]$. Denote by $\varepsilon:\mathbb{C}[x]\to\mathbb{C}$ the evaluation at $0$, i.e., $\varepsilon(p)=p(0)$. Then the relations mean that $\rho$ is a unital *-representation, $\eta$ is a $1$-$\rho$-$\varepsilon$-cocycle, and $(p,q)\mapsto \langle\eta(p^*),\eta(q)\rangle$ is the $2$-$\varepsilon$-$\varepsilon$-coboundary of $\psi$ in the usual Hochschild cohomology of associative algebras (where we view $D$ and $\mathbb{C}$ as $(\rho,\varepsilon)$-bimodule and as $(\varepsilon,\varepsilon)$-bimodule, respectively). This is actually a special case of a more general construction for generating functionals on augmented *-algebras, see \cite{schurmann,franz+skalski}. In the constructions below we will only need the values $R=\rho(x)$, $e=\eta(x)$, and $b=\psi(x)$ of these maps.

Let $\psi$ be a generating functional. Then there exists a triple $(b,\sigma,\nu)$, called the \emph{characteristic triple} of $\psi$, where $b,\sigma\in\mathbb{R}$, $\sigma\ge 0$, and $\nu$ is a finite non-negative measure on $\mathbb{R}$ with $\nu(\{0\})=0$ and which admits moments of all orders, such that
\begin{equation}\label{eq-LK-type}
\psi(p) = b p'(0) + \frac{\sigma^2}{2} p''(0) + \int_{\mathbb{R}\backslash\{0\}} \frac{p(u)-p(0)-p'(0)u}{u^2} \nu({\rm d}u)
\end{equation}
for $p\in\mathbb{C}[x]$. If $\psi(x^2)=0$, then one sets $b=\psi(x)$, $\sigma=0$, $\nu=0$. If $\psi(x^2)\not=0$, then one chooses a probability measure  $\mu$ that solves the moment problem $\int_\mathbb{R}u^k\mu({\rm d}u)= \psi(x^{k+2})/\psi(x^2)$, $k\in\mathbb{N}$, and sets $b=\psi(x)$, $\sigma^2=\psi(x^2)\mu(\{0\})$, and $\nu = \psi(x^2)\big(\mu-\mu(\{0\})\delta_0\big)$.
Note that $\nu$ is obtained as solution of a moment problem, and it is in general not uniquely determined.

Let $H=\mathbb{C}\oplus L^2(\mathbb{R},\frac{1}{x^2}\nu)$. Then one can check that $\eta(p)=\left(\sigma p'(0),p-p(0)\right)$ for $p\in\mathbb{C}[x]$, $D=\eta(\mathbb{C}[x])\subseteq H$,
and $\rho(p)(\lambda, f)=(p(0)\lambda,pf)$ for $(\lambda,f)\in\mathbb{C}\oplus D$, $p\in\mathbb{C}[x]$, defines a Sch\"urmann triple for $\psi$. The operator $\rho(p)$ may be unbounded, but it is well-defined on the space $D$, which is spanned by polynomials, because all moments of $\nu$ are finite.

Let $e=\eta(x)$, $R=\rho(x)$, then we have
\begin{equation}\label{eq-psi-triple}
\psi(x^n) =
\left\{\begin{array}{cl}
0 & \mbox{ if } n=0, \\
b & \mbox{ if } n=1, \\
\langle e,e\rangle = \sigma^2 + \nu(\mathbb{R}) &\mbox{ if } n=2, \\
\langle e,R^{n-2}e \rangle = \int x^{n-2}\nu({\rm d}x) & \mbox{ if } n\ge 3.
\end{array}\right.
\end{equation}

Let us summarize this result.

\begin{proposition}\label{prop-coeff-schurmann}
Let $\psi:\mathbb{C}[x]\to\mathbb{C}$ be a generating functional. Then there exists a pre-Hilbert space $D$, a symmetric operator $R$ acting on $D$, a vector $e\in D$, and a real number $b$ such that $\psi$ is given by \eqref{eq-psi-triple}.
\end{proposition}

For Belton's construction of not necessarily stationary monotone increment processes we have to construct a family of triples $(R_t, e_t, b_t)_{t\ge 0}$ acting on a fixed Hilbert space $H$ for a time-dependent family of generating functionals $(\psi_t)_{t\ge 0}$. So let $(\psi_t)_{t\ge 0}$ be a family of generating functionals such that the following conditions hold:
\begin{itemize}
\item[(M)]
Measurability: the functions $\mathbb{R}_+\ni t\mapsto \psi_t(x^k)\in \mathbb{C}$ are measurable for all $k\in\mathbb{N}$;
\item[(L)]
Local integrability and boundedness: $\psi_t(x),\psi_t(x^2)\in L^1_{\rm loc}(\mathbb{R}_+)$, i.e.,
\[
\int_0^T |\psi_t(x)|{\rm d}t <\infty \quad\mbox{ and }\quad \int_0^T |\psi_t(x^2)|{\rm d}t <\infty 
\]
for all $T>0$, and the function $M:\mathbb{R}_+\to\mathbb{R}_+$,
\[
M(t) =
\left\{\begin{array}{ll}
0 & \mbox{ if } \psi_t(x^2)=0, \\
\sup_{k\ge 1} \left|\frac{\psi_t(x^{k+2})}{\psi_t(x^2)}\right|^{1/k} & \mbox{ else } 
\end{array}\right.
\]
is locally bounded, i.e.
\[
\sup_{0\le t\le T} M(t) < \infty
\]
for all $T>0$.
\end{itemize}

We start by rephrasing these conditions in terms of the characteristic triples.
\begin{proposition}\label{prop-cond-psi}
Let $(\psi_t)_{t\ge 0}$ be a family of generating functionals on $\mathbb{C}[x]$, and let\\ $\big((b_t,\sigma_t,\nu_t)\big)_{t\ge 0}$ be a family of characteristic triples associated to it as in Equation \eqref{eq-LK-type}.

Then $(\psi_t)_{t\ge 0}$ satisfies conditions (M) and (L) if and only if $\big((b_t,\sigma_t,\nu_t)\big)_{t\ge 0}$ satisfies the following three conditions:
\begin{itemize}
\item[(CP)] Compactness:
for all $T>0$, the measures $\nu_t$ in the triple $(b_t,\sigma_t,\nu_t)$ associated to $\psi_t$ as in Equation \eqref{eq-LK-type} are supported in some compact interval $[-M_T,M_T]$ for $0\le t\le T$;
\item[(WM)] Weak measurability:
the function $\mathbb{R}_+\ni t\to \nu_t$ is weakly measurable, i.e., $\mathbb{R}_+\ni t\to \int_\mathbb{R} f \nu_t({\rm d}x) \in\mathbb{C}$ is measurable for all $f\in C_b(\mathbb{R})$;
\item[(B)] Local boundedness:
the functions $\mathbb{R}_+\ni t\to b_t\in\mathbb{R}$ and $\mathbb{R}_+\ni t\to \nu_t(\mathbb{R})\in\mathbb{R}_+$ are locally integrable, $\mathbb{R}_+\ni t\to \sigma_t\in\mathbb{R}_+$ is locally square integrable. 
\end{itemize}
\end{proposition}
\begin{proof}
Let us first show that if $\psi_t(x^2)\not=0$ for some $t\in\mathbb{R}_+$, then $M(t)$ as defined in Condition (L) is the smallest real number $M>0$ such that $\nu_t$ is supported in $[-M,M]$. Indeed, if the probability measure $\mu_t$ denotes a solution of the moment problem $\int_\mathbb{R} x^k{\rm d}\mu_t=\psi_t(x^{k+2})/\psi_t(x^2)$, then in (L) we take the supremum of
\[
\left|\frac{\psi_t(x^{k+2})}{\psi_t(x^2)}\right|^{1/k} \le \| x\|_{L^k(\mu_t)}
\]
with equality for even $k$. Since the right-hand-side increases to $\|x\|_{L^\infty(\mu_t)}=\|x\|_{L^\infty(\nu_t)}$ as $k\to\infty$, we have proved our claim, and we can deduce that condition (CP) is equivalent to local boundedness of the function $M$ in Condition (L).

Condition (CP) insures that the triples $(b_t,\sigma_t,\nu_t)$ are uniquely determined by $\psi_t$, because the moment problem for compactly supported measures is determinate.

Since bounded continuous functions can be approximated by polynomials on some compact interval containing the support of $\nu_t$ for all $0\le t\le T$, we see that (M) and (WM) are equivalent, when (CP) holds. That local integrability of $\psi_t(x)$ and $\psi_t(x^2)$ is equivalent to (B) follows immediately from the relations $\psi_t(x)=b_t$ and $\psi_t(x^2)=\nu_t(\mathbb{R})+\sigma_t^2$.
\end{proof}

Let us now look at the equivalent conditions for the triples $\big((R_t,e_t,b_t)\big)_{t\ge 0}$.

\begin{proposition}\label{prop-coeff-belton}
Let $(\psi_t)_{t\ge 0}$ be a family of generating functionals satisfying conditions (M) and (L).

Then there exists a Hilbert $H$ and a family of triples $\big((R_t,e_t,b_t)\big)_{t\ge 0}$ such that
\begin{itemize}
\item[(i)]
$\psi_t$ is determined by $(R_r,e_t,b_t)$ via formula \eqref{eq-psi-triple} for all $t\ge0$;
\item[(ii)]
$\mathbb{R}_+\ni t\mapsto R_t\in B(H)$ is weakly measurable, $\mathbb{R}_+\ni t\mapsto e_t\in H$ is weakly measurable, and $\mathbb{R}_+\ni t\mapsto b_t\in \mathbb{R}$ is measurable;
\item[(iii)]
$\int_0^T |b_t|\,{\rm d}t + \left(\int_0^T \|e_t\|^2\,{\rm d}t\right)^{1/2} + \sup_{t\in[0,T]} \|R_t\|$ is finite for all $T>0$.
\end{itemize}
\end{proposition}
\begin{proof}
By Proposition \ref{prop-coeff-schurmann}, there exists a triple $(\tilde{R}_t,\tilde{e}_t,\tilde{b}_t)$ constructed on the pre-Hilbert space $\tilde{D}=\mathbb{C}\oplus L^2(\mathbb{R},\frac{1}{x^2}\nu_t)$ representing $\psi_t$ as in Equation \eqref{eq-psi-triple} for each $t\ge 0$. Weak measurability of $(\psi_t)_{t\ge0}$ immediately gives the measurability of $t\mapsto \tilde{b}_t=\psi_t(x)$.

We construct a family of triples $\big((R_t,e_t,b_t)\big)_{t\ge0}$ acting on $D=\mathbb{C}\oplus \ell^2$, using the theory of orthogonal polynomials. To each of the finite measure $\nu_t$, $t\ge 0$, we can associate a family of polynomials $(p^t_n)_{n=0}^{N(t)}$ with $p^t_0=1$ and  ${\rm deg}(p^t_n)=n$ that are mutually orthogonal with respect to the measure $\nu_t$. If the support of $\nu_t$ is finite, then $N(t)+1$ is equal to its cardinality, otherwise we have $N(t)=\infty$. We normalize these polynomials such that $\int \overline{p_n^t(x)}p^t_n(x)\nu_t({\rm d}x) = \nu_t(\mathbb{R})$ for $n=0,1,\ldots,N(t)$. We can define an isometry $\imath_1: L^2(\mathbb{R},\nu_t)\to \ell^2$ by $\imath_1(p^t_n)=\sqrt{\nu_t(\mathbb{R})} v_n$, $n=0,1,\ldots,N(t)$, where $(v_n)_{n\ge0}$ denotes the canonical orthonormal basis of $\ell^2$. Furthermore, the Hilbert spaces $L^2(\mathbb{R},\frac{1}{x}\nu_t)$ and $L^2(\mathbb{R},\nu_t)$ are isomorphic via $\imath_2:L^2(\mathbb{R},\frac{1}{x^2}\nu_t)\ni f \mapsto \frac{f}{x}\in L^2(\mathbb{R},\nu_t)$.

We set
\[
R_t= \imath \tilde{R}_t\imath^*, \quad e_t=\imath(\tilde{e}_t),\quad b_t=\tilde{b}_t,
\]
for $t\ge 0$, with $\imath=({\rm id}_\mathbb{C}\oplus \imath_1\circ\imath_2):\mathbb{C}\oplus L^2(\mathbb{R},\frac{1}{x^2}\nu_t)\to \mathbb{C}\oplus L^2(\mathbb{R},\nu_t)$. It is clear that this triple also satisfies (i).

Since we have $\tilde{e}_t=(\sigma_t,x)$ and $p_0^t=1$, we get $e_t= (\sigma_t,\sqrt{\nu_t(\mathbb{R})}v_0)$. If $(\alpha^t_n)_{n\in\mathbb{N}}$ and $(\beta^t_n)_{n\in\mathbb{N}}$ denote the Jacobi parameters from the three-term-recurrence relation
\[
xp_n^t = \alpha^t_n p_{n+1}^t + \beta_n^t p_n^t + \alpha_{n-1}^t p^t_{n-1},
\]
then $R_t$ acts as
\[
R_t (\lambda, v_n) = (0,\alpha^t_n v_{n+1} + \beta^t_nv_n+\alpha^t_{n-1} v_{n-1})
\]
on $\mathbb{C}\oplus \ell^2$. So weak measurability of $t\mapsto e_t$ and $t\mapsto R_t$, i.e., measurability of the functions $t\mapsto \langle v,e_t\rangle$ and $t\mapsto \langle v, R_t w\rangle$ for all $v,w\in\ell^2$, is equivalent to measurability of $t\mapsto \nu_t(\mathbb{R})$, $t\mapsto \sigma_t$, and of the $t$-dependence of the Jacobi parameters.

Let $A=\{t\ge 0;\psi_t(x^2)\not=0\}$, this is a measurable set by the weak measurability of $(\psi_t)_{t\ge0}$. Recall that $\sigma_t=0$ and $\nu_t=0$ for $t\not\in A$, and that for $t\in A$ we can define $\nu_t$ as $\nu_t=\psi_t(x^2) \big(\mu_t-\mu_t(\{0\})\delta_0\big)$, where $\mu_t$ is the solution of the moment problem $\int u^k\mu_t({\rm d}u)=\psi_t(x^{k+2})/\psi_t(x^2)$, which is unique thanks to Condition (CP). Using again Condition (CP), we have
\[
G_{\mu_t}(z) = \sum_{k\ge 0} \frac{\psi_t(x^{k+2})}{\psi_t(x^2)z^{k+1}}
\]
for sufficiently large $z$. This is clearly measurable as a function of $t$ for each $z$ with sufficiently large $|z|$, which implies that its analytic continuation to $\mathbb{C}^+$ is also measurable for all $z\in\mathbb{C}^+$. Therefore, by Stieltjes inversion \eqref{eq:atom2}, we get the measurability of
\[
\sigma^2_t = \psi_t(x^2) \mu_t(\{0\}) = \psi_t(x^2) \lim_{\varepsilon\searrow 0} i \varepsilon G_{\mu_t}(i\varepsilon),
\]
and of $\nu_t(\mathbb{R})=\psi_t(x^2)\big(1-\mu_t(\{0\})\big)$.

Similarly we can prove measurability of $\nu_t(B)$ for all Borel subsets $B$ (see the arguments in Theorem \ref{theo-Markov-proc}), of $\int f {\rm d}\nu_t$ for all bounded continuous $f$, and finally of the Jacobi parameters of $\nu_t$.

The compactness condition (CP) and the boundedness assumption (B) insure that (iii) holds.
\end{proof}

\begin{remark}
It is not difficult to check that conversely, if we have a family of triples $\big((R_r,e_t,b_t)\big)_{t\ge0}$ satisfying conditions (ii) and (iii), and define $(\psi_t)_{t\ge0}$ via Equation \eqref{eq-psi-triple}, the $(\psi_t)_{t\ge0}$ will satisfy (M) and (L).
\end{remark}

\index{generating functional|)}
\index{Sch\"urmann triple|)}


\subsection{A construction of monotone L\'evy processes using Sch\"urmann's representation theorem}
\index{monotone L\'evy process!additive|(}

Let $\mathcal{A}=\mathbb{C}[x]$ be the algebra of polynomials in one self-adjoint variable. Note that $x\mapsto 0\in\mathbb{C}$ and $x\mapsto x_1 + x_2\in \mathbb{C}[x_1,x_2]\cong \mathcal{A}\coprod\mathcal{A}$ extend to unique unital  *-homomorphisms $\varepsilon:\mathcal{A}\to\mathbb{C}$ and  $\Delta:\mathcal{A}\to \mathcal{A}\coprod \mathcal{A}$ (where $\coprod$ denotes the coproduct in the category of associative unital *-algebra, i.e., the free product with identification of the unit elements), and that $(\mathcal{A},\Delta,\varepsilon)$ is a dual semigroup in the sense of \cite[Section 2.1]{franz-unif}.

Let $X=(X_t)_{t\ge 0}\subseteq \mathcal{L}(\mathcal{H})$ be a SAIP over some concrete quantum probability space $(\mathcal{H},\Omega)$, where $\mathcal{H}$ is a pre-Hilbert space. Note that the condition $X\subseteq \mathcal{L}(\mathcal{H})$ insures that the definition of monotone independence in Definition \ref{def-mon} makes sense even if $\mathcal{H}$ is only a pre-Hilbert space, and that we can work with $(\mathcal{L}(\mathcal{H}),\Phi_\Omega)$ as an abstract quantum probability space. Then all moments of $X$ exist and, if the distributions of the increments $X_t-X_s$ of $X$ depend only on $t-s$, then $x\mapsto j_{st}(x)=X_t-X_s\in \mathcal{L}(\mathcal{H})$ defines a quantum stochastic process $J=\big(j_{st}:\mathcal{A}\to (\mathcal{L}(H),\Phi_\Omega)\big)_{0\le s\le t}$ with values in the abstract probability space $(\mathcal{L}(\mathcal{H}),\Phi_\Omega)$, which is a monotone L\'evy process in the sense of \cite[Definition 2.5]{franz-unif}.

Let $(\mu_t)_{t\ge 0}$ be the monotone convolution semigroup consisting of the distributions of $X$ w.r.t\ to the state vector in which the increments are monotonically independent. We denote by $\psi$ the \index{generating functional}generating functional defined as the derivative at $t=0$ of $(\mu_t)_{t\ge 0}$. For our construction we will need the pre-Hilbert space $D$, the operator $R=\rho(x)$, the vector $e=\eta(x)$, and the scalar $b=\psi(x)$ obtained from the \index{Sch\"urmann triple}Sch\"urmann triple of $\psi$ in the previous paragraph, see Proposition \ref{prop-coeff-schurmann}.

Let $\mathcal{P}$ be the two-dimensional *-algebra generated by one projection $p$. We equip $\tilde{\mathcal{A}}=\mathcal{A}\coprod\mathcal{P}$ with the *-bialgebra structure given in \cite[Proposition 3.1]{franz-unif} for the monotone case. Franz proved in \cite[Theorem 3.3]{franz-unif} that we have a one-to-one correspondence between monotone L\'evy processes $J$ on $\mathcal{A}$ and a class of L\'evy processes in the sense of Sch\"urmann $\tilde{J}=\big(\tilde{\jmath}_{st}:\tilde{\mathcal{A}}\to (\mathcal{L}(\mathcal{H}),\Phi_\Omega)\big)_{0\le s\le t}$ on the *-bialgebra $\tilde{\mathcal{A}}$ and over some algebraic probability space $(\mathcal{L}(\mathcal{H}),\Phi_\Omega)$.
By Sch\"urmann's representation theorem, cf.\ \cite[Theorem 2.5.3]{schurmann}, $\tilde{J}$ has a realisation on a symmetric Fock space $\Gamma$ acting as possibly unbounded operators on some common invariant dense subspace $\mathcal{H}\subseteq \Gamma$. The operators $\tilde{X}_{t}=\tilde{\jmath}_{0t}(x)$ and $Q_{t}=\tilde{\jmath}_{0t}(p)$ are obtained as solutions of the quantum stochastic differential equations
\begin{eqnarray*}
Q_t &=& {\rm id}_{\Gamma} - \int_0^t Q_s{\rm d}\Lambda_s({\rm id}_D), \\
\tilde{X}_t &=& \int_0^t \left( {\rm d}A_s^+(e) + {\rm d}\Lambda_s(R-{\rm id}_D) + {\rm d}A_s(e) + \psi(x){\rm d}s - \tilde{X}_s {\rm d}\Lambda_s({\rm id}_D)\right).
\end{eqnarray*}
This follows from the extension of the \index{Sch\"urmann triple}Sch\"urmann triple on $\mathcal{A}$ to a Sch\"urmann triple on $\tilde{\mathcal{A}}$ given in \cite[Proposition 3.9]{franz-unif}. Note that we can explicitly solve the first of these equations, $Q_t$ is the second quantization of multiplication by the indicator function $\mathbf{1}_{[t,+\infty)}$ on $L^2(\mathbb{R}_+,D)$, see \cite[Proposition 3.10]{franz-unif}. Equivalently, it is the tensor product $Q_t=P_\Omega\otimes I_{\Gamma_{[t,+\infty)}}$ of the orthogonal projection onto the vacuum vector $\Omega$ with the identity operator, if we use the factorization $\Gamma_{\mathbb{R}_+}\cong \Gamma_{[0,t]}\otimes\Gamma_{[t,+\infty)}$ of the Fock space.

By \cite[Theorem 3.7]{franz-unif}, we can recover a monotone L\'evy process on $\mathbb{C}[x]$ and therefore a SAIP $Y=(Y_t)_{t\ge 0}$ by setting
\[
Y_t = \tilde{X}_t P_t,
\]
where $P_t$ is the tensor product $P_t=I_{\Gamma_{[0,t]}}\otimes P_\Omega$ of the identity operator with the orthogonal projection onto the vacuum vector $\Omega$, w.r.t.\ to the factorization $\Gamma_{\mathbb{R}_+}\cong \Gamma_{[0,t]}\otimes\Gamma_{[t,+\infty)}$. The process $Y$ constructed in this manner is equivalent to $X$, i.e., all their expectation values agree.

To realize a monotone L\'evy processes on a symmetric Fock space we only need to know its generating functional.

We can summarize this as follows.

\begin{theorem}
Let $(c_n)_{n\ge 1}$ be a sequence of real numbers that is conditionally positive definite in the sense that the functional $\psi:\mathbb{C}[x]\to\mathbb{C}$ determined by $\psi(x^n)=c_n$ for $n\ge 1$ and $\psi(\mathbf{1})=0$ is positive on all positive polynomials vanishing at $0$.

Then there exists a family of symmetric operators $X=(X_t)_{t\ge 0}$ defined on a common invariant subspace of the symmetric Fock space $\Gamma=\Gamma(L^2(\mathbb{R},D))$ constructed from some pre-Hilbert space $D$, containing the vacuum vector $\Omega\in\Gamma$, such that the following are true:
\begin{itemize}
\item[(i)]
The algebras $\mathcal{A}_1={\rm span}\{X_{t_1}^k;k\ge 1\}$, $\mathcal{A}_2={\rm span}\{(X_{t_2}-X_{t_1})^k;k\ge 1\}$, $\ldots$, $\mathcal{A}_n={\rm span}\{(X_{t_n}-X_{t_{n-1}})^k;k\ge 1\}$ generated by the increments of $X$ are monotonically independent w.r.t.\ the vacuum state $\Phi_\omega(\cdot)=\langle\Omega,\,\cdot \, \Omega\rangle$ for $n\ge 1$ and any choice of $0<t_1<\cdots< t_n$.

\item[(ii)]
The functions $\mathbb{R}_+\ni t\mapsto \langle\Omega,(X_{s+t}-X_s)^n\Omega\rangle\in\mathbb{C}$ are differentiable in $t$ for $s,t>0$ and $n\ge 1$, and we have
\[
\lim_{t\searrow 0} \frac{1}{t}\langle\Omega, (X_{s+t}-X_s)^n\Omega\rangle = c_n.
\]
\end{itemize}
\end{theorem}

\begin{proof}
This follows from the results in \cite{franz-unif,schurmann} mentioned above.
\end{proof}

\begin{remark}
If $(\mu_t)_{t\ge0}$ is a monotone convolution semigroup of compactly supported probability measures and $(F_t)_{t\ge0}$ denotes its reciprocal Cauchy transforms, then its generator $A(z) = \left.\frac{{\rm d}}{{\rm d}t}\right|_{t=0} F_t(z)$ can be written as $-A(z) = \alpha + \int \frac{1}{z-x} \tau({\rm d}x)$ with a real number $\alpha$ and a compactly supported finite non-negative measure $\tau$.

To construct the SAIP associated to $(\mu_t)_{t\ge0}$, we have to choose the sequence $(c_n)_{n\in\mathbb{N}}$ given by $c_1=\alpha$, and $c_n=\int x^{n-2}\tau({\rm d}x)$ for $n\ge2$, since
\[
\sum_{n=0}^\infty c_n z^{-n-1} = \psi\left(\frac{1}{z-x}\right) = \left. \frac{{\rm d}}{{\rm d}t}\right|_{t=0} \int \frac{1}{z-x} \mu_t({\rm d}x) = - \frac{A(z)}{z^2}.
\]
\end{remark}

\index{monotone L\'evy process!additive|)}


\subsection{Belton's construction}

\index{monotone increment process!additive (SAIP)|(}
\index{quantum stochastic integral|(}

Let us now discuss Belton's approach using vacuum-adapted stochastic calculus to realize not necessarily stationary quantum stochastic processes with monotonically independent increments, cf.\ \cite{belton-mon}.

We continue to use the more basic and more explicit notation of \cite{schurmann,franz+skalski} for stochastic integrals, rather than the more elegant notation of Lindsay and Belton in  \cite{lindsay,belton-mon,belton-iso}. In this construction all stochastic integrals will be vacuum-adapted, i.e., the integrands are of the form
\[
E_t = \hat{E}(s) \otimes P_\Omega \quad \mbox{ on }\quad \Gamma\cong \Gamma_{0,t}\otimes \Gamma_{t,\infty}
\]
for all $t\ge 0$. We denote the vacuum projection acting on the future of $t$ again by $P_t=I_{\Gamma_{[0,t]}}\otimes P_{\Omega}$. For a family of triples $(R_t,e_t,b_t)_{t\ge0}$ as in Proposition \ref{prop-coeff-belton}, we define the associated SAIP $(X_t)_{t\ge0}$ by the quantum stochastic integral
\begin{equation}\label{eq-def-X}
X_t = \int_0^t P_s \left({\rm d}\Lambda(R_s) + {\rm d}A^+(e_s) + {\rm d}A(e_s) + b_s{\rm d}s\right), \qquad t\ge 0.
\end{equation}
Note that in Belton's notation the stochastic integral in Equation \eqref{eq-def-X} is written
\[
X_t = \int_0^t \check{E}\, {\rm d}\Lambda,
\]
with
\[
\check{E}(s) = \left(\begin{array}{cc} b_s & \langle e_s| \\ |e_s\rangle & R_s \end{array}\right)\otimes P_s, \qquad s\in\mathbb{R}_+,
\]
where $\langle e_s|$ and $|e_s\rangle$ denote the linear operators $\langle e_s|:\ell^2\ni v \mapsto \langle e_s,v\rangle\in\mathbb{C}$, $|e_s\rangle:\mathbb{C}\ni\lambda \mapsto \lambda e_s\in\ell^2$.

By the Ito formula for so-called $\Omega$-processes we get the following formulas for the powers of $X_t$.
\begin{lemma}
Let $(X_t)_{t\ge0}$ be defined by \eqref{eq-def-X}. Then we have the following integral representations
\[
X_t^n = \sum_{k=0}^{n-1} \int_0^t X_s^k P_s \,{\rm d}\mathcal{I}^{(n,k)}_s
\]
for $n\ge 2$, where the integrators ${\rm d} \mathcal{I}^{(n,k)}_s$ are given by
\begin{align*}
{\rm d}\mathcal{I}_s^{(n,0)} &= {\rm d}\Lambda(R_s^n) + {\rm d} A^+(R_s^{n-1}e_s)+{\rm d}A(R_s^{n-1}e_s) + \underbrace{\langle e_s,R_s^{n-2} e_s\rangle}_{=\psi_s(x^n)} {\rm d}s, \\
{\rm d}\mathcal{I}_s^{(n,k)} &=  {\rm d}A^+(R_s^{n-1-k}e_s) + {\rm d}A(R_s^{n-1-k}e_s)+ (k+1) \underbrace{\langle e_s,R_s^{n-2-k} e_s\rangle}_{=\psi_s(x^{n-k})} {\rm d}s, \\ & \qquad\qquad\qquad\qquad\qquad\qquad\qquad\qquad \mbox{ for } \quad k\in\{1,\ldots,n-2\} \\
{\rm d}\mathcal{I}^{(n,n-1)} &=  {\rm d}A^+(e_s) + {\rm d} A(e_s) + n \underbrace{b_s}_{=\psi_s(x)} {\rm d}s.
\end{align*}
\end{lemma}
\begin{proof}
Rewrite \cite[Corollary 2.1]{belton-mon} in our notation and reorder the sums. The identification of the coefficients of the `${\rm d}s$'-term with values of $\psi_t$ comes from Equation \eqref{eq-psi-triple}.
\end{proof}
Applying \cite[Theorem 2.1]{belton-mon} with $f=g=0$, we get the following integral equations for the vacuum expectation of $X_t^n$:
\begin{equation}\label{eq-int-eq}
\langle \Omega, X_t^n\Omega\rangle = \int_0^t \sum_{k=0}^{n-1} (k+1) \psi_s(x^{n-k}) \langle\Omega, X_s^k\Omega\rangle\, {\rm d}s.
\end{equation}
Since $(X_t)_{t\ge0}$ is bounded, we can develop the Cauchy transform
\[
G_{X_t}(z) = \langle \Omega, (z-X_t)^{-1}\Omega\rangle
\]
of the law of $X_t$ into a power series
\[
G_{X_t}(z) = \sum_{n=0}^\infty z^{-n-1}\langle\Omega, X_t^n\Omega\rangle
\]
for $z\in \mathbb{C}$ with $|z|> \| X_t\|$.

Using the integral equation \eqref{eq-int-eq}, we can get the following differential equation for $G_X$:
\begin{eqnarray*}
\frac{\partial}{\partial t} G_{X_t}(z) &=& \sum_{n=0}^\infty \sum_{k=0}^{n-1} (k+1) z^{-n-1}\psi_t(x^{n-k}) \langle\Omega, X_t^k\Omega\rangle  \\
&=& \sum_{\ell=0}^\infty z^{-\ell-1} \psi_t(x^\ell) \sum_{k=0}^\infty (k+1) z^{-k} \langle\Omega, X_t^k\Omega\rangle \\
&=& A_t(z) \frac{\partial}{\partial z} G_{X_t}(z)
\end{eqnarray*}
for sufficiently large $|z|$, where
\[
- A_t(z) = z^2\sum_{\ell=0}^\infty z^{-\ell-1} \psi_t(x^\ell) .
\]
In terms of the reciprocal Cauchy transform we get the same equation,
\begin{equation}\label{eq-loew-2}
\frac{\partial}{\partial t} F_{X_t}(z) = A_t(z) \frac{\partial}{\partial z} F_{X_t}(z).
\end{equation}
Let $(b_t,\sigma_t,\nu_t)$ be the characteristic triple associated to $\psi_t$ in Equation \eqref{eq-LK-type}. Then the function $A_t$ can also be written as
\[
-A_t(z) = b_t + \int \frac{1}{z-x} \left(\nu_t+\sigma_t^2 \delta_0\right)({\rm d} x).
\]
\begin{theorem}
Let $(\psi_t)_{t\ge0}$ be a family of \index{generating functional}generating functionals on $\mathbb{C}[z]$ satisfying (M) and (L). Denote by $\big((R_t,e_t,b_t)\big)_{t\ge0}$ the triples associated to $(\psi_t)_{t\ge0}$ by Proposition \ref{prop-coeff-belton}.

Then the vacuum adapted quantum stochastic integral in Equation \eqref{eq-def-X} defines a quantum stochastic process $(X_t)$. Furthermore, this quantum stochastic process is a SAIP and the reciprocal Cauchy transforms
\[
F_{st}(z)= \frac{1}{\langle \Omega, \big(z-(X_t-X_s)\big)^{-1}\Omega\rangle}, \qquad z\in\mathbb{C}^+,\quad 0\le s\le t,
\]
of the distributions of its increments are the \index{transition mappings}transition mappings of the \index{Loewner chain!additive}additive Loewner chain given by the solution of Equation \eqref{eq-loew-2} with initial condition $F_{ss}(z)=z$.
\end{theorem}

\begin{proof}
By Propositions \ref{prop-cond-psi} and \ref{prop-coeff-belton}, the family $\big((R_t,e_t,b_t)\big)_{t\ge0}$ satisfies the properties (ii) and (iii) of that proposition, therefore \cite[Theorem 3.1]{belton-mon} insures that $(X_t)_{t\ge0}$ is well-defined and a SAIP.

Equation \eqref{eq-loew-2} is the Loewner partial differential equation \eqref{EV_Loewner} for the additive Herglotz vector field $M(z,t)=A_t(z)$. This shows that the reciprocal Cauchy transforms of $(X_t)_{t\ge0}$ form indeed an additive Loewner chain, as claimed.
\end{proof}

\index{monotone increment process!additive (SAIP)|)}
\index{quantum stochastic integral|)}




\chapter[UMIPs, Markov Processes, and Loewner Chains]{Multiplicative Monotone Increment Processes, Classical Markov Processes, and Loewner Chains}\label{section_processes_mult}

In this section we will construct a bijection between a  class of unitary operator processes, multiplicative Loewner chains, and some class of Markov processes on the unit circle.

\section[UMIPs to Loewner chains]{From multiplicative monotone increment processes to multiplicative Loew\-ner chains}\label{from_mult_processes_to_Loewner}
\index{monotone increment process!multiplicative (UMIP)|(}

We will consider the following classes of unitary quantum stochastic processes.

\begin{definition}\label{def_umip}
Let $(H,\xi)$ be a concrete quantum probability space and $(U_t)_{t\ge 0}\subseteq B(H)$ a family of unitary operators
with $U_0=I$. We will call $(U_t)$ a \emph{unitary multiplicative monotone increment process (UMIP)} if the following conditions are satisfied: \vspace{0.3em}
\begin{itemize}
	\item[(a)] The mapping $(s,t)\mapsto \mu_{st}$ is continuous w.r.t.\ weak convergence, where $\mu_{st}$ denotes the distribution of $U_s^{-1}U_t$.
    
 %

\vspace{0.3em}

\item[(b)] The tuple 
\[
(U_{t_1}-I, U_{t_1}^{-1}U_{t_2}-I,\ldots,U_{t_{n-1}}^{-1}U_{t_n}-I)
\]
is \index{monotone independence}monotonically independent for all $n\in\mathbb{N}$ and all $t_1,\ldots,t_n\in\mathbb{R}$ s.t.\ $0\le t_1\le t_2\le\cdots\le t_n$.
\end{itemize}
The random variables $\{U_s^{-1} U_t: 0 \leq s \leq t\}$ are called \emph{increments}.  If $U_s^{-1} U_t$ has the same law as $U_{t-s}$ for all $0\leq s \leq t$ we say that increments are stationary and $(U_t)$ is a \index{monotone L\'evy process!multiplicative}\emph{unitary multiplicative monotone L\'evy process}. 
We call $(U_t)$ \emph{centered} if $\langle\xi,U_t\xi\rangle\in\mathbb{R}$ for all $t\ge 0$.
We will call $(U_t)$ \emph{normalized} if $\langle\xi,U_t\xi\rangle=e^{-t}$ for all $t\ge 0$. 
\end{definition}


\begin{theorem}\label{from_UMIP_to_Loewner}
Let $(U_t)_{t\ge 0}$ be a UMIP in a quantum probability space $(H,\xi)$. Let
\[
\psi_{st}(z) = \left\langle \xi, \frac{zU_s^*U_t}{1-zU_s^*U_t}\xi \right\rangle
\text{~~~and~~~}
\eta_{st}(z) = \frac{\psi_{st}(z)}{1+\psi_{st}(z)}
\]
for $0\le s\le t$ and $|z|<1$, and let $\eta_t=\eta_{0t}.$ Then $(\eta_{t})_{t\geq 0}$ is a \index{Loewner chain!multiplicative}multiplicative Loewner chain with \index{transition mappings}transition mappings $(\eta_{st})$.
\end{theorem}
\begin{proof} The continuity of $(s,t)\mapsto \mu_{st}$ implies continuity of $(s,t)\mapsto \eta_{st}$ 
 due to Lemma \ref{lemmaconvergenceunit}.
The result now follows readily from $U_s^* U_u = U_s^* U_t U_t^* U_u$, the fact that the pair $(U_s^* U_t-I,U_t^* U_u-I)$ is monotonically independent, Remark \ref{u-I}, and the formula \eqref{eq:multiplicative_monotone_convolution}.
\end{proof}

\begin{theorem}\label{thm:equiv_umip}\index{equivalence!of quantum stochastic processes} Two UMIPs are equivalent (in the sense of Definition \ref{def:equiv_saip}) if and only if the distributions of increments coincide. 
\end{theorem}
\begin{proof} 
The only if part is similar to the proof of Theorem \ref{thm:equiv_saip}: we now take $n=1$ and $f_1(x)=x z /(1-x z)$ for $z \in \disk$, and use $\eta$-transforms with Theorem \ref{from_UMIP_to_Loewner}. 

For the if part, by approximation it suffices to consider $f_i(x)=x^{k_i}$ for integers $k_i$. For example if $n=3$, $k_1, k_2 \geq 0 > k_3$ and $t_2  \leq t_1 \leq t_3$ we write
\begin{equation*}
\langle \xi, f_1(U_{t_1}) f_2(U_{t_2})f_3(U_{t_3})\xi\rangle = \langle \xi, (U_{t_2}U_{t_2}^{-1}U_{t_1})^{k_1}(U_{t_2})^{k_2} ((U_{t_1}^{-1}U_{t_3})^*U_{t_1}^*)^{|k_3|}\xi  \rangle. 
\end{equation*}
The product of unitaries in the RHS can further be expressed as a polynomial of $U_{t_2} -I, U_{t_2}^{-1}U_{t_1}-I, (U_{t_1} -I)^*$ and $(U_{t_1}^{-1}U_{t_3}-I)^*$. Then thanks to monotone independence for the increments the inner product of each monomial factorizes, and each factor can be computed in terms of the distribution of an increment. The final expression only depends on the distributions of increments.   
\end{proof}

\index{monotone increment process!multiplicative (UMIP)|)}

\section[From Loewner chains to Markov processes]{From multiplicative Loewner chains to $\submm$-homogeneous Markov processes}

\index{Markov process!$\submm$-homogeneous|(}

\begin{definition}\label{def_equi_mult_Markov}
A probability kernel $k$ on $\T$ is called \emph{$\submm$-homogeneous} if it satisfies
\[
\delta_x\submm k(y,\,\cdot\,) = k(x y,\,\cdot\,)
\]
for all $x,y\in\tor$. A Markov process $(M_t)_{t\geq0}$ with values in $\T$ is called a \emph{$\submm$-homogeneous Markov process} if its transition kernels $(k_{st})_{0\le s\le t}$
 satisfy the following conditions:
 \vspace{0.3em}
 
\begin{itemize}
\item[(a)] The mapping $(s,t)\mapsto k_{st}(x,\cdot)$ is continuous w.r.t.\ weak convergence for all $x\in\T$.

%

\vspace{0.3em}

\item[(b)] The kernel $k_{st}$ is \emph{$\submm$-homogeneous} for all $0\le s\le t$.
\end{itemize}
\end{definition}

\begin{theorem}\label{mult-Markov-proc}
Let $(\eta_{t})_{t\geq 0}$ be a \index{Loewner chain!multiplicative}multiplicative Loewner chain with \index{transition mappings}transition mappings
 $(\eta_{st})_{0\le s\le t}$. Then there exists a $\submm$-homogeneous Markov process $(M_t)_{t\ge 0}$ with $M_0=1$ whose transition kernels $(k_{st})_{0\le s\le t}$ are determined by
\begin{equation}\label{eq:kernel2}
\int_\mathbb{T} \frac{zy}{1-zy} k_{st}(x,{\rm d}y) = \frac{\eta_{st}(z)x}{1-\eta_{st}(z)x} ~\Big(= \psi_{\delta_x}(\eta_{st}(z)) \Big)
\end{equation}
for $0\le s\le t$ and $x\in\mathbb{T}$.
\end{theorem}
\begin{proof}
The proof is similar to that of Theorem \ref{theo-Markov-proc}; one only needs to use the kernel function $zy/(1-zy)$ instead of $1/(z-y)$. The existence of a probability measure $k_{st}(x,\cdot)$ follows from \eqref{eq:kernel2} and Lemma \ref{characterizationpsi}. The weak continuity of $(s,t)\mapsto k_{st}(x,\cdot)$ follows from Lemma \ref{lemmaconvergenceunit}. The measurability of $x\mapsto k_{st}(x,B)$ follows from the inversion formulas \eqref{eq:Herglotz-inv1} and \eqref{eq:atom3}.
\end{proof}

Let $(M_t)$ be the Markov process constructed in Theorem \ref{mult-Markov-proc}. The Markov property \eqref{eq:Markov2} entails 
\begin{equation}\label{eq:Markovianity2}
\mathbb E\left[\frac{z M_t}{1- z M_t} \Bigg| \mathcal F_s\right] =\frac{\eta_{st}(z)M_s}{1-\eta_{st}(z)M_s}~~a.s.,
\end{equation}
which plays an important role in the construction of UMIPs in the next section.

\index{Markov process!$\submm$-homogeneous|)}

\section[From Markov processes to UMIPs]{From $\submm$-homogeneous Markov processes to multiplicative monotone increment processes}

\index{monotone increment process!multiplicative (UMIP)|(}

In the unitary case there is no technical difficulty of unbounded operators that arose in the additive case.

Let $(M_t)_{t\ge 0}$ be a \index{Markov process!$\submm$-homogeneous}$\submm$-homogeneous Markov process on $\T$ with $M_0=1$ adapted to a filtration $(\mathcal F_t)_{t\geq0}$. Let $(\Omega, \mathcal F, \mathbb P)$ be the underlying probability space. 
We introduce a family of unitary operators $(U_t)_{t\geq0}$ on $L^2(\Omega, \mathcal F, \mathbb P)$ by
\begin{equation}\label{eq:UMIP}
U_t = I + (M_t -I) P_t, 
\end{equation}
where $M_t$ is regarded as a multiplication operator and $P_t$ is the conditional expectation $\mathbb E [\,\cdot\,| \mathcal F_t]$ as before. We first check that each $U_t$ is unitary. Recalling that $M_t$ and $P_t$ commute, we have
\begin{align}
U_s^* U_t
&= (I + (M_s^*-I)P_s)(I+P_t(M_t-I)) \notag\\
&= I + (M_s^*-I)P_s+ (M_t-I) P_t + (M_s^*-I)P_s(M_t-I)  \label{eq:unitary_increment}\\
&= I + (M_t-I) P_t + (M_s^*-I)P_s M_t. \notag
\end{align}
Specializing to $s=t$ shows that $U_t^*U_t =I$. Similarly we can prove that $U_t U_t^*=I$.
Moreover, $(U_t)$ is a UMIP as shown below. 

\begin{proposition}\label{prop:unitary_key}
The unitary process $(U_t)$ satisfies 
\begin{equation}\label{eq:unitary_monotone_increment}
P_s \frac{z U_s^* U_t}{1-z U_s^* U_t} P_s = \psi_{st}(z) P_s
\end{equation}
for $0 \le s \le t$ and $z \in \C\setminus\tor$, where $\psi_{st}(z)$ is defined by 
$$
\psi_{st}(z) = \int_{\mathbb T} \frac{z y}{1-z y} k_{st}(1,{\rm d}y).
$$ 
In particular, the distribution of the unitary operator $U_s^* U_t$ with respect to the state $\langle \mathbf 1_{\Omega}, \cdot \mathbf 1_\Omega\rangle$ is given by $k_{st}(1,\cdot)$.
\end{proposition}
\begin{proof} 
We prove the formula \eqref{eq:unitary_monotone_increment} firstly for $z\in \disk$. The $\submm$-homogeneity condition implies
$
k_{st}(x,{\rm d}y) = \delta_x \mm k_{st}(1,{\rm d}y)
$
and hence
$$
\int_{\mathbb T} \frac{z y}{1-z y} \,k_{st}(x,{\rm d}y) = \psi_{\delta_x} (\eta_{st}(z)) = \frac{\eta_{st}(z)x}{1-\eta_{st}(z)x},
$$
where $\eta_{st}:= \psi_{st}/(1+\psi_{st})$. Therefore, the Markov property \eqref{eq:Markovianity2} holds, which can also be expressed as 
\begin{equation}\label{eq:Markov}
P_s \frac{z M_t}{1-z M_t} P_s = \frac{\eta_{st}(z) M_s}{1-\eta_{st}(z)M_s} P_s, \qquad 0 \leq s \leq t, z\in \disk. 
\end{equation}
First we start from writing
\begin{align*}
P_s \frac{z U_s^* U_t}{1-z U_s^* U_t} P_s
&=-P_s + \sum_{n=0}^\infty z^n P_s \left[ I+ (M_t-I)P_t + (M_s^*-I)P_s M_t \right]^n P_s
\end{align*}
for $z\in \disk$.  In the expansion of $P_s \left[ I+ (M_t-I)P_t + (M_s^*-I)P_s M_t \right]^n P_s$, the projection $P_t$ does nothing, so that we may replace $P_t$ by the identity. This gives us
 \begin{align}
P_s \frac{z U_s^* U_t}{1-z U_s^* U_t} P_s
&=-P_s + \sum_{n=0}^\infty P_s \left[z M_t + z (M_s^*-I)P_s M_t \right]^n P_s. \label{eq:unitary1}
\end{align}
Introducing the simplified notation $Z_t = (M_t-I)P_t$ and expanding the power, the above can be expressed in the following way:
\begin{align*}
&\eqref{eq:unitary1} = -P_s \\ & + \sum_{n=0}^\infty \sum_{k=0}^n \sum_{\substack{ p_1, \dots, p_{k+1} \geq0 \\ p_1 + \dots + p_{k+1}= n-k}} P_s (z M_t)^{p_1} (z Z_s^* M_t)  (z M_t)^{p_2}  (z Z_s^* M_t) \cdots (z Z_s^* M_t) (z M_t)^{p_{k+1}} P_s \\
 &= -P_s + \sum_{k=0}^\infty \sum_{p_1, \dots, p_{k+1} \geq0} z^{p_1 + \dots + p_{k+1}+k} P_s M_t^{p_1} Z_s^* M_t^{p_2+1}  Z_s^* M_t^{p_3+1} \cdots Z_s^* M_t^{p_{k+1}+1} P_s \\
 &= -P_s + \sum_{k=0}^\infty \sum_{q_1\geq0} \sum_{q_2, \dots, q_{k+1} \geq 1} z^{q_1 + \dots + q_{k+1}} P_s M_t^{q_1} Z_s^* M_t^{q_2}  Z_s^* M_t^{q_3} \cdots Z_s^* M_t^{q_{k+1}} P_s \\
  &= -P_s + \sum_{k=0}^\infty P_s \frac{1}{1-z M_t} Z_s^* \frac{z M_t}{1-z M_t}  Z_s^* \cdots Z_s^* \frac{z M_t}{1-z M_t} P_s.
\end{align*}
Using the relations $Z_s^* P_s = P_s Z_s^* = Z_s^*$ and the Markov property \eqref{eq:Markov},  we can compute the above further to get
\begin{align*}
\eqref{eq:unitary1}&=  -P_s + \sum_{k=0}^\infty P_s \frac{1}{1-\eta_{st}(z) M_s}   \left(\frac{\eta_{st}(z) M_s}{1-\eta_{st}(z) M_s}\right)^k (Z_s^*)^k P_s \\
&= -P_s + P_s \frac{1}{1-\eta_{st}(z) M_s}  \frac{1}{1- \frac{\eta_{st}(z) (I-M_s)}{1- \eta_{st}(z)M_s}} P_s \\
&= -P_s + P_s \frac{1}{1-\eta_{st}(z)} P_s = \frac{\eta_{st}(z)}{1-\eta_{st}(z)}P_s = \psi_{st}(z)P_s,
\end{align*}
as desired for $z\in\disk$. To show this identity for $|z|>1$ we use the following identity
$$
\frac{z x}{1- z x} = \overline{-1 - \frac{\overline{z}^{-1} x}{1-\overline{z}^{-1} x}}, \qquad |z|>1, |x|=1, 
$$
and then setting $x=U_s^* U_t$ and sandwiching by $P_s$ yield 
\begin{align*}
P_s\frac{z U_s^* U_t}{1- z U_s^* U_t} P_s 
&= P_s\left(-1 - \frac{\overline{z}^{-1} U_s^* U_t}{I-\overline{z}^{-1} U_s^* U_t}\right)^* P_s = (-1 - \overline{\psi_{st}(\overline{z}^{-1}}))P_s \\
&= \int_{\tor}\overline{\left(-1- \frac{\overline{z}^{-1} x}{1-\overline{z}^{-1} x} \right) }k_{st}(1,{\rm d}x)P_s = \psi_{st}(z)P_s. 
\end{align*}

Finally applying \eqref{eq:unitary_monotone_increment} to the constant function $\mathbf 1_\Omega$ proves that
\begin{equation}\label{eq:expectation}
\left\langle \mathbf 1_\Omega, \frac{z U_s^* U_t}{1-zU_s^* U_t}\mathbf 1_\Omega \right\rangle = \psi_{st}(z),
\end{equation}
which shows the last statement.
\end{proof}

In order to generalize the functions $zx/(1-zx)$ to all continuous functions in the above proposition, we need an approximation lemma. 

\begin{lemma}\label{lem:approx2}
The set of functions 
$$
\text{\rm span}_\C\{1,zx(1-z x)^{-1}: z\in \C\setminus \tor\}
$$
is dense in $C(\tor)$. 
\end{lemma}
\begin{proof} The proof is similar to Lemma \ref{lem:approx}. 
\end{proof}

\begin{theorem}\label{construction_UMIP} The process $(U_t)_{t\geq0}$ defined in \eqref{eq:UMIP} is a UMIP.
\end{theorem}
\begin{proof} As $M_0=1$, we have $U_0=I.$ Due to Proposition \ref{prop:unitary_key}, the distribution of $U_s^*U_t$ is equal to $k_{st}(1,\cdot)$, and the mapping $(s,t)\mapsto k_{st}(1,\cdot)$ is weakly continuous by assumption. Formula \eqref{eq:unitary_increment} implies that 
$
(U_s^* U_t -I)P_u = P_u(U_s^* U_t -I) = U_s^* U_t -I
$ 
for all $0\leq s \leq t \leq u$, and hence
\begin{equation}\label{eq:unitary_key1}
f(U_s^* U_t -I) P_u = P_u f(U_s^* U_t -I) =f(U_s^* U_t -I)
\end{equation}
for all $f \in C_b(\C)$ with $f(0)=0$.  
By Proposition \ref{prop:unitary_key}, formula \eqref{eq:expectation} and Lemma \ref{lem:approx2} we obtain $P_s g(U_t^*U_u)P_s = \langle \mathbf{1}_\Omega, g(U_t^*U_u)\mathbf{1}_\Omega\rangle P_s$ for any $0\leq s \leq t \leq u$ and $g \in C(\tor)$.  This also implies that 
\begin{equation}\label{eq:unitary_key2}
P_s f(U_t^*U_u-I)P_s = \langle \mathbf{1}_\Omega, f(U_t^*U_u-I)\mathbf{1}_\Omega\rangle P_s
\end{equation}
for all $f \in C_b(\C)$. 
With the key formulas \eqref{eq:unitary_key1} and \eqref{eq:unitary_key2} the remaining proof is similar to that of Theorem \ref{thm:additive_monotone_construction}. 
\end{proof}

\index{monotone increment process!multiplicative (UMIP)|)}

\section{Summary of the one-to-one correspondences}${}$\label{summary_multiplicative}

All in all, Theorems \ref{from_UMIP_to_Loewner}, \ref{mult-Markov-proc}, and
\ref{construction_UMIP} yield one-to-one correspondences between

\begin{enumerate}[\quad(A)]
\item \label{item:A2}
\index{Loewner chain!multiplicative}multiplicative Loewner chains $(\eta_{t})_{t\geq 0}$ in $\mathbb{D}$ (Def. \ref{EV_def:evolution_family}),

\item \label{item:B2}
$\T$-valued \index{Markov process!$\submm$-homogeneous}$\submm$-homogeneous Markov processes $(M_t)_{t\ge 0}$ with $M_0=1$ up to equivalence (Def.\ \ref{def_equi_mult_Markov} and Def.\ \ref{def:Markov_equiv}),

\item\label{item:C2} \index{monotone increment process!multiplicative (UMIP)}UMIPs: unitary multiplicative monotone increment processes $(U_t)_{t\ge 0}$ up to equivalence 
(Def.\ \ref{def_umip} and Def.\ \ref{def:equiv_saip}).

\end{enumerate}
Moreover, the above objects also correspond to:

\begin{enumerate}[\quad(A)]
\setcounter{enumi}{3}
\item\label{item:D2} families $(\mu_{st})_{0\leq s \leq t}$ of probability measures on $\T$ such that

\begin{enumerate}[(i)]
\item $\mu_{tt} = \delta_1$ for all $t\geq0,$

\item\label{hemigroup2} $\mu_{su} = \mu_{st} \submm \mu_{tu}$ for all $0\leq s\leq t\leq u,$

\item $(s,t)\mapsto \mu_{st}$ is weakly continuous.

\end{enumerate}
\end{enumerate}
\eqref{item:C2} $\Rightarrow$ \eqref{item:D2}: Given a UMIP $(U_t)$ we define $\mu_{st}$ to be the law of $U_s^{-1}U_t$. This is independent of a choice of a UMIP in the same equivalence class by Theorem \ref{thm:equiv_umip}. \eqref{item:D2} $\Rightarrow$ \eqref{item:A2}: Given $(\mu_{st})$ we define the \index{transition mappings}transition mappings $\eta_{st}=\eta_{\mu_{st}}$. Then $(\eta_{0t})$ forms a multiplicative Loewner chain. Thus our constructions yield bijections between the objects  \eqref{item:A2}--\eqref{item:D2}.

Following the additive case, we also call such a family of probability measures satisfying the three conditions in \eqref{item:D2} a \emph{weakly continuous $\submm$-convolution hemigroup}. By replacing $\submm$ by another convolution $\star$ we also define a notion of \index{hemigroup}\emph{weakly continuous $\star$-convolution hemigroups}. 
 

Furthermore, we can relate properties of the three notions as follows: 

\begin{itemize}
 \item[(1)]The Markov process $(M_t)_{t\geq0}$ is centered for all $t\ge 0$ if and only if
 the quantum process $(U_t)_{\geq0}$ is centered if and only if
  $\eta'_{t}(0)\in \R$ for all $t\ge 0.$

 \item[(2)] The Markov process $(M_t)_{t\geq0}$ is normalized with $\mathbb E[M_t]=e^{-t}$ for all $t\ge 0$ if and only if the
quantum process $(U_t)_{t\geq0}$ is normalized if and only if
$\eta'_t(0)=e^{-t}$ for all $t\ge 0.$

\item[(3)] The Markov process $(M_t)_{t\geq0}$ is stationary if and only if $(U_t)$ is stationary, i.e.\ $U_s^{*}U_t$ and $U_{t-s}$ has the same law for all $0\leq s \leq t$, if and only if the Loewner chains form a semigroup: $\eta_{s} \circ \eta_t = \eta_{s+t}$.
\end{itemize}

\section[Multiplicative free increment processes]{Construction of $\submm$-homogeneous Markov processes from multiplicative free increment processes}\label{mult_con_from_free}

\index{Markov process!$\submm$-homogeneous|(}
\index{free increment process!multiplicative}

Assume that $(\mu_{st})_{0\leq s \leq t < \infty}$ are probability distributions on $\tor$ coming from the increments of a free multiplicative increment process started at $1$, i.e. they form a weakly continuous $\boxtimes$-convolution hemigroup (see Section \ref{summary_multiplicative}) and $\mu_{tt}=\delta_1$. It is know that each $\mu_{st}$ is $\boxtimes$-infinitely divisible. There exists a family of probability measures $(\nu_{st})_{0\leq s \leq t}$ on $\tor$ such that
\[
\mu_{0t} = \mu_{0s} \mm \nu_{st}
\]
for $0 \leq s \leq t$ \cite{Bia98}. These measures are unique because the $\eta$-transforms can be obtained via
$\eta_{\nu_{st}}(z) = \eta_{\mu_{0s}} ^{-1} \circ \eta_{\mu_{0t}}(z)$ in a neighborhood of $0$. Note that the $\boxtimes$-infinite divisibility of $\mu_{0s}$ implies that $\eta_{\mu_{0s}}$ is univalent, and hence $\eta_{\mu_{0s}}^{-1}$ is defined in a neighborhood of $0$. The weak continuity for the map $(s,t)\mapsto\nu_{st}$ now follows from \cite[Proposition 2.5]{BV92}. Thus $(\nu_{st})$ forms a weakly continuous $\submm$-convolution hemigroup, and hence we can construct a $\submm$-homogeneous Markov process via the correspondence in Section \ref{summary_multiplicative}.

In particular, in the stationary case $\mu_{st}$ can be expressed as $\mu_{t-s}$, where $(\mu_t)_{t\geq0}$ forms a weakly continuous $\boxtimes$-convolution semigroup on $\mathbb{T}$ with $\mu_0 = \delta_1$. The $\eta$-transform of $\nu_{st}$ above can be expressed as
\begin{equation}\label{eq:subordination_multiplicative}
\frac{\eta_{\nu_{st}}(z)}{z} = \Big(\frac{\eta_{\mu_t}(z)}{z}\Big) ^{1-s/t},
\end{equation}
which was essentially proved in \cite{BB05}. Notice that one has to choose a suitable branch of the power function $z \mapsto z^t$ in order to define the RHS of \eqref{eq:subordination_multiplicative}. The equation \eqref{eq:subordination_multiplicative} also shows that
\begin{equation}
\nu_{st} = \mu_t ^{\hutimes (1-s/t)},
\end{equation}
where $\utimes$ is multiplicative Boolean convolution defined in \eqref{bool_convolutions_mult}.

\section[Generators, Feller property, and martingale property]{Generators, Feller property, and martingale property of $\submm$-homogeneous Markov processes}\label{analysis_mult_markov}

\subsection{Generators in the stationary case}
We compute \index{Hunt's formula}Hunt's formula for the generators of stationary $\submm$-homogeneous Markov processes on $\T$ with initial distribution $\delta_1$. Let $(k_t)$ be the transition kernels. Then the distributions $\mu_t:= k_t(1,\cdot) = \mathbb P[M_t \in \cdot]$ are weakly continuous regarding $t$ and form a $\submm$-convolution semigroup, namely
$$
\mu_s \submm \mu_t = \mu_{s+t}, \qquad s,t\geq0.
$$
As in the additive case the Hunt formula for Markov processes is closely related to the \index{L\'evy-Khintchine representation!monotone}L\'evy-Khintchine representation for the $\submm$-convolution semigroups proved by Bercovici \cite[Theorem 4.2]{B05}.
\begin{theorem}\label{UMID}
\begin{enumerate}[\rm(1)]
\item\label{CSUMMI} Let $\{\mu_t \}_{t\geq 0} \subset\cP(\tor)$ be a weakly continuous $\submm$-convolution semigroup such that $\mu_0 = \delta_1$, and let $\eta_t := \eta_{\mu_t}$. Then the right derivative $B(z) = \frac{{\rm d}}{{\rm d}z}\big|_{t=0} \eta_{t}(z)$ exists and $\{\eta_t\}$ satisfies the differential equation
\begin{equation} \label{ODE}
\frac{{\rm d}}{{\rm d}t}\eta_t (z) = \eta_t(z) B(\eta_t(z)), \qquad \eta_0(z)=z, \qquad t\geq0, z \in \D.
\end{equation}
Moreover, the function $B$ is of the Herglotz form
\begin{equation}\label{VectorFUMMI}
B(z) = i\alpha - \int_{\tor} \frac{1+z \zeta}{1-z\zeta}\rho({\rm d}\zeta),
\end{equation}
where $\alpha \in \R$ and $\rho$ is a finite, non-negative measure on $\tor$. The pair $(\alpha,\rho)$ is unique and is called the generating pair.

\item Conversely, given an analytic map $B$ of the form \eqref{VectorFUMMI}, the solution $\{\eta_t\}$ to the differential equation \eqref{ODE} can be expressed as $\eta_t=\eta_{\mu_t}$ for a weakly continuous $\submm$-convolution semigroup $\{\mu_t\}_{t\geq0}$ such that $\mu_0=\delta_1$.
\end{enumerate}
\end{theorem}

We will firstly relate $k_t(x,\cdot)$ to the generating pairs $(\alpha,\rho)$. The proof is similar to and easier than the additive case (see Section \ref{sec:Hunt}) because any family of uniformly bounded finite, non-negative measures on $\T$ is tight.

\begin{lemma}\label{lem:multiplicative_generator}
Let $(M_t)_{t\geq0}$ be a stationary $\submm$-homogeneous Markov process with $M_0=1$ and with transition kernels $(k_t)_{t\geq0}$. Let $(\alpha,\rho)$ be the generating pair in \eqref{VectorFUMMI} associated to the $\submm$-semigroup $\{k_t(1,\cdot)\}_{t\geq0}$. Then for all $x\in \T$ we have, as $t\downarrow0$,
\begin{equation*}
\frac{1}{t}\Re(1-\bar{x}y)\,k_t(x,{\rm d}y) \to \rho({\rm d}y) ~~{\it(weakly)},
\end{equation*}
 and
\begin{equation}\label{eq:multiplicative_shift}
\frac{1}{t}\int_\T \Im(\bar{x}y)\, k_t(x,{\rm d}y) \to \alpha.
\end{equation}
\end{lemma}
\begin{proof} Let $\psi_{x;t}$ and $\eta_{x;t}$ be the moment generating function and its $\eta$-transform of the distribution $k_t(x,\cdot)$.
Let $B$ be the infinitesimal generator for the semigroup $\{k_t(1,\cdot)\}_{t\geq0}$ given in \eqref{VectorFUMMI}. Computing the derivative of $\psi_{x;t}(z) = x \eta_{0;t}(z)/(1-x \eta_{0;t}(z))$ at $t=0$ we have
\begin{align*}
B(z)
&= -\lim_{t\to 0}\frac{1}{t}\int_\T \frac{(1-z x)(1-\bar{x}y)}{1-z y}\,k_{t}(x,{\rm d}y) \\
&= -\lim_{t\to 0}\frac{1}{t}\left(\int_\T \frac{1+z y}{1-z y}\Re(1-\bar{x}y)\,k_{t}(x,{\rm d}y) + i\int_{\T} \Im(1-\bar{x}y)\,k_{t}(x,{\rm d}y) \right).
\end{align*}
Substituting $z=0$ we get the uniform boundedness for the family of finite, non-negative measures $\{t^{-1}\Re(1-\bar{x}y)\,k_{t}(x,{\rm d}y): 0<t<1\}$. Therefore we can extract a weak limit $\rho'$, which is a finite, non-negative measure on $\T$.  By the uniqueness of the Herglotz representation we get $\rho'=\rho$, and the convergence \eqref{eq:multiplicative_shift} as well.
\end{proof}

To describe \index{Hunt's formula}Hunt's formula for the unitary case, we identify a function $f$ on $\T$ with the function $x\mapsto f(e^{ix})$ defined on $\R$. Similarly, we identify a measure $\mu$ on $\tor$ with the measure $B\mapsto \mu(e^{i B})$, $B$ Borel subset of $\R$.
Instead of the free difference quotient \eqref{eq:FDQ} we introduce the following operator $\delta: C^1(\T) \to C(\T^2)$:
\begin{equation}
(\delta f)(\theta,\phi) =
\begin{cases}
\frac{f(\theta)-f(\phi)}{\tan\left((\theta-\phi)/2\right)}, & \theta-\phi \notin \pi\Z, \\
2 f'(\theta),& \theta-\phi \in \pi (2\Z),\\
0,& \theta-\phi\in\pi(2\Z+1).
\end{cases}
\end{equation}
Then, for $f \in C^2(\T)$ we have the formula
\begin{equation}
(\partial_\theta\delta f)(\theta,\phi) =
\begin{cases}
\frac{f(\phi)-f(\theta) - \sin(\phi-\theta) f'(\theta)}{1-\cos(\theta-\phi)}, & \theta-\phi\notin \pi(2\Z), \\
f''(\theta), & \theta-\phi\in \pi(2\Z).
\end{cases}
\end{equation}

Let $\mathcal B_b(\tor)$ be the set of bounded Borel measurable functions from $\tor$ to $\C$. 
\begin{theorem}\label{thm:Hunt_multiplicative}
Let $(M_t)_{t\geq0}$ be a stationary $\submm$-homogeneous Markov process with $M_0=1$ and with transition kernels $(k_t)_{t\geq0}$. Let $T_t\colon \mathcal B_b(\tor)\to \mathcal B_b(\tor)$ be its transition semigroup
\begin{equation}
(T_t f)(\theta) = \int_{[0,2\pi)} f(\phi) \,k_t(\theta,{\rm d}\phi), \qquad f \in \mathcal B_b(\tor),
\end{equation}
which satisfies $T_s T_t = T_{s+t}$ for $s,t\geq0$. The generator of the transition semigroup is then given by
\begin{align*}
(G f)(\theta) &:= \left.\frac{{\rm d}}{{\rm d} t}\right|_{t=0} (T_t f)(\theta) =  \alpha f'(\theta) + \int_{[0,2\pi)}(\partial_\theta\delta f)(\theta,\phi)\, \rho({\rm d}\phi)
\end{align*}
for $f \in C^2(\T)$ and $\theta \in [0,2\pi)$, where $(\alpha,\rho)$ is the pair in \eqref{VectorFUMMI} associated to the $\submm$-semigroup $\{k_t(1,\cdot)\}_{t\geq0}$.
\end{theorem}
\begin{proof} For $\theta, \phi \in [0,2\pi)$ the identity
\begin{align*}
&\int_{[0,2\pi)} f(\phi) \,k_t(\theta,{\rm d}\phi) -f(\theta) = \int_{[0,2\pi)\setminus\{\theta\}} \{f(\phi)-f(\theta)\} \,k_t(\theta,{\rm d}\phi)\\
&=\int_{[0,2\pi)\setminus\{\theta\}}\frac{f(\phi)-f(\theta)-\sin(\phi-\theta)f'(\theta)}{1-\cos(\phi-\theta)} (1-\cos(\phi-\theta)) k_t(\theta,{\rm d}\phi) \\
& \qquad + f'(\theta)\int_{[0,2\pi)\setminus\{\theta\}} \sin(\phi-\theta) k_t(\theta,{\rm d}\phi) \\
&=\int_{[0,2\pi)}(\partial_\theta\delta f)(\theta,\phi) (1-\cos(\phi-\theta)) k_t(\theta,{\rm d}\phi) + f'(\theta)\int_{[0,2\pi)}\sin(\phi-\theta) k_t(\theta,{\rm d}\phi)
\end{align*}
holds. Then Lemma \ref{lem:multiplicative_generator} completes the proof.
\end{proof}

\begin{example}
\begin{enumerate}[(1)]
\item The \index{Brownian motion!monotone}unitary monotone Brownian motion (see Example \ref{UMBM}) has $\submm$-infinitely divisible distributions characterized by $\alpha=0$ and $\rho=(1/2)\delta_1$ on $\T$. Then
\begin{equation*}
(G f)(\theta) =
\begin{cases}
\frac{f(0)-f(\theta)+(\sin\theta) f'(\theta)}{2(1-\cos\theta)}, & \theta \notin \pi(2\Z),\\
\frac{1}{2}f''(0), & \theta \in \pi(2\Z).
\end{cases}
\end{equation*}

\item The probability measure $\mu_t = (1-e^{-t})\haar +e^{-t}\delta_1$ forms a $\submm$-convolution semigroup with $\mu_0=\delta_1$. The $\eta$-transform is given by
$\eta_{\mu_t}(z) = e^{-t}z [1-(1-e^{-t})z]^{-1}$, and the function $B$ in \eqref{VectorFUMMI} is computed as $B(z)=z-1$. The pair $(\alpha,\rho)$ is given by $\alpha=0$ and ${\rm d}\rho(\zeta) = [1-\Re(\zeta)] \haar({\rm d}\zeta)$, and hence the generator for the associated Markov process is
\begin{equation*}
(G f)(\theta) = \int_{[0,2\pi)}(\partial_\theta\delta f)(\theta,\phi)\, (1-\cos \phi) \frac{{\rm d}\phi}{2\pi}.
\end{equation*}
for $f \in C^2(\T)$ and $\theta \in [0,2\pi)$.  The transition kernel $k_t(x,\cdot):= \delta_x \submm \mu_t$ is characterized by
\begin{equation*}
\psi_{k_t(x,\cdot)}(z) = \frac{x \eta_{\mu_t}(z)}{1-x \eta_{\mu_t}(z)} = \frac{x z}{e^t - (e^t +x-1)z},
\end{equation*}
{}from which we can prove that $k_t(x,\cdot)$ is absolutely continuous with respect to $\haar$ unless $x=1$.
\end{enumerate}

\end{example}

\subsection{Feller property in the stationary case}

\index{Feller property}

Similarly to the additive case in Section \ref{sec:Feller_additive}, $\submm$-homogeneous Markov processes also have the Feller property (see Definition \ref{def:Feller}). Let $S_t$ be the restriction of $T_t$ to $C(\tor)$, where $(T_t)_{t\geq0}$ is the transition semigroup on $\mathcal B_b(\tor)$ defined in Theorem \ref{thm:Hunt_multiplicative}.

\begin{theorem}
The family $(S_t)_{t\geq0}$ is a Feller semigroup on $C(\tor)$.
\end{theorem}
\begin{proof} Let $f_z(x) =zx/(1-zx)$ and 
$$
 D = \text{\rm span}_\C\{1, f_z: z \in \C\setminus \tor\} \subset C(\T).
$$
By Lemma \ref{lem:approx2} $D$ is dense in $C(\T)$. The arguments in the proof of Theorem \ref{thm:Feller} can be used to prove that $S_t D \subset D$ for all $t\geq0$, $(S_t)_{t\geq0}$ is a contraction semigroup on $C(\T)$, and $(S_t f)(x) \to f(x)$ as $t\downarrow0$ for all $f \in C(\T)$ and $x\in \T$.
\end{proof}

Let $L$ be the infinitesimal generator of $(S_t)_{t\geq0}$ with domain
\begin{equation}
\Dom(L) = \left\{f \in C(\T): \lim_{t\downarrow0} \frac{S_t f -f}{t} ~\text{exists in the uniform norm}\right\}.
\end{equation}
We can prove by direct computation that $D \subset \Dom(L)$, which implies that for functions in $D$ the convergence of the Hunt formula in Theorem \ref{thm:Hunt_multiplicative} holds uniformly. Since $D$ is a dense subspace of $C(\T)$ and is invariant under $\{S_t: t\geq0\}$, we conclude that $D$ is a core for $L$.

\subsection{Martingale property}

\index{Martingale property}

As well as the additive case, (not necessarily stationary) $\submm$-homogeneous Markov processes satisfy a martingale property. The proof is easier in this case since all moments are finite.

\begin{proposition} Let $(M_t)_{t\ge 0}$ be a $\submm$-homogeneous Markov process with a filtration $(\mathcal F_t)$ such that $M_0=1$ and let $(\eta_t)_{t\geq0}$ be the associated multiplicative Loewner chain. Then, for each fixed $u\geq0$ and $z \in \eta_u(\disk)$ the process $(N_t^{z})_{0\leq t \leq u}$ defined by 
$$
N_t^{z}=\frac{\eta_t^{-1}(z)M_t}{1-\eta_t^{-1}(z)M_t} 
$$ 
is an $(\mathcal{F}_t)$-martingale.
\end{proposition}
\begin{proof}
This is a direct consequence of \eqref{eq:Markovianity2}. Note that each map $\eta_t$ is univalent by Theorem \ref{EV_univalence}. 
\end{proof}

\begin{corollary} Let $(M_t)_{t\ge 0}$ be a $\submm$-homogeneous Markov process on $\T$ with a filtration $(\mathcal F_t)$ such that $M_0=1$. Then $\mathbb E[M_t]\neq0$ for every $t\geq0$ and the process $(M_t/ \mathbb E[M_t])_{t\geq0}$ is an $(\mathcal F_t)$-martingale.
\end{corollary}
\begin{proof}
Recall that $\eta_t$ is the $\eta$-transform of $M_t$ (since $M_0=1$) and hence $\eta_t(z)=\mathbb E[M_t]z +O(z^2)$. The map $\eta_t$ is univalent and hence $\mathbb E[M_t]\neq0$. Then $\eta_t^{-1}(z)=z/\mathbb E[M_t] +O(z^2)$ in a neighborhood of $z=0$ and then taking the derivative of $\mathbb E[N_t^z | \mathcal{F}_s]=N_s^z$ at $z=0$ concludes the proof.  
\end{proof}

\index{Markov process!$\submm$-homogeneous|)}



\chapter[Limit Theorems and Geometric Function Theory]{Limit Theorems for Additive Monotone Convolution and Geometric Function Theory}\label{section_limits_additive}
The goal of this section is to study limit theorems for additive monotone convolution, to characterize some subclasses of probability measures with univalent Cauchy transforms in terms of geometric function theory, and to find relations between the above two results.

Recall that non-commutative probability theory provides us with four further additive convolutions of probability measures on $\R$, see Section \ref{intro_convolutions}. While anti-monotone convolution is simply a reversion of the monotone convolution, we will find several interesting relations to Boolean, classical, and free convolutions.
\section{Khintchine's limit theorem and univalent Cauchy transforms}

\subsection{Preliminaries from complex analysis}

We denote by $\Univ(\R)$ the set of all probability measures on $\R$ having a univalent Cauchy transform. The Hurwitz theorem and Lemma \ref{lemmaconvergence} show that $\Univ(\R)$ is a weakly closed subset of  $\cP(\R)$.\\

The following criterion for univalence was shown by Noshiro \cite{N34} and Warschawski \cite{W35}. While its proof is simple, it is quite useful.
\begin{lemma}\label{NW}
Let $D\subset \C$ be a convex open set and let $f\colon D \to \C$ be analytic. Suppose for some $\theta \in(0,2\pi]$ it holds that $\Im(e^{i\theta} f'(z))>0$ for $z\in D$. Then
$f$ is univalent in $D$.
\end{lemma}
\begin{proof}
We can assume that $\theta=0$ by considering the function $e^{i\theta}f(z)$ instead of $f$ itself.
The following identity holds:
\[
\begin{split}
f(z_0)-f(z_1)&= \int_0^1 \frac{{\rm d}}{{\rm d}t} f((1-t)z_0+tz_1)\, {\rm d}t\\
&=(z_0-z_1)\int_0^1 f'((1-t)z_0+tz_1)\, {\rm d}t.
\end{split}
\]
Hence $|f(z_0)-f(z_1)| \geq C(z_0,z_1) |z_0-z_1|$, where
$C(z_0,z_1)=\inf\{\Im[f'((1-t)z_0+tz_1)] \mid t\in[0,1]\}>0$.
\end{proof}

The following result can be extracted from the proof of Muraki \cite[Lemma 6.3]{M00} (see also \cite[Lemma 2.3]{Has10}).
\begin{lemma}\label{finitevariance0}\index{monotone convolution!additive $\rhd$}
If $\lambda=\mu\rhd\nu$ and $\lambda$ has a finite variance, then $\mu,\nu$ also have finite variances and
$$
m_1(\lambda)=m_1(\mu)+m_1(\nu),\quad \sigma^2(\lambda)=\sigma^2(\mu)+\sigma^2(\nu).
$$
\end{lemma}
The following result can be extracted from the proof of \cite[Theorem 2.4]{Has10}, but we show it here.  For $t \in \R$, let $\C_t$ be the set of all $z \in \C$ with $\Im (z) >t$.
\begin{lemma}\label{univfinitevariance}
If $\mu$ has a finite variance $\sigma^2(\mu)$, then \index{F-transform@$F$-transform}$F_\mu$ is univalent in $\C_{\sigma(\mu)}$.
\end{lemma}
\begin{proof}
$F_\mu$ has a Nevanlinna representation
$$
F_\mu(z)=z-m_1(\mu)+\int_{\R}\frac{1}{x-z}\tau({\rm d}x),
$$
where $\tau$ is a finite, non-negative measure. It holds that $\tau(\R)=\sigma^2(\mu)$ (see \cite[Proposition 2.2]{M92}. Maassen assumed that $m_1(\mu)=0$,
but his result easily extends to the case $m_1(\mu)\neq0$).
We have
\[
\begin{split}
|F_\mu(z_0)-F_\mu(z_1)|
&=\left| z_0-z_1 + \int_{\R}\frac{z_0-z_1}{(x-z_0)(x-z_1)}\,\tau({\rm d}x)\right| \\
&\geq |z_0-z_1|\left(1-\int_{\R}\frac{1}{\Im (z_0)\, \Im (z_1)}\, \tau({\rm d}x)\right) \\
&= |z_0-z_1|\left(1-\frac{\sigma^2(\mu)}{\Im (z_0) \,\Im (z_1)}\right).
\end{split}
\]
If $z_0,z_1 \in \C_{\sigma(\mu)}$, then $1-\frac{\sigma^2(\mu)}{\Im (z_0)\, \Im (z_1)}>0$ and hence $F_\mu$ is univalent in $ \C_{\sigma(\mu)}$.
\end{proof}
\begin{remark} Maassen proved that $F_\mu\colon\C_{\sigma(\mu)} \to \C_{\sigma(\mu)}$ assumes every point of $\C_{2\sigma(\mu)}$ exactly once.
\end{remark}

\begin{lemma}\label{lemmainfinity}
For any probability measure $\mu$ on $\R$ and for any $t>0$, we have
$$
\lim_{z\in\C_t, |z|\to \infty}|F_{\mu}(z)|=\infty.
$$
\end{lemma}
\begin{proof}
We show the equivalent statement
$
\lim_{z\in\C_t, |z|\to \infty}|G_{\mu}(z)|=0.
$
Assuming this were not the case, one would find $\epsilon>0, t>0$ and $z_n \in \C_t$ such that $|z_n|\geq n$ and $|G_\mu(z_n)| \geq \epsilon$ for any integer $n \geq1$.
Since $|1/(z_n-x)| \leq 1/t$ for any $x \in \R$ and $n \geq1$ and since $1/(z_n-x) \to 0$ as $n\to \infty$ for each $x$, one may use the dominated convergence theorem to conclude
that $G_\mu(z_n) \to 0$, a contradiction.
\end{proof}

\begin{lemma}\label{equivconvergence}
\begin{enumerate}[\rm(1)]
\item\label{unique} Let $\mu,\nu,\nu' \in\cP(\R)$ such that $\mu \rhd \nu =\mu\rhd \nu'$. Then $\nu=\nu'$.

\item\label{m-conv2} Let $\mu, \mu_n, \nu_n$, $n=1,2,3,\ldots$ be probability measures on $\R$.
Assume that $\mu_n\rhd \nu_n\wto\mu$. Then $\mu_n \wto \mu$ if and only if $\nu_n \wto \delta_0$.

\item\label{m-conv1} Let $\lambda_n,\mu_n,\nu_n\in\cP(\R)$ such that $\lambda_n=\mu_n\rhd\nu_n$ and $\mu_n$ converges weakly to some probability measure. Then $\lambda_n$ converges weakly to some probability measure if and only if $\nu_n$ converges weakly to some probability measure.
\end{enumerate}
\end{lemma}
\begin{proof}\eqref{unique} \,\,\, From (the proof of) Bercovici and Voiculescu \cite[Proposition 5.4]{BV93} there exist truncated cones $\Gamma_{\lambda,M}, \Gamma_{\lambda',M'}$ such that $F_\nu, F_{\nu'}$ are univalent on $\Gamma_{\lambda',M'}$ and\\ $F_\nu(\Gamma_{\lambda',M'})$, $F_{\nu'}(\Gamma_{\lambda',M'})\subset \Gamma_{\lambda,M}$, and $F_\mu$ is univalent on $\Gamma_{\lambda,M}$. Hence we may apply the inverse $F_\mu^{-1}$ to $F_\mu \circ F_\nu(z)= F_\mu \circ F_{\nu'}(z)$ from the left for $z \in\Gamma_{\lambda',M'}$. This implies by analytic continuation that $F_\nu=F_{\nu'}$ on $\C^+$ and hence $\nu=\nu'.$

\eqref{m-conv2}\,\,\, Suppose that $\mu_n\wto\mu$.
{}From \cite[Propositions 5.4 and 5.7]{BV93} and their proofs, there exist $\lambda,M,\lambda',M',\lambda'',M''>0$ such that  $F_\mu, F_{\mu_n}$ are all univalent on $\Gamma_{\lambda',M'}$ and $F_{\mu_n\rhd\nu_n}$ are all univalent on $\Gamma_{\lambda'',M''}$ such that \[F_\mu(\Gamma_{\lambda',M'}), F_{\mu_n}(\Gamma_{\lambda',M'})\supset \Gamma_{\lambda,M} \supset F_{\mu_n\rhd\nu_n}(\Gamma_{\lambda'',M''}).\] Hence the left compositional inverses of $F_\mu$, $F_{\mu_n}$ may be defined on $\Gamma_{\lambda,M}$, and moreover, we may assume that  $F_{\mu_n}^{-1}$ converge to $F_\mu^{-1}$ locally uniformly in $\Gamma_{\lambda, M}$.
Hence, for each $z \in \Gamma_{\lambda'',M''}$,
\begin{equation}
F_{\nu_n}(z)=F_{\mu_n}^{-1}(F_{\mu_n \rhd \nu_n}(z)) \to F_\mu^{-1} (F_\mu(z))=z
\end{equation}
as $n\to \infty$ (more precisely, the identity $F_{\nu_n}(z)=F_{\mu_n}^{-1}(F_{\mu_n \rhd \nu_n}(z))$ may be first justified for $z=i y$ for sufficiently large $y$ and then for all $z\in \Gamma_{\lambda'',M''}$ by the identity theorem). This implies that $\nu_n \wto \delta_0$ by Lemma \ref{lemmaconvergence}.

Conversely, if $\nu_n \wto \delta_0$, then there exists a domain $\Gamma_{\beta, L}$ such that all $F_{\nu_n}^{-1}$ are defined on $\Gamma_{\beta,L}$ as the right
compositional inverses of $F_{\nu_n}$ and that $F_{\nu_n}^{-1}(z)$ converge to $z$ locally uniformly in $\Gamma_{\beta, L}$.
Hence
$$
F_{\mu_n}(z)=F_{\mu_n\rhd\nu_n}(F_{\nu_n}^{-1}(z)) \to F_\mu(z),\quad z\in \Gamma_{\beta, L}
$$
as $n\to \infty$. This implies that $\mu_n \wto \mu$ again by Lemma \ref{lemmaconvergence}.

\eqref{m-conv1}\,\,\, The proof of \eqref{m-conv2} works with slight modifications.
\end{proof}
\begin{remark}
A similar result is shown by Wang in \cite[Proposition 2.2]{W12}.
\end{remark}

\subsection{Infinitesimal arrays and additive monotone convolution}\label{inf_arrays_additive}

\index{infinitesimal array!of measures on $\R$|(}
\index{monotone convolution!additive $\rhd$|(}

We prove under some assumptions that the limits of monotone convolutions of infinitesimal triangular arrays coincide with probability measures with univalent Cauchy transforms.
A family of probability measures $\{\mu_{n,j} \mid 1\leq j \leq k_n, n\geq1\}$ is called an \emph{infinitesimal (triangular) array} if $k_n\uparrow\infty$ and for any $\delta>0$,
\begin{equation}\label{inf_tri_array}
\lim_{n\to \infty}\sup_{1 \leq j \leq k_n} \mu_{n,j}([-\delta,\delta]^c)=0.
\end{equation}
For an associative binary operation $\star$ on $\cP(\R)$, we say that a probability measure $\mu$ is the \emph{$\star$-limit of an infinitesimal array} if the weak convergence
$$
\mu_{n,1}\star \cdots \star \mu_{n,k_n} \wto \mu\quad \text{as}~n\to \infty
$$
holds for some infinitesimal triangular array $\{\mu_{n,j} \mid 1\leq j \leq k_n, n\geq1\}$. The set of all $\star$-limits of infinitesimal arrays is denoted by $\IA(\star)$ or $\IA(\star,\R)$.

A probability measure $\mu$ on $\R$ is said to be \index{infinite divisibility}\emph{$\star$-infinitely divisible} if for every $n\in\N$ there exists $\mu_n \in \cP(\R)$ such that $\mu=\mu_n^{\star n}$ ($n$ fold convolution). The set of $\star$-infinitely divisible distributions is denoted by $\ID(\star,\R)$ or simply $\ID(\star)$. It is known that $\ID(\star)$ is closed with respect to weak convergence when $\star\in\{\ast,\boxplus\}$, but whether $\ID(\rhd,\R)$ is weakly closed is not known.

For classical convolution $\ast$, Khintchine proved that
\begin{equation}\label{CK}
\IA(\ast) = \ID(\ast),
\end{equation}
 see \cite[\S 24, Theorem 2]{GK54}. Also for free convolution $\boxplus$,  Bercovici and Pata \cite{BP00} proved the analogue of Khintchine's theorem:
 \begin{equation}\label{FK}
 \IA(\boxplus)=\ID(\boxplus).
\end{equation}

Our concern in this section is the class $\IA(\rhd,\R)$. We start from a general result valid for any convolution.

\begin{proposition}\label{closednessIA} Let $\star$ be an associative binary operation on $\cP(\R)$.
\begin{enumerate}[\rm(1)]
\item\label{CConv} The set $\IA(\star)$ is closed under the convolution $\star$.
\item\label{WC} The set $\IA(\star)$ is closed with respect to weak convergence.
\end{enumerate}
\end{proposition}
\begin{proof}
\eqref{CConv}\,\,\, Obvious.

\eqref{WC}\,\,\,
Let $\mu^{(m)}$ be a sequence of probability measures which are limits of infinitesimal arrays, and suppose that $\mu^{(m)}$ converges to a probability measure $\mu$ weakly.
Take an infinitesimal triangular array $\{\mu_{n,j}^{(m)} \mid 1\leq j \leq k_n^{(m)}, n\geq1\}$ such that for each $m \geq1$,
 $\mu_{n,1}^{(m)} \star \cdots \star \mu_{n,k_n^{(m)}}^{(m)}\wto \mu^{(m)}$ as $n \to \infty$.
We take a distance $d(\cdot,\cdot)$ which is compatible with the weak convergence (for example $d$ can be taken to be the L\'evy-Prokhorov distance).
Then, for each integer $m \geq 1$, there exists a positive integer $n(m)$ such that
$$
\sup_{1 \leq j \leq k_n^{(m)}} \mu_{n,j}^{(m)}([-m^{-1}, m^{-1}]^c) < \frac{1}{m} \quad \text{for}\quad n \geq n(m)
$$
and
$$
d(\mu_{n,1}^{(m)} \star \cdots \star \mu_{n,k_n^{(m)}}^{(m)}, \mu^{(m)}) < \frac{1}{m} \quad \text{for}\quad n \geq n(m).
$$
If we define $\mu_{m,j}:=\mu_{n(m),j}^{(m)}$ for $1 \leq j \leq k_m:= k_{n(m)}^{(m)},~m\geq1$, then this is an infinitesimal array converging weakly to $\mu$. Indeed, for any $\epsilon,\delta>0$, we can find an integer $M \geq1$ such that $\frac{1}{M} < \epsilon, \delta$. Then
$$
\sup_{1 \leq j \leq k_m} \mu_{m,j}([-\delta, \delta]^c) \leq \sup_{1 \leq j \leq k_{n(m)}^{(m)}} \mu_{n(m),j}^{(m)}([-m^{-1}, m^{-1}]^c) < \frac{1}{m} < \epsilon
$$
for  $m \geq M.$ By the triangular inequality,
\begin{equation*}
\begin{split}
d(\mu_{m,1} \star \cdots \star \mu_{m,k_m}, \mu)
&\leq d(\mu_{n(m),1}^{(m)} \star \cdots \star \mu_{n(m),k_{n(m)}^{(m)}}^{(m)}, \mu^{(m)}) + d(\mu^{(m)},\mu) \\
&\leq \frac{1}{m} + d(\mu^{(m)},\mu) \to 0,
\end{split}
\end{equation*}
as $m \to \infty$, which completes the proof.
\end{proof}

\begin{theorem}\label{thminfinitesimalarray} \begin{enumerate}[\rm(1)]
\item\label{IA1} \index{Cauchy transform!univalent}
\index{infinitesimal array!of measures on $\R$}$\Univ(\R)\subset\IA(\rhd)$.
\item\label{IA2} Assume that an infinitesimal array $\{\mu_{n,j} \mid 1\leq j \leq k_n, n\geq1\}$ satisfies the condition
$$
\sup_{1\leq j\leq k_n}\sigma^2(\mu_{n,j}) \to 0 \quad \text{as~} n\to\infty.
$$
If  $\mu_{n,1} \rhd \cdots \rhd \mu_{n,k_n}$ converges weakly to a probability measure $\mu$, then $G_\mu$ is univalent.
\end{enumerate}
\end{theorem}
\begin{proof} \eqref{IA1}\,\,\, By Theorem \ref{EV_embed_F} (a), we have $F_\mu = F_1$ for an \index{Loewner chain!additive}additive Loewner chain
$(F_{\mu_t})_{t\geq0}$ with \index{transition mappings}transition mappings $(F_{\mu_{st}})_{0\leq s\leq t}$. For $n\geq 1$ and $1\leq j \leq n$ we define $\mu_{n,j}=\mu_{(j-1)/n, j/n}$. This defines an infinitesimal triangular because the continuity of $(s,t)\mapsto F_{\mu_{st}}$ 
 implies \eqref{inf_tri_array}.\\
We have \[\mu=\mu_{0,1} = \mu_{0,1/n}\rhd ... \rhd \mu_{(n-1)/n,1}\]
for all $n\geq 1.$ Hence, $\mu \in \IA(\rhd).$\\
\eqref{IA2}\,\,\, For any $\epsilon>0$, there exists $N\geq1$ such that $\sigma^2_n:=\sup_{1\leq j\leq k_n}\sigma^2(\mu_{n,j}) <\epsilon$ for $n \geq N$.
By Lemma \ref{univfinitevariance}, $F_{\mu_{n,j}}$ is univalent in $\C_\epsilon \subset \C_{\sigma_n}$ for any $j$,
and since $F_{\mu_{n,j}}(\C_\epsilon) \subset \C_\epsilon$, the composition $F_{\mu_{n,1}}\circ \cdots \circ F_{\mu_{n,k_n}}$ is univalent in $\C_\epsilon$
for $n\geq N$. By taking the limit $n\to\infty$, $F_\mu$ (which is not a constant) is also univalent in $\C_\epsilon$. Since $\epsilon>0$ is arbitrary, we conclude that $F_\mu$ is univalent in $\C^+.$
\end{proof}

The second statement is hopefully the case without any assumption on the variances, and such is true for the case of the unit circle (see Theorem \ref{Khintchineunit}).
Hence, let us pose a conjecture here:
\begin{conjecture}\label{IA}
\index{Cauchy transform!univalent}
\index{infinitesimal array!of measures on $\R$}
$\Univ(\R)=\IA(\rhd)$.
\end{conjecture}

\begin{remark}\label{PA}
Anshelevich and Arizmendi \cite{AA} introduced a class of probability measures $\cL$ that is characterized by the property $F_\mu(z+2\pi) = F_\mu(z)+2\pi$ for all $z\in\C^+$. They proved Theorem \ref{thminfinitesimalarray}\eqref{IA1} when $\mu \in \cL\cap \Univ(\R)$ by using Theorem \ref{Khintchineunit} that we prove later; see \cite[Remark 51]{AA}. Adopting a similar argumentation also shows the following partial answer to Conjecture \ref{IA}: If $\{\mu_{n,j}\}_{1\leq j \leq k_n, 1\leq n} \subset\cL$ is an infinitesimal array, and if  $\mu_{n,1} \rhd \cdots \rhd \mu_{n,k_n}$ weakly converges to a probability measure $\mu$, then $\mu \in \Univ(\R) \cap \cL$. This partial answer makes the conjecture more reasonable.
\end{remark}

For monotone convolution, the inclusion $\ID(\rhd)\subset \IA(\rhd)$ is rather immediate from differential equations; see Section \ref{subsec_mon_inf_div_add}. By contrast to classical and free probabilities, we can prove that $\IA(\rhd)$ is strictly larger than $\ID(\rhd)$.  For this it suffices to prove that $\Univ(\R)\setminus \ID(\rhd) \neq \emptyset$ by Theorem \ref{thminfinitesimalarray}.
For this we give three counterexamples using monotone cumulants introduced in \cite{HS11} and one counterexample with a geometric proof. 

Recall that a probability measure with compact support is \index{infinite divisibility!additive monotone convolution}$\rhd$-infinitely divisible if and only if its \index{monotone cumulants}monotone cumulant sequence $\{r_n\}_{n\geq1}$ is conditionally positive definite, namely the determinant of the $n \times n$ matrix $\{r_{i+j}\}_{i,j=1}^n$ is non-negative for any $n \in\N$; see  \cite[Theorem 8.5]{H11} which contains the proof for the monotone case.

\begin{example}\label{semicircle not ID}
The standard \index{semicircle distribution}semicircle distribution is in $\Univ(\R)$ but is not in $\ID(\rhd)$. Let $\{r_n\}_{n\geq1}$ be the monotone cumulants. The values of some $r_n$ are computed in \cite[Appendix]{CGW18} (we can also use \cite[Theorem 4.8]{HS11} or \cite[Proposition 4.7]{HS11} for the computation of $r_n$):  
\begin{equation*}
\{r_n\}_{n=1}^{10} = \left\{0,1,0,\frac{1}{2},0,\frac{1}{2},0,\frac{7}{12},0,\frac{2}{3}\right\}, \qquad
\det\{r_{i+j}\}_{i,j=1}^5
= -\frac{1}{3456}<0,
\end{equation*}
which shows that the measure is not $\rhd$-infinitely divisible. However, since the semicircle distribution is $\boxplus$-infinitely divisible it has a univalent Cauchy transform (see Proposition \ref{FID}).

Moreover, numerical simulation suggests that the semicircle distribution with any mean is not $\rhd$-infinitely divisible. More precisely, let $r_n(a)$ be the monotone cumulants of the semicircle distribution with mean $a$ and variance $1$. Let $h_n(a):= \det \{r_{i+j}(a)\}_{i,j=1}^n$. From the graph drawn by simulation, the function $\min\{h_2(a),h_3(a),h_5(a)\}$ seems negative for all $a\in\R$, but a rigorous proof is perhaps difficult since $h_5(a)$ is a polynomial of degree $20$.
\end{example}

The shift of a probability measure may break the $\rhd$-infinitely divisibility.

\begin{example}\label{arcsine not ID}
The \index{arcsine distribution}arcsine law with mean $a \in \R$ and variance $t>0$
\begin{equation*}
\AS_{a,t} = \frac{1}{\pi\sqrt{2t -(x-a)^2}} \mathbf{1}_{(-\sqrt{2t }+a,\sqrt{2t}+a)}(x) \,{\rm d}x
\end{equation*}
is $\rhd$-infinitely divisible if and only if $a =0$. Indeed, it is well known that $\AS_{0,t}$ is $\rhd$-infinitely divisible (see \cite{M00}). For $a\neq0$, let $\{r_n\}_{n\geq1}$ be the monotone cumulants. We can see that
\begin{equation*}
\{r_n\}_{n=1}^4 =\left\{\sqrt{t} a, t, \frac{t^{3/2} a}{2}, \frac{t^2 a^2}{6}\right\}, \qquad \begin{vmatrix}
r_2 & r_3 \\
r_3 & r_4
\end{vmatrix}
= -\frac{t^3 a^2}{12}<0.
\end{equation*}
This shows that, when $a\neq0$, the measure $\AS_{a,t}$ is not $\rhd$-infinitely divisible. On the other hand, the Cauchy transform $G_{\AS_{a,t}}(z)= 1/\sqrt{(z-a)^2-2 t}$ is univalent on $\C^+$.
\end{example}

Note that the right monotone shift $\mu \rhd \delta_a$ is the usual shift, but the left one $\delta_a \rhd \mu$ is in general different. It also turns out that the left monotone shift may break the $\rhd$-infinite divisibility.
\begin{example}\label{arcsine not ID2} Let $a \in \R$ and $t>0$.
The left monotone shift $\delta_a \rhd \AS_{0,t}$
is $\rhd$-infinitely divisible if and only if $a =0$. Indeed, let $\{r_n\}_{n\geq1}$ be the monotone cumulants. We can see that
\begin{equation*}
\{r_n\}_{n=1}^4 =\left\{\sqrt{t} a, t, -\frac{t^{3/2} a}{2}, \frac{t^2 a^2}{6}\right\}, \qquad \begin{vmatrix}
r_2 & r_3 \\
r_3 & r_4
\end{vmatrix}
= -\frac{t^3 a^2}{12}<0.
\end{equation*}
This shows that, when $a\neq0$, the measure $\delta_a\rhd \AS_{0,t}$ is not $\rhd$-infinitely divisible. On the other hand, the reciprocal Cauchy transform $F_{\delta_a \rhd \AS_{0,t}}(z)= \sqrt{z^2-2 t} -a $ is univalent on $\C^+$.
\end{example}

\begin{remark}\label{semicircle not ID2}
We can also give a geometric proof of $\ID(\rhd) \not= \IA(\rhd)$. 
Choose $\mu$ such that $F_\mu$ is univalent and $F_\mu(\Ha)=\Ha\setminus \gamma(0,1]$ for a simple curve 
$\gamma:[0,1]\to\overline{\Ha}$ with $\gamma(0)\in\R$ and $\gamma(0,1]\subset \Ha$. This is possible for any such curve due to Theorem \ref{ftransform_images}. Assume that $\gamma$ is not a vertical line segment. 
Then $\mu\in \IA(\rhd)$ 
due to Theorem \ref{thminfinitesimalarray}.

Furthermore, assume that the curve is parametrized by \index{half-plane capacity}half-plane capacity, i.e. the unique conformal mapping 
$g_t:\Ha\setminus\gamma(0,t]\to\Ha$ with hydrodynamic normalization has half-plane capacity $t$, which means $g_t(z)=z+\frac{t}{z}+...$ 
at $\infty$. 
Then $f_t=g_t^{-1}$ is an \index{Loewner chain!additive}additive Loewner chain satisfying the Loewner equation from Proposition \ref{EV_normal_add} with 
$$M(z,t)= \frac {1}{U(t)-z},$$
where $U(t)=g_t(\gamma(t))$, see \cite[Prop. 4.4]{lawler05}. 
Any other Loewner chain generating $F_\mu$ corresponds to a time change of the Loewner chain $(f_t)$.

Now assume that $\mu\in \ID(\rhd)$. Then, by Theorem \ref{MIDLK}, $F_\mu$ can be embedded into a Loewner chain $(h_t)$ which is a semigroup. The additivity of the half-plane capacity implies that 
$h_t(z)=z-\frac{c t}{z}+...$ for some $c>0$. A time change yields $c=1$ and we have $(h_t)=(f_t)$, which implies $M(z,t)$ does not depend on 
$t$, i.e. $U(t)\equiv u\in\R$. In other words, $\gamma[0,1]$ must be a vertical line segment connecting $u$ to some $u+iT$, $T>0$, which is a contradiction to our assumption.
Hence, $\mu\not\in \ID(\rhd)$.
\end{remark}

\index{infinitesimal array!of measures on $\R$|)}
\index{monotone convolution!additive $\rhd$|)}

\section[Univalent Cauchy transforms]{Subclasses of probability measures with univalent Cauchy transforms}

\index{Cauchy transform!univalent|(}

The class $\Univ(\R)$ of probability measures on $\R$ with univalent Cauchy transforms is important in view of Theorem \ref{thminfinitesimalarray}. We present several of its subclasses.

\subsection{Monotonically infinitely divisible distributions}\label{subsec_mon_inf_div_add}

\index{infinite divisibility!additive monotone convolution|(}

We start from a basic characterization of the class $\ID(\rhd)$ proved by Muraki \cite{M00} in the finite variance case and Belinschi \cite{Bel05} in the general case.

\begin{theorem} \label{MIDLK}
\begin{enumerate}[\rm(1)]
\item If $\mu \in \ID(\rhd)$, then there exists a unique weakly continuous $\rhd$-convolution semigroup $\{\mu_t\}_{t\geq0} \subset\cP(\R)$ such that $\mu_0=\delta_0$ and $\mu_1=\mu$.
\item If $\{\mu_t\}_{t\geq0} \subset\cP(\R)$ is a weakly continuous $\rhd$-convolution semigroup such that $\mu_0=\delta_0$, then $\mu_1 \in \ID(\rhd)$.
\end{enumerate}
\end{theorem}

Theorem \ref{MIDLK} and Theorem \ref{MLK} provide one-to-one correspondences between the following sets:

\begin{itemize}
\item[(i)]
$\ID(\rhd)$; 
\item[(ii)] the set of weakly continuous $\rhd$-convolution semigroups   $\{\mu_t\}_{t\geq0}$ such that $\mu_0=\delta_0$; 
\item[(iii)] the set of analytic mappings $A$ of the form \eqref{MLK1}. 
\end{itemize}

It is well known that
\begin{equation}
\ID(\rhd) \subset \Univ(\R)
\end{equation}
since for $\mu \in \ID(\rhd)$ the map $F_\mu$ is obtained as a time 1 map of the flow as described in Theorem \ref{MLK}.

As a subclass of $\Univ(\R)$, one missing property of $\ID(\rhd)$ is the following. 
\begin{conjecture}
The set $\ID(\rhd)$ is weakly closed.
\end{conjecture}

There are not many examples of $\rhd$-infinitely divisible distributions available in the literature. This is because proving a specific distribution to be $\rhd$-infinitely divisible is equivalent to embedding the map $F_\mu$ into a flow, which is a hard problem. For example it is not known whether the standard \index{normal distribution}normal distribution $N(0,1)$ is in $\ID(\rhd)$ or not. We present one family of explicit examples below.

\begin{example}\label{mst} The \index{monotonically stable distribution}monotonically stable distribution $\bfm_{\alpha,\rho,t},  \alpha\in(0,2], \rho\in[0,1]\cap [1-1/\alpha, 1/\alpha], t >0 $ is introduced in \cite{Has10} and is characterized by
\begin{equation*}
G_{\bfm_{\alpha,\rho,t}}(z) = (z^\alpha+t e^{i \alpha \rho \pi})^{-1/\alpha}, \qquad  z\in \C^+,
\end{equation*}
where the power functions $w^\beta$ are defined continuously for angles $\arg w \in(0,2\pi)$.
It is Lebesgue absolutely continuous and the density is studied in \cite{HS15}. It satisfies the semigroup property
\begin{equation}\label{MST semigroup}
\bfm_{\alpha,\rho,s} \rhd \bfm_{\alpha,\rho,t} =\bfm_{\alpha,\rho,s+t}
\end{equation}
 and hence is $\rhd$-infinitely divisible. The analytic vector field $A$ associated to the semigroup $\{\bfm_{\alpha,\rho,t}\}_{t\geq0}$ in Theorem \ref{MLK} is given by
\begin{equation*}
A(z) = \frac{1}{\alpha} e^{i\alpha\rho\pi}z^{1-\alpha}.
\end{equation*}
In particular, the case $\alpha=2$ (then only $\rho=1/2$ is allowed) corresponds to the centered \index{arcsine distribution}arcsine law with variance $t/2$
\begin{equation*}
\bfm_{2,1/2,t}({\rm d}x) = \AS_{0,t/2}= \frac{1}{\pi\sqrt{t-x^2}} \mathbf{1}_{[-\sqrt{t},\sqrt{t}]}(x) \,{\rm d}x,
\end{equation*}
and the case $\alpha=1$ corresponds to the \index{Cauchy distribution}Cauchy distribution.
\end{example}

\index{infinite divisibility!additive monotone convolution|)}

\subsection{Freely infinitely divisible distributions}\label{subsec_free_inf_div_add}

\index{infinite divisibility!additive free convolution|(}

A $\boxplus$-infinitely divisible measure has a free analogue of the \index{L\'evy-Khintchine representation!free}L\'{e}vy-Khintchine representation in terms of the \index{Voiculescu transform}Voiculescu transform \eqref{eq:Voiculescu_transform}.
\begin{theorem}[\cite{BV93}] \label{thmBV93}
For a probability measure $\mu$ on $\R$, the following statements are equivalent.
\begin{enumerate}[\rm(1)]
\item $\mu \in \ID(\boxplus)$.
\item For any $t>0$, there exists $\mu^{\boxplus t} \in\cP(\R)$ with the property $\varphi_{\mu^{\boxplus t}}(z) = t\varphi_\mu(z).$
\item $-\varphi_\mu$ extends to a Pick function, i.e.\ an analytic map of $\C^+$ into $\C^+ \cup \R$.
\item\label{FLK1} There exist $\gamma \in \R$ and a finite, non-negative measure $\rho$ on $\R$ such that
\begin{equation*}
\varphi_\mu(z)=\gamma +\int_{\mathbb{R}}\frac{1+z x}{z-x} \rho({\rm d}x) ,\qquad z\in \C^+.
\end{equation*}
The pair $(\gamma,\rho)$ is unique.
\end{enumerate}
Moreover, given $\gamma\in\R$ and a finite, non-negative measure $\rho$ on $\R$, there exists a unique $\boxplus$-infinitely divisible distribution $\mu$ which has the Voiculescu transform of the form \eqref{FLK1}.
\end{theorem}

The following is a well known result, whose proof is provided for completeness.
\begin{proposition}\label{FID}
$\ID(\boxplus) \subset \Univ(\R)$.
\end{proposition}
\begin{proof}
The function $F_\mu^{-1}$ defined by $z+\varphi_\mu(z)$ extends analytically to $\C^+$ due to Theorem \ref{thmBV93}. Such defined $F_\mu^{-1}$ coincides with the right inverse of $F_\mu$ in a domain of the form $\Gamma_{\lambda,M}$, and so $F_\mu^{-1}(F_\mu(z))=z$ for $z\in \Gamma_{\lambda,M}$. By the identity theorem, this is the case for all $z\in\C^+$ and hence $F_\mu$ is univalent in $\C^+$.
\end{proof}

\begin{example}\label{exa:semicircle}
 The \index{free stable distribution}free (strictly) stable distribution $\bff_{\alpha,\rho,t}, \alpha\in(0,2], \rho\in[0,1]\cap [1-1/\alpha, 1/\alpha], t >0 $ is introduced in \cite{BV93} and is characterized by
\begin{equation*}
\varphi_{\bff_{\alpha,\rho,t}}(z) = - t e^{i\alpha\rho\pi}z^{1-\alpha}, \qquad z\in\C^+.
\end{equation*}
It is Lebesgue absolutely continuous and its density is studied in the Appendix of \cite{BP99} and in \cite{Dem11,HK14}. The density can be written explicitly in special cases when $\alpha=1/2, 1,2$.
In particular, $\bff_{1,\rho,t}$ is the \index{Cauchy distribution}Cauchy distribution and coincides with $\bfm_{1,\rho,t}$. The most important case is $\alpha=2$ and corresponds to the \index{semicircle distribution}semicircle distribution with mean 0 and variance $t$
\begin{equation*}
\bff_{2,1/2,t}({\rm d}x) = \frac{1}{2\pi t} \sqrt{4 t -x^2} \mathbf{1}_{[-2\sqrt{t},2\sqrt{t}]}(x)\,{\rm d}x,
\end{equation*}
which has the Cauchy transform
\begin{equation*}
G_{\bff_{2,1/2,t}}(z) = \frac{z - \sqrt{z^2-4 t}}{2 t}, \qquad z \in\C^+,
\end{equation*}
where the square root $\sqrt{w}$ is defined continuously on angles $\arg w \in(0,2\pi)$. In this case the range is the half-disk
\begin{equation}\label{half-disk}
G_{\bff_{2,1/2,t}}(\C^+) = \{z= x + i y: x^2 + y^2 < 1/t, y<0\}.
\end{equation}
\end{example}

\begin{example} The \index{free Poisson distribution}free Poisson (or Marchenko-Pastur) distribution $\MP_\lambda,\lambda>0$ is given by
\begin{equation*}
\max\{1-\lambda,0\} \delta_0 + \frac{1}{2\pi x}\sqrt{((1+\sqrt{\lambda})^2-x)(x-(1-\sqrt{\lambda})^2)}\,\mathbf{1}_ {((1-\sqrt{\lambda})^2,(1+\sqrt{\lambda})^2)}(x)\,{\rm d}x,
\end{equation*}
which has the \index{Voiculescu transform}Voiculescu and Cauchy transforms
\begin{align*}
&\varphi_{\MP_\lambda}(z)= \frac{\lambda z}{z-1},\qquad z\in\C^+, \\
& G_{\MP_\lambda}(z) = \frac{z+1-\lambda-\sqrt{(z+1-\lambda)^2-4 z}}{2 z},
\end{align*}
where the square root $\sqrt{w}$ is defined continuously on angles $\arg w \in(0,2\pi)$.
\end{example}

Recent works have found many probability measures in $\ID(\boxplus)$ including the \index{normal distribution}normal distribution \cite{BBLS11}. For other examples see \cite{Has16} and references therein.

One may wonder whether the $\ast$-infinitely divisible distributions form a subclass of $\Univ(\R)$. This is not the case in general as Section \ref{sec reg} shows.

\index{infinite divisibility!additive free convolution|)}

\subsection{Unimodal distributions}\label{subsec_unimodal_add} 

\index{unimodal distribution!on $\R$|(}

A Borel measure $\mu$ on $\R$ is said to be \emph{unimodal} with mode
$c \in \R$ if there exist a non-decreasing function $f\colon(-\infty, c)\mapsto [0,\infty)$ and a non-increasing function $g:(c,\infty)\mapsto [0,\infty)$ and $\lambda\in[0,\infty]$ such that
\begin{equation}
\mu({\rm d}x) = f(x)\mathbf{1}_{(-\infty,c)}(x)\,{\rm d}x + \lambda\delta_c +  g(x)\mathbf{1}_{(c,\infty)}(x)\,{\rm d}x.
\end{equation}
Note that $c$ need not be unique. For instance it can be any point in the support of a uniform distribution. The set of unimodal probability measures on $\R$ is denoted by $\UM(\R)$. It is closed with respect to weak convergence, see e.g.\ \cite[Exercise 29.20]{Sat13}, and the inclusion $\UM(\R)\subset \Univ(\R)$ holds as the following Theorem \ref{anotherkhintchine} shows.

The unimodal probability measures have been characterized in several way in the literature from different interests. 
\begin{theorem}\label{anotherkhintchine} Let $\mu$ be a probability measure on $\R$.  The following statements are mutually equivalent.
\begin{enumerate}[\rm(1)]
\item\label{another1} $\mu$ is unimodal with mode $c$.
\item\label{another2} $\Im((z-c)G_\mu'(z)) \geq0$ for all $z\in\C^+$.
\item\label{another3} There exists an $\R$-valued random variable $X$, independent of a uniform random variable $U$ on $(0,1)$, such that $\mu$ is the law of $U X +c$.  

\item\label{another4} The following three assertions hold:
\begin{itemize}
\item $G_\mu$ is univalent in $\C^+$.
\item \index{Cauchy transform!horizontally convex}$G_\mu(\C^+)$ is horizontally convex, namely if $z_1,z_2 \in G_\mu(\C^+)$ with the same imaginary part, then $(1-t)z_1+t z_2 \in G_\mu(\C^+)$ for any $t\in(0,1)$.
\item There exist points $z_n \in \C^+$ such that $z_n \to c$ and
$$
\lim_{n\to\infty} \Im (G_\mu(z_n)) = \inf_{z\in \C^+} \Im (G_\mu(z)).
$$
\end{itemize}
\end{enumerate}
\end{theorem}
\begin{remark}
\begin{enumerate}[\rm(a)]
\item Kaplan \cite{K52} proved that if $\mu$ does not have an atom, then \eqref{another1} implies \eqref{another2} and \eqref{another4}, but it seems that he did not prove the converse statements.

\item The equivalence between \eqref{another1} and \eqref{another2} is essentially proved by Isii \cite[Theorem 3.2']{Isi57}.
\end{enumerate}
\end{remark}
\begin{proof} For simplicity we assume that $c=0$. The general case follows by a simple transformation. 

\eqref{another1} $\Leftrightarrow$ \eqref{another3} is Khintchine's characterization \cite{K38} (see also \cite[\S32, Theorem 2]{GK54}), saying that a probability measure $\mu$ on $\R$ is unimodal with mode 0 if and only if
there exists a probability measure $\nu$ such that
\begin{equation}
\widehat{\mu}(t)=\frac{1}{t}\int_0^t\widehat{\nu}(s)\,{\rm d}s,
\end{equation}
where $\widehat{\mu}$ is the characteristic function of $\mu$. This is equivalent to saying that $\mu$ is the law of a random variable $U X$ where $U$ is a uniform random variable on $(0,1)$ and $X$ is any $\R$-valued random variable independent of $U$.

\eqref{another2} $\Rightarrow$ \eqref{another3}: since $z G_\mu'(z)$ is a Pick function and $\lim_{y\to\infty} iy (iy G_\mu'(iy))=-1$, there exists a probability measure $\nu$ such that
$zG_\mu'(z)=-G_\nu(z)$. Integration gives us
\begin{equation}\label{mixture uniform}
G_\mu(z)= -\int_{\R\setminus\{0\}}\frac{1}{x}\log\left(\frac{z-x}{z}\right)\,\nu({\rm d}x)+ \frac{\nu(\{0\})}{z}.
\end{equation}
Since the Cauchy transform of the \index{uniform distribution}uniform distribution on $(0,x)$ (or $(x,0)$ if $x<0$) is equal to $-\frac{1}{x}\log\left(\frac{z-x}{z}\right)$, we conclude that
$\mu$ is the law of $U X$ where $X$ has the law $\nu$.

\eqref{another3} $\Rightarrow$ \eqref{another2}: \eqref{another3} implies the representation \eqref{mixture uniform}, which implies $zG_\mu'(z)=-G_\nu(z)$.

\eqref{another2} $\Leftrightarrow$ \eqref{another4}: This follows from Hengartner and Schober \cite{HS70} with a suitable Moebius transformation from the unit disk onto the upper half-plane.
 \end{proof}

 \begin{example} The \index{uniform distribution}uniform distribution $\bfu_t$ on $(0,t), t>0$, has the Cauchy transform
\begin{equation*}
G_{\bfu_t}(z)=\frac{1}{t} \log \frac{z}{z-t},
\end{equation*}
which appeared in the proof of Theorem \ref{anotherkhintchine}.
The range $G_{\bfu_t}(\C^+)$ is the strip $\{z \in \C^-:  -\pi/t < \Im(z)< 0 \}$ which becomes smaller as $t$ becomes larger.  This domain is horizontally convex.
\end{example}

\begin{example}
The range of the Cauchy transform of the \index{semicircle distribution}semicircle distribution with mean 0 and variance $t$ is the half-disk \eqref{half-disk}, which is horizontally convex. This domain is also starlike in the sense of Definition \ref{def:starlike}.
\end{example}

Subclasses of $\UM(\R)$ are provided in Sections \ref{sec CSD} and \ref{sec FSD}.

\index{unimodal distribution!on $\R$|)}

\subsection{Selfdecomposable distributions} \label{sec SD}

\index{selfdecomposable distribution}

In classical probability, a subclass of $\IA(\ast)= \ID(\ast)$ can be provided by the selfdecomposable distributions. The free analogue was defined by Barndorff-Nielsen and Thorbj{\o}rnsen \cite{BNT02}.  We also discuss the monotone version. It turns out that all these classes are contained in $\Univ(\R)$; see Sections \ref{sec CSD}, \ref{sec FSD} for the classical and free cases and Section \ref{sec MSD} for the monotone case.

Let $\DD_c\mu$ be the scaling of a probability measure $\mu$ by a constant $c\in \R$, namely $(\DD_c\mu)(A)=\mu(A/c)$ for Borel sets $A \subset \R$ when $c\neq0$ and $\DD_0\mu=\delta_0$.
\begin{definition}\label{defMSD}
Let $\star$ be an associative binary operation on $\cP(\R)$. A probability measure $\mu$ on $\R$ is said to be \emph{$\star$-selfdecomposable} if for any $c\in(0,1)$ there exists a probability measure $\bmu^c$ such that $\mu=(\DD_c\mu)\star\bmu^c$. This class is denoted by $\SD(\star)$.
\end{definition}

The following property readily follows from the definition.
\begin{proposition}
Suppose that an associative binary operation $\star$ is commutative and satisfies $\DD_c(\mu\star\nu) = (\DD_c \mu )\star (\DD_c\nu)$ for all $c\in\R$ and $\mu,\nu \in \cP(\R)$. Then $\SD(\star)$ is closed under the operation $\star$.
\end{proposition}
For $\star=\ast$ or  $\boxplus$, it is known that $\SD(\star)$ is weakly closed and the measure $\bmu^c$ is unique. We later show that the same is true for monotone convolution.

\subsection{Classically selfdecomposable distributions} \label{sec CSD}

\index{selfdecomposable distribution!additive classical convolution|(}

The class $\SD(\ast)$ of classically selfdecomposable distributions is known to be a weakly closed subset of $\ID(\ast)$.
Yamazato \cite{Yam78} proved that all $\ast$-selfdecomposable distributions are unimodal:
\begin{equation*}
\SD(\ast) \subset \UM(\R),
\end{equation*}
and thus $\SD(\ast) \subset \Univ(\R)$. The \index{classical stable distribution}classical stable distributions and in particular the \index{normal distribution}normal distribution are $\ast$-selfdecomposable.

A further subclass of $\SD(\ast)$ is the class $\GGC$ of \index{generalized gamma convolutions}\emph{generalized gamma convolutions} \cite{Bon92}. This class is defined to be the weak closure of the set
\begin{equation*}
\{\gammabm(p_1,\theta_1)\ast \cdots \ast \gammabm(p_n,\theta_n)\mid p_k,\theta_k>0, n\in\N, k=1,\dots,n\},
\end{equation*}
where $\gammabm(p,\theta)$ is the \index{gamma distribution}gamma distribution
\begin{equation*}
\frac{1}{\theta^{p}\Gamma( p)} x^{p-1} e^{-x/\theta}\,\mathbf{1}_{(0,\infty)}(x)\,{\rm d}x,\qquad p,\theta>0.
\end{equation*}
It is known that all gamma distributions are $\ast$-selfdecomposable, and hence $\GGC \subset \SD(\ast)$.

Generalized gamma convolutions are all supported on $[0,\infty)$. Bondesson introduced the class of \emph{extended GGCs} (denoted by $\EGGC$, also called the Thorin class). It is the weak closure of the set
\begin{equation*}
\{\gammabm(p_1,\theta_1)\ast \cdots \ast \gammabm(p_n,\theta_n)\mid p_k>0, \theta_k\in\R, k=1,\dots,n, n\in\N\},
\end{equation*}
where $\gammabm(p,\theta)$ is defined to be the scaling $\DD_{\theta}(\gammabm(p,1))$ for all $\theta \in\R$ (this definition coincides with the original one when $\theta>0$).
Since $\SD(\ast)$ is closed under the convolution $\ast$ and $\gamma(p,\theta) \in \SD(\ast)$ for all $p>0$ and $\theta\in\R$, we conclude that $\EGGC\subset \SD(\ast)$. Thus, those well known classes $\SD(\ast)$, $\GGC$ and $\EGGC$ are all contained in $\Univ(\R)$.

\index{selfdecomposable distribution!additive classical convolution|)}

\subsection{Freely selfdecomposable distributions}\label{sec FSD}

\index{selfdecomposable distribution!additive free convolution|(}

The class $\SD(\boxplus)$ of freely selfdecomposable distributions was originally introduced by Barndorff-Nielsen and Thorbj{\o}rnsen \cite{BNT02}. It is known that $\SD(\boxplus)$ is a weakly closed subset of $\ID(\boxplus)$. This class may be characterized in terms of the free cumulant transform $C_\mu(z) = z \varphi_\mu(1/z)$, where $\varphi_\mu$ is the Voiculescu transform. 

\begin{theorem}[\cite{HST}]
For a probability measure $\mu$ on $\R$ the following statements are
equivalent.
\begin{enumerate}[{\rm (i)}]
\item\label{aa} $\mu\in \SD(\boxplus)$.
\item\label{bb} $C_\mu$ extends to an analytic map $C_{\mu}\colon\C^{-}\to\C$ such that the derivative $C_\mu'$  satisfies that
$\Im(C_{\mu}'(z))\leq 0$ for any $z\in\C^{-}$.
\item\label{cc}
There exists $\beta$ in $\R$ and a non-negative measure $\sigma$ on $\R$, satisfying that
$\int_{\R}\log(|x|+2)\,\sigma({\rm d}x)<\infty$, such that
$C_\mu'$ may be extended to all of $\C^-$ via the formula:
\begin{align}\label{corform}
C_{\mu}'(z)=\beta + \int_{\R}\frac{x+z}{1-z x }\,\sigma({\rm d}x), \qquad z\in\C^{-}.
\end{align}
\end{enumerate}
If \eqref{aa}--\eqref{cc} are satisfied, then the pair $(\beta,\sigma)$ in
\eqref{cc} is unique. Additionally, given a pair $(\beta, \sigma)$ of a real number and a non-negative measure on $\R$ satisfying the integrability condition above, then there exists a unique $\mu \in \SD(\boxplus)$ such that $C_\mu'$ is represented in the form \eqref{corform}.
\end{theorem}

Hasebe and Thorbj\o rnsen \cite{HT16} showed the free analogue of Yamazato's theorem:
\begin{equation*}
\SD(\boxplus) \subset \UM(\R).
\end{equation*}
Furthermore we may show the following.
\begin{proposition}\label{FSD}
$\SD(\boxplus)\subset \SD(\rhd)$.
\end{proposition}

\begin{proof}
By definition, for any $c\in (0,1)$, there exists $\bmu^c$ such that $\mu=(\DD_c\mu)\boxplus \bmu^c$. From the subordination property for free convolution (see e.g.\ \cite{BB04} or the original article \cite{Bia98}), there exists another probability measure $\overline{\nu}^c$ such that $\mu=(\DD_c\mu)\rhd \overline{\nu}^c$, and hence $\mu \in \SD(\rhd)$.
\end{proof}

The above results and arguments show that
\begin{equation*}
\SD(\boxplus)\subset\ID(\boxplus)\cap \SD(\rhd) \cap \UM(\R).
\end{equation*}

\begin{example}The \index{free stable distribution}free stable distribution has the semigroup property $\bff_{\alpha,\rho,s} \boxplus \bff_{\alpha,\rho,t} = \bff_{\alpha,\rho,s+t}$ and the stability $\DD_c( \bff_{\alpha,\rho,t}) =  \bff_{\alpha,\rho,c^\alpha t}, c>0$.  These conditions imply that
\begin{equation*}
\bff_{\alpha,\rho,t} = (\DD_{c}\bff_{\alpha,\rho,t}) \boxplus \bff_{\alpha,\rho,(1-c^\alpha)t}, \qquad c\in(0,1),
\end{equation*}
and hence the $\boxplus$-selfdecomposability.
\end{example}

Other examples of freely selfdecomposable distributions are some \index{free Meixner distribution}free Meixner distributions and the \index{normal distribution}classical normal distributions (see \cite{HST}).

\index{selfdecomposable distribution!additive free convolution|)}

\subsection{Univalence and regularity of probability measures} \label{sec reg}
We have seen several sufficient conditions for a probability measure to have a univalent Cauchy transform. This section presents a necessary condition in terms of some regularity property of a probability measure. This result is a slight generalization of \cite[Theorem 3.5]{Has10}, but the proof is almost the same.
\begin{proposition}\label{Atm}
Let $\mu \in \Univ(\R)$. Suppose that $\mu$ has an isolated atom at $a \in \R$, namely $\mu(\{a\})>0$ and $\mu(\{a\})=\mu((a-\epsilon,a+\epsilon))$ for some $\epsilon>0$.
Then  $\mu|_{\R\setminus \{a\}}$ is Lebesgue absolutely continuous with $L^\infty$ density.
\end{proposition}
\begin{proof}
We prove in fact that
\begin{equation}\label{Cinfty}
\text{$G_\mu(B_\epsilon(a)^c\cap\C^+)$ is a bounded subset of $\C^-$},
\end{equation}
where $B_\epsilon(a)$ is the ball at $a$ with radius $\epsilon>0$.
We postpone the proof of \eqref{Cinfty} and suppose for now that it is the case.

Note that, in general, the singular part of a probability measure $\mu$ is supported on the set
\begin{equation}\label{support}
\{x \in \R: \lim_{y\downarrow0}\Im(G_\mu(x+i y))=-\infty\},
\end{equation}
which follows from the simple estimate
\begin{equation}\label{Der}
\frac{\mu((x-h,x+h))}{2h} \leq \int_{|u-x|<h} \frac{h}{h^2+(u-x)^2}\,\mu({\rm d}u) \leq - \Im(G_\mu(x+i h)),\quad x\in \R, h>0,
\end{equation}
and the basic fact (see \cite[Theorem 7.15]{Rud87}) that the singular part is supported on
\begin{equation*}
\left\{x \in\R: \lim_{h\downarrow0} \frac{\mu((x-h,x+h))}{2h} =\infty\right\}.
\end{equation*}
Obviously \eqref{Cinfty} implies that the set \eqref{support} is $\{a\}$, and hence $\mu|_{\R\setminus\{a\}}$ is Lebesgue absolutely continuous. The density is essentially bounded since, thanks to \eqref{Cinfty}, the right hand side of \eqref{Der} is bounded by a uniform constant independent of $h$ and $x \in (a-\epsilon,a+\epsilon)^c$.

Now we prove the key fact \eqref{Cinfty}. Let $\tau$ be a finite, non-negative measure such that $\mu =\lambda \delta_a + \tau$ and $\lambda=\mu(\{a\})$. It is supported on $(a-\epsilon,a+\epsilon)^c$ and hence the Cauchy transform $G_\tau$ is analytic in $B_\epsilon(a)$.

Suppose that \eqref{Cinfty} is not the case. Then we can find  a sequence of points $z_1,z_2,z_3,\dots$ in $\C^+ \cap B_\epsilon(a)^c$ such that $F_\mu(z_n) \to 0$.
We look for a point $z \in \C^+$ and $n\in\N$ such that $F_\mu(z)=F_\mu(z_n)$. A solution $z$ to the equation $F_\mu(z) = F_\mu(z_n)$ is a zero of the function
\begin{equation}
f_n(z) := z-a + F_\mu(z_n)(\lambda + (z-a) G_\tau(z)).
\end{equation}
For sufficiently large $n$, the function $f_n(z) - (z-a)$ is smaller in absolute value than $f(z)$ on $\partial B_\epsilon(a)$, and hence $f$ has a unique zero in $B_\epsilon(a)$ by Rouch\'e's theorem. This zero is not $a$ since $\lambda F_\mu(z_n) \neq 0 $. Therefore, we found a point $z \neq a$ such that $F_\mu(z) = F_\mu(z_n)$. It remains to show that $z\in\C^+$. Indeed, $F_\mu$ maps $\C^-$ into $\C^-$ and $(a-\epsilon,a+\epsilon)$ into $\R\cup\{\infty\}$, and hence $z\in \C^+$. This contradicts the univalence of $F_\mu$.
\end{proof}

Since the Poisson distribution contains two isolated atoms, we obtain the following consequence.
\begin{corollary}
$\ID(\ast)$ is not a subset of $\Univ(\R)$.
\end{corollary}

Some problems are presented below. 

\begin{problem}\label{non-isolated}
Is it possible to extend Proposition \ref{Atm} to the case where $a$ is an atom which is not isolated?
\end{problem}
As a clue to this problem, consider the probability measure
 \begin{equation}
  \mu= \frac{1}{2} \delta_0 + \frac{1}{2}\mathbf{1}_{[0,1]}(x)\,{\rm d}x,
\end{equation}
 which is unimodal and hence has a univalent Cauchy transform. The Cauchy transform is
 \begin{equation}
  G_\mu(z) = \frac{1}{2}\left(\frac{1}{z} + \log \frac{z}{z-1} \right),
\end{equation}
 where $\log$ is the principal value. Obviously we see that $G_\mu(z)\to\infty$ as $z\to1$. This example shows that we cannot expect the property \eqref{Cinfty} for any $\epsilon>0$ for a measure in $\Univ(\R)$ with a non-isolated atom. 

\begin{problem}
Is there a singular probability measure in $\Univ(\R)$?
\end{problem}

\subsection{Limit theorems and infinitely divisible distributions}\label{LThm}

As we saw many univalent Cauchy transforms can be provided by infinitely divisible distributions, which are known to appear in limit theorems. We make some observations on limit theorems and pose some problems.

The \index{infinite divisibility!additive classical convolution}\index{L\'evy-Khintchine representation!classical}classical L\'evy-Khintchine representation says that
\begin{equation*}
\ID(\ast,\R) =\{\mu_\ast^{\gamma,\rho}: \gamma\in\R, \text{$\rho$ is a finite, non-negative measure on $\R$}\},
\end{equation*}
where $\mu_\ast^{\gamma,\rho}$ is the probability measure characterized by
\begin{equation}\label{CLK1}
\int_{\R} e^{z x} \mu_\ast^{\gamma,\rho}({\rm d}x) =\exp\left(\gamma z + \int_{\R} \left(e^{z x}-1- \frac{z x}{1+x^2} \right)\frac{1+x^2}{x^2}\,\rho({\rm d}x)\right), \qquad z\in i\R.
\end{equation}
Note that the integrand $\left(e^{z x}-1- \frac{z x}{1+x^2} \right)\frac{1+x^2}{x^2}$ is defined to be $\frac{1}{2}z^2$ at $x=0$, so that it becomes a continuous function on $\R$.
Similarly, Theorem \ref{thmBV93} implies that \index{infinite divisibility!additive free convolution}
\begin{equation*}
\ID(\boxplus,\R) = \{\mu_\boxplus^{\gamma,\rho}:\gamma\in\R, \text{$\rho$ is a finite, non-negative measure on $\R$}\},
\end{equation*}
where $\mu_\boxplus^{\gamma,\rho}$ is the probability measure characterized by the parameter $(\gamma,\rho)$ in Theorem \ref{thmBV93}\eqref{FLK1}.

In addition to classical and free Khintchine's theorems \eqref{CK} and \eqref{FK}, classical and free limit theorems are equivalent in the following sense.

\begin{theorem}\label{NIID} Let $\{\mu_{n,j}\}_{1\leq j \leq k_n, n\geq1}$ be an infinitesimal array of probability measures on $\R$ and let $a_n \in \R$.
The following statements are equivalent.
\begin{enumerate}[\rm(1)]
\item\label{CLimit} $\delta_{a_n} \ast \mu_{n, 1} \ast \cdots \ast \mu_{n, k_n} \wto \mu_\ast^{\gamma,\rho}$.
\item\label{FLimit} $\delta_{a_n} \boxplus \mu_{n, 1} \boxplus \cdots \boxplus \mu_{n, k_n} \wto \mu_\boxplus^{\gamma,\rho}$.
\item\label{PairLimit} For the shifted measures $\mathring{\mu}_{n, j}(B) := \mu_{n, j}(B + a_{n, j})$ ($B$ is a Borel set) and $a_{n, j} = \int_{|x|<r} x \, \mu_{n, j}({\rm d}x)$ ($r>0$ is any fixed number), the following convergence holds:
\begin{align*}
&\sum_{j=1}^{k_n}\frac{x^2}{1+x^2}\,\mathring{\mu}_{n,j}({\rm d}x)\wto \rho, \qquad a_n+\sum_{j=1}^{k_n}\left(a_{n,j}+\int_\R\frac{x}{1+x^2}\,\mathring{\mu}_{n,j}({\rm d}x)\right) \to \gamma.
\end{align*}
\end{enumerate}
\end{theorem}
The equivalence between \eqref{CLimit} and \eqref{PairLimit} is well known in classical probability \cite[Chapter 4]{GK54}, and the equivalence between \eqref{FLimit} and \eqref{PairLimit} was proved by Chistyakov and Goetze \cite[Theorem 2.2]{CG08}. Note that a Boolean version also holds \cite{Wan08}. However, we prove that the monotone version fails to hold even if $a_n=0$.
\begin{proposition}\label{m-analogue} In the setting of Theorem \ref{NIID}, suppose moreover that $a_n=0$. Then the statement
\begin{equation}\label{MLimit}
\mu_{n,1} \rhd \cdots \rhd \mu_{n,k_n} \wto \mu_\rhd^{\gamma,\rho}
\end{equation}
is not equivalent to the statements \eqref{CLimit}--\eqref{PairLimit} in Theorem \ref{NIID}.
\end{proposition}
\begin{proof}
Take $k_n=2 n$ and $\mu_{n,1}= \cdots =\mu_{n,n}$ to be the symmetric \index{arcsine distribution}arcsine law $\AS_{0,1/n}$ with variance $1/n$, and take $\mu_{n,n+1} = \cdots =\mu_{n,2n}=\delta_{a/n}, a\neq0$.  The semigroup property \eqref{MST semigroup} implies that $\mu_{n, 1} \rhd \cdots \rhd \mu_{n, 2n}$ is the shifted arcsine law $\AS_{a,1}$ which is not even $\rhd$-infinitely divisible (see Example \ref{arcsine not ID}). However, the central limit theorem says that the measure $\mu_{n, 1} \ast \cdots \ast \mu_{n, 2n}$ converges to $N(a,1)$.
\end{proof}

Therefore, the following question remains unsolved.
\begin{problem}\label{m-analogue2}
Characterize the convergence \eqref{MLimit}.
\end{problem}

Note that Anshelevich and Williams \cite{AW14} proved that \eqref{MLimit} is equivalent to assertions \eqref{CLimit}--\eqref{PairLimit} in Theorem \ref{NIID} when $a_n=0$ and the distributions are identical, namely $\mu_{n,1}= \cdots = \mu_{n,k_n}$. In this identical setting, the Khintchine type theorem is still open.
\begin{conjecture} \label{m-conj}
Suppose that $\mu,\mu_n \in \cP(\R)$, $n\in\N$, and $\{k_n\}_{n\geq1}$ is a sequence of strictly increasing natural numbers.  If $\mu_n^{\rhd k_n} \wto \mu$, then $\mu\in \ID(\rhd)$.
\end{conjecture}
This Khintchine type theorem fails for non-identical distributions since we know that $\IA(\rhd)$ is strictly larger than $\ID(\rhd)$; see Theorem \ref{thminfinitesimalarray} and Example \ref{semicircle not ID}. Moreover, no integral representation is known for the whole class $\IA(\rhd)$. So, even if one solves Problem \ref{m-analogue2}, there is no clue about how the convergence $\mu_{n,1} \rhd \cdots \rhd \mu_{n , k_n} \wto \mu$ could be characterized (as in Theorem \ref{NIID} \eqref{PairLimit}) for a general $\mu\in \IA(\rhd)$ and an infinitesimal array $\{\mu_{n,j}\}$.
However, the subclasses $\ID(\rhd)$ and $\UM(\R)$ of $\IA(\rhd)$ have integral representations; see \eqref{MLK1} and \eqref{mixture uniform}. As mentioned, the class $\ID(\rhd)$ already has a  characterization due to Anshelevich and Williams \cite{AW14} in terms of a limit theorem for identical distributions without a shift. It is then natural to search for limit theorems converging to measures in $\UM(\R)$.

\begin{problem}
Find a limit theorem that characterizes the class $\UM(\R)$ as the set of all possible weak limits. 
\end{problem}

\index{Cauchy transform!univalent|)}

\section[L\'evy's limit theorem]{L\'evy's limit theorem, monotone selfdecomposability, starlike Cauchy transforms, and Markov transform} \label{additive_selfdecom}

\index{selfdecomposable distribution!additive monotone convolution|(}

The most important subclass of $\Univ(\R)$ in this paper is the class of monotonically selfdecomposable distributions. We give three complete characterizations of this class in terms of L\'evy's limit theorem, starlike Cauchy transforms and the Markov transform.

\subsection{L\'evy's limit theorem}\label{sec:Levy}
We discussed in Section \ref{LThm} some limit theorems. The classes of classical and free selfdecomposable distributions also have nice characterizations in terms of a limit theorem.

\begin{definition}
For $\star \in \{\ast,\boxplus,\rhd\}$, we say that a probability measure belongs to class $\LL(\star)$ if it is the weak limit of the probability measures 
\begin{equation}\label{LevyLimit}
\DD_{b_n}(\mu_1\star \cdots  \star \mu_n),\qquad, n=1,2,3,\dots, 
\end{equation} where $b_n$ are positive real numbers and $\mu_{n}$ are probability measures on $\R$ such that\\ $\{\DD_{b_n}(\mu_k)\}_{1 \leq k \leq n, 1\leq n}$ forms an infinitesimal array.
\end{definition}
\begin{remark}
It is common to include a shift and consider the convergence of $\delta_{a_n}\star \DD_{b_n}(\mu_1\star \cdots  \star \mu_n)$ when $\star=\ast$ or $\star=\boxplus$, but in order to avoid the subtlety of shifts for monotone convolution, we  confine ourselves to the case $a_n=0$.
\end{remark}
Obviously, $\LL(\star)$ is a subset of $\IA(\star)$. L\'evy proved that (see \cite[Theorem 56]{Lev54} or \cite[\S29, Theorem 1]{GK54})
\begin{equation}\label{CL}
\SD(\ast) = \LL(\ast).
\end{equation}
The free analogue $\SD(\boxplus)=\LL(\boxplus)$ was proved by Chistyakov and Goetze \cite[Theorem 2.10]{CG08}, which is actually a consequence of \eqref{CL} and Theorem \ref{NIID}. Sakuma \cite{Sak09} gave a more direct proof.

We can prove the analogous statement for monotone convolution. Recall that, from Definition \ref{sec SD}, for every $\mu \in\SD(\rhd)$ and $c\in (0,1)$ there exists $\bmu^c\in \cP(\R)$ such that
\begin{equation}\label{msd0}
\mu = (\DD_c\mu) \rhd \bmu^c, \qquad \text{or}\quad G_\mu(z) = c^{-1} G_\mu(c^{-1} F_{\bmu^c}(z)).
\end{equation}
Moreover, we may define $\bmu^0:= \mu$ and $\bmu^1:=\delta_0$. Lemma \ref{equivconvergence} \eqref{unique} implies that $\bmu^c$ is unique and Lemma \ref{equivconvergence} \eqref{m-conv1} shows that the map $[0,1]\to\cP(\R), c\mapsto \bmu^c$ is weakly continuous.

\begin{theorem}\label{SL}
$\SD(\rhd) = \LL(\rhd).$
\end{theorem}
\begin{proof}
$\SD(\rhd) \subset \LL(\rhd).$ Take the decomposition $\mu=(\DD_c\mu)\rhd\bmu^c$ for $0\leq c<1$ with convention that $\bmu^0=\mu$. Let $\mu_k:= \DD_k(\bmu^{(k-1)/k})$ for $k\in\N$. It satisfies the identity $\DD_{k}\mu=(\DD_{k-1}\mu) \rhd \mu_k$ for all $k\in\N$, and iterating this identity yields that
\begin{equation*}
\mu = \DD_{1/n}(\mu_1 \rhd \cdots \rhd \mu_n).
\end{equation*}
It then remains to prove that $\{\DD_{1/n}\mu_k\}_{1\leq k \leq n, 1\leq n}$ forms an infinitesimal array. Indeed, recalling from Lemma \ref{equivconvergence}\eqref{m-conv2} that $\bmu^c \wto \delta_0$ as $c\uparrow1$, for each $\delta,\epsilon>0$ we may take $k_0=k_0(\delta,\epsilon)\in\N$ such that
$\bmu^{(k-1)/k}([-\delta,\delta]^c) <\epsilon$ for all $k \geq k_0$, and hence $(\DD_{1/n}\mu_k)([-\delta,\delta]^c) <\epsilon$ for all $k_0 \leq k \leq n$ and $n\in\N$ too. Then there exists $n_0=n_0(\delta,\epsilon)\in\N$ such that $(\DD_{1/n_0}\mu_k)([-\delta,\delta]^c) <\epsilon$ for all $1 \leq k < k_0$. Thus we have
\begin{equation*}
\sup_{1\leq k \leq n}(\DD_{1/n}\mu_k)([-\delta,\delta]^c) <\epsilon,\qquad n \geq n_0,
\end{equation*}
showing the infinitesimality.

$\SD(\rhd) \supset \LL(\rhd).$ Take $\mu \in \LL(\rhd)$ and assume that $\mu$ is not $\delta_0$. We may take probability measures $\mu_n$ and positive numbers $b_n$ such that $\DD_{b_n}(\mu_1\rhd \cdots  \rhd \mu_n)\wto\mu$ and such that $\{\DD_{b_n}(\mu_k)\}_{1 \leq k \leq n, 1\leq n}$ forms an infinitesimal array.

Step 1. We show that
\begin{equation*}
\lim_{n\to\infty} b_n=0, \qquad \lim_{n\to\infty} \frac{b_{n+1}}{b_n}=1.
\end{equation*}
Indeed, there exists $k_0 \in \N$ such that $\mu_{k_0}$ is not equal to $\delta_0$. The infinitesimality of $\{\DD_{b_n}(\mu_k)\}_{1 \leq k \leq n, 1\leq n}$ then shows that $\DD_{b_n}\mu_{k_0} \wto \delta_0$ as $n\to\infty$ and hence $b_n\to0$. For the second limit, let $\lambda_n:=\DD_{b_n}(\mu_1\rhd \cdots  \rhd \mu_n)$. In the obvious identity
\begin{equation*}
\lambda_{n+1}= (\DD_{b_{n+1}/b_n}\lambda_n) \rhd (\DD_{b_{n+1}}\mu_{n+1}),
\end{equation*}
the first and the third measures converge to $\mu$ and $\delta_0$ respectively (the latter follows from the infinitesimality), and hence Lemma \ref{equivconvergence} \eqref{m-conv2} shows that $\DD_{b_{n+1}/b_n}\lambda_n \wto \mu$.  Since $\lambda_n\wto\mu\neq\delta_0$, we must have $b_{n+1}/b_n \to 1$ (see \cite[\S10, Theorem 2]{GK54}).

Step 2. Let $c\in (0,1)$. From Step 1 there exist subsequences $\{m(k)\}_{k\geq1}$ and $\{n(k)\}_{k\geq1}$ of $\N$ such that $m(k) < n(k)$ and $b_{n(k)}/b_{m(k)} \to c$; see the proof of \cite[Theorem 15.3]{Sat13} for details. We denote $(m(k),n(k))$ simply by $(m,n)$. In the identity
\begin{equation}
\lambda_n = (\DD_{b_n/b_m}\lambda_m) \rhd \DD_{b_n}(\mu_{m+1} \rhd \cdots \rhd \mu_n),
\end{equation}
the first and second probability measures converge to $\mu$ and $\DD_c\mu$ respectively, and hence Lemma \ref{equivconvergence} \eqref{m-conv1} shows that the third one converges to some probability measure. Thus we have $\mu \in \SD(\rhd)$.
\end{proof}
The above limit theorem shows that the reciprocals of starlike functions can be characterized as the limits of some iterated compositions of analytic self-maps.

We still lack a characterization of convergence of $\DD_{b_n}(\mu_1\rhd \cdots  \rhd \mu_n)$ to a $\rhd$-selfdecom\-posable distribution.

\begin{problem} Given $\mu \in \SD(\rhd)$, $\mu_n \in\cP(\R)$ and $b_n>0$ such that $\{\DD_{b_n}(\mu_k)\}_{1 \leq k \leq n, 1\leq n}$ forms an infinitesimal array, find a necessary and sufficient condition for the convergence $\DD_{b_n}(\mu_1\rhd \cdots  \rhd \mu_n)\wto \mu$ in a way similar to Theorem \ref{NIID} \eqref{PairLimit}.
\end{problem}

We prove a few basic properties of the class $\SD(\rhd)$.

\begin{proposition}\label{msd}
$\SD(\rhd)$ is a weakly closed subset of $\cP(\R)$.
\end{proposition}
\begin{proof}Take $\mu_n \in \SD(\rhd)$ and suppose that $\mu_n\wto \mu \in\cP(\R)$. By definition, for any $c\in(0,1)$ and $n\in\N$ there exists a $\bmu_n^c \in \cP(\R)$ such that
$\mu_n=(\DD_c\mu_n)\rhd\bmu_n^c. $
Lemma \ref{equivconvergence} \eqref{m-conv1} shows that, as $n\to\infty$, the measure $\bmu_n^c$ weakly converges to some $\bmu^c \in \cP(\R)$, and hence $\mu=(\DD_c\mu) \rhd \bmu^c$.
\end{proof}

\begin{proposition}\label{prop:MSD_affine}
If $\mu \in \SD(\rhd)$, $a\in\R$ and $b>0$, then $(\DD_b \mu) \rhd \delta_a \in \SD(\rhd)$.
\end{proposition}
\begin{remark}
The distribution $\delta_a\rhd(\DD_b\mu)$ may not belong to $\SD(\rhd)$, see Example \ref{SD delta}.
\end{remark}
\begin{proof}
The identity \eqref{msd0} implies $(\DD_b\mu)\rhd \delta_a = \DD_c((\DD_b\mu )\rhd \delta_a) \rhd (\delta_{ -c a} \rhd(\DD_b\bmu^c) \rhd \delta_a)$.
\end{proof}

\subsection{Monotone selfdecomposability and starlike Cauchy transform} \label{sec MSD}

\index{Cauchy transform!starlike}

We characterize the set $\SD(\rhd)$. The key concept is the starlikeness.
\begin{definition}\label{def:starlike}
An analytic map $G\colon \C^+\to \C$ having the non tangential limit $G(\infty)=0$ is said to be \emph{starlike} if $G$ is univalent in $\C^+$ and $c G(\C^+) \subset G(\C^+)$ for any $c \in (0,1)$.
 We denote by $\Star(\R)$ the set of probability measures on $\R$ that have starlike Cauchy transforms.
\end{definition}
 Note that $\mu\in\Star(\R) $ if and only if $F_\mu$ is univalent and $cF_\mu(\Ha)\subset F_\mu(\Ha)$ for all $c\in(1,\infty),$ i.e. $F_\mu$ is starlike w.r.t. $\infty.$
\begin{theorem}\label{msd2}
 $\SD(\rhd)=\Star(\R)$.
\end{theorem}
\begin{proof}
$\SD(\rhd)\supset\Star(\R)$.  Suppose that $G_\mu$ is starlike. It is by definition univalent. It also satisfies that
$G_{\mu}(\C^+) \subset G_{\DD_c\mu}(\C^+)$ for every $c\in(0,1)$. Then we may define the analytic univalent map $F_c:=G_{\DD_c\mu}^{-1}\circ G_\mu\colon \C^+\to\C^+$. Since $G_\mu(iy) =\frac{1}{iy}(1+o(1))$ as $y\to\infty$, we have the asymptotic form $F_c(iy)=iy(1+o(1))$ as $y\to\infty$. Lemma \ref{Julia} implies the existence of some $\bmu^c \in \cP(\R)$ such that $F_c=F_{\bmu^c}$, and hence $\mu=(\DD_c\mu)\rhd \bmu^c$, as desired.

$\SD(\rhd)\subset\Star(\R)$. Take $\mu \in \SD(\rhd)$. It is easy to see from \eqref{msd0} that $c G_\mu(\C^+)\subset G_\mu(\C^+)$ for all $c\in(0,1)$. It then remains to prove that $\mu\in \Univ(\R)$.  The relation \eqref{msd0} implies that $\DD_t\mu = (\DD_s\mu) \rhd (\DD_t \bmu^{s/t} )$ for all $0<s\leq t \leq 1$. Let $F_t(z):=F_{\DD_t\mu}(z)$ and $f_{s,t}(z):= F_{\DD_t \bmu^{s/t}}(z)$, which are continuous with respect to $s$ and $t$ for every fixed $z$. Take $0<s\leq t \leq 1$. Then we obtain
\begin{equation*}
F_t  = F_s  \circ f_{s,t}.
\end{equation*}
Therefore, $(F_t)_{0\leq  t \leq 1}$ is (a part of) an \index{Loewner chain!additive}additive Loewner chain with \index{transition mappings}transition mappings $(f_{st})$, and hence each map $f_{st}$ is univalent due to
Theorem \ref{EV_univalence}. Therefore, by taking $t=1$, we conclude that $\bmu^s \in \Univ(\R)$ for $0<s<1$ and hence the weak limit $\mu =\lim_{s\downarrow 0}\bmu^s$ belongs to $\Univ(\R)$ too.
\end{proof}

\subsection{Monotone selfdecomposability and Markov transform}\label{label_markov_transform}

\index{Markov transform}

Next we obtain a certain integral representation for probability measures in $\SD(\rhd)=\Star(\R)$. Actually it is related to the Markov transform that is known to be a useful tool in asymptotic representation theory \cite{K98}.

A \index{Rayleigh measure}\emph{Rayleigh measure} on $\R$ is a finite signed Borel measure $\nu$ that satisfies
\begin{align}
& 0\leq \nu((-\infty,x]) \leq 1, \qquad x\in \R, \\
& \nu(\R) =1, \\
& \int_{\R} \log(1+|x|)\,|\nu|({\rm d}x) <\infty. \label{Ray3}
\end{align}
Integration by part shows that \eqref{Ray3} may be written in the equivalent form
\begin{equation}\label{Ray4}
\int_{0}^\infty \frac{1-\DF_\nu(x)}{1+x} \,{\rm d}x <\infty, \qquad \int_{-\infty}^{0} \frac{\DF_\nu(x)}{1+|x|} \,{\rm d}x <\infty,
\end{equation}
where $\DF_\nu(x)=\nu((-\infty,x])$ is the distribution function.

Let $\cR(\R)$ be the set of Rayleigh measures on $\R$. The \emph{Markov transform} \cite{K98,FF16} is a bijection $\KM\colon \cR(\R) \to \cP(\R)$ defined by
\begin{equation}\label{KM}
G_{\KM(\nu)}(z) = \exp\left(-\int_\R \log(z-x)\, \nu({\rm d}x)\right), \qquad z\in\C^+,
\end{equation}
where $\log$ is the principal value. It satisfies the master equation
 \begin{equation*}
 \frac{{\rm d}}{{\rm d} z} G_{\KM(\nu)}(z) = -G_\nu(z) G_{\KM(\nu)}(z).
\end{equation*}
Then we have the following.

\begin{theorem}\label{starlike1} Let $\mu$ be a probability measure on $\R$. The following are equivalent.
\begin{enumerate}[\rm(1)]
\item\label{star1} $\mu\in\SD(\rhd)$.
\item\label{star2} $\Im\left(\frac{G_\mu'(z)}{G_\mu(z)}\right) \geq 0$ for all $z\in\C^+$.
\item\label{star3} There exists a probability measure $\nu$ on $\R$ satisfying the integrability condition
$$
\int_{\R} \log(1+|x|)\,\nu({\rm d}x)<\infty
$$
such that $\mu = \KM(\nu)$. Moreover, if $\mu$ has a finite variance, then so does $\nu$, and in this case the mean and the variance of $\nu$ are given by $m_1(\nu)=m_1(\mu)$ and $\sigma^2(\nu)=2\sigma^2(\mu)$.
\end{enumerate}
\end{theorem}
\begin{remark} (i) The question of characterizing the range of the Markov transform of probability measures was raised in \cite[Section 7]{FF16}. 

(ii) The equivalence between (1) and (2) is known in Lecko \cite{L01} and Lecko and Lyzzaik \cite{LL03} in a more general setting by using Julia's lemma.
We give another proof similar to that in \cite[Theorem 2.5]{P75}.
\end{remark}

\begin{proof}
\eqref{star1} $\Rightarrow $ \eqref{star2}: Take $\mu \in \SD(\rhd)=\Star(\R)$. We follow several steps below.

(a)\,\,\, We first show that $G_\mu(\cdot+i\sigma) \colon\C^+\to\C^-$ is also starlike for any $\sigma>0$.
We fix $c\in(0,1)$ and $\sigma>0$ for the moment. From the argumentation in the proof of Theorem \ref{msd2} and Lemma \ref{Julia}, there exists a univalent function $F_c\colon \C^+\to\C^+$ satisfying $\Im (F_c(z))\geq \Im (z)$ for any $z\in\C^+$. This implies that $F_c(\C_\sigma) \subset \C_\sigma$ for any $\sigma>0$ (recall that $\C_\sigma$ stands for the set of $z \in \C$ such that $\Im(z)>\sigma$). Therefore, we obtain the inclusion
$$
cG_\mu(\C_\sigma) =c G_{\DD_c\mu}(F_c(\C_\sigma)) =G_\mu(c^{-1}F_c(\C_\sigma)) \subset G_\mu(c^{-1}\C_\sigma) \subset G_\mu(\C_\sigma).
$$
This implies that $G_\mu(\cdot+i\sigma) \colon\C^+\to\C^-$ is starlike.

(b)\,\,\, Since  $\gamma_\sigma=\{G_\mu(t+i\sigma)\}_{t\in\R}$ is a Jordan curve and it is the boundary of $G_\mu(\C_\sigma)$,
{}from a geometric observation of starlikeness, for each $\sigma>0$ we have either $\frac{{\rm d}}{{\rm d}t}\arg G_\mu(t+i\sigma) \geq 0$ for all $t\in\R$ or $\frac{{\rm d}}{{\rm d}t}\arg G_\mu(t+i\sigma) \leq 0$ for all $t\in\R$.

(c{})\,\,\, Next we will show that for any $\sigma>0$, $\frac{{\rm d}}{{\rm d}t}\arg G_\mu(t+i\sigma)$ is not identically 0. If it is identically 0 for some $\sigma>0$, then $G_\mu(t+i\sigma)$ lies on a line passing through 0. Lemma \ref{lemmainfinity} implies that $|G_\mu(t+i\sigma)|\to0$ as $t\to\pm\infty$, which contradicts the univalence of $G_\mu$.

(d{})\,\,\, We show that either $\frac{{\rm d}}{{\rm d}t}\arg G_\mu(t+i\sigma) \geq 0$ for any $t\in\R, \sigma>0$ or $\frac{{\rm d}}{{\rm d}t}\arg G_\mu(t+i\sigma) \leq 0$ for any $t\in\R, \sigma>0$. If there exist $\sigma_1, \sigma_2>0$ such that $\frac{{\rm d}}{{\rm d}t}\arg G_\mu(t+i\sigma_1) \geq 0$ and
$\frac{{\rm d}}{{\rm d}t}\arg G_\mu(t+i\sigma_2) \leq 0$ for any $t\in\R$, then we may assume without loss of generality that $\sigma_1<\sigma_2$. Then the supremum
$$
\sigma_3:=\sup\left\{\sigma \in[\sigma_1,\sigma_2]~ \Bigg|~ \frac{{\rm d}}{{\rm d}t}\arg G_\mu(t+i\sigma) \geq 0 \text{~for all $t\in\R$}\right\}
$$
exists in $[\sigma_1,\sigma_2]$. By continuity we must have $\frac{{\rm d}}{{\rm d}t}\arg G_\mu(t+i\sigma_3) = 0$ for any $t\in\R$, which contradicts property (c{}).

(e)\,\,\, Note that $\frac{{\rm d}}{{\rm d}t}\arg G_\mu(t+i\sigma)= \Im(\frac{{\rm d}}{{\rm d}z}\log G_\mu(z))|_{z=t+i\sigma}=\Im\left(\frac{G_\mu'(z)}{G_\mu(z)}\right)\Big|_{z=t+i\sigma}$.
Since $\frac{G_\mu'(z)}{G_\mu(z)} =-\frac{1}{z}(1+o(1))$ as $z \to \infty$ non tangentially, we have $\Im\left(\frac{G_\mu'(z)}{G_\mu(z)}\right) \geq 0$ for all $z\in\C^+$.

\eqref{star2} $\Rightarrow$ \eqref{star3}:
Since $\frac{G_\mu'(z)}{G_\mu(z)}=-\frac{1}{z}(1+o(1))$ as $z\to\infty$ non-tangentially, there exists a probability measure $\nu$ such that $\frac{G_\mu'(z)}{G_\mu(z)}=-G_\nu(z)$.
Taking integration leads to
$$
G_\mu(z)=\exp\left(-\int^z G_\nu(w)\,{\rm d}w+c \right),
$$
where $c \in \C$ is a constant and $\int^z G_\nu(w)\,{\rm d}w$ is an indefinite integral whose derivative is $G_\mu(z)$.
By integration by part,
\[
G_\nu(z) = \int_{\R} \frac{1}{z-x} \,\mu({\rm d}x)=-\int_{\R} \frac{1}{(z-x)^2} \DF_\nu(x)\,{\rm d}x,
\]
where $\DF_\nu(x)=\nu((-\infty,x])$ is the distribution function of $\nu$.
Therefore, up to constants we obtain
\begin{equation}\label{Int}
\int^z G_\nu(w)\,{\rm d}w= \int_{\R} \left(\frac{1}{z-x}+\frac{x}{1+x^2}\right)\DF_\nu(x)\,{\rm d}x.
\end{equation}

We will show that $\int_{-\infty}^0\frac{|x|}{1+x^2}\DF_\nu(x)\,{\rm d}x<\infty$ and $\int_{0}^\infty \frac{x}{1+x^2}(1-\DF_\nu(x))\, {\rm d}x<\infty$ which are equivalent to the desired integrability condition.
Combining \eqref{Int}  with the identity
$$
0 =-\log z + i\pi +\int_{0}^\infty \left(\frac{1}{z-x}+\frac{x}{1+x^2}\right){\rm d}x
$$
gives
\[
\begin{split}
-\int^z G_\nu(w)\,{\rm d}w+c
&= -\log z+ c +i\pi +\int_{0}^\infty \left(\frac{1}{z-x}+\frac{x}{1+x^2}\right)(1-\DF_\nu(x))\,{\rm d}x \\
&\qquad\qquad\qquad\qquad- \int_{-\infty}^0 \left(\frac{1}{z-x}+\frac{x}{1+x^2}\right)\DF_\nu(x)\,{\rm d}x.
\end{split}
\]
In order that  $\exp(-\int^z G_\nu(w)\,{\rm d}w+c)$ defines a Cauchy transform of a probability measure,
the function
$$
f(z)=\int_{0}^\infty \left(\frac{1}{z-x}+\frac{x}{1+x^2}\right)(1-\DF_\nu(x))\,{\rm d}x - \int_{-\infty}^0 \left(\frac{1}{z-x}+\frac{x}{1+x^2}\right)\DF_\nu(x)\,{\rm d}x
$$
must be bounded as $z \to\infty$ non-tangentially. After some computation, we have
$$
\Re[f(iy)]=\int_{0}^\infty \frac{x(y^2-1)}{(x^2+y^2)(1+x^2)}(1-\DF_\nu(x) + \DF_\nu(-x))\,{\rm d}x,
$$
which converges by monotone convergence to $
\int_{0}^\infty \frac{x}{1+x^2}(1-\DF_\nu(x) + \DF_\nu(-x))\,{\rm d}x
$
as $y\to \infty.$
This must be bounded and hence the integrability condition on $\nu$ follows.

\eqref{star3} $\Rightarrow$ \eqref{star1}. Lemma \ref{NW} shows that $\log G_\mu(z)$ is univalent.
Assume $z=x+i\sigma \in G_\mu(\C^+)$. Since $\frac{{\rm d}}{{\rm d}t}\arg G_\mu(t+i\sigma) >0$ from \eqref{star3}, the curve $\gamma_\sigma:=(G_\mu(t+i\sigma))_{t\in\R}$ does not contain the point $cz$ for any $c\in(0,1)$. By Lemma \ref{lemmainfinity},  the curve $\gamma_\sigma \cup\{0\}$ is Jordan closed curve and it
surrounds $cz$, and hence $cz \in G_\mu(\C_\sigma)$.

Suppose that the probability measure $\mu$ has a finite variance.
We can show that
$$
G_\nu(z)=-\frac{G_\mu'(z)}{G_\mu(z)} = \frac{1}{z}\left(1+\frac{m_1(\mu)}{z}+\frac{2m_2(\mu)-m_1(\mu)^2}{z^2}+o(z^{-2}) \right).
$$
Hence $m_1(\nu)=m_1(\mu)$ and $m_2(\nu)=2m_2(\mu)-m_1(\mu)^2$ and hence $\sigma^2(\nu)=2\sigma^2(\mu)$.
\end{proof}

The above characterization enables us to prove the absence of atoms.

\begin{proposition}\label{sd_no_atom}
If $\mu$ is $\rhd$-selfdecomposable and is not a delta measure, then $\mu$ does not have an atom.
\end{proposition}
\begin{remark}
In classical and free probabilities it is known that any $\ast$- or $\boxplus$-selfdecom\-posable distribution is Lebesgue absolutely continuous; see \cite[Theorem 27.13]{Sat13} for the classical case, and see \cite[Theorem 3.4]{HS} for the free case. Our result is a partial analogy with these results. Note that in \cite[Theorem 2.4.5]{K98} it is stated that a measure in $\SD(\rhd)$ with compact support is Lebesgue absolutely continuous, but we could not find a reference with a proof. Thus we pose the following problem.
\end{remark}
 \begin{problem}
Is there a $\rhd$-selfdecomposable distribution that has a non-zero singular continuous part?
\end{problem}

\begin{proof}[Proof of Proposition \ref{sd_no_atom}] By Theorem \ref{starlike1} there exists a probability measure $\nu$ such that $\int_{\R} \log (1+|x|) \nu({\rm d}x)<\infty$ and $\mu=\KM(\nu)$. Note that $\mu= \delta_b$ if and only if $\nu=\delta_b$. Therefore, our assumption says that $\nu$ is not a Dirac delta measure.

It is well known (and follows readily from the dominated convergence theorem) that $\mu(\{a\})= \lim_{y\downarrow} i y G_\mu(a +iy)$ for all $a\in\R$.
Using \eqref{KM} then yields
\begin{equation*}
\mu(\{a\}) = \lim_{y\downarrow0}i y G_\mu(a+i y) = \lim_{y\downarrow0}\exp\left(\int_\R \log \frac{i y}{i y + a -x} \nu({\rm d}x) \right).
\end{equation*}
Now, by the monotone convergence theorem,
\begin{equation*}
\Re \int_\R \log \frac{i y}{i y + a -x} \nu({\rm d}x) = \frac{1}{2}\int_\R \log \frac{y^2}{y^2 + (a -x)^2} \nu({\rm d}x) \to -\infty \cdot \nu(\R\setminus\{a\})  =-\infty
\end{equation*}
as $y\downarrow0$. This shows that $\mu(\{a\})=0$ for any $a\in\R$.
 \end{proof}

\begin{example}\label{SD delta}
Take $\nu=\sum_{k=1}^n w_k \delta_{a_k}$, where $-\infty < a_1 < \cdots < a_n <\infty, n \geq2$ and $\sum_{k=1}^n w_k=1, w_k >0$. Then
\begin{equation*}
G_{\KM(\nu)}(z) = \prod_{k=1}^n (z-a_k)^{-w_k},
\end{equation*}
where the powers are the principal value. The Stieltjes inversion shows that the probability measure $\KM(\nu)$ is absolutely continuous with density
\begin{equation*}
p(x) =
\begin{cases} 0,& x > a_n \text{~or~} x<a_1, \\[2mm]
\displaystyle \frac{\sin \pi(w_{p+1}+\cdots + w_n)}{\pi} \prod_{k=1}^n |x-a_k|^{-w_k}, & x \in (a_p,a_{p+1}), 1 \leq p \leq n-1.
\end{cases}
\end{equation*}
The range $G_{\KM(\nu)}(\C^+)$ is of the form $\C^- \setminus \bigcup_{p=1}^{n-1} L_p$, where $L_p= \{r e^{i\theta_p}: r \geq r_p\}$ are half-lines with
\begin{align*}
r_p &= \min_{x \in (a_{p}, a_{p+1})} \prod_{k=1}^n |x-a_k|^{-w_k}, \\
\theta_p &=-\pi (w_{p+1}+ \cdots + w_n).
\end{align*}
In particular, when $\nu = \frac{1}{2}(\delta_{-1}+\delta_1)$, then $\KM(\nu)$ is the symmetric \index{arcsine distribution}arcsine law on $[-1,1]$ and the range $G_{\KM(\nu)}(\C^+)$ is given by $\C^- \setminus i [-1,\infty)$.
Moreover, take $a \neq 0$. Since $F_{\delta_a\rhd \KM(\nu)}(z) = F_{\KM(\nu)}(z)-a$, we have $F_{\delta_a\rhd \KM(\nu)}(\C^+) = -a+ \C^+ \setminus i (0,1]$.\\
 This easily shows that $r F_{\delta_a\rhd \KM(\nu)}(\C^+) \not \subset F_{\delta_a\rhd \KM(\nu)}(\C^+)$ for every $r>1$ and hence $\delta_a \rhd \KM(\nu)$ is not in $\SD(\rhd)$.
\end{example}

\begin{example} The \index{Boolean stable distribution}Boolean stable distribution $\bfb_{\alpha,\rho,t}$, with $\alpha\in(0,2], \rho\in[0,1]\cap [1-1/\alpha, 1/\alpha], t >0 $ was introduced in \cite{SW97} and is characterized by
\begin{equation}\label{BST}
G_{\bfb_{\alpha,\rho,t}}(z) = \frac{1}{z + t e^{i \alpha \rho \pi}z^{1-\alpha}}, \qquad  z\in \C^+.
\end{equation}
It is $\rhd$-selfdecomposable if and only if $\alpha \in (0,1]$ since it satisfies
\begin{equation*}
\bfb_{\alpha,\rho,t} = \KM(((1-\alpha)\delta_0+\alpha\delta_1) \circledast \bfb_{\alpha,\rho,t}),
\end{equation*}
which was computed in \cite[Example 5.8 (2)]{AH16b} only for $\rho=t=1$, but extends to all parameters.
Various properties of Boolean stable distributions are investigated in \cite{AH14,HS15,AH16a}.
\end{example}

\begin{example} The \index{monotonically stable distribution}monotonically stable distribution $\bfm_{\alpha,\rho,t}$ is $\rhd$-selfdecomposable since it satisfies $\bfm_{\alpha,\rho,t}= (\DD_c \bfm_{\alpha,\rho,t}) \rhd \bfm_{\alpha,\rho,(1-c^\alpha)t}$. Arizmendi and Hasebe \cite[Example 5.8 (1)]{AH16b} proved that
\begin{equation*}
\bfm_{\alpha,\rho,t} = \KM(\bfb_{\alpha,\rho,t})
\end{equation*}
for $\rho=t=1$, where $\bfb_{\alpha,\rho,t}$ is the \index{Boolean stable distribution}Boolean stable distribution \eqref{BST}. This formula extends to all parameters.

\end{example}

\begin{example} The \index{free stable distribution}free stable distribution is $\boxplus$-selfdecomposable and hence it is also \\ $\rhd$-selfdecomposable (see Section \ref{sec FSD}). Arizmendi and Hasebe \cite[Example 5.8 (3)]{AH16b} proved that
\begin{equation*}
\bff_{\alpha,\rho,t} =\KM(\bff_{\alpha,\rho,t}^{\uplus \alpha})
\end{equation*}
 for $\rho=t=1$, where $\uplus$ is Boolean convolution.  This result may be generalized to arbitrary parameters.
\end{example}

We give counterexamples to some inclusions.

\begin{example}
\begin{enumerate}[\rm(1)]
\item $\SD(\rhd) \not\subset \ID(\rhd)$ since the \index{semicircle distribution}semicircle distribution is in $\SD(\boxplus)\subset \SD(\rhd)$ but is not $\rhd$-infinitely divisible (see Example \ref{semicircle not ID}).

\item $\SD(\rhd) \not \subset \UM(\R)$ since the centered \index{arcsine distribution}arcsine law is not unimodal but it is\\ $\rhd$-selfdecomposable.

\item $\UM(\R)\not\subset \SD(\rhd)$ since there is a bounded simply connected domain that is horizontally convex but not starlike. Note that such a domain can be realized as the range of $G_\mu$ for some probability measure $\mu \in \Univ(\R)$ with compact support  by Theorem \ref{ftransform_images} (b).
\end{enumerate}
\end{example}

\index{selfdecomposable distribution!additive monotone convolution|)}



\chapter{Limit Theorems for Multiplicative Monotone Convolution}\label{sectionunitcircle}

This section is the multiplicative analogue of Section \ref{section_limits_additive}, i.e. we study the convolution
$\submm$ and probability measures on $\T$ with univalent $\eta$-transform.\\

Non-commutative probability theory provides us with four further multiplicative convolutions of probability measures on $\T$, see Section \ref{mono_convolutions_mult}. Again, we will also encounter the Boolean, classical, and free convolution on $\cP(\T).$

Recall that the normalized Haar measure $\haar$ on the unit circle plays a special role for all those convolutions.

\section{Preliminaries}\label{sec Mconv}

We denote by $\Univ(\tor)$ the set of all probability measures $\mu$ on $\tor$ which have univalent $\psi_\mu$. Note that
$\psi_\mu$ is univalent if and only if $\eta_\mu$ is univalent.
The Hurwitz theorem and Lemma \ref{lemmaconvergenceunit} then show that $\Univ(\tor)\cup\{\haar\}$ is a weakly closed subset of the set $\cP(\tor)$ of all probability measures on $\tor$.

A useful and interesting example is the Poisson kernel. By Lemma \ref{characterizationeta}, for every $c\in\disk\setminus\{0\}$ there exists a probability measure $\pk_c$ on $\tor$ such that $\eta_{\pk_c}(z)=c z$, and this measure turns out to be the Poisson kernel:
\begin{equation*}
\pk_c({\rm d}\xi) = \frac{1-|c|^2}{|1-\overline{c}\xi|^2} \haar({\rm d}\xi) =\frac{1-|c|^2}{1-2|c| \cos(\theta- \arg c) + |c|^2 }\, \frac{{\rm d} \theta}{2\pi},
\end{equation*}
where $\xi= e^{i\theta}$ and $\arg c $ is of any branch. Then we have $\Sigma_{\pk_c}(z) = 1/c$, and hence $\eta_{\mu\boxtimes\pk_c}= \eta(c z)$ which implies that $\mu\boxtimes \pk_c = \mu\submm\pk_c$. Moreover, $m_n(\pk_c) = c^n$ for $n\in\N$ and hence $m_n(\mu\circledast \pk_c)=c^n m_n(\mu)$. This implies that $\psi_{\mu\circledast \pk_c}(z) = \psi_\mu(c z)$ and hence $\eta_{\mu\circledast \pk_c}(z)=\eta_\mu(cz)$. Thus we obtain
\begin{equation*}
\mu\boxtimes \pk_c = \mu\submm\pk_c = \mu \circledast \pk_c
\end{equation*}
and
\begin{equation}\label{pk2}
\pk_c^{\boxtimes n}  = \pk_c^{\mmm n} = \pk_c^{\circledast n} = \pk_{c^n}.
\end{equation}
Note that $\pk_c$ converges weakly to $\delta_\zeta$ as $c$ tends to $\zeta \in \tor$, and $\pk_c$ converges to $\haar$ as $c\to0$. These results can be verified via Lemma \ref{lemmaconvergenceunit}. Thus we may extend the parameter of $\pk_c$ to $c\in\overline{\disk}$ by weak continuity.

\section[Khintchine's limit theorem]{Khintchine's limit theorem and univalent moment generating functions}\label{subsec_khin_mult}

We prove that the set of possible limits of monotone convolution of infinitesimal arrays on $\tor$ is exactly the set of all $\mu$ for which $\psi_\mu$ is univalent together with the Haar measure $\haar$. For additive monotone convolution we have met a technical difficulty and we needed assumptions of finite variance, but for multiplicative convolution on $\tor$ we can show the complete result.
\begin{definition}
A family of probability measures  $\{\mu_{n,j} \mid 1\leq j \leq k_n, n\geq1\}$ on $\tor$ is called an \index{infinitesimal array!of measures on $\tor$}\emph{infinitesimal (triangular) array} if $k_n\uparrow\infty$ and for any $\delta\in(0,\pi)$,
$$
\lim_{n\to \infty}\sup_{1 \leq j \leq k_n} \mu_{n,j}([-\delta,\delta]^c)=0,
$$
where the arc $\{e^{i \theta}: \theta \in [-\delta,\delta]\}$ is identified with the interval $[-\delta,\delta]$, and $[-\delta,\delta]^c$ denotes the complement of the arc $[-\delta,\delta]$ in $\tor$.
\end{definition}

Take an associative binary operation $\star$ on $\cP(\tor)$. A probability measure $\mu$ on $\tor$ is called the \emph{$\star$-limit of an infinitesimal array} if there exists an infinitesimal array $\{\mu_{n,j}\}$ such that
\begin{equation*}
\mu_{n,1} \star \cdots \star \mu_{n,k_n} \wto \mu  \quad \text{as~}n\to\infty.
\end{equation*}
The set of all $\star$-limits of infinitesimal arrays is denoted by $\IA(\star,\tor)$ or $\IA(\star)$. Following the arguments of Proposition \ref{closednessIA}, we can show that $\IA(\star,\tor)$ is (operationally) closed under $\star$ and (topologically) closed with respect to weak convergence.

A probability measure $\mu$ on $\tor$ is said to be \index{infinite divisibility}\emph{$\star$-infinitely divisible} if for every $n\in\N$ there exists $\mu_n \in \cP(\tor)$ such that $\mu=\mu_n^{\star n}$ ($n$ fold convolution). The set of $\star$-infinitely divisible distributions is denoted by $\ID(\star,\tor)$ or simply $\ID(\star)$. Note that the iterative use of \eqref{Haar} shows that
\begin{equation*}
\haar =\haar^{\star n},\qquad n\in\N, \star \in\{\mmm,\boxtimes,\circledast\}
\end{equation*}
and hence the Haar measure is $\star$-infinitely divisible for the three convolutions.

The Khintchine type result is also true on $\tor$.
\begin{proposition}\index{infinite divisibility!multiplicative classical convolution}
$\IA(\circledast,\tor) = \ID(\circledast,\tor)$.
\end{proposition}
\begin{proof}
The inclusion $\IA(\circledast,\tor) \subset\ID(\circledast,\tor)$ is known from \cite[Theorem 5.2]{Par67}. The converse inclusion can also be proved by using several results from \cite{Par67}. Namely,
every $\mu \in \ID(\circledast,\tor)$ has the form $\mu= \lambda \circledast\nu$, where $\lambda$ is the normalized Haar measure on a compact subgroup of $\tor$, namely on $\tor$ or on $\Z_p:=\{e^{2\pi i k/p}:0 \leq k \leq p-1\}$ for some $p\geq2$, and $\nu$ is an infinitely divisible distribution without an idempotent factor \cite[p106, Theorem 7.2]{Par67}. Since $\IA(\circledast,\tor)$ is closed under convolution, it suffices to show that $\lambda$ and $\nu$ both are limits of infinitesimal arrays.
We start from $\nu$. It has a \index{L\'evy-Khintchine representation!classical}L\'evy-Khintchine representation \cite[p103, Theorem 7.1]{Par67}, with which we can naturally define its convolution roots $\nu_n$ such that $\nu_n^{\circledast n}=\nu$ and $\nu_n \wto \delta_1$. Thus we conclude that $\nu \in \IA(\circledast)$.
For $\lambda$, we consider two cases. If $\lambda$ is the normalized Haar measure $\haar$ on $\tor$ (then in fact $\mu=\haar$), then we can see from \eqref{pk2} and Lemma \ref{lemmaconvergenceunit} that
\begin{equation}\label{pk conv0}
(\pk_{1-1/n})^{\circledast n^2}=\pk_{(1-1/n)^{n^2}} \wto \haar, \qquad \text{as~}n\to\infty,
\end{equation}
and so $\lambda \in \IA(\circledast)$. If $\lambda$ is the normalized Haar measure on a finite group $\Z_p$ for some $p\geq1$, then for $\lambda_n:=(1-1/n)\delta_1 +(1/n)\delta_{e^{2\pi i/p}}$ we can prove that
\begin{equation*}
\lim_{n\to\infty}m_k(\lambda_n^{\circledast n^2}) = m_k(\lambda)=
\begin{cases}
0, & k \notin p\Z, \\
1,& k\in p\Z,
\end{cases}
\end{equation*}
and hence by Lemma \ref{lemmaconvergenceunit} \eqref{moment_converge} $\lambda_n^{\circledast n^2}\wto \lambda$, and thus $\lambda \in \IA(\circledast)$. Altogether, we conclude the inclusion $ \ID(\circledast,\tor) \subset \IA(\circledast,\tor)$.
\end{proof}

The free analogue holds as well.
\begin{proposition}\index{infinite divisibility!multiplicative free convolution}
$ \IA(\boxtimes,\tor) = \ID(\boxtimes,\tor).$
\end{proposition}
\begin{proof}
Belinschi and Bercovici \cite{BB08} proved the inclusion $\subset$. For the converse inclusion, recall that only the Haar measure is $\boxtimes$-infinitely divisible with zero mean \cite[Lemma 6.1]{BV92}. If $\mu \in \ID(\boxtimes,\tor)$ is not the Haar measure, then there exists a function $u$ as  described in Theorem \ref{MFID} below. Defining a probability measure $\mu_n$ having the function $(1/n)u$, we obtain $\mu =\mu_n^{\boxtimes n}$ for every $n\in\N$, and also $\mu_n \wto \delta_1$ using \cite[Proposition 2.9]{BV92}, and hence $\mu\in\IA(\boxtimes)$. If $\mu=\haar$, then we can see from \eqref{pk2} and Lemma \ref{lemmaconvergenceunit} that
\begin{equation}\label{pk conv}
(\pk_{1-1/n})^{\boxtimes n^2}=\pk_{(1-1/n)^{n^2}} \wto \haar, \qquad \text{as~}n\to\infty,
\end{equation}
and hence $\haar \in \IA(\boxtimes)$.
\end{proof}

The goal of this section is to demonstrate that $\IA(\submm,\tor)=\Univ(\tor) \cup \{\haar\}$. For this we need some estimates.
\begin{lemma}\label{lemmaesitmateunit} Let $\mu$ be a probability measure on $\tor$, let $\delta\in(0,\pi/2)$ and $n \in\N$. Then we have
\begin{align}
\mu([-\delta,\delta]^c) &\leq  \frac{\pi^2}{2\delta^2}|1-m_1(\mu)|, \label{estimate11} \\
|1-m_1(\mu)|^2 &\leq 2\delta^2 + 10 (\mu([-\delta,\delta]^c))^2,\label{estimate12} \\
|1-m_n(\mu)| &\leq \frac{\pi^2 n}{2}|1-m_1(\mu)|. \label{estimate13}
\end{align}
\end{lemma}
\begin{proof} We start from the obvious inequality
\begin{equation}\label{interestimate11}
\begin{split}
\mu([-\delta,\delta]^c)
&\leq\int_{[-\delta,\delta]^c} \frac{1-\cos x}{1-\cos\delta}\,\mu({\rm d}x) \leq  \frac{1}{1-\cos \delta}\int_{(-\pi,\pi]}(1-\cos x)\,\mu({\rm d}x).
\end{split}
\end{equation}
Elementary calculus shows that $\sin x \geq 2x/\pi $ for $0\leq x\leq\pi/2$ and hence $1-\cos\delta=2\sin^2(\delta/2) \geq 2\delta^2/\pi^2$.
Thus \eqref{interestimate11} implies
\begin{equation}\label{interestimate12}
\begin{split}
\mu([-\delta,\delta]^c)
&\leq \frac{\pi^2}{2\delta^2} \left|\int_{(-\pi,\pi]}(1-e^{ix})\,\mu({\rm d}x)\right| = \frac{\pi^2|1-m_1(\mu)|}{2\delta^2},
\end{split}
\end{equation}
which is the first desired inequality.
The second inequality is verified as
\begin{equation}\label{interestimate13}
\begin{split}
|1-m_1(\mu)|^2
&= \left( \int_{|x| \leq\delta}(1-\cos x)\,\mu({\rm d}x)+\int_{[-\delta,\delta]^c}(1-\cos x)\,\mu({\rm d}x) \right)^2\\
&\quad +\left( \int_{|x| \leq\delta}\sin x\,\mu({\rm d}x)+\int_{[-\delta,\delta]^c}\sin x\,\mu({\rm d}x) \right)^2\\
&\leq (1-\cos \delta +2\mu([-\delta,\delta]^c))^2 +(\sin \delta +\mu([-\delta,\delta]^c))^2\\
&\leq 2(1-\cos \delta)^2 + 2 (2\mu([-\delta,\delta]^c))^2 + 2 \sin^2 \delta + 2(\mu([-\delta,\delta]^c))^2\\
&\leq 2\delta^2+ 10  (\mu([-\delta,\delta]^c))^2.
\end{split}
\end{equation}
For the third inequality, note that the following elementary inequalities hold:
\begin{align}
 1-\cos x &\geq \frac{2}{\pi^2}x^2,\qquad x\in[-\pi,\pi], \label{supplement1}\\
|1-e^{ix}| &\leq |x|,\qquad x\in\R. \label{supplement2}
\end{align}
On one hand \eqref{supplement1} shows that
\begin{equation}\label{interestimate14}
\begin{split}
|1-m_1(\mu)|^2
&\geq \left( \int_{-\pi}^\pi(1-\cos x)\,\mu({\rm d}x) \right)^2 \geq \frac{4}{\pi^4}\int_{-\pi}^{\pi} x^2\,\mu({\rm d}x).
\end{split}
\end{equation}
On the other hand, using \eqref{supplement2} and the Schwarz inequality shows that
\begin{equation}\label{interestimate15}
\begin{split}
|1-m_n(\mu)|
&\leq \int_{-\pi}^\pi |1-e^{inx}|\,\mu({\rm d}x) \leq  n\int_{-\pi}^\pi |x|\,\mu({\rm d}x)\leq n\left(\int_{-\pi}^\pi x^2\,\mu({\rm d}x)\right)^{1/2}.
\end{split}
\end{equation}
The third inequality follows from (\ref{interestimate14}) and (\ref{interestimate15}).
\end{proof}

Now we show the Khintchine type theorem.

\begin{theorem}\label{Khintchineunit}\index{infinitesimal array!of measures on $\tor$}
\index{moment generating function!univalent}
$\IA(\submm)=\Univ(\tor)\cup \{\haar\} $.
\end{theorem}
\begin{proof}
We can show that $\haar \in \IA(\submm)$ from exactly the same arguments in \eqref{pk conv} thanks to \eqref{pk2}. Similarly, delta measures are contained in $\IA(\submm)$ since $\delta_{e^{i\beta}}=\delta_{e^{i\beta/n}}\mm \cdots \mm\delta_{e^{i\beta/n}}$ ($n$ fold). 

Suppose that $\psi_\mu$ is univalent and $\mu$ is not a delta measure. This implies that $0<|m_1(\mu)| <1$ and so there exist $\alpha > 0, \beta\in [0,2\pi)$ such that $\eta_\mu'(0)=m_1(\mu)=e^{-\alpha+i\beta}$. 
By Theorem \ref{embed_multiplicative} (b), we find a \index{Loewner chain!multiplicative}multiplicative Loewner chain $\{\eta_t\}_{t\geq0}$ such that $\eta_1=\eta_{\delta_{e^{-i\beta}} \submm\mu}$ and $\eta_t'(0)=e^{-\alpha t}$. Let $\eta_{st}:= \eta_s^{-1}\circ\eta_t\colon \disk\to\disk$ be the corresponding \index{transition mappings}transition mappings. Lemma \ref{characterizationeta} shows that there exists a family of probability measures $(\mu_{st})_{0\leq s\leq t \leq \infty}$ on $\tor$ such that $\eta_{st} = \eta_{\mu_{st}}$. We have the following properties:
\begin{enumerate}[(a)]
\item\label{embedding11} $\mu_{0,1}=\delta_{e^{-i\beta}} \mm\mu$;
\item $\mu_{ss}=\delta_1$ for any $s\in [0,\infty)$;
\item\label{variancepropertyunit} $m_1(\mu_{st})=e^{-\alpha(t-s)}$;
\item \label{hemigrouppropertyunit} $\mu_{st} \mm \mu_{tu}=\mu_{su}$ for $0\leq s\leq t\leq u <\infty$.
\end{enumerate}
For $\delta\in(0,\pi/2)$ and any $0\leq s\leq t$, using (\ref{estimate11}) and the inequality $1-e^{-x} \leq x$ for $x\geq0$ shows that
\begin{equation}\label{interestimate16}
\begin{split}
\mu_{st}([-\delta,\delta]^c)
&\leq \frac{\pi^2}{2\delta^2}|1-m_1(\mu_{st})|\leq \frac{\pi^2\alpha(t-s)}{2\delta^2}.
\end{split}
\end{equation}

Let us put $\nu_{n,j}:=\mu_{\frac{j-1}{n}, \frac{j}{n}}$ ($1 \leq j \leq n$). Properties (\ref{embedding11}) and (\ref{hemigrouppropertyunit})  imply that
\begin{equation*}
\mu=\delta_{e^{i\beta}} \mm\nu_{n,1} \mm \cdots\mm\nu_{n,n}.
\end{equation*}
It then follows from \eqref{interestimate16} that
\begin{equation*}
\sup_{1 \leq j \leq n}\nu_{n,j}([-\delta, \delta]^c) \leq \frac{\pi^2\alpha}{2\delta^2 n},
\end{equation*}
and so $\{\nu_{n,j}\}$ is an infinitesimal array. Thus $\mu$ is the $\submm$-limit of an infinitesimal array.

Conversely, suppose that $\mu$ is the $\submm$-limit of an infinitesimal array $\{\mu_{n,j}: 1 \leq j \leq k_n, 1\leq n\}$ and $\mu \neq \haar$.
 The estimate \eqref{estimate12} entails 
\begin{equation}\label{interestimate17.5}
\sup_{1\leq j \leq k_n}|1-m_1(\mu_{n,j})| \to 0, \quad n\to \infty. 
\end{equation}
 Note that $\psi_{\delta_1}(z)=z/(1-z)$ and $\eta_{\delta_1}(z)=z$. Let $\nu$ be a probability measure on $\tor$.
Using \eqref{estimate13} then shows, for all $z\in\disk$,
\begin{equation}\label{interestimate17}
\begin{split}
\left|\psi_{\nu}(z) - \frac{z}{1-z}\right| &= \left|\sum_{n=1}^\infty(1-m_n(\nu))z^n\right| \\
&\leq \frac{\pi^2}{2}|1-m_1(\nu)|\sum_{n=1}^\infty n|z|^n \leq \frac{\pi^2|1-m_1(\nu)|}{2(1-|z|)^2}.
\end{split}
\end{equation}
Combining the inequalities \eqref{interestimate17.5} and \eqref{interestimate17} yields
\begin{equation*}\label{interestimate18}
\sup_{1\leq j \leq k_n, z\in r\disk}|\psi_{\mu_{n,j}}(z) - \psi_{\delta_1}(z)|\to0,\quad n\to \infty
\end{equation*}
for any $0<r<1$, and 
\begin{equation*}
\begin{split}
|\eta_{\mu_{n,j}}(z)-z|
&= \left|\frac{\psi_{\mu_{n,j}}(z)}{1+\psi_{\mu_{n,j}}(z)}-\frac{\psi_{\delta_1}(z)}{1+\psi_{\delta_1}(z)}\right|
=\frac{|\psi_{\mu_{n,j}}(z)- \psi_{\delta_1}(z)|}{|1+\psi_{\mu_{n,j}}(z)| |1+\psi_{\delta_1}(z)|} \\
&\leq 4|\psi_{\mu_{n,j}}(z)- \psi_{\delta_1}(z)|,
\end{split}
\end{equation*}
since $\Re[\psi_{\mu_{n,j}}(z)]\geq-\frac{1}{2}$. Hence, we have $\sup_{z\in r\disk, 1 \leq j \leq k_n}|\eta_{\mu_{n,j}}(z)-z| \to 0$  as $n\to \infty$.
Fix any numbers $0<r < r' <1$. Applying Cauchy's integral formula to the derivative yields that
\begin{equation*}
\frac{{\rm d}}{{\rm d}z}(\eta_{\mu_{n,j}}(z)-z) = \frac{1}{2\pi i}\int_{\partial (r'\disk)}\frac{\eta_{\mu_{n,j}}(\xi)-\xi}{(\xi-z)^2}\,{\rm d}\xi,
\end{equation*}
and hence
$\sup_{1\leq j \leq k_n,z\in r\disk}|\frac{{\rm d}}{{\rm d}z}(\eta_{\mu_{n,j}}(z)-z)| \to 0$ as $n\to\infty$.
In particular,
\begin{equation*}
\inf_{1 \leq j \leq k_n, z\in r\disk}\Re\left[ \frac{{\rm d}}{{\rm d}z}\eta_{\mu_{n,j}}(z) \right] >0
\end{equation*}
 for sufficiently large $n$. Now Lemma \ref{NW} shows that $\eta_{\mu_{n,j}}$ is univalent in $r\disk$. Since $\eta_{\mu_{n,j}}(r\disk) \subset r\disk$ (see Lemma \ref{characterizationeta}), the composition $\eta_{\mu_{n,1}}\circ \cdots \circ \eta_{\mu_{n,k_n}}$ is also univalent in $r\disk$, and hence its limit $\eta_\mu$ is univalent in $r\disk$ since it is not a constant (recall that the $\eta$-transform is constant only for the Haar measure).
Since $0<r<1$ was arbitrary, the map $\eta_\mu$ is univalent in $\disk$.
\end{proof}

Examples of probability measures in $\IA(\submm)$ are shown in Section \ref{sec:univ}.

\section[Univalent moment generating functions]{Subclasses of probability measures with univalent moment generating functions}\label{sec:univ}

We have seen that the class $\Univ(\tor)$ characterizes the monotone limit distributions for infinitesimal arrays. We now introduce several subclasses of $\Univ(\tor)$ which can be characterized by simple geometric or analytic conditions.

\subsection{Monotonically infinitely divisible distributions on $\tor$}\label{mon_inf_mult}

\index{infinite divisibility!multiplicative monotone convolution|(}

Bercovici \cite[Theorem 4.4]{B05} characterized  $\submm$-infinitely divisible distributions on the unit circle.
\
\begin{theorem} \label{UMIDLK}

\begin{enumerate}[\rm(1)]
\item\label{UM-embed} If $\mu \in \ID(\submm,\T)\setminus\{\haar\}$, then there exists a weakly continuous $\submm$-convolution semigroup $\{\mu_t\}_{t\geq0} \subset \cP(\T)$ such that $\mu_0=\delta_1$ and $\mu_1=\mu$.
\item If $\{\mu_t\}_{t\geq0} \subset\cP(\R)$ is a weakly continuous $\submm$-convolution semigroup such that $\mu_0=\delta_1$, then $\mu_1 \in \ID(\submm,\T)\setminus\{\haar\}$.
\end{enumerate}
\end{theorem}

In the statement \eqref{UM-embed}, the semigroup $\{\mu_t\}_{t\geq0}$ need not be unique. The non-uniqueness is studied in \cite[Theorem 5.8]{Has13}.

Recall that the $\eta$-transforms of a weakly continuous $\submm$-convolution semigroup satisfies a  differential equation \eqref{ODE}, which entails the univalence of each map $\eta_t$. Therefore, we obtain the following.
\begin{corollary}\label{MID univ}
$\ID(\submm,\tor) \subset \Univ(\tor) \cup \{\haar\}$. In particular, 
if $\mu\in\ID(\submm,\tor)\setminus\{\haar\}$, then 
the first moment of $\mu$ is different from $0$.
\end{corollary}
\begin{proof}Note that the first moment of 
$\mu$ is equal to $\eta_\mu'(0)$. 
\end{proof}
Let $\alpha\in[-\pi/2, \pi/2]$. 
A univalent function $f:\D\to\C$ with $f(0)=0$ is called 
$\alpha$-\emph{spirallike} if $e^{-e^{i\alpha}t}w \in f(\D)$ for every $t\geq0$ and $w\in f(\D)$. \\
In the special case $\alpha=0$, $f$ is a \emph{starlike} mapping. If $\alpha=\pm \pi/2$, 
then $f$ maps $\D$ onto a disc, which implies $f(z)=az$ for some $a\in\C\setminus\{0\}$.\\

An analytic $f:\D\to\C$ with $f(0)=0$ is univalent and $\alpha$-spirallike if and only if 
\begin{equation}\label{spirallike_ch}\Re\left[e^{-i\alpha}\frac{zf'(z)}{f(z)}\right]\geq0
\end{equation} on $\D$, see \cite[Theorem 6.6]{P75}. (Note that in 
\cite{P75}, $f$ is called spirallike of type $\alpha$ if $e^{-e^{-i\alpha}t}w \in f(\D)$.) 

\begin{corollary}\label{starlike_id} Let $\mu\in \cP(\T)$. 
Then $\mu \in \ID(\submm,\T)\setminus\{\haar\}$ if and only if there exist $\beta \in [-\pi/2,\pi/2]$, $R \geq0$, and a $\beta$-spirallike mapping 
$f:\D\to\C$ with $f(0)=0, f'(0)=1$, such that \[\eta_\mu = f^{-1}\circ (e^{-Re^{i\beta}}\cdot f).\]
If this holds true, then $e^{-Re^{i\beta}}$ is the first moment of $\mu$. 
\end{corollary}

\begin{proof}
First we consider the special case $\mu=\delta_{\alpha}$. Then $\eta_\mu(z)=\alpha z$. We can take $f(z)=z$.\\

Suppose that $\mu \in \ID(\submm,\T)\setminus\{\haar\}$ is not a delta measure (and hence $\eta_\mu$ is not a rotation). By Theorem \ref{UMIDLK}, $\eta_\mu$ can be embedded into 
a \index{Loewner chain!multiplicative}multiplicative Loewner chain $(\eta_{t})_{t\geq 0}$ at $t=1$ and the \index{transition mappings}transition mappings 
$\eta_{st}$ satisfy $\eta_{st}=\eta_{0,t-s}=\eta_{t-s}$. Furthermore we have $\eta_{0t}'(0)=e^{-at}$ for some $a\in\C$. By the Schwarz Lemma we must have $\Re(a)>0$, and hence $a=Re^{i\beta}$ for some $R>0$ and $\beta \in (-\pi/2,\pi/2)$. 

Equation \eqref{EV_ord} yields
 \begin{equation*} \frac{\partial}{\partial t} \eta_{st}(z) = M(\eta_{st}(z)), \quad \text{$s\leq t$},
\end{equation*}
where $M(z)=-zp(z)$ for a holomorphic $p:\D\to \C$ with $p(0)=a$ and $\Re(p(z))>0$ for all $z\in\D$. (If $\Re(p(z))=0$ for some $z$, then $p$ is constant and $\mu$ is a delta measure, which is excluded now.) 
According to \cite[Theorem 2.3]{GHKK08} or \cite[Lemma 1]{becker}, the locally uniform limit $\lim_{t\to\infty}e^{at}\eta_{s,t}(z)=:f_s(z)$ 
exists and $(f_t)_t$ 
is an increasing Loewner chain with $f_t\circ \eta_{st}=f_s$, and $f_0(0)=0, f_0'(0)=1$. 
As $\eta_{st}=\eta_{0,t-s}$, we have 
\[f_0 = \lim_{t\to\infty}e^{at-as}\eta_{0,t-s}(z)=
e^{-as}\lim_{t\to\infty}e^{at}\eta_{s,t}(z)=e^{-as} f_s.\]
Hence, the Loewner chain has the simple form $(e^{at}f_0)_{t\geq0}$, which implies that $f_0$ is a $\beta$-spirallike mapping.
We conclude that $f_0= f_1\circ \eta_{01}=(e^{a}f_0)\circ \eta_{01}$ and thus $\eta_{01}=\eta_\mu = f_0^{-1}\circ(e^{-Re^{i\beta}}f_0)$.\\

Conversely, if $\eta_\mu$ has the form $\eta_\mu = f^{-1}\circ (e^{-Re^{i\beta}}\cdot f)$, then 
$(f_t)_{t\geq0}:=(e^{at}f)_{t\geq 0}$ with $a=Re^{i\beta}$ is an increasing Loewner chain 
and $\eta_{0t}=f_t^{-1}\circ f_0$ defines a semigroup of analytic mappings of $\disk$ with $\eta_{0t}(0)=0$ and $\eta_{01}=\eta_\mu$. Lemma \ref{characterizationeta} and Theorem \ref{UMIDLK} imply that 
$\mu \in \ID(\submm,\T)\setminus\{\haar\}$.
\end{proof}

\begin{corollary}
The set $\ID(\submm,\tor)$ is weakly closed.
\end{corollary}
\begin{proof}The class of all univalent functions $f:\D\to\C$ with $f(0)=0$ and $f'(0)=1$ is compact with respect to locally uniform convergence, see \cite[Theorem 1.7]{P75}. By using \eqref{spirallike_ch}, we see that the set of all functions of the form $f^{-1}\circ (e^{-Re^{i\beta}}\cdot f)$ from Corollary \ref{starlike_id} 
together with the function $0$ forms a compact set. The result now follows by using Lemmas \ref{characterizationeta} 
and \ref{lemmaconvergenceunit}.
\end{proof}

\begin{example}\label{UMBM} The distribution $\{\mu_t\}_{t>0}$ of a \index{Brownian motion!monotone}unitary monotone Brownian motion was introduced and studied by Hamdi \cite{Ham15}. Its $\submm$-infinitely divisible distribution is characterized by $\alpha=0$ and $\rho=\frac{1}{2}\delta_1$ in \eqref{VectorFUMMI}. The moment generating function is
\begin{equation*}
\psi_{\mu_t}(z) = -\frac{1}{2} + \frac{1+z}{2\sqrt{z^2-2(2e^{-t/2}-1)z +1}}.
\end{equation*}
It is absolutely continuous with respect to $\haar$ and the density is
\begin{equation*}
\frac{{\rm d}\mu_t}{{\rm d} \haar}(e^{i\theta})=\frac{\cos(\theta/2)}{\sqrt{\cos^2(\theta/2) - e^{-t/2}}} \mathbf{1}_{(-2\arccos( e^{-t/4}), ~ 2\arccos(e^{-t/4}) )}(\theta),  \qquad -\pi < \theta < \pi,
\end{equation*}
where $\arccos$ is a strictly decreasing function from $(-1,1)$ onto $(0,\pi)$. The density diverges at the edges of the support, and hence is not unimodal. It can be observed that $\mu_t \wto \haar$ as $t\to\infty$.
\end{example}

\index{infinite divisibility!multiplicative monotone convolution|)}

\subsection{Freely infinitely divisible distributions on $\tor$} \label{FIDunit}

\index{infinite divisibility!multiplicative free convolution|(}

Bercovici and Voiculescu investigated the unit circle case, see \cite{BV92}. We exclude the measures with vanishing mean to define the $\Sigma$-transform, and correspondingly we denote by $\ID_\times(\boxtimes,\tor)$ the set of $\boxtimes$-infinitely divisible distributions with non-zero mean. It is known, see  \cite[Lemma 6.1]{BV92}, that
\begin{equation*}
\ID_\times(\boxtimes,\tor)= \ID(\boxtimes,\tor)\setminus\{\haar\}.
\end{equation*}
The class $\ID_\times(\boxtimes,\tor)$ is characterized as follows.
\begin{theorem}\label{MFID}
Let $\mu \in \cP_\times(\tor)$. The following three statements are equivalent.
\begin{enumerate}[\rm(1)]
\item $\mu\in \ID(\boxtimes,\tor)$.
\item\label{CSUMFI} There exists a weakly continuous $\boxtimes$-convolution semigroup $\{\mu_t \}_{t\geq 0} \subset\cP(\tor)$ such that $\mu_0 = \delta_1$ and $\mu_1 = \mu$.
\item\label{UMFI} There exists an analytic map $u\colon\disk \to \bH \cup i \R$ such that $\Sigma_\mu(z) = \exp(u(z))$.
\end{enumerate}
\end{theorem}
The analytic map $u$ in \eqref{UMFI} above can be characterized by the Herglotz representation
\begin{equation}\label{VectorFUMFI}
u(z) = -i \alpha +\int_{\tor} \frac{1+ z \zeta}{1-z\zeta}\rho({\rm d} \zeta),
\end{equation}
where $\alpha \in \R$ and $\rho$ is a finite, non-negative measure on $\tor$. 

Note that $\alpha$ is not unique due to the transformation $\alpha\mapsto \alpha +2\pi n$ for $n\in\Z$. Also the convolution semigroup $\{\mu_t \}_{t\geq 0}$ in \eqref{CSUMFI} is not unique either up to the transformation $\{\mu_t\}_{t\geq0} \mapsto \{\DD_{e^{2\pi n t i}}(\mu_t)\}_{t\geq0}$ for $n\in\Z$. However, a canonical bijection between $\{\mu_t\}_{t\geq0}$ and $u$ can be given by $\Sigma_{\mu_t}(z) = \exp(t u(z))$.
A non-canonical choice appears only when we go from $\mu \in \ID_\times(\boxtimes, \tor)$ to $\{\mu_t\}_{t\geq0}$ or from $\mu \in \ID_\times(\boxtimes, \tor)$ to $u$.

Proposition \ref{FID} has the following multiplicative analogue.
\begin{proposition}
$\ID(\boxtimes,\tor) \subset \Univ(\tor) \cup \{\haar\}$.
\end{proposition}
\begin{proof}
Suppose that $\mu \in \ID_\times(\boxtimes,\tor)$. Theorem \ref{MFID} shows that $f(z):= z \Sigma_\mu(z)$ extends to an analytic map defined in $\disk$. By the definition of $\Sigma_\mu$, the identity  $f(\eta_\mu(z))=z$ holds in neighborhood of $0$, and so in $\disk$ by the identity theorem. Hence $\eta_\mu$ is univalent in $\disk$.
\end{proof}
Note that $\ID(\boxtimes,\tor)$ is weakly closed since it coincides with $\IA(\boxtimes,\tor)$.

\begin{example}
The distributions $\{\mu_t\}_{t\geq0}$ of \index{Brownian motion!free}unitary free Brownian motion, introduced by Biane \cite{Bia97a}, is characterized by $\Sigma_{\mu_t}(z)= \exp(\frac{t(1+z)}{2(1-z)})$. Biane proved that it is Lebesgue absolutely continuous at any $t>0$ \cite[Proposition 10]{Bia97b}, and Zhong proved that the density is unimodal \cite[Theorem 5.4]{Zho14}. 
\end{example}

\index{infinite divisibility!multiplicative free convolution|)}

\subsection{Unimodal distributions on $\tor$} \label{subsec_unimodal_mult}
We investigate unimodal distributions on $\T$.

\index{unimodal distribution!on $\tor$|(}

\begin{definition} \label{defunit}
Let $\alpha , \beta \in \R$ such that $0\leq \beta-\alpha \leq 2\pi$. A measure $\mu$ on $\tor$ is said to be \emph{unimodal} with antimode $e^{i\alpha}$ and mode $e^{i\beta}$
if there exist $\lambda\in[0,\infty)$ and a function $f\colon (\alpha, \alpha + 2\pi)\to[0,\infty)$, non-decreasing on $(\alpha,\beta)$ and non-increasing on $(\beta,\alpha+2\pi)$ such that
$$
\mu({\rm d}\theta) = f(\theta)\,{\rm d}\theta + \lambda\delta_{\beta}.
$$
If $\alpha=\beta$ and $\alpha+2\pi = \beta$, we understand that $f$ is non-increasing on $(\alpha,\alpha+2\pi)$ and $f$ is non-decreasing on $(\alpha,\alpha+2\pi)$, respectively. A measure $\mu$ on $\tor$ is said to be unimodal if it is unimodal with some antimode and mode.
\end{definition}
The set of unimodal distributions on $\tor$ is denoted by $\UM(\tor)$. Similarly to the case of $\R$, the set of unimodal distributions is closed with respect to weak convergence (this fact can also be deduced from Lemma \ref{lemmaconvergenceunit} and Theorem \ref{unimodalunit}).

\begin{proposition}\label{vertical}
$\UM(\tor) \subset \Univ(\tor)\cup\{\haar\}$.
\end{proposition}
\begin{proof}
If $\mu$ is unimodal with antimode $1$ and mode $-1$ without an atom, then we can directly use \cite[Theorem 3]{K52} (see also \cite[Theorem 40]{AA75}). If $\mu$ has an atom, then we can resort to approximation. For a general mode and an antimode, if they are different points then we can apply a suitable Moebius transformation (see the map $T$ in the proof of Lemma \ref{unimodal-thm} below).  Finally, if the mode and antimode coincide, then we can easily find  approximating unimodal distributions whose mode and antimode are different.
\end{proof}

An analogue of Theorem \ref{anotherkhintchine} can be formulated as follows.

\begin{theorem}\label{unimodalunit}
Let $\mu$ be a probability measure on $\tor$ that is not a Haar measure. Let $\alpha, \beta \in \R$ such that $0\leq \beta-\alpha \leq 2\pi$. The following are equivalent.
\begin{enumerate}[\rm(1)]
\item\label{unimodalunitfirst} $\mu$ is unimodal on $\tor$ with antimode $e^{i\alpha}$ and mode $e^{i\beta}$.
\item\label{unimodalunitsecond} $\Re\left(e^{\frac{i}{2}(\alpha+\beta-\pi)} (z-e^{-i\alpha})(z-e^{-i\beta})\psi'_\mu(z) \right) \geq0$ in $\disk$.
\end{enumerate}
\end{theorem}

\begin{remark}
A geometric characterization of unimodality exists and strengthens Proposition \ref{vertical}. If $e^{i\alpha} \neq e^{i\beta}$, then the above conditions \eqref{unimodalunitfirst} and \eqref{unimodalunitsecond} are also equivalent to the following four geometric conditions: 
\begin{enumerate}[\rm(a)]
\item \label{eq:unimodal_a} $\psi_\mu$ is univalent;
\item\label{eq:unimodal_b} The domain $\psi_\mu(\disk)$ is \index{moment generating function!vertically convex}vertically convex, namely for any $z_1,z_2\in \psi_\mu(\disk)$ having the same real part and for any $t\in(0,1)$, the point $(1-t)z_1 +tz_2$ also belongs to $\psi_\mu(\disk)$;
\item\label{eq:unimodal_c} There exist points $z_n$ in $\disk$ converging to $e^{-i\beta}$ such that
$$
\lim_{n\to\infty}\Re (\psi_\mu(z_n))=\sup_{z\in\disk} \Re(\psi_\mu(z));
$$
\item\label{eq:unimodal_d} There exist points $z'_n$ in $\disk$ converging to $e^{-i\alpha}$ such that
$$
\lim_{n\to\infty}\Re (\psi_\mu(z_n'))=\inf_{z\in\disk} \Re(\psi_\mu(z)).
$$
\end{enumerate}
The proof can be obtained from \cite[Theorem 1]{HS70} (or the proof of \cite{RZ76}) with a suitable Moebius transformation (see the transformation $T$ in Lemma \ref{unimodal-thm} below). 

If $e^{i\alpha} = e^{i\beta}$, then we can use results in \cite[Section 6]{HS70} with a simple rotation $z \mapsto \gamma z$ for some $\gamma \in \T$. 
More specifically, if $\alpha = \beta$, then the equivalence still holds with the above \eqref{eq:unimodal_b} replaced by
\begin{enumerate}[\rm(a')]
\setcounter{enumi}{1}
\item\label{eq:unimodal_b2} $z + i y \in \psi_\mu(\disk)$ for any $z \in \psi_\mu(\disk)$ and $y \geq 0$,   
\end{enumerate}
and if $\alpha +2\pi= \beta$, then the equivalence holds with \eqref{eq:unimodal_b} replaced by
\begin{enumerate}[\rm(a'')]
\setcounter{enumi}{1}
\item\label{eq:unimodal_b3} $z + i y \in \psi_\mu(\disk)$ for any $z \in \psi_\mu(\disk)$ and $y \leq 0$.  
\end{enumerate}
\end{remark}

\begin{example}
The range domains that satisfy both (\ref{eq:unimodal_b2}') and (\ref{eq:unimodal_b3}'') are only the vertical stripes and the half-planes. These domains are realized respectively by the \index{uniform distribution}uniform distributions on arcs of $\T$ and the Dirac delta measures. For example the uniform distribution $\mu$ on the upper semicircle $\{z \in \T: \arg z \in[0,\pi]\}$ has the moment generating function 
$$
\psi_\mu(z) = \frac{i}{\pi}\log \frac{1+z}{1-z}, 
$$
which is a bijection of the unit disk onto the vertical stripe $\{w \in \C: -\frac{1}{2}< \Re(w) < \frac{1}{2} \}$. This distribution is unimodal with any antimode in the lower semicircle and  any mode in the upper semicircle. The special choice $(\text{antimode},\text{mode})=(1,1)$ corresponds to the case $\alpha=\beta=0$, and the choice $(\text{antimode},\text{mode})=(-1,-1)$ corresponds to $\alpha +2\pi=\beta$, where $\alpha=\pi$. 
\end{example}

The proof of Theorem \ref{unimodalunit} follows from the lemmas below, which strengthen \cite[Theorem 3]{K52}. Note that the number $e^{\frac{i}{2}(\alpha+\beta-\pi)}$ in Theorem \ref{unimodalunit} is different from $e^{\frac{i}{2}(\pi-\alpha-\beta)}$ in Lemma \ref{unimodal-thm} since the integral kernel is $(e^{-i\theta} +z)/(e^{-i\theta}-z)$ in the moment generating function \eqref{eq:psi2}, while the kernel $(e^{i\theta} +z)/(e^{i\theta}-z)$ is used in the lemmas below.

\begin{lemma}
\label{Converse-of-Kaplan}
Let $f$ be holomorphic on $\D$. Then $f$ satisfies
\begin{equation}
\label{Def-vertical-direction}
\inf_{z\in\D}\Re [f(z)]>-\infty \text{~~~and~~~} \Re (z^{2}-1)f'(z) \geq 0 \text{~on $\D$}
\end{equation}
if and only if there exist $c\geq0, d \in \R$ and a function $k\colon (-\pi,\pi) \to \R$ in $L^1({\rm d}\theta)$, bounded below, non-increasing on $(-\pi,0)$ and non-decreasing on $(0,\pi)$, such that $f$ is written as
\begin{equation}
\label{integral-rep}
f(z) = \int_{-\pi}^{\pi} \frac{e^{i\theta} +z}{e^{i\theta} -z} k(\theta)  \,{\rm d}\theta + c\cdot \frac{1-z}{1+z}+i d\hspace{15pt}(z \in \D).
\end{equation}
\end{lemma}
\begin{remark}
There might exist a condition weaker than $\inf_{z\in \D}\Re(f(z)) >-\infty$ such that the condition of $k$ being bounded below can be removed, but we do not pursue this direction.
\end{remark}

\begin{proof}
By the Herglotz formula (Lemma \ref{inversionunit}), there exists a non-negative finite measure $\rho$ on $\T$ such that
\begin{equation*}
\label{herglotz0}
(z^{2}-1)f'(z) = \int_{(-\pi,0)\cup(0,\pi)} \frac{e^{i\theta} +z}{e^{i\theta} -z} \, \rho({\rm d}\theta)
		- i \Im f'(0)
		+ \rho(\{0\})\frac{1+z}{1-z}
		+ \rho(\{\pi\})\frac{1-z}{1+z}.
\end{equation*}
Then we have
\begin{align}
f(z) &= i \int_{(-\pi,0)\cup(0,\pi)} \left(\frac{\cos\theta-1}{2} \log(1+z) -\frac{\cos\theta+1}{2}\log(1-z) +\log(1-ze^{-i\theta} )\right) \frac{\rho({\rm d}\theta)}{\sin\theta} \notag \\
&\qquad  - \frac{i\Im f'(0)}2 \log \frac{1-z}{1+z}-\frac{\rho(\{0\})z}{1 - z}- \frac{\rho(\{\pi\})z}{1+z} +f(0).\label{eq:f1}
\end{align}
To prove some integrability of the measure $\rho$, we look at
\begin{equation*}
-\Re f(r) + \Re f(0) = \int_{(-\pi,0)\cup(0,\pi)}\arg(1 -re^{-i\theta}) \frac{\rho({\rm d}\theta)}{\sin \theta} +\frac{\rho(\{0\})r}{1 - r} + \frac{\rho(\{\pi\})r}{1+r}
\end{equation*}
for $r \in (0,1)$. Since $\arg(1-e^{-i\theta}) = (\pi - |\theta|)\text{sign}(\theta)/2$ and $\Re f \geq m$ for some $m\in\R$, we conclude by Fatou's lemma that
\begin{align*}
-m + \Re f(0)
&	\ge \varliminf_{r \uparrow 1} \left(
		 \int_{(-\pi,0)\cup(0,\pi)}\arg(1 -re^{-i\theta}) \frac{\rho({\rm d}\theta)}{\sin \theta}
			+\frac{\rho(\{0\})r}{1 - r}
				+\frac{\rho(\{\pi\})r}{1+r}\right)\\
&	\ge \int_{(-\pi,0)\cup(0,\pi)}\varliminf_{r \uparrow 1}
		\arg(1 -re^{-i\theta}) \frac{\rho({\rm d}\theta)}{\sin \theta}
			+\varliminf_{r \uparrow 1}\frac{\rho(\{0\})r}{1 - r}
				+\frac{\rho(\{\pi\})}{2} \\
&	= \int_{(-\pi,0)\cup(0,\pi)}(\pi-|\theta|)\frac{\rho({\rm d}\theta)}{2|\sin \theta|}
			+\varliminf_{r \uparrow 1}\frac{\rho(\{0\})r}{1 - r}
				+\frac{\rho(\{\pi\})}{2},
\end{align*}
which proves that $\rho(\{0\}) =0$ and ${\rm d}\rho(\theta)/ |\sin \theta|$ is a finite measure on any compact subset of $(-\pi,\pi)$.

Now define the non-negative function
\begin{equation*}
h(\theta) := \int_{0}^{\theta}\frac{\rho({\rm d}\phi)}{2\sin \phi}, \qquad \theta \in (-\pi,\pi),
\end{equation*}
which is non-increasing on $(-\pi,0)$ and non-decreasing on $(0,\pi)$. Note that $h$ belongs to $L^1((-\pi,\pi),{\rm d}\theta)$.
Performing integration by parts for \eqref{eq:f1} implies that
\begin{align*}
f(z) &= i\int_{-\pi}^\pi \left(\sin\theta \cdot\log\frac{1+z}{1-z}- \frac{2z i}{e^{i\theta}-z} \right)h(\theta)\,{\rm d}\theta- \frac{i\Im f'(0)}2 \log \frac{1-z}{1+z}- \frac{\rho(\{\pi\})z}{1+z} +f(0)\\
&= \int_{(-\pi,\pi)} \frac{e^{i\theta} +z}{e^{i\theta} -z}  h(\theta)\,{\rm d}\theta +\frac{\rho(\{\pi\})}{2}\frac{-1+z}{-1-z} + i a \log \frac{1-z}{1+z} + C
\end{align*}
where $a$ is a real number and $C$ is a complex number. According to the identities
\begin{equation}\label{eq:integral}
 \int_0^{\pi}\frac{e^{i\theta} +z}{e^{i\theta} -z} \,{\rm d}\theta =2 i \log \frac{1-z}{1+z} + \pi, \qquad  \int_{-\pi}^{\pi}\frac{e^{i\theta} +z}{e^{i\theta} -z} \,{\rm d}\theta = 2\pi,
\end{equation}
 the function
\begin{equation*}
k(\theta) =
h(\theta) +\frac{\Re(C)}{2\pi} + \frac{1}{4\pi}+\frac{a}{4}\text{sign}(\theta)
\end{equation*}
and the numbers $c=\rho(\{\pi\})/2$ and $d= \Im(C)$ give the desired formula \eqref{integral-rep}.

The converse statement can be obtained by mostly tracing the above arguments backwards. When proving the boundedness of $\Re f$ from below, one should use
$$
\Re \frac{e^{i\theta}+z}{e^{i\theta}-z} \geq0
$$
and \eqref{eq:integral}.
\end{proof}

\begin{lemma}
\label{unimodal-thm}
Let $f$ be holomorphic on $\D$ and $\alpha,\beta \in \R$ such that $0<\beta-\alpha <2\pi $. Then $f$ satisfies
\begin{equation}
\label{Def-vertical-direction02}
\inf_{z\in \D} \Re[f(z)] > -\infty \text{~~~and~~~}\Re \left[e^{i\frac{\pi -\alpha -\beta }{2}}(z-e^{i\alpha})(z-e^{i\beta})f'(z)\right] \geq 0 \text{~on $\D$}
\end{equation}
if and only if there exist $c\geq0,d \in \R$ and a function $k\colon (\alpha,\alpha + 2\pi) \to \R$ in $L^1({\rm d}\theta)$, bounded below, non-decreasing on $(\alpha, \beta)$ and non-increasing on $(\beta, \alpha + 2\pi)$, such that $f$ is written as
\begin{equation}
\label{integral-rep-02}
f(z) = \int_{-\pi}^{\pi} \frac{e^{i\theta} +z}{e^{i\theta} -z} k(\theta)  \,{\rm d}\theta + c \cdot\frac{e^{i\beta} +z}{e^{i\beta} -z } + i d \hspace{15pt}(z \in \D).
\end{equation}
\end{lemma}

\begin{proof}
Let $w= T(z):= \gamma \frac{z+i a}{1-i a z}$, where $\gamma = e^{i \frac{\alpha + \beta -\pi}{2}}$ and $a \in (-1,1)$ is uniquely determined by
$(a+i) / (a-i) = e^{i\frac{\beta-\alpha+\pi}{2}}$. Then $T$ defines a homeomorphism of $\overline{\D}$ and it is an analytic bijection of $\D$. By direct computation, a branch of $\theta \mapsto \arg T(e^{i\theta})$ is strictly increasing on $[0,2\pi)$, and $T(1) = e^{i\alpha}$ and $T(-1) = e^{i\beta}$. Defining $\tilde{f}(z) = f(T(z)) = f(w)$ and by direct computation one has
$$
(z^2-1) \tilde f'(z) = \frac{1+a^2}{1-a^2} \frac{1}{\gamma} (w- e^{i\alpha})(w-e^{i\beta}) f'(w).
$$
Thus we apply Lemma \ref{Converse-of-Kaplan} to $\tilde f$ and obtain an integral representation
$$
\tilde f(z) = \int_{-\pi}^{\pi} \frac{e^{i\theta} +z}{e^{i\theta} -z} \tilde k(\theta)  \,{\rm d}\theta + \tilde c\cdot \frac{-1+ z}{-1-z} +i \tilde d .
$$
Going back to $f(w)$, we apply the change of variable $e^{i\theta} = T^{-1}(e^{i\varphi})$ to get
$$
f(w) = \int_{-\pi}^{\pi} \frac{T^{-1}(e^{i\varphi}) + z}{T^{-1}(e^{i\varphi}) - z} \tilde k(\arg T^{-1}(e^{i\varphi})) \frac{(T^{-1})'(e^{i\varphi})}{T^{-1}(e^{i\varphi})} e^{i\varphi}  \,{\rm d}\varphi + \tilde c \cdot \frac{-1+z}{-1-z}+i\tilde d.
$$
Defining $b=i a \overline{\gamma}$ for simplicity, the following identity holds:
$$
\frac{T^{-1}(e^{i\varphi}) + z}{T^{-1}(e^{i\varphi}) - z}\cdot\frac{(T^{-1})'(e^{i\varphi})}{T^{-1}(e^{i\varphi})} e^{i\varphi} = \frac{e^{i\varphi} +w}{e^{i\varphi}-w} - \frac{2i \Im(b e^{i\varphi})}{|1+b e^{i\varphi}|^2}.
$$
Hence
$$
f(w) = \int_{-\pi}^{\pi} \frac{e^{i\varphi} + w}{e^{i\varphi} - w} \tilde k(\arg T^{-1}(e^{i\varphi})) \,{\rm d}\varphi + c \cdot \frac{e^{i\beta}+w}{e^{i\beta}-w} + C,
$$
where $c = \tilde c (1-a^2)/(1+a^2)$ and $C \in \C$ is a constant. We can incorporate the real part of $C$ into $\tilde k(\arg T^{-1}(e^{i\varphi}))$ using the formula \eqref{eq:integral}. The function $\varphi\mapsto \tilde k(\arg T^{-1}(e^{i\varphi}))$ satisfies the desired monotonicity. It is also integrable since $\tilde k$ is integrable and the determinant regarding the change of variable $e^{i\varphi}= T(e^{i\theta})$ is continuous and strictly positive on the circle.

Lastly, the converse statement can be obtained by tracing the above arguments backwards.
\end{proof}

In Lemma \ref{unimodal-thm} we have assumed $\alpha \neq \beta \mod 2\pi$. It is tempting to prove the statement for $\alpha = \beta \mod 2\pi$ by some approximation arguments, but it seems not very easy. Instead we give a proof following similar lines of the proof of Lemma \ref{unimodal-thm}.

\begin{lemma}
\label{convex-in-one-drct}
Let $f$ be holomorphic on $\D$ and $\alpha\in\R$.
Then $f$ satisfies
\begin{equation}
\label{convex-in-one-drct01}
\inf_{z\in \D} \Re[f(z)] > -\infty \text{~~~and~~~} \Re \left[ e^{i(\frac{\pi}{2} - \alpha)}(z-e^{i\alpha})^{2}f'(z)\right] \geq 0 \text{~on  $\D$}
\end{equation}
if and only if there exist $c \geq0, d\in \R$ and a function $k \colon (\alpha,\alpha + 2\pi) \to \R$ in $L^1({\rm d}\theta)$, non-increasing and bounded below, such that $f$ is written as
\begin{equation}
\label{convex-in-one-drct02}
f(z) = \int_{\alpha}^{\alpha+ 2\pi} \frac{e^{i\theta} +z}{e^{i\theta} -z} k(\theta)  \,{\rm d}\theta + c\cdot\frac{e^{i\alpha}+z}{e^{i\alpha}-z} + i d
\hspace{15pt}(z \in \D).
\end{equation}
\end{lemma}

\begin{proof}
Let $w = e^{i(\alpha+\pi)}z$ and $\tilde f(z) := f(e^{i(\alpha+\pi)}z) = f(w)$. Then
\begin{equation}
\label{convex-in-one-drct03}
e^{i(\frac{\pi}{2} - \alpha)}(w-e^{i\alpha})^{2}f'(w) = -i (z+1)^2 \tilde f'(z)
\end{equation}
for all $z \in \D$. Hence the Herglotz formula
\begin{equation}
\label{herglotz}
-i(z+1)^{2}\tilde f'(z) = \int_{(-\pi,\pi)} \frac{e^{i\theta} +z}{e^{i\theta} -z} \, \rho({\rm d}\theta)
		+  \rho(\{\pi\}) \frac{1-z}{1+z} + i a
\end{equation}
exists for some $a \in \R$ and a finite non-negative measure $\rho$ on $\T$.
Then we have
\begin{equation}\label{eq:tilde}
\begin{split}
-i\tilde f(z) &= \int_{(-\pi,\pi)}\left(\log(1+z) - \log(1-ze^{-i\theta}) +i \sin \theta\frac{z}{1+z}  \right)\frac{ \rho({\rm d}\theta)}{1 + \cos \theta} \\
&\qquad+\frac{ \rho{(\{\pi\})} z}{(1+z)^{2}} - \frac{i a}{1+z} + C
\end{split}
\end{equation}
for some $C\in \C$. Notice that applying Taylor's theorem around $\theta=\pi$ shows that the integral converges.

In order to derive an integrability of $\rho({\rm d}\theta)/(1+\cos \theta)$, we take the curve $z= i y /(1-i y) = (-y^2 +iy)/(y^2+1), y \in \R,$ which is contained in $\D$ and tends to $-1$ as $y\to\infty$. On this curve one can see that
\begin{equation}
\begin{split}
-\Re \tilde f(z) &= \int_{(-\pi,\pi)}\left(\arg (1+i y) - \arg\left(1+ i \frac{-y \cos \theta - y^2 \sin \theta}{1+y^2(1+\cos \theta) -y\sin\theta}\right) \right)\frac{ \rho({\rm d}\theta)}{1 + \cos \theta}  \\
&\qquad + \rho(\{\pi\})y  - a + \Im(C).
\end{split}
\end{equation}
Notice that $1+y^2(1+\cos \theta) -y\sin\theta = (1+y^2) \Re (1-ze^{-i\theta}) >0$. Then one can prove that the integrand is non-negative for $y>0$. Hence we can apply Fatou's lemma as $y\to\infty$ using  the assumption of $\Re \tilde f(z)$ being bounded below to conclude that
$$
\rho (\{\pi\})=0, \qquad \infty > \int_{(-\pi,\pi)} \left(\frac{\pi}{2} + \frac{\theta}{2} \right) \frac{ \rho({\rm d}\theta)}{1 + \cos \theta}.
$$
This shows that one can integrate out the term $i \sin \theta z/(1+z)$ from \eqref{eq:tilde}, and the function
\begin{equation}
h(\theta) :=\frac12\int_{\theta}^{\pi} \frac{ \rho({\rm d}\phi)}{1+\cos\phi}
\end{equation}
is finite, non-increasing on $(-\pi,\pi)$ and $h \in L^1((-\pi,\pi),{\rm d}\theta)$.
Then one obtains, for some $b \in \R$ and $D \in \C$,
\begin{align*}
-i\tilde f(z)
&= -2\int_{(-\pi,\pi)}\left(\log(1+z) - \log(1-ze^{-i\theta}) \right){\rm d}h(\theta)  +\frac{i b}{1+z} + D \\
&=    \int_{(-\pi,\pi)}\frac{-2i z}{e^{i\theta}-z}h(\theta)\,{\rm d}\theta  +\frac{i b}{1+z} + D
\end{align*}
and hence
\begin{align*}
\tilde f(z)
&=\int_{-\pi}^\pi\frac{e^{i\theta}+z}{e^{i\theta}-z}h(\theta)\,{\rm d}\theta -\frac{b}{1+z} + i D - \int_{-\pi}^\pi h(\theta) \,{\rm d}\theta \\
& =\int_{(-\pi,\pi)}\frac{e^{i\theta}+z}{e^{i\theta}-z}\tilde k(\theta)\,{\rm d}\theta +\frac{\tilde b(-1+z)}{-1-z} + i \tilde a,
\end{align*}
where $\tilde k(\theta)=h(\theta) + d$ for some $d, \tilde a, \tilde b \in\R$.  The assumption $\inf_{z \in \D}\Re \tilde f (z) >-\infty$ and the dominated convergence theorem imply $2\tilde b=\lim_{r\downarrow-1}\Re(1+r)\tilde f (r) \geq 0$.
The desired statement for $f$ can be obtained by the rotation $z \mapsto e^{-i(\pi+ \alpha)}z$.

The converse statement can be obtained by tracing the above arguments backwards.
\end{proof}

\index{unimodal distribution!on $\tor$|)}

\subsection{Starlike moment generating functions}

\index{moment generating function!starlike|(}

We find a class of probability measures on $\tor$ analogous to $\SD(\rhd)$, namely those which may be characterized by the starlikeness of $\psi$-transform.

\begin{definition}
Let $\mu \in \cP(\tor)$. We say that $\psi_\mu$ is \emph{starlike} if $\psi_\mu$ is univalent and $\psi_\mu(\disk)$ is star-shaped with respect to $0$, i.e.\ $c\psi_\mu(\disk) \subset \psi_\mu(\disk)$ for every $c\in(0,1)$. The set of probability measures with starlike moment generating functions is denoted by $\Star(\tor)$.
\end{definition}

By contrast to probability measures on the real line, there is no natural concept of ``dilation'' on the unit circle. In order to find a probabilistic characterization of $\Star(\tor)$, we propose an alternative operation. Observe first that for all $\mu \in\cP(\tor)$ the identity
\begin{equation*}
\psi_{(1-c)\haar + c \mu} = c \psi_\mu
\end{equation*}
holds.
\begin{definition}
A probability measure $\mu$ on $\tor$ is said to be of type H if $\mu \in \cP_\times(\tor)$ and for every $c\in(0,1)$ there exists $\bmu^c \in \cP(\tor)$ such that $(1-c)\haar + c \mu = \mu \submm \bmu^c$.
The set of such probability measures is denoted by $\cH(\submm)$.
\end{definition}

\begin{remark}
We need to assume that $\mu \in \cP_\times(\tor)$; otherwise Theorem \ref{starlikeunit} below does not hold. Indeed, take $\mu = \frac{1}{2}(\delta_1+\delta_{-1})\notin\cP_\times(\tor)$ and define $\eta_c(z)=\sqrt{c} z / \sqrt{1-(1-c)z^2}$ for $c\in(0,1)$. We can check that $\eta_c(0)=0$ and $\eta_c(\disk) \subset \disk$, and so by Lemma \ref{characterizationeta} there exists a probability measure $\bmu^c \in \cP(\tor)$ such that $\eta_c=\eta_{\bmu^c}$. We can easily check that $c \psi_\mu =  \psi_\mu\circ \eta_c$ for all $c \in (0,1)$, but $\psi_\mu(z) = z^2/(1-z^2)$ is not univalent.
\end{remark}

For $\mu \in \cH(\submm)$, the condition $\mu \in \cP_\times(\tor)$ implies that the inverse series $\psi_\mu^{-1}(z)= (1/m_1(\mu))z + \cdots$ exists and converges in a neighborhood of $0$ and hence $\bmu^c$ is uniquely determined by the formula
\begin{equation}\label{TH}
\eta_{\bmu^c}(z) = \psi_\mu^{-1}(c \psi_\mu(z))
\end{equation}
around the origin. We may also define $\bmu^0=\haar$ and $\bmu^1=\delta_1$. Then \eqref{TH} and Lemma \ref{lemmaconvergenceunit} imply that $[0,1]\mapsto \cP(\tor), c \mapsto \bmu^c$ is weakly continuous. Note that $\haar \notin \cH(\submm) \cup \Star(\tor)$.

\begin{theorem}\label{starlikeunit}
$\Star(\tor)=\cH(\submm)$.
\end{theorem}
\begin{proof}
In the proof we adopt the notation $H_c \mu:=(1-c)\haar + c \mu$.

$\Star(\tor)\subset \cH(\submm)$. Let $\mu \in \Star(\tor)$. Since $\psi_\mu$ is univalent, $\psi_\mu'(0) \neq 0$ and hence $\mu\in\cP_\times(\tor)$. From the inclusion $\psi_{H_c\mu}(\disk)= c \psi_\mu(\disk)\subset \psi_\mu(\disk)$ we may define the univalent map $\eta_c=\psi_\mu^{-1}\circ \psi_{H_c\mu}\colon \disk\to\disk$. It satisfies $\eta_c(0)=0$, and then applying Lemma \ref{characterizationeta} shows that $\eta_c=\eta_{\bmu^c}$ for some $\bmu^c \in \cP(\tor)$. This $\bmu^c$ satisfies $\mu \submm \bmu^c =H_c\mu$.

$\Star(\tor)\supset \cH(\submm)$.
Take $\mu \in \cH(\submm)$ and take the decomposition $H_c\mu = \mu \submm \bmu^c$. This relation obviously shows that $c \psi_\mu(\disk)\subset \psi_\mu(\disk)$ for all $c\in (0,1)$, and so it suffices to prove the univalence of $\psi_\mu$. In a neighborhood of $0$ we have \eqref{TH}, which implies that $\bmu^c \submm \bmu^d = \bmu^{c d}$. Introducing the reparametrization $\mu_t:= \bmu^{\exp(-t)}$ then shows that
\begin{equation}\label{star-semigroup}
\mu_s \submm \mu_t =\mu_{s+t}, \qquad s,t\geq0, \qquad \mu_0=\delta_1.
\end{equation}
Therefore, $\{\mu_t\}_{t\geq0}$ is a weakly continuous $\submm$-convolution semigroup, and by Corollary \ref{MID univ}, we have $\mu_t \in \Univ(\tor)$ for all $t\geq0$.

Take and fix $\epsilon\in(0,1)$ such that $\psi_\mu$ is univalent in $\epsilon \disk$. Take $0<r<1$. Since $\mu_t \wto \haar$ as $t\to\infty$, by Lemma \ref{lemmaconvergenceunit} there exists $t =t(\epsilon,r)>0$ such that $\sup_{z\in r\disk}|\eta_{\mu_t}(z)| < \epsilon$. Since $\eta_{\mu_t}$ is univalent on $r\disk$, the function $\psi_\mu = e^t \psi_\mu \circ \eta_{\mu_t}$ is also univalent on $r\disk$. Since $0<r<1$ was arbitrary, we conclude that $\psi_\mu$ is univalent on $\disk$.
\end{proof}

The starlike functions have an well known characterization; see \cite[Theorem 2.6]{P75}.

\begin{theorem}
Suppose that $\mu \in \cP_\times(\tor)$. The following are equivalent.
\begin{enumerate}[\rm(1)]
\item $\mu \in \cH(\submm)$.
\item $\Re\left(\frac{z \psi_\mu'(z)}{\psi_\mu(z)}\right) \geq0$.
\end{enumerate}
If the above conditions hold, then the infinitesimal generator $B$ for the convolution semigroup \eqref{star-semigroup} is given by $B(z)= -\frac{\psi_\mu(z)}{z \psi_\mu'(z)}$.
\end{theorem}
\begin{remark}
If the above conditions hold, then there exists $c\in\overline{\disk}\setminus\{0\}$ and a probability measure $\rho$ on $\tor$ such that
\begin{equation}\label{int:star}
\psi_\mu(z) = c z \exp\left(-2 \int_\tor \log (1- \xi z)\,\rho({\rm d}\xi) \right).
\end{equation}
However, the parameters $c$ and $\rho$ must satisfy a (seemingly complicated) additional condition in order to have $\Re[\text{RHS of \eqref{int:star}}] \geq -1/2$.
\end{remark}

\begin{example}
The distribution $\mu_t$ of \index{Brownian motion!monotone}unitary monotone Brownian motion (see Example \ref{UMBM}) has a univalent $\psi_{\mu_t}$ which has the starlike range
\begin{equation*}
\psi_{\mu_t}(\disk) = \{z \in \C: \Re(z) >-1/2\} \setminus [ ((1-e^{-t/2})^{-1/2}-1)/2,\infty).
\end{equation*}
\end{example}

\begin{example}
Take $\rho = w_1 \delta_{e^{-i \theta_1}} +w_2 \delta_{e^{-i\theta_2}}$, where $-\pi \leq \theta_1 < \theta_2 <\pi$ and $w_1+w_2=1, w_1,w_2>0$. Then \eqref{int:star} at $z=e^{i\theta}$ reads
\begin{equation*}
\psi(e^{i\theta}) =
\begin{cases}
\frac{c e^{i(w_1\theta_1+w_2 \theta_2) +i(w_1-w_2)\pi}}{2 (1-\cos(\theta-\theta_1))^{w_1} (1-\cos(\theta-\theta_2))^{w_2}},  & \theta_1 < \theta < \theta_2, \\[2mm]
\frac{-c e^{i(w_1\theta_1+w_2 \theta_2)}}{2 (1-\cos(\theta-\theta_1))^{w_1} (1-\cos(\theta-\theta_2))^{w_2}}, & \theta_2 < \theta < \theta_1 + 2\pi.
\end{cases}
\end{equation*}
The range is $\C \setminus (L_1 \cup L_2)$ where $L_k$ are half-lines of the form $\{r e^{i \alpha_k}: r \geq r_k\}$, and hence $\Re[\psi(z)] \geq -1/2$ can never be satisfied.
\end{example}

\begin{example}
Take $\rho = (1-t)\haar + t \delta_1, t\in(0,1)$. Then the RHS of \eqref{int:star} at $z=e^{i\theta}$ reads
\begin{equation*}
\psi(e^{i\theta}) = \frac{c e^{i(1-t)\theta + i t \pi}}{|2 \sin (\theta/2)|^{2 t}}, \qquad 0 < \theta < 2\pi.
\end{equation*}
The range of $\psi$ is contained in $\{z\in\C: \Re(z)> -1/2\}$ if $0<c \leq  2^{2t -1}$ and $0<t < 1/2$. In this case the probability measure $\mu$ such that $\psi_\mu=\psi$ is Haar absolutely continuous with density
\begin{equation*}
\frac{{\rm d}\mu}{{\rm d}\haar}(e^{i\theta})=\frac{c \cos ((1-t)\theta + t \pi) }{2^{2t-1} \sin^{2t }(\theta/2)} +1, \qquad 0 < \theta < 2\pi.
\end{equation*}
\end{example}

\index{moment generating function!starlike|)}

\subsection{Univalence and regularity of probability measures} The multiplicative version of Proposition \ref{Atm} holds true. 
\begin{proposition}\label{MAtm}
Let $\mu \in \Univ(\tor)$. Suppose that $\mu$ has an isolated atom at $\zeta \in \tor$. Then $\mu|_{\tor\setminus \{\zeta\}}$ is absolutely continuous with respect to the Haar measure $\haar$ and its density is $L^\infty$. 
\end{proposition}
\begin{proof}
The proof is similar to that of Proposition \ref{Atm}. The functions $\psi_\mu$ and $\eta_\mu$ play the roles of $G_\mu$ and $F_\mu$, respectively. For example the inequality \eqref{Der} can be replaced by 
$$
\frac{\mu(\{e^{i \varphi}: |\varphi-\theta| <h\})}{h} \leq \int_{|\varphi-\theta| <h} \frac{1-r^2}{1+r^2-2r\cos(\varphi-\theta)}\,\mu({\rm d}e^{i\varphi}) \leq \Re(2\psi_\mu(r e^{-i\theta})+1), 
$$
where $h = 1-r \in(0,1)$ and $\theta \in [0,2\pi)$. The remaining arguments are omitted here.  
\end{proof}
\begin{remark}
Anshelevich and Arizmendi proved a weaker version using the wrapping map \cite[Proposition 56]{AA}. We are not sure if the above result can also be proved with the wrapping map.
\end{remark}
\begin{corollary}
$\ID(\circledast)$ is not a subset of $\Univ(\tor)\cup\{\haar\}$.
\end{corollary}
\begin{proof}
Let $\mu = p \delta_{1}+ (1-p) \delta_{-1}, 1/2 \leq p<1$. It embeds into the $\circledast$-convolution semigroup
\begin{equation*}
\mu^{\circledast t} =  \frac{1+ (2p-1)^t}{2} \delta_1 +  \frac{1- (2p-1)^t}{2} \delta_{-1}, \qquad t\geq0
\end{equation*}
and hence $\mu$ is $\circledast$-infinitely divisible. Proposition \ref{MAtm} shows that $\mu$ is not in $\Univ(\tor)$.
\end{proof}

\subsection{Problems} We can ask several multiplicative versions of questions that appeared in the additive case.
\begin{enumerate}[\rm(1)]
\item Can we characterize the unimodal distributions in terms of a limit theorem?

\item We characterized unimodal distributions in terms of $\psi_\mu$, and characterized a class of probability measures which have starlike $\psi_\mu$. On the other hand many other classes of univalent functions on the unit disk are known in the literature, e.g.\ close-to-convex functions and spiral-like functions. Can we characterize probability measures whose moment generating function $\psi_\mu$ belong to those  classes?

\item Are the multiplicative versions of Problem \ref{m-analogue} and Conjecture \ref{m-conj} true? Cf.\ \cite{AW14}.

\end{enumerate}

\chapter*{Acknowledgements}
The authors thank Octavio Arizmendi, Ikkei Hotta, and Tadahiro Miyao for useful discussions and providing proofs of some results.  

\appendix



\newpage
\printindex

\end{document}